\documentclass[10pt,a4paper]{amsart}

\usepackage{fullpage}
\usepackage{amsmath}
\usepackage{amsfonts}
\usepackage{amssymb}
\usepackage{mathrsfs}
\usepackage{bbm}
\usepackage{tikz-cd}
\usepackage{url}
\usepackage[english]{babel}
\usepackage[latin1]{inputenc}
\usepackage{enumitem}
\usepackage{multicol}

\DeclareRobustCommand{\gobblefour}[4]{}

\usepackage{calc}

\makeatletter
\renewcommand{\tocsection}[3]{%
  \indentlabel{\@ifnotempty{#2}{\ignorespaces#1 \makebox[\widthof{00.}][l]{#2.}\quad}}#3}
\renewcommand{\tocsubsection}[3]{%
  \indentlabel{\@ifnotempty{#2}{\ignorespaces#1 \makebox[\widthof{00.0.}][l]{#2.}\quad}}#3}
\makeatother

\setlist[itemize]{leftmargin=*}
\setlist[enumerate]{leftmargin=*}

\newtheorem{thm}{Theorem}[section]
\newtheorem{defin}[thm]{Definition}
\newtheorem{rem}[thm]{Remark}

\newtheorem{prop}[thm]{Proposition}
\newtheorem{cor}[thm]{Corollary}

\newtheorem{lemma}[thm]{Lemma}

\newcommand{\1}{\mathbbm{1}}
\newcommand{\A}{\mathbb A}
\newcommand{\B}{\mathbb B}
\newcommand{\C}{\mathbb C}
\newcommand{\E}{\mathbb E}
\newcommand{\F}{\mathbb F}
\newcommand{\G}{\mathbb G}
\newcommand{\HH}{\mathbb H}
\newcommand{\I}{\mathbb I}
\newcommand{\J}{\mathbb J}

\newcommand{\N}{\mathbb N}
\newcommand{\Q}{\mathbb Q}
\newcommand{\R}{\mathbb R}
\newcommand{\T}{\mathbb T}
\newcommand{\Z}{\mathbb Z}

\newcommand{\ab}{\mathrm{ab}}
\newcommand{\Ad}{\mathrm{Ad}\,}
\newcommand{\ad}{\mathrm{ad}\,}

\newcommand{\Aut}{\mathrm{Aut}}
\newcommand{\aux}{\mathrm{aux}}
\newcommand{\bad}{\mathrm{bad}}
\newcommand{\car}{\mathrm{char}}
\newcommand{\card}{\mathrm{card}}

\newcommand{\cl}{\mathrm{cl}}

\newcommand{\cont}{\mathrm{cont}}
\newcommand{\cris}{\mathrm{cris}}

\newcommand{\Dim}{\mathrm{dim}}
\newcommand{\DGamma}{\mathrm{D}\Gamma}
\newcommand{\diag}{\mathrm{diag}\,}
\newcommand{\dR}{\mathrm{dR}}
\newcommand{\End}{\mathrm{End}}
\newcommand{\ev}{\mathrm{ev}}
\newcommand{\ext}{\mathrm{ext}}

\newcommand{\Frob}{\mathrm{Frob}}
\newcommand{\Gal}{\mathrm{Gal}}
\newcommand{\GL}{\mathrm{GL}}
\newcommand{\GSp}{\mathrm{GSp}}
\newcommand{\hol}{\mathrm{hol}}
\newcommand{\Hom}{\mathrm{Hom}}

\newcommand{\id}{\mathrm{id}}

\newcommand{\im}{\mathrm{Im}\,}

\newcommand{\irr}{\mathrm{irr}}

\newcommand{\lcm}{\mathrm{lcm}}
\newcommand{\Lie}{\mathrm{Lie}}
\newcommand{\loc}{\mathrm{loc}}

\newcommand{\Mat}{\mathrm{M}}

\newcommand{\Min}{\mathrm{min}}

\newcommand{\ncr}{\mathrm{ncr}}
\newcommand{\norm}{\mathrm{norm}}

\newcommand{\ord}{\mathrm{ord}}

\newcommand{\PGamma}{\mathrm{P}\Gamma}

\newcommand{\PGSp}{\mathrm{PGSp}}
\newcommand{\PSp}{\mathrm{PSp}}

\newcommand{\rig}{\mathrm{rig}}

\newcommand{\SL}{\mathrm{SL}}
\newcommand{\Sp}{\mathrm{Sp}}
\newcommand{\Spec}{\mathrm{Spec}\,}
\newcommand{\Spf}{\mathrm{Spf}\,}
\newcommand{\Spm}{\mathrm{Spm}\,}
\newcommand{\slo}{\mathrm{sl}}
\newcommand{\st}{\mathrm{st}}
\newcommand{\Std}{\mathrm{Std}}
\newcommand{\Sym}{\mathrm{Sym}}
\newcommand{\Tr}{{\mathrm{Tr}}}

\newcommand{\univ}{\mathrm{univ}}
\newcommand{\ur}{\mathrm{ur}}

\newcommand{\Zar}{\mathrm{Zar}}

\newcommand{\xto}{\xrightarrow}
\newcommand{\into}{\hookrightarrow}
\newcommand{\onto}{\twoheadrightarrow}
\newcommand{\ovl}{\overline}

\newcommand{\ccirc}{\kern0.5ex\vcenter{\hbox{$\scriptstyle\circ$}}\kern0.5ex}

\newcommand{\cB}{\mathcal{B}}
\newcommand{\cC}{\mathcal{C}}
\newcommand{\cD}{\mathcal{D}}
\newcommand{\cE}{\mathscr{E}}
\newcommand{\cG}{\mathcal{G}}
\newcommand{\cI}{\mathcal{I}}
\newcommand{\cJ}{\mathcal{J}}
\newcommand{\cK}{\mathcal{K}}

\newcommand{\cN}{\mathcal{N}}
\newcommand{\cO}{\mathcal{O}}

\newcommand{\ccR}{\mathcal{R}}
\newcommand{\cR}{\mathscr{R}}
\newcommand{\cS}{\mathscr{S}}
\newcommand{\cV}{\mathcal{V}}
\newcommand{\cW}{\mathcal{W}}
\newcommand{\cZ}{\mathcal{Z}}
\newcommand{\calH}{\mathcal{H}}

\newcommand{\bB}{{\mathbf{B}}}
\newcommand{\bD}{{\mathbf{D}}}
\newcommand{\bG}{{\mathbf{G}}}

\newcommand{\uk}{{\underline{k}}}

\newcommand{\Fp}{\overline{\F}_p}
\newcommand{\Qp}{\overline{\Q}_p}
\newcommand{\Zp}{\overline{\Z}_p}

\newcommand{\fa}{{\mathfrak{a}}}
\newcommand{\fA}{{\mathfrak{A}}}
\newcommand{\fb}{{\mathfrak{b}}}
\newcommand{\fc}{{\mathfrak{c}}}

\newcommand{\fg}{{\mathfrak{g}}}
\newcommand{\fG}{{\mathfrak{G}}}

\newcommand{\fI}{{\mathfrak{I}}}
\newcommand{\fJ}{{\mathfrak{J}}}
\newcommand{\fl}{{\mathfrak{l}}}

\newcommand{\fm}{{\mathfrak{m}}}

\newcommand{\fP}{{\mathfrak{P}}}

\newcommand{\fQ}{{\mathfrak{Q}}}

\newcommand{\ft}{{\mathfrak{t}}}
\newcommand{\fu}{{\mathfrak{u}}}
\newcommand{\fU}{{\mathfrak{U}}}

\newcommand{\fgl}{{\mathfrak{gl}}}
\newcommand{\fgsp}{{\mathfrak{gsp}}}
\newcommand{\fsl}{{\mathfrak{sl}}}
\newcommand{\fsp}{{\mathfrak{sp}}}

\title{Galois level and congruence ideal\\for $p$-adic families of finite slope Siegel modular forms}
\author{Andrea Conti}

\begin{document}

\begin{abstract}
We consider families of Siegel eigenforms of genus $2$ and finite slope, defined as local pieces of an eigenvariety and equipped with a suitable integral structure. Under some assumptions on the residual image, we show that the image of the Galois representation associated with a family is big, in the sense that a Lie algebra attached to it contains a congruence subalgebra of non-zero level. We call Galois level of the family the largest such level. We show that it is trivial when the residual representation has full image. When the residual representation is a symmetric cube, the zero locus defined by the Galois level of the family admits an automorphic description: it is the locus of points that arise from overconvergent eigenforms for $\GL_2$, via a $p$-adic Langlands lift attached to the symmetric cube representation. 
Our proof goes via the comparison of the Galois level with a ``fortuitous'' congruence ideal, that describes the zero- and one-dimensional subvarieties of symmetric cube type appearing in the family. We show that some of the $p$-adic lifts are interpolated by a morphism of rigid analytic spaces from an eigencurve for $\GL_2$ to an eigenvariety for $\GSp_4$. 
The remaining lifts appear as isolated points on the eigenvariety.
\end{abstract}

\maketitle

\numberwithin{thm}{section}
\numberwithin{equation}{section}

\setcounter{tocdepth}{1}
\tableofcontents

\section{Introduction}

Drawing inspiration from earlier work of H. Hida and J. Lang, the paper \cite{cit} studied the image of the Galois representations associated with $p$-adic families of modular forms, more precisely eigenforms of finite slope for the action of a Hecke algebra unramified outside of a fixed tame level. Such a family is defined by equipping a local piece of the eigencurve of the given tame level with an integral structure. A result of \cite{cit} states that the Galois representation attached to a family has big image in the following sense: there is a ring $\B$ and a Lie subalgebra $\fG$ of $\fgl_2(\B)$ attached to $\im\rho$, in a meaningful way, such that $\fG$ contains $\fl\cdot\fsl_2(\B)$ for a non-zero ideal $\fl$ of $\B$. This can be seen as an analogue, for a $p$-adic family, of a classical result of Ribet and Momose on the image of the $p$-adic Galois representation attached to a classical eigenform \cite{ribetI, momose}. We call Galois level of the family the largest ideal $\fl$ with the above property. The arguments in \cite{cit} rely strongly on the work of Hida and J. Lang for ordinary families \cite{hida,lang}, in particular on the study by J. Lang of the self-twists of the Galois representations attached to families. A new ingredient in the positive slope case is relative Sen theory, that replaces ordinarity in some crucial steps. Another result of \cite{cit} is an automorphic description of the Galois level of a family: the geometric points of its zero locus are the $p$-adic CM points of the family. This is also a generalization of a theorem of Hida in the ordinary case. The proof goes via the comparison of the Galois level with a fortuitous congruence ideal, that encodes the information on the CM specializations of the family. We call this ideal ``fortuitous'' because, contrary to what happens in the ordinary case, the CM specializations of a non-CM family do not correspond to congruences with CM families, that do not exists when the slope is positive. 

In this paper we find analogous results for $p$-adic families of Siegel modular forms of genus $2$ and finite slope. We think that our work in this setting shows that the big image properties of Galois representations and their relations to congruences are part of a picture that can be extended to more general reductive groups. 
We remark that Hida and Tilouine already have some results for ordinary $p$-adic families of $\GSp_4$-eigenforms that are residually of ``twisted Yoshida type'' \cite{hidatil}. Their arguments rely on the Galois ordinarity of the families and on $R=T$ results, both of which are not available when the slope is positive. They obtain congruences between families that are lifts from $\GL_{2/F}$, for a quadratic field $F$, and families that are not; their congruence ideals are then traditional ones and not fortuitous ones. In light of the results of the present paper, we think that fortuitous congruences should be regarded as general phenomena, that appear whenever we consider families of eigenforms for a reductive group that arise as $p$-adic Langlands lifts from a group of smaller rank.

The paper can be divided in two parts. In the first one (Sections 2 to 9) we define two-parameter families of $\GSp_4$-eigenforms of finite slope and we attach Galois representations to them; we then prove that the image of these representations is big in a Lie theoretic sense, assuming that the residual representation is either of full image or a symmetric cube. In the second part (Sections 10 to 16) we prove that the size of the Galois representation attached to a two-parameter family is related to the congruences of the family with lifts of eigenforms for a smaller group, constructed via a $p$-adic Langlands transfer. In the first half we need to solve many technical problems when passing from genus $1$ to genus $2$, whereas the second half is substantially different from its genus $1$ counterpart.  
We present our results and arguments in more detail below.

Fix a prime $p$ and an integer $M$ not divisible by $p$. Let $\calH_2^M$ be an abstract Hecke algebra unramified outside $Mp$ and of Iwahoric level at $p$. In their paper \cite{aip}, Andreatta, Iovita and Pilloni constructed a rigid analytic object $\cD_2$, that we call the $\GSp_4$-eigenvariety, and a map from $\calH_2^M$ to the ring of analytic functions on $\cD_2$, interpolating the systems of Hecke eigenvalues associated with the $p$-stabilized Siegel modular forms of genus $2$ and tame level $M$. The eigenvariety $\cD_2$ is equipped with a map to the two-dimensional weight space $\cW_2$, that is the rigid analytic space associated with the formal scheme $\Spf\Z_p[[(\Z_p^\times)^2]]$ by Berthelot's construction \cite[Section 7]{dejong}. To our purposes it is important that families be defined integrally, so we cannot work globally on irreducible components of the eigenvariety. We consider instead an admissible domain $D_h$ on $\cD_2$ consisting of the points of slope bounded by a rational number $h$ and of weight in a wide open disc in the weight space. If the radius of this disc is sufficiently small with respect to $h$, the restriction of the weight map to $D_h$ is a finite map thanks to a result of Bella\"\i che (Proposition \ref{locfin}). A suitably chosen integral structure on the weight disc induces an integral structure on $D_h$. This means that we can define a local profinite ring $\I^\circ$ and a map $\calH^M_2\to\I^\circ$ that interpolates the systems of Hecke eigenvalues of the classical eigenforms appearing in $D_h$. An argument by Chenevier gives a Galois pseudocharacter on $D_h$, that we lift to a representation $\rho\colon G_\Q\to\GSp_4(\I^\circ)$ (Lemma \ref{sympform}). We define the ``self-twists'' of $\rho$ as automorphisms of $\I^\circ$ that induce an isomorphism of $\rho$ with one of its twists by a Dirichlet character (Definition \ref{selftwist}). We write $\I_0^\circ$ for the subring of elements of $\I^\circ$ fixed by all the self-twists. We define a certain completion $\B$ of $\I^\circ_0[1/p]$ and a Lie subalgebra $\Lie(\im\rho)$ of $\fgsp_4(\B)$ attached to $\im\rho$ (see Section \ref{biglie}). We assume that $\rho$ is $\Z_p$-regular (Definition \ref{Zpreg}) and that the residual representation $\ovl\rho$ is either full or of symmetric cube type (Definition \ref{sctype}). Our first main result is the following.

\begin{thm}\label{intro1}(Theorem \ref{thexlevel})
There exists a non-zero ideal $\fl$ of $\B$ such that $\fl\cdot\fsp_4(\B)\subset\Lie(\im\rho)$.
\end{thm}

We call \emph{Galois level} of the family the largest ideal $\fl$ satisfying the inclusion of Theorem \ref{intro1}. 
We give here a summary of the proof of Theorem \ref{intro1}, that takes up Sections 6 to 9. We first show that, under our assumptions on $\ovl\rho$, there exists a classical weight such that $\rho$ specializes to a representation with big image at all points of this weight appearing on the family (Theorem \ref{classbigim}). Here we need the recent classicality result contained in \cite[Theorem 5.3.1]{bijpilstr}. Another essential ingredient is a result of Pink (Theorem \ref{pink}), that we use to show that the representation associated with a $\GSp_4$-eigenform that is not a lift from a smaller group has big image with respect to the ring fixed by its self-twists. This is an analogue of the result of Ribet and Momose for $\GL_2$-eigenforms. We rely on a result proved in the second part of the paper (see Corollary \ref{sym3automcor}) to show that a form which is not a lift satisfies the assumptions of Pink's theorem. 

Once a classical weight with the desired properties is chosen, we follow a strategy of J. Lang to obtain some information on the image of $\rho$. As a first step we need to show that a big image result holds for the product of the specializations of $\rho$ of a given weight, rather than just for a single one (Proposition \ref{prodresfull}). 
The argument here relies on Goursat's Lemma and on the classification of subnormal subgroups of symplectic groups by Tazhetdinov. Afterwards we use the result of the first step to construct some non-trivial unipotent elements in the image of $\rho$. 
In order to do this we need to prove an analogue of \cite[Theorem 3.1]{lang}, that allows us to lift the self-twists of the specializations of $\rho$ at our chosen weight to self-twists of $\rho$ itself. The arguments of J. Lang about the lifting of the self-twists to automorphisms of a suitable deformation ring can be translated to the genus $2$ case with little effort, but descending to a self-twist of the family requires some specific ingredients. Precisely, we prove that we can twist a family of $\GSp_4$-eigenforms by a Dirichlet character to obtain a new family (Lemma \ref{interptwists}) and we rely on the étaleness of the eigenvariety above our chosen weight. 

In Section \ref{sen} we show how the relative Sen theory of \cite[Section 5]{cit} can be extended to the group $\GSp_4$, in order to associate a Sen operator with $\rho$. The eigenvalues of this operator are given explicitly by the interpolation of the Hodge-Tate weights of the classical specializations of the family (Proposition \ref{eigensen}). The exponential of the Sen operator induces by conjugation a structure of $\Z_p[[T_1,T_2]]$-Lie algebra on $\Lie(\im\rho)$, so that the special elements we constructed generate a non-trivial congruence subalgebra. This proves Theorem \ref{intro1}.

When $\ovl\rho$ has full image the Galois level of the family is trivial (Corollary \ref{fullcong}), so the main focus of the rest of the paper is the case where $\ovl\rho$ is a symmetric cube. 
We can give two definitions of a \emph{symmetric cube locus} on the eigenvariety: an automorphic one, as the locus of points whose system of Hecke eigenvalues is obtained from that of an overconvergent $\GL_2$-eigenform via a symmetric cube morphism of Hecke algebras, and a Galois one, as the locus of points whose Galois representation is the symmetric cube of that associated with an overconvergent $\GL_2$-eigenform. An important result is the following.

\begin{thm}\label{intro2}(Theorem \ref{sym3type})
The automorphic and Galois definitions of the symmetric cube locus are equivalent.
\end{thm}

Theorem \ref{intro2} plays an essential role in describing the Galois level of the family by automorphic means. Note that this result and its role in our work are completely new with respect to the genus $1$ case: there the only possible congruences are of CM type and it is trivial to see that a point of small Galois image, contained in the normalized of a torus, is a $p$-adic CM point (see \cite[Remark 3.11]{cit}).

The proof of Theorem \ref{intro2} goes via the theory of $(\varphi,\Gamma)$-modules. It is known by Emerton's work that a Galois representation is associated with an overconvergent $\GL_2$-eigenform, up to a twist, if and only if it is trianguline. Thanks to the recent work of Kedlaya, Pottharst and Xiao on triangulations over eigenvarieties, we know that the ``only if'' part also holds for overconvergent $\GSp_4$-eigenforms (Theorem \ref{famtri}). By combining these results we reduce Theorem \ref{intro2} to the proposition below. Let $V$ be a two-dimensional representation of the absolute Galois group of $\Q_p$.

\begin{prop}\label{intro3}(Proposition \ref{sym3tri})
If $\Sym^3V$ is trianguline then $V$ is a twist of a trianguline representation by a character.
\end{prop}

We prove Proposition \ref{intro3} by adapting to our situation some arguments of Di Matteo \cite{dimat}. We also show an analogue of Proposition \ref{intro3} where ``trianguline'' is replaced by ``de Rham'' (Corollary \ref{sym3dR}); in this case the proof goes via nonabelian cohomology. Proposition \ref{intro3} allows us to prove that if a $p$-old point of symmetric cube type of $\cD_2^M$ is classical, then it is obtained from a classical point of an eigencurve for $\GL_2$, via the classical Langlands lift attached to the symmetric cube representation by Kim and Shahidi \cite{kimsha}.

We study further the symmetric cube locus and show that it is Zariski-closed with zero- and one-dimensional irreducible components. The one-dimensional part of the locus can be constructed as the image of a morphism from an eigencurve for $\GL_2$, of a suitable tame level, to $\cD_2^M$ (Section \ref{morpheigen}). 
This morphism is obtained by interpolating $p$-adically the classical symmetric cube Langlands lift. This interpolation argument goes back to Chenevier's work on the $p$-adic Jacquet-Langlands correspondence \cite{chenevier}, but we prefer to use some results of Bella\"\i che and Chenevier \cite[Section 7.2.3]{bellaiche} that allow us to move more easily from an eigenvariety to the other when changing of weight spaces, Hecke algebras and compact operators (see Section \ref{changeBC}). 
The zero-dimensional components of the symmetric cube locus are given by isolated $p$-adic Langlands lifts, that cannot be interpolated due to the fact that their slopes do not vary analytically. The appearance of such points is related to the existence of more than one crystalline period for the corresponding Galois representation (Remark \ref{dimsym}).

Restricting once again our attention to a local piece of the eigencurve describing a family, we define a \emph{symmetric cube congruence ideal} that measures the locus of symmetric cube specializations of the family (Definition \ref{sym3cong}). We call it a \emph{fortuitous} congruence ideal: since there are no two-parameter families of symmetric cube type, the congruences detected by this ideal are symmetric cube specializations of a family that is not globally a symmetric cube. 
Thanks to Theorem \ref{intro2}, that serves as a bridge between the automorphic and Galois sides, we can relate the congruence ideal with the Galois level of the family.

\begin{thm}\label{intro4}(Theorem \ref{comparison})
The sets of prime divisors of the Galois level and of the symmetric cube congruence ideal coincide outside of a finite and explicit bad locus.
\end{thm}

We think that the results of this paper can be generalized by allowing for different residual representations, hence different types of congruences, or by replacing $\GSp_4$ by other reductive groups for which an eigenvariety has been constructed. 
We hope to come back to this problem in a later work.

\bigskip

\textbf{Acknowledgments.} The results in this paper were obtained as part of the work for my Ph.D. thesis at the Laboratoire Analyse, Géométrie et Applications of Université Paris 13. I would like to thank my advisor, Jacques Tilouine, for having suggested that I work on the subject presented here and for having guided me in the preparation of the thesis. I was supported by the Programs ArShiFo ANR-10-BLAN-0114 and PerColaTor ANR-14-CE25-0002-01. I wish to thank all the members of the équipe Arithmétique et Géométrie Algébrique of Université Paris 13 of the past three years for creating an excellent working environment. The preparation of the paper was completed while I was a postdoctoral fellow at Concordia University in Montreal. I thank Adrian Iovita for inviting me there and for his helpful suggestions. I also wish to thank Anne-Marie Aubert, Kevin Buzzard, Denis Benois, Joel Bella\"\i che, Huan Chen, Ga\"etan Chenevier, Giovanni Di Matteo, Taiweng Deng, Mladen Dimitrov, Yiwen Ding, Haruzo Hida, Jaclyn Lang, Shinan Liu, Giovanni Rosso, Benjamin Schraen and Beno\^\i t Stroh for useful discussions.

\bigskip

\textbf{Notations.} We fix some notations and conventions. In the text $p$ will always denote a prime number strictly larger than $3$. Most argument work for every odd $p$; we specify when this is not sufficient. We choose algebraic closures $\ovl{\Q}$ and $\ovl{\Q}_p$ of $\Q$ and $\Q_p$, respectively. If $K$ is a finite extension of $\Q$ or $\Q_p$ we denote by $G_K$ its absolute Galois group. We equip $G_K$ with its profinite topology. We denote by $\cO_K$ the ring of integers of $K$. If $K$ is local, we denote by $\fm_K$ the maximal ideal of $\cO_K$. 
For every prime $p$ we fix an embedding $\iota_p\colon\ovl{\Q}\into\ovl{\Q}_p$, identifying $G_{\Q_p}$ with a decomposition group of $G_\Q$. This identification will be implicit everywhere. We fix a valuation $v_p$ on $\ovl{\Q}_p$ normalized so that $v_p(p)=1$. It defines a norm given by $\vert\cdot\vert=p^{-v_p(\cdot)}$. We denote by $\C_p$ the completion of $\ovl{\Q}_p$ with respect to this norm.

All rigid analytic spaces will be considered in the sense of Tate (see \cite[Part C]{bgr}). 
Let $K/\Q_p$ be a field extension and let $X$ be a rigid analytic space over $K$. We denote by $\cO(X)$ the $K$-algebra of rigid analytic functions on $X$, and by $\cO(X)^\circ$ the $\cO_K$-subalgebra of functions with norm bounded by $1$ (we often say ``functions bounded by $1$'' meaning that they are bounded in norm). When $f\colon X\to Y$ is a map of rigid analytic spaces, we denote by $f^\ast\colon\cO(Y)\to\cO(X)$ the map induced by $f$. 
There is a Grothendieck topology on $X$, called the Tate topology; we refer to \cite[Proposition 9.1.4/2]{bgr} for the definition of its admissible open sets and admissible coverings. 

We say that $X$ is a wide open rigid analytic space if there exists an admissible covering $\{X_i\}_{i\in\N}$ of $X$ by affinoid domains $X_i$ such that, for every $i$, $X_i\subset X_{i+1}$ and the map $\cO(X_{i+1})\to\cO(X_i)$ induced by the previous inclusion is compact.

There is a notion of irreducible components for a rigid analytic space $X$; see \cite{conradirr} for the details. 
We say that $X$ is equidimensional of dimension $d$ if all its irreducible components have dimension $d$. 

We denote by $\A^d$ the $d$-dimensional rigid analytic affine space over $\Q_p$. Given a point $x\in\A^d(\C_p)$ and $r\in p^\Q$, we denote by $B_d(x,r)$ the $d$-dimensional closed disc of centre $x$ and radius $r$. It is an affinoid domain defined over $\C_p$. We denote by $B_d(x,r^-)$ the $d$-dimensional wide open disc of centre $x$ and radius $r$, defined as the rigid analytic space over $\C_p$ given by the increasing union of the $d$-dimensional affinoid discs of centre $x$ and radii $\{r_i\}_{i\in\N}$ with $r_i<r$ and $\lim_{i\mapsto +\infty}r_i=r$. With an abuse of terminology we refer to $B_d(x,r)$ as the $d$-dimensional ``closed disc'' and to $B_d(x,r^-)$ as the $d$-dimensional ``open disc'', even though both are open sets in the Tate topology. 

Let $X$ be an affinoid or a wide open rigid analytic space. 
We denote by $\cO(X)\{\{T\}\}$ the ring of power series $\sum_{i\geq 0}a_iT^i$ with $a_i\in\cO(X)$ and $\lim_i\vert a_i\vert r^i\to 0$ for every $r\in\R^+$. This is the ring of rigid analytic functions on $X\times\A^1$.

Let $S$ be any subset of $X(\C_p)$. We say that $S$ is:
\begin{enumerate}
\item a discrete subset of $X(\C_p)$ if $S\cap A$ is a finite set for any open affinoid $A\subset X(\C_p)$;
\item a Zariski-dense subset of $X(\C_p)$ if, for every $f\in\cO(X)$ vanishing at every point of $S$, $f$ is identically zero;
\item an accumulation subset of $X(\C_p)$ if for every $x\in S$ there exists a basis $\cB$ of affinoid neighborhoods of $x$ in $X$ such that for every $A\in\cB$ the set $S\cap A(\C_p)$ is Zariski-dense in $A$ (this term is borrowed from \cite[Section 3.3.1]{bellchen}).
\end{enumerate}

Let $g\geq 1$ be an integer and let $s$ be the $g\times g$ antidiagonal unit matrix $(\delta_{i,n-i}(i,j))_{1\leq i,j\leq g}$. Let $J_g$ be the $2g\times 2g$ matrix $\left(\begin{array}{cc} 0&s\\-s&0\end{array}\right)$. We denote by $\GSp_{2g}$ the algebraic group of symplectic similitudes for $J_g$, defined over $\Z$; for every ring $R$ the $R$-points of this group are given by
\[ \GSp_{2g}(R)=\{A\in\GL_4(R)\,\vert\,\exists\, \nu(A)\in R^\times \mbox{ s.t. } {}^tAJA=\nu(A)J\}. \]
For $g=1$ we have $\GSp_2=\GL_2$. The map $A\to\nu(A)$ defines a character $\nu\colon\GSp_4(R)\to R^\times$. We refer to $\nu$ as the similitude factor and we set $\Sp_{2g}(R)=\{A\in\GSp_{2g}(R)\,\vert\,\nu(A)=1\}$.

We denote by $B_g$ the Borel subgroup of $\GSp_{2g}$ such that for every ring $R$ the $R$-points of $B_g$ are the upper triangular matrices in $\GSp_{2g}(R)$. We let $T_g$ be the maximal torus such that for every ring $R$ the $R$-points of $T_g$ are the diagonal matrices in $\GSp_{2g}(R)$. We write $U_g$ for the unipotent radical of $B_g$. We have $B_g=T_gU_g$. We will always speak of weights and roots for $\GSp_{2g}$ with respect to the previous choice of Borel subgroup and torus. For every root $\alpha$ we denote by $U^\alpha$ the corresponding one-parameter unipotent subgroup of $\GSp_{2g}$. For every prime $\ell$, we write $I_{g,\ell}$ for the Iwahori subgroup of $\GSp_{2g}(\Q_\ell)$ corresponding to our choice of Borel subgroup. For every $n\geq 1$ we denote by $\1_n$ the $n\times n$ unit matrix.

Let $g$ be a positive integer. For every prime $\ell$ and every integer $n\geq 0$ we define some compact open subgroups of $\GSp_{2g}(\A_\Q)$ by:
\begin{enumerate}
\item $\Gamma^{(g)}(\ell^n)=\{h\in\GSp_{2g}(\widehat{\Z}) \,\vert\, h_\ell\cong\1_{2g}\pmod{\ell^n}\}$;
\item $\Gamma^{(g)}_1(\ell^n)=\{h\in\GSp_{2g}(\widehat{\Z}) \,\vert\, h_\ell\pmod{\ell^n}\in U_g(\Z/\ell^n\Z)\}$;
\item $\Gamma^{(g)}_0(\ell^n)=\{h\in\GSp_{2g}(\widehat{\Z}) \,\vert\, h_\ell\pmod{\ell^n}\in B_g(\Z/\ell^n\Z)\}$.
\end{enumerate}
In particular for $n=1$ the $\ell$-component of $\Gamma_0(\ell)$ is the Iwahori subgroup of $\GSp_4(\Q_\ell)$.
Let $N$ be an arbitrary positive integer. Write $N=\prod_i\ell_i^{n_i}$ for some distinct primes $\ell_i$ and some $n_i\in\N$. We set $\Gamma^{(g)}_?(N)=\bigcap_i\Gamma^{(g)}_?(\ell_i^{n_i})$ for $?=\varnothing,0,1$. For $g=1$ we will omit the upper index (1).

We denote by $\fgsp_{2g}$ the Lie algebra of $\GSp_{2g}$ and by $\fsp_{2g}$ its derived Lie algebra, which is the Lie algebra of $\fsp_{2g}$. We denote by $\Ad\colon\GSp_{2g}\to\Aut(\fsp_{2g})$ the adjoint action of $\GSp_{2g}$ on $\fsp_{2g}$. It is an irreducible representation of $\GSp_{2g}$. 

By ``classical modular form for $\GSp_4$'' we always mean a \emph{vector-valued} modular form.

\bigskip

\section{Preliminaries on eigenvarieties}\label{eigenvar}

In this section we define the basic objects with which we are going to work: weight spaces, Hecke algebras and eigenvarieties. We recall some of their properties.

\subsection{The weight spaces}\label{weightsp}

We choose once and for all $u=1+p$ as a generator of $\Z_p^\times$. This choice determines an isomorphism $\Z_p^\times\cong(\Z/(p-1)\Z)\times\Z_p$. 
Let $g$ be a positive integer.
Consider the Iwasawa algebra $\Z_p[[(\Z_p^\times)^g]]$. 
A construction by Berthelot \cite[Section 7]{dejong} attaches to the formal scheme $\Spf\Z_p[[(\Z_p^\times)^g]]$ a rigid analytic space that we denote by $\cW_g$. 
If $A$ is a $\Q_p$-algebra, the $A$-points of $\cW_g$ are the continuous characters $(\Z_p^\times)^g\to A^\times$. Denote by $\widehat{(\Z/(p-1)\Z)^g}$ the group of characters of $(\Z/(p-1)\Z)^g$. The following map gives an isomorphism from $\cW_g$ to a disjoint union of $g$-dimensional open discs $B_g(0,1^-)$ indexed by $\widehat{(\Z/(p-1)\Z)^g}$: 
\begin{gather*} 
\eta_g\colon\cW_g\to\widehat{(\Z/(p-1)\Z)^g}\times B_g(0,1^-), \\
\kappa\mapsto (\kappa\vert_{(\Z/(p-1)\Z)^g},(\kappa(u,1,\ldots,1)-1,\kappa(1,u,1,\ldots,1)-1,\ldots,\kappa(1,\ldots,1,u)-1)).
\end{gather*}

We write $\Lambda_g$ for the algebra $\Z_p[[T_1,T_2,\ldots,T_g]]$ of formal series in $g$ variables over $\Z_p$. It is the ring of rigid analytic functions bounded by $1$ on a connected component of the weight space. 

We denote by $\kappa_{\cW_g}\colon\Z_p^\times\to\Z_p[[(\Z_p^\times)^g]]^\times$ the universal character of $\cW_g$. 
For every affinoid domain $A=\Spm R$ and every inclusion $\iota_A\colon A\into\cW_g$ we set $\kappa_A=\iota_A^\ast\ccirc\kappa_{\cW_g}$. We call $\kappa_A$ the universal character associated with $A$. 
By \cite[Proposition 8.3]{buzzard} there exists $r\in p^\Q$ such that $\kappa_A$ is $r$-analytic, in the sense that it can be extended to a character $((\Z_p^\times)^g\cdot B_g(1,r))\to R^\times$. The radius of analyticity of $\kappa_A$ is the largest such $r$; we denote it by $r_{\kappa_A}$.
 
We call \emph{arithmetic primes} the primes of $\Z_p[[(\Z_p^\times)^g]]\widehat{\otimes}_{\Z_p}\C_p$ of the form $P_{\uk,\varepsilon}=(k_1,k_2,\ldots,k_g,1+T_1-\varepsilon_1(u)u^{k_1},1+T_2-\varepsilon_2(u)u^{k_2},\ldots,1+T_g-\varepsilon_g(u)u^{k_g})$ for a $g$-tuple of integers $\uk=(k_1,k_2,\ldots,k_g)$ and a finite order character $\varepsilon\colon(\Z_p^\times)^g\to\C_p^\times$. We will always take as $\varepsilon$ the trivial character $1$; in this case we write $P_\uk=P_{\uk,1}$. We say that a $\Q_p$-point $\kappa\colon\Z_p^\times\to\Q_p^\times$ of $\cW^\circ_g$ is \emph{classical} if it is the specialization of $\kappa_{\cW_g}$ at $P_\uk$ for some $\uk\in\Z^g$. 

\subsection{The abstract Hecke algebras}

The abstract Hecke algebras we consider are tensor products of a Iwahori-Hecke algebra at $p$ and of the spherical Hecke algebras at all primes outside of a finite set containing $p$.

\subsubsection{The abstract spherical Hecke algebra}

Let $\ell$ be a prime. Let $G$ be a $\Z$-subgroup scheme of $\GSp_{2g}$ and let $K\subset G(\Q_\ell)$ be a compact open subgroup. For $\gamma\in G(\Q_\ell)$ we denote by $\1([K\gamma K])$ the characteristic function of the double coset $[K\gamma K]$. Let $\calH(G(\Q_\ell),K)$ be the $\Q$-algebra generated by the functions $\1([K\gamma K])$ for $\gamma\in G(\Q_\ell)$, equipped with the convolution product. We call \emph{spherical (or unramified) Hecke algebra of $\GSp_{2g}$ at $\ell$} the $\Q$-algebra $\calH(\GSp_{2g}(\Q_\ell),\GSp_{2g}(\Z_\ell))$. 
It is generated by the elements $T^{(g)}_{\ell,i}=\1([\GSp_{2g}(\Z_\ell)\diag(\1_i,\ell\1_{2g-2i},\ell^2\1_i)\GSp_{2g}(\Z_\ell)])$, for $i=0,1,\ldots g$, and $(T^{(g)}_{\ell,0})^{-1}$. 
Note that our operator $T^{(g)}_{\ell,0}$ is often denoted by $S^{(g)}_{\ell}$ in the literature. 

\subsubsection{The abstract dilating Iwahori-Hecke algebra}\label{dilIw} 

The Hecke algebra $\calH(T_g(\Q_\ell),T_g(\Z_\ell))$ carries a natural action of the Weyl group $W_g=\cS_g\ltimes (\Z/2\Z)^g$ of $\GSp_{2g}$, where $\cS_g$ is the group of permutations of $\{1,2,\ldots,g\}$: if $\diag(\nu t_1,\ldots,\nu t_g,t_g^{-1},\ldots,t_1^{-1})$ is an element of the torus, $\cS_g$ acts by permuting the $t_i$'s and the non-trivial element in each $\Z/2\Z$ sends $t_i$ to $t_i^{-1}$. We denote the action of $w\in W_g$ on $t\in T(\Q_\ell)$ by $t\mapsto w.t$. 
%
The twisted Satake transform $S_{\GSp_{2g}}^{T_g}\colon\calH(\GSp_{2g}(\Q_\ell),\GSp_{2g}(\Z_\ell))\to\calH(T_g(\Q_\ell),T_g(\Z_\ell))$ induces an isomorphism of $\calH(\GSp_{2g}(\Q_\ell),\GSp_{2g}(\Z_\ell))$ onto its image, which is the subalgebra of $\calH(T_g(\Q_\ell),T_g(\Z_\ell))$ consisting of $W_g$-invariant elements. In particular $\calH(T_g(\Q_\ell),T_g(\Z_\ell))$ is a Galois extension of $\calH(\GSp_{2g}(\Q_\ell),\GSp_{2g}(\Z_\ell))$ of Galois group $W_g$. 

For $i=0,1,\ldots,g$ let $t^{(g)}_{\ell,i}=\1([\diag(\1_i,\ell\1_{2g-2i},\ell^2\1_i)\T_g(\Z_\ell)])$. Note that $t^{(g)}_{\ell,0}=S_{\GSp_{2g}}^{T_g}(T^{g}_{\ell,0})$. The set $(t^{(g)}_{\ell,i})_{i=1,\ldots,g}$ generates the extension $\calH(T_g(\Q_\ell),T_g(\Z_\ell))$ over $\calH(\GSp_{2g}(\Q_\ell),\GSp_{2g}(\Z_\ell))$.

We call an element $\gamma\in T_g(\Z_\ell)$ dilating if $v_p(\alpha(\gamma))\leq 0$ for every positive root $\alpha$. Let $T_g(\Z_\ell)^-$ be the subset of $T_g(\Z_\ell)$ consisting of dilating elements and let $\calH(T_g(\Q_\ell),T_g(\Z_\ell))^-$ be the $\Q$-subalgebra of $\calH(T_g(\Q_\ell),T_g(\Z_\ell))$ generated by the functions $\1([\gamma T_g(\Z_\ell)])$ with $\gamma\in T_g(\Q_\ell)^-$. 
The functions $\1([\gamma T_g(\Z_\ell)])$ with $\gamma\in T_g(\Q_\ell)^-$ also form a basis of $\calH(T_g(\Q_\ell),T_g(\Z_\ell))^-$ as a $\Q$-vector space. 

\begin{rem}\label{extchar}
Every $\gamma\in T(\Q_\ell)$ can be written in the form $\gamma=\gamma_1\gamma_2^{-1}$ with $\gamma_1,\gamma_2\in T(\Z_\ell)^-$. A character $\chi\colon\calH(T_g(\Q_\ell),T_g(\Z_\ell))^-\to\Qp$ can be extended uniquely to a character $\chi^\ext\colon\calH(T_g(\Q_\ell),T_g(\Z_\ell))\to\Qp$ by setting $\chi^\ext([\gamma T(\Z_\ell)])=\chi([\gamma_1 T(\Z_\ell)])\chi([\gamma_2 T(\Z_\ell)]^{-1})$ for some $\gamma_1$ and $\gamma_2$ as before. It can be easily checked that $\chi^\ext$ is well-defined.
\end{rem}

Let $\calH(\GSp_{2g}(\Q_\ell),I_{g,\ell})^-$ be the subalgebra of $\calH(\GSp_{2g}(\Q_\ell),I_{g,\ell})$ generated by the functions $\1([I_{g,\ell}\gamma I_{g,\ell}])$ with $\gamma\in T(\Z_\ell)^-$. 
We call $\calH(\GSp_{2g}(\Q_\ell),I_{g,\ell})^-$ the \emph{dilating Iwahori-Hecke algebra} at $\ell$. 
It is generated by the elements $U^{(g)}_{\ell,i}=\1([I_{g,\ell}\diag(\1_i,\ell\1_{2g-2i},\ell^2\1_i)I_{g,\ell}])$, for $i=0,1,\ldots,g$, and $(U^{(g)}_{\ell,0})^{-1}$.

We define a morphism of $\Q$-algebras $\iota_{I_{g,\ell}}^{T_g}\colon\calH(\GSp_{2g}(\Q_\ell),I_{g,\ell})^-\to\calH(T_g(\Q_\ell),T_g(\Z_\ell))^-$ by sending $\1(I_{g,\ell}\gamma I_{g,\ell})$ to $\1(T_g(\Z_\ell)\gamma T_g(\Z_\ell))$ for every $\gamma\in T(\Z_\ell)^-$. 
The map $\iota_{I_{g,\ell}}^{T_g}$ is an isomorphism; this can be proved as \cite[Proposition 6.4.1]{bellchen}.

Let $p$ be a prime and $N$ be a positive integer such that $(N,p)=1$. 
Set
\[ \calH_g^{Np}=\bigotimes_{\Q,\ell\nmid Np}\calH(\GSp_{2g}(\Q_\ell),\GSp_{2g}(\Z_\ell)) \]
and
\[ \calH_g^N=\calH_g^{Np}\otimes_\Q\calH(\GSp_{2g}(\Q_p),I_{g,p})^-. \]
We call $\calH_g^{N}$ the \emph{abstract Hecke algebra spherical outside $N$ and Iwahoric dilating at $p$}.

The algebra $\calH_g^N$ acts on the space of classical vector-valued modular forms for $\GSp_{2g}(\Q)$ of level $\Gamma_1(N)\cap\Gamma_0(p)$. 
With an abuse of notation we will consider the elements of one of the local algebras as elements of $\calH_g^N$ via the natural inclusion (tensoring by $1$ at all the other primes).

\subsubsection{The Hecke polynomials}\label{weyl}

We record here some explicit formulas for the minimal polynomials $P_\Min(t^{(g)}_{\ell,i};X)$ of the elements $t_{\ell,i}^{(g)}$ over $\calH(\GSp_{2g}(\Q_\ell),\GSp_{2g}(\Z_\ell))$ when $g$ is $1$ or $2$.

For $g=1$, the element $t^{(1)}_{\ell,1}=\1([\diag(1,\ell)T_1(\Z_\ell)])$ generates the degree two extension $\calH(T_1(\Q_\ell),T_1(\Z_\ell))$ of $\calH(\GL_2(\Q_\ell),\GL_2(\Z_\ell))$. Let $w$ be the only non-trivial element of the Weyl group of $\GL_2$. 
The minimal polynomial of $t^{(1)}_{\ell,1}$ is $P_\Min(t^{(1)}_{\ell,1})(X)=(X-t^{(1)}_{\ell,1})(X-(t^{(1)}_{\ell,1})^w)$. 
An explicit calculation gives
\begin{equation}\label{minpol1} P_\Min(t^{(1)}_{\ell,1};X)=(X-t^{(1)}_{\ell,1})(X-(t^{(1)}_{\ell,1})^w)=X^2-T^{(1)}_\ell X+\ell T^{(1)}_{\ell,0}. \end{equation}

For $g=2$, the degree eight extension $\calH(T_2(\Q_\ell),T_2(\Z_\ell))$ over $\calH(\GSp_4(\Q_\ell),\GSp_4(\Z_\ell))$ is generated by $t^{(2)}_{\ell,1}=\1([\diag(1,\ell,\ell,\ell^2)T_2(\Z_\ell)])$ and $t^{(2)}_{\ell,2}=\1([\diag(1,1,\ell,\ell)T_2(\Z_\ell)])$. Each of them has an orbit of order four under the action of the Weyl group. 
If $t=\diag(\nu t_1,\nu t_2,t_1^{-1},t_2^{-1})$ is an element of the torus we denote by $w_0$, $w_1$, $w_2$ the generators of the Weyl group satisfying $t^{w_0}=\diag(\nu t_2,\nu t_1,t_2^{-1},t_1^{-1})$, $t^{w_1}=\diag(\nu t_1^{-1},\nu t_2,t_1,t_2^{-1})$, $t^{w_2}=\diag(\nu t_1,\nu t_2^{-1},t_1^{-1},t_2)$. Note that $t^{(2)}_{\ell,2}$ is invariant under $w_0$. The calculation in the proof of \cite[Lemma 3.3.35]{andrquad} gives
\begin{equation}\label{minpol2}\begin{gathered}
P_\Min(t^{(2)}_{\ell,2};X)=(X-t^{(2)}_{\ell,2})(X-(t^{(2)}_{\ell,2})^{w_1})(X-(t^{(2)}_{\ell,2})^{w_2})(X-(t^{(2)}_{\ell,2})^{w_1w_2})= \\
=X^4-T^{(2)}_{\ell,2}X^3+((T^{(2)}_{\ell,2})^2-T^{(2)}_{\ell,1}-\ell^2T^{(2)}_{\ell,0})X^2-\ell^3T^{(2)}_{\ell,2}T^{(2)}_{\ell,0} X+\ell^6(T^{(2)}_{\ell,0})^2. \end{gathered}\end{equation}
Since $t^{(2)}_{\ell,1}=(t^{(2)}_{\ell,2})(t^{(2)}_{\ell,2})^{w_1}$ is invariant under $w_1$, we can also write
\begin{equation}\begin{gathered}\label{polweyl1}
P_\Min(t^{(2)}_{\ell,1})(X)=(X-t^{(2)}_{\ell,2}(t^{(2)}_{\ell,2})^{w_1})(X-(t^{(2)}_{\ell,2})^{w_2}(t^{(2)}_{\ell,2})^{w_1w_2})(X-t^{(2)}_{\ell,2}(t^{(2)}_{\ell,2})^{w_2})(X-(t^{(2)}_{\ell,2})^{w_1}(t^{(2)}_{\ell,2})^{w_1w_2}). 
\end{gathered}\end{equation}

\subsubsection{Normalized systems of Hecke eigenvalues}\label{normsyst}

For this reason we introduce their standard normalization, depending on the weight, before passing to the $p$-adic setting. 
Let $f$ be a classical $\GSp_{2g}$-eigenform of level $\Gamma_1(N)\cap\Gamma_0(p)$ and weight $\uk=(k_1,k_2,\ldots,k_g)$. 
Let $\chi\colon\calH_g^N\to\Qp$ be the system of Hecke eigenvalues associated with $f$. 

\begin{defin}\label{normsystdef}
For $g\in\{1,2\}$, let $\chi^\norm\colon\calH_g^N\to\Qp$ be the character defined by
\begin{itemize}[label={--}]
\item $\chi^\norm\vert_{\calH_g^{Np}}=\chi\vert_{\calH_g^{Np}}$;
\item $\chi^\norm(U_{p,i}^{(g)})=p^{-\sum_{j=1}^{g-i}(k_j-j)}$ for $i=1,2,\ldots,g$ (where the exponent of $p$ is $0$ for $i=g$).
\end{itemize}
We call $\chi^\norm$ the normalized system of Hecke eigenvalues associated with $f$. 
\end{defin}

\subsection{The eigenvariety machine}\label{eigenmac}

We recall some elements of Buzzard's ``eigenvariety machine'' \cite{buzzard}. We call \emph{eigenvariety datum} a $5$-tuple $(\cW,\calH,(M(A,w))_{A,w},(\phi_{A,w})_{A,w},\eta)$ where:
\begin{enumerate}
\item there exists an integer $g\geq 1$ such that $\cW=\cW_g$ is the $g$-dimensional weight space defined in the previous section;
\item $(A,w)$ varies over the couples consisting of an affinoid $A\subset\cW$ and $w\in\Q$ satisfying $p^{-w}\leq r_{\kappa_A}$;
\item for every $(A,w)$ with $A=\Spm R$, $M(A,w)$ is a projective Banach $R$-module;
\item $\calH$ is a commutative ring;
\item $\phi_{A,w}\colon\calH\to\End_{R,\cont}(M(A,w))$ is an action of $\calH$ on $M(A,w)$;
\item $\eta\in\calH$ is an element such that $\phi_{A,w}(\eta)$ is a compact operator on $M(A,w)$ for every $(A,w)$;
\item when $A$ and $w$ vary the modules $M(A,w)$ with their $\calH$-actions satisfy the compatibility properties assumed in \cite[Lemma 5.6]{buzzard}.
\end{enumerate}


Let $K$ be a finite extension of $\Q_p$. A morphism $\lambda\colon\calH\to K$ is called a \emph{$K$-system of eigenvalues} for the given datum if there exists a point $\kappa\in\cW(K)$, an affinoid $A=\Spm R$ containing $\kappa$, a rational $w$ and an element $m\in M(A,w)\otimes_R K$ (where $R\to K$ is the evaluation at $\kappa$) such that $\phi_{A,w}(T)m=\lambda(T)m$ for all $T\in\calH$.

\begin{thm}\label{theigenmac}
For every eigenvariety datum $(\cW,\calH,(M(A,w))_{A,w},(\phi_{A,w})_{A,w},\eta)$ there exists a triple $(\cD,\psi,w)$ consisting of
\begin{enumerate}
\item a rigid analytic space $\cD$ over $\Q_p$,
\item a morphism of $\Q_p$-algebras $\psi\colon\calH\to\cO(\cD)^\circ$,
\item a morphism of rigid analytic spaces $w\colon\cD\to\cW$ (called the \emph{weight morphism}),
\end{enumerate}
with the following properties:
\begin{enumerate}
\item $\psi(\eta)$ is invertible in $\cO(\cD)$;
\item for every finite extension $K/\Q_p$ the map
\begin{equation}\label{heckeeigen}\begin{gathered} \cD(K)\to\Hom(\calH,K), \\
x\mapsto(T\mapsto \psi(T)(x)), \end{gathered}\end{equation}
induces a bijection between the $K$-points of $\cD$ and the $K$-systems of eigenvalues for the given datum.
\end{enumerate}
We call $(\cD,\psi,w)$ the \emph{eigenvariety} for the given datum.
\end{thm}

We often leave $\psi$ and $w$ implicit and just refer to $\cD$ as the eigenvariety. 
Since the space $\cW_g$ is equidimensional of dimension $g$, \cite[Proposition 6.4.2]{chenfam} implies that $\cD$ is also equidimensional of dimension $g$.

Thanks to property (1) in Theorem \ref{theigenmac}, we can give the following definition.

\begin{defin}\label{slopedef}
Let $\slo\colon\cD(\C_p)\to\R^{\geq 0}$ be the function defined by $\slo(x)=v_p(\psi(\eta)(x))$ for every $x\in\cD(\C_p)$. We call $\slo(x)$ the \emph{slope} of $x$. 
\end{defin}
%

\begin{rem}\label{affslope}\mbox{ }
The function $\slo\colon\cD(\C_p)\to\R^+$ is locally constant. 
In particular $\slo$ is bounded over $A(\C_p)$ for every affinoid subdomain $A$ of $\cD$.
\end{rem}


\begin{defin}\label{ordeigen}
We call \emph{ordinary eigenvariety} for the given datum the largest open subvariety $\cD^\ord$ of $\cD$ with the property that $\psi(\eta)\vert_{\cD^\ord}\in(\cO(\cD^\ord)^\circ)^\times$.
\end{defin}

\subsection{The cuspidal $\GSp_{2g}$-eigencurve}\label{gspeigen}

Let $g$ be a positive integer. Let $p$ be an odd prime and let $N$ be a positive integer such that $(N,p)=1$. Let $\calH_g^N$ be the abstract Hecke algebra for $\GSp_{2g}$, spherical outside $N$ and Iwahoric dilating at $p$. Let $\cW_g$ be the $g$-dimensional weight space. For every affinoid $A=\Spm R\subset\cW_g$ and every sufficiently large rational number $w$, Andreatta, Iovita and Pilloni \cite[Section 8.2]{aip} defined a Banach $R$-module $M_g(A,w)$ of $w$-overconvergent cuspidal $\GSp_{2g}$-modular forms of weight $\kappa_A$ and tame level $\Gamma_1(N)$.  
For each $(A,w)$ there is an action $\phi^{g}_{A,w}\colon\calH^N_g\to\End_{R,\cont}M_g(A,w)$. Set $U_p^{(g)}=\prod_{i=1}^gU^{(g)}_{p,g}$. It is shown in \cite[Section 8.1]{aip} that $(\cW_g,\calH^N_g,(M_g(A,w))_{A,w},(\phi^{g})_{A,w},U_p^{(g)})$ is an eigenvariety datum. The eigenvariety machine constructs from this datum a rigid analytic variety over $\Q_p$, equidimensional of dimension $g$. We call it the \emph{$\GSp_{2g}$-eigenvariety} of tame level $N$ and we denote it by $\cD_g^{N}$. It is equipped with a weight morphism $w_g\colon\cD_g^N\to\cW_g$ and a map $\psi_g\colon\calH_g^N\to\cO(\cD_g^N)$, that interpolates the normalized systems of Hecke eigenvalues of classical cuspidal $\GSp_{2g}$-eigenforms of tame level $\Gamma_1(N)$. The images of the elements $T_{i,\ell}^{(g)}$ and $U_{i,p}^{(g)}$, $1\le i\le g$, belong to $\cO(\cD_g^N)^\circ$.

When $g=1$ we call $\cD_1^N$ the \emph{eigencurve}. It was constructed by Coleman and Mazur in \cite{colmaz} for $N=1$ and $p>2$, building on earlier ideas of Coleman. Their construction was extended to all $N$ and $p$ by Buzzard in \cite{buzzard}.

We call a point $x\in\cD_g^{N}(\C_p)$ \emph{classical} if the system of Hecke eigenvalues associated with $x$ by the map \eqref{heckeeigen} is that of a classical modular form $f$ of level $\Gamma_1(N)\cap\Gamma_0(p)$ and weight $w_g(x)$. In this case $w_g(x)$ is clearly a classical weight.

There is a slope function $\slo\colon\cD_g^N(\C_p)\to\R^+$ given by Definition \ref{slopedef}. 
By Coleman's classicality result, a point of $\cD_1^M$ of weight $k\geq 2$ and slope $h<k-1$ is classical. 
An analogue for general genus is given by the result below. Let $x$ be a $\Qp$-point of $\cD_2^{N}$ of weight $\uk=(k_1,k_2,\ldots,k_g)\in\Z^g$, so that $k_1\geq k_2\geq\ldots\geq k_g$.

\begin{prop}\label{siegclslopes}(\cite[Theorem 5.3.1]{bijpilstr}, see also Remark 1 in the Introduction of \emph{loc. cit.})
If $\slo(x)<k_g-\frac{g(g+1)}{2}$ then the point $x$ is classical.
\end{prop}

\subsubsection{The non-CM eigencurve}

%
We say that a classical point of $\cD_1^N$ is a \emph{CM point} if it corresponds to a classical CM modular form. 
We say that an irreducible component of $\cD_1^N$ is a {CM component} if all its classical specializations are CM points.

\begin{rem}\label{cmpoints}
By \cite[Proposition 5.1]{hida}, if an ordinary irreducible component of the eigencurve contains a classical CM eigenform of weight $k\geq 2$ then the component is CM. In particular there exist CM irreducible components of the ordinary eigencurve, and every ordinary CM classical point belongs to a CM component. On the contrary, the CM classical points of the positive slope eigencurve form a discrete set (recall that this means that they are finite in each affinoid domain). This is a consequence of \cite[Corollary 3.6]{cit}, where it is shown that the eigencurve $\cD^{+,\leq h}$ contains a finite number of CM classical points.
\end{rem}

Let $\cD_1^{N,\cG}$ be the Zariski-closure in $\cD_1^N$ of the set of non-CM classical points. We call $\cD_1^{N,\cG}$ the \emph{non-CM eigencurve}. 
The upper index $\cG$ stands for ``general'', since CM components are exceptional among the irreducible components of $\cD_1^N$. 

\begin{rem}\label{nonCMdense}
It follows from Remark \ref{cmpoints} that $\cD_1^{N,\cG}$ is the union of all the non-CM irreducible components of $\cD_1^N$. In particular $\cD_1^{N,\cG}$ is equidimensional of dimension $1$ and it contains the positive slope eigencurve. Moreover the set of non-CM classical points is an accumulation and Zariski-dense subset of $\cD_1^{N,\cG}$. 
\end{rem}

\bigskip

\section{The Galois pseudocharacters on the eigenvarieties}\label{biggalps}

In this section $p$ is a fixed prime, $M$ is a positive integer prime to $p$ and $g$ is $1$ or $2$. For a point $x\in\cD^M_g(\C_p)$ we denote by $\ev_x\colon\cO(\cD_g^M)\to\C_p$ both the evaluation at $x$ and the map $\GSp_{2g}(\cO(\cD^M_g))\to\GSp_{2g}(\C_p)$ induced by $\ev_x$. 
Recall that the $\GSp_{2g}$-eigenvariety $\cD^M_g$ is endowed with a morphism
\[ \psi_g\colon\calH_g^M\to\cO(\cD^M_g) \] 
that interpolates the normalized systems of Hecke eigenvalues associated with the cuspidal $\GSp_{2g}$-eigenforms of level $\Gamma_1(N)\cap\Gamma_0(p)$. Also recall that the images of $T_{i,\ell}^{(g)}$ and $U_{i,p}^{(g)}$, $1\le i\le g$, are elements of $\cO(\cD_g^M)^\circ$.
For a classical point $x\in\cD^M_g(\Qp)$ let $\psi_x=\ev_x\ccirc\psi_g$. Let $f_x$ be the classical $\GSp_{2g}$-eigenform having system of Hecke eigenvalues $\psi_x$ and let $\rho_x\colon G_\Q\to\GSp_{2g}(\Qp)$ be the $p$-adic Galois representation attached to $f_x$. 
When $x$ varies, the traces of the representations $\rho_x$ can be interpolated into a pseudocharacter with values in $\cO(\cD^M_g)^\circ$. This is the main result of this section. 
Unfortunately the pseudocharacter obtained this way cannot be lifted to a representation with coefficients in $\cO(\cD^M_g)^\circ$. We will be able to obtain a lift only by working over a sufficiently small admissible subdomain of $cD^M_g$ (see Section \ref{galrepfam}).

\subsection{Classical results on pseudocharacters}

We refer to \cite[Sections 2 and 3]{rouquier} for the definition and basic properties of pseudocharacters. 
In this subsection $d$ is a positive integer, $A$ is a commutative ring with unit and $R$ is an $A$-algebra with unit (not necessarily commutative). If $G$ is a group and $T\colon G\to A$ is a map, we say that $T$ is a pseudocharacter if it can be extended to a pseudocharacter $A[G]\to A$. 
Recall that if $\tau\colon R\to\Mat_d(A)$ is a representation, the map $\Tr(\tau)\colon R\to A$ is a pseudocharacter of dimension $d$.
Thanks to the following result of Carayol, $\tau$ is uniquely determined by the pseudocharacter $\Tr(\tau)$. 

\begin{thm}\label{carayol}\cite{cartrace} 
Suppose that $A$ is a complete noetherian local ring. Let $A^\prime$ be a semilocal extension of $A$. Let $\tau^\prime\colon R\to\Mat_d(A^\prime)$ be a representation. Suppose that the traces of $\tau^\prime$ belong to $A$. Then there exists a representation $\tau\colon R\to\Mat_d(A)$, unique up to isomorphism over $A$, such that $\tau$ is isomorphic to $\tau^\prime$ over $A^\prime$. 
\end{thm}

%
Under some hypotheses on the ring $A$ it is known that every pseudocharacter arises as the trace of a representation. 
The first of the following two theorems is due to Taylor when $\car(A)=0$ and Rouquier when $\car(A)>d$; the second one was proved independently by Nyssen and Rouquier.

\begin{thm}\cite{taylorgal,rouquier}
\label{pseudotaylor}
Suppose that $A$ is an algebraically closed field of characteristic either $0$ or greater than $d$. Let $T\colon R\to A$ be a $d$-dimensional pseudocharacter. Then there exists a representation $\tau\colon R\to\Mat_d(A)$ such that $\Tr(\tau)=T$.
\end{thm}
%

\begin{thm}
\cite{nyssen}[Corollary 5.2]\cite{rouquier}
\label{pseudolift}
Suppose that $A$ is a local henselian ring in which $d!$ is invertible and let $\F$ denote the residue field of $A$. Let $T\colon R\to A$ be a pseudocharacter of dimension $d$ and $\ovl{T}\colon R\to\F$ be its reduction modulo the maximal ideal of $A$. 
Suppose that there exists an irreducible representation $\ovl{\tau}\colon R\to\Mat_d(\F)$ such that $\Tr(\ovl{\tau})=\ovl{T}$. Then there is an isomorphism $R/\ker T\cong\Mat_d(A)$ and the projection $R\to R/\ker T$ is a representation lifting $\ovl{\tau}$.
\end{thm}

\subsection{The characteristic polynomial of a pseudocharacter}

We introduce a notion of characteristic polynomial of a pseudocharacter. Let $\tau\colon G\to\GL_d(A)$ be a representation and let $T=\Tr(\tau)$. 
For $g\in G$ let $\alpha_1,\alpha_2,\ldots,\alpha_d$ be the eigenvalues of $\tau(g)$. For every $n\in\N$ we have $T(g^n)=\sum_{i=1}^d\alpha_i^n$, so the functions $T(g^n)$ generate over $\Q$ the ring of symmetric polynomials with rational coefficients in the variables $\alpha_1,\alpha_2,\ldots,\alpha_n$. We deduce that there exist polynomials $f_1,f_2,\ldots,f_d\in \Q[x_1,x_2,\ldots,x_d]$, independent og $g$, such that $\det(1-X\tau(g))=1+\sum_{i=1}^df_1(T(g),T(g^2),\ldots,T(g^d))X^i$.

\begin{defin}\label{charpoldef}
If $T\colon G\to A$ is a $d$-dimensional pseudocharacter, we let $P_\car(T)\colon G\to A[X]^{\deg=d}$ be the polynomial defined by
\[ P_\car(T)=1+\sum_{i=1}^df_1(T(g),T(g^2),\ldots,T(g^d))X^i, \]
where $f_1,f_2,\ldots,f_d$ are as in the discussion above. 
We call $P_\car(T)$ the \emph{characteristic polynomial} of $T$.
\end{defin}

For example for $d=2$ we have 
\begin{equation}\label{detpseudo} P_\car(T)(g)=1-T(g)X+\left(\frac{T(g)^2-T(g^2)}{2}\right)X^2. \end{equation}
Note that Chenevier constructed in \cite{chendet} objects called ``determinants'' that include information on the characteristic polynomials of the elements of the group algebra. In our setting we do not need this tool; the polynomial associated with $T$ by Definition \ref{charpoldef} is the one mentioned in \cite[Note 7]{chendet}.

For later use we introduce the notion of symmetric cube of a two-dimensional pseudocharacter.

\begin{defin}\label{sym3pseudo}
Let $T\colon G\to A$ be a two-dimensional pseudocharacter. The \emph{symmetric cube of $T$} is the pseudocharacter $\Sym^3T\colon G\to A$ defined by
\[ \Sym^3T(g)=\frac{T(g)^2(3T(g^2)-T(g)^2)}{2} \]
\end{defin}

This definition is justified by the remark below. 

\begin{rem}\label{sym3pseudorem}
Let $\tau\colon G\to\GL_2(A)$ be a representation and let $T=\Tr(\tau)$. 
Then the trace of the representation $\Sym^3\tau\colon G\to\GSp_4(A)$ is $\Sym^3T$. In particular $P_\car(\Sym^3T)(g)=\Sym^3P_\car(T)(g)$. 
We deduce from the definition of $P_\car$ that the equality $P_\car(\Sym^3T)=\Sym^3P_\car(T)$ must hold for every pseudocharacter $T\colon G\to A$ (not necessarily defined as the trace of a representation). 
\end{rem}

\subsection{Interpolation of the classical pseudocharacters}

As before let $g\in\{1,2\}$. Every classical point of $\cD^M_g$ admits an associated Galois representation. In this section we interpolate the trace pseudocharacters attached to these representations to construct a pseudocharacter over the eigenvariety. 

We remind the reader that for every ring $R$ we implicitly extend a character of the Hecke algebra $\calH^M_g\to R^\times$ to a morphism of polynomial algebras $\calH^M_g[X]\to R[X]$ by applying it to the coefficients. Recall that we fixed an embedding $G_{\Q_\ell}\into G_\Q$ for every prime $\ell$, hence an embedding of the inertia subgroup $I_\ell$ in $G_\Q$. As usual $\Frob_\ell$ denotes a lift of the Frobenius at $\ell$ to $G_{\Q_\ell}$.

Let $S^\cl$ denote the set of classical points of $\cD_g^M$. 
Let $x\in S^\cl$. We keep the notations $\ev_x$, $\psi_x$, $\rho_x$ as in the beginning of the section. We let $T_x\colon G_\Q\to\Qp$ be the pseudocharacter defined by $T_x=\Tr(\rho_x)$.

\begin{prop}\label{biggalthm} 
There exists a pseudocharacter $T_{g}\colon G_\Q\to\cO(\cD_g^M)$ 
of dimension $2g$ with the following properties:
\begin{enumerate}
\item for every prime $\ell$ not dividing $Np$ and every $h\in I_\ell$ we have $T_{g}(h)=2$, where $2\in\cO(\cD_g^M)$ denotes the function constantly equal to $2$;
\item for every prime $\ell$ not dividing $Np$ we have $P_\car(T_g)(\Frob_\ell)(X)=\psi_g(P_\Min(t_{\ell,g}^{(g)};X))$; 
\item for every $x\in S^\cl$ we have $\ev_x\ccirc T_{g}=T_x$.  
\end{enumerate}
\end{prop}

\begin{proof}
The pseudocharacter $T_g$ is constructed via the interpolation argument of \cite[Proposition 7.1.1]{chenfam}. Its properties are a consequence of those of the classical representations. See \cite[Theorem 3.5.10]{contith} for a detailed proof of the proposition.
\end{proof}

\begin{rem}\label{nonclassrep}\mbox{}
\begin{enumerate}
\item Let $x\in\cD_g^M(\Qp)$. Consider the $2g$-dimensional pseudocharacter $T_x\colon G_\Q\to\Qp$ defined by $T_x=\ev_x\ccirc T_{2g}$. 
By Theorem \ref{pseudotaylor} there exists a Galois representation $\rho_x\colon G_\Q\to\GL_4(\Qp)$ satisfying $T_x=\Tr(\rho_x)$. 
We will see in Section \ref{galrepfam} that, when $\ovl{\rho}_x$ is absolutely irreducible, $\rho_x$ is isomorphic to a representation $G_\Q\to\GSp_4(\Qp)$.
\item When $x$ varies in a connected component of $\cD_g^M$, the residual representation $\ovl\rho_x\colon G_\Q\to\GSp_{2g}(\Qp)$ is independent of $x$. We call it the residual representation associated with the component.
\end{enumerate}
\end{rem}

\bigskip

\section{Big image of Galois representations attached to $\GSp_4$-eigenforms}\label{resbigim}

Let $N$ be a positive integer and let $p$ be a prime not dividing $N$. Let $F$ be a $\GSp_4$-eigenform of level $\Gamma_1(N)$. 
Let $\rho_{F,p}\colon G_\Q\to\GSp_4(\Qp)$ be the $p$-adic Galois representation associated with $F$. It is defined over a $p$-adic field $K$. Under the technical condition of ``$\Z_p$-regularity'' of $\rho_{F,p}$ and an assumption on the associated residual representation, we prove that the image of $\rho_{F,p}$ is ``big'' when $F$ is not a lift from a $\GL_2$-eigenform. 
An important ingredient of the proof is a result that we will prove later, Theorem \ref{sym3autom}(i). More precisely we need the corollary that we state below.

\begin{cor}\label{sym3automcor}(Corollary \ref{sym3kimsha})
Suppose that there exists a representation $\rho^\prime\colon G_\Q\to\GL_2(\Qp)$ satisfying $\rho_{F,p}\cong\Sym^3\rho^\prime$. Then $F$ is the symmetric cube lift of a $\GL_2$-eigenform, defined by Corollary \ref{formtransf}. 
\end{cor}

Another crucial ingredient is a theorem of Pink that we recall below (Theorem \ref{pink}). 

In the following definitions, let $E$ be a finite extensions of $\Q_p$. Let $R$ be a local ring with maximal ideal $\fm_R$ and residue field $\F$. Let $\tau\colon G_E\to\GSp_4(R)$ be a representation. Let $\PGSp_4(R)=\GSp_4(R)/R^\times$, where $R^\times$ is identified with the subgroup of scalar matrices. 
We denote by $\ovl{\tau}\colon G_E\to\GSp_4(\F)$ the reduction of $\tau$ modulo $\fm_R$. 
Recall that $T_2$ is the torus consisting of diagonal matrices in $\GSp_4$. 
We give a notion of $\Z_p$-regularity of $\tau$, analogous to that in \cite[Lemma 4.5(2)]{hidatil}.

\begin{defin}\label{Zpreg}
%
We say that $\tau$ is \emph{$\Z_p$-regular} if there exists $d\in\im\tau\cap T_2(R)$ with the following property: if $\alpha$ and $\alpha^\prime$ are two distinct roots of $\GSp_4$ then $\alpha(d)\neq\alpha^\prime(d)\pmod{\fm_R}$. If $d$ has this property we call it a \emph{$\Z_p$-regular element}.
\end{defin}

From now on we focus on representations that are either ``residually full'' or ``residually of symmetric cube type'', in the sense of the definition below. Note that these two types of representations appear in \cite[Section 5.8]{pillab} as examples of those for which he can construct a sequence of Taylor-Wiles primes. 

\begin{defin}\label{sctype}
We say that $\tau$ is:
\begin{enumerate}
\item \emph{residually full} if there exists a non-trivial subfield $\F^\prime$ of $\F$ and an element $g\in\GSp_4(\F)$ such that
\[ \Sp_4(\F^\prime)\subset g(\im\ovl{\tau})g^{-1}\subset\GSp_4(\F^\prime); \] 
\item \emph{residually of symmetric cube type} if there exist a non-trivial subfield $\F^\prime$ of $\F$ and an element $g\in\GSp_4(\F)$ such that
\[ \Sym^3\SL_2(\F^\prime)\subset g(\im\ovl{\tau})g^{-1}\subset\Sym^3\GL_2(\F^\prime). \] 
\end{enumerate}
We also say that $\ovl\tau$ is \emph{full} in case (i) and \emph{of symmetric cube type} in case (ii).
\end{defin}

We write $\fsp_4(K)$ for the Lie algebra of $\Sp_4(K)$ and $\Ad\colon\GSp_4(K)\to\End(\fsp_4(K))$ for the adjoint representation. Let $F$ and $\rho_{F,p}\colon G_\Q\to\GSp_4(\cO_K)$ be as in the beginning of the section. Let $E$ be the subfield of $K$ generated over $\Q_p$ by the set $\{\Tr(\Ad(\rho(g)))\}_{g\in G_\Q}$. Let $\cO_E$ be the ring of integers of $E$. For a $\GL_2$-eigenform $f$, we denote by $\rho_{f,p}$ the associated $p$-adic Galois representation. We will prove the following result.

\begin{thm}\label{classbigim}
Assume that $\rho_{F,p}$ is $\Z_p$-regular and that one of the following two conditions is satisfied:
\begin{enumerate}[label=(\roman*)]
\item $\rho_{F,p}$ is residually full;
\item $F$ is not a $p$-stabilization of the symmetric cube lift of a $\GL_2$-eigenform, defined by Corollary \ref{formtransf}. 
\end{enumerate}
Then the image of $\rho_{F,p}$ contains a principal congruence subgroup of $\Sp_4(\cO_E)$.
\end{thm} 

For use in the proof of of Theorem \ref{classbigim} we state a result of Pink. 

\begin{thm}\cite[Theorem 0.7]{pink}\label{pink}
Let $L$ be a local field and let $H$ be an absolutely simple connected adjoint group over $L$. Let $\Gamma$ be a compact Zariski-dense subgroup of $H(L)$. Suppose that the adjoint representation of $\Gamma$ is irreducible. Then there exists a closed subfield $E$ of $L$ and a model $H_E$ of $H$ over $E$ such that $\Gamma$ is an open subgroup of $H_E(E)$. 
\end{thm}

We also need the following lemma.

\begin{lemma}\label{subsym3}
Let $\cG$ be a profinite group and let $\cG_1$ be a normal open subgroup of $\cG$. Let $L$ be a field. Let $\tau\colon\cG\to\GSp_4(L)$ be a continuous representation. Suppose that:
\begin{enumerate}
\item there exists a representation $\tau_1^\prime\colon\cG_1\to\GL_2(L)$ such that $\tau\vert_{\cG_1}\cong\Sym^3\tau_1^\prime$;
\item the image of $\tau_1^\prime$ contains a principal congruence subgroup of $\SL_2(L)$; 
\item there exists a character $\eta\colon\cG\to L^\times$ such that $\det\tau\cong\eta^6$.
\end{enumerate}
Then there exists a finite extension $\iota\colon L\into L^\prime$ and a representation $\tau^\prime\colon\cG\to\GL_2(L^\prime)$ such that $\iota\ccirc\tau\cong\Sym^3\tau^\prime$.
\end{lemma}

\begin{proof}
We show that there is a finite extension $L^\prime$ of $L$ such that $\iota\ccirc\tau(\cG)\subset\Sym^3\GL_2(L^\prime)$. For $g\in\GSp_4(L)$ let $\Ad(g)\colon\GSp_4(L)\to\GSp_4(L)$ be conjugation by $g$. Since $\cG_1$ is an open normal subgroup of $\cG$, $\tau(\cG)$ normalizes $\tau(\cG_1)$. Let $g$ be an arbitrary element of $\tau(\cG)$. The map $\Ad(g)$ restricts to an automorphism $\Ad(g)\vert_{\tau(\cG_1)}$ of $\tau(\cG_1)$. Since $\tau\vert_{\cG_1}\cong\Sym^3\tau_1^\prime$, the symmetric cube map induces an isomorphism $\tau(\cG_1)\cong\tau_1^\prime(\cG_1)$. Hence $\Ad(g)$ induces an automorphism $\Ad(g)^\prime$ of $\tau_1^\prime(\cG_1)$, which is a subgroup of $\GL_2(L)$ containing a congruence subgroup of $\SL_2(L)$. By applying Corollary \ref{congisom} to the map $\Ad(g)^\prime\colon\tau_1^\prime(\cG_1)\to\tau_1^\prime(\cG_1)$ we deduce that there exists $h_g\in\GL_2(L)$, a field automorphism $\sigma$ of $L$ and a character $\varphi\colon\tau_1^\prime(\cG_1)\to L^\times$ such that 
\begin{equation}\label{deltag} \Ad(g)^\prime(x)=\varphi(x)h_gx^\sigma h_g^{-1} \end{equation}
for every $x\in\cG_1$. Since every operation in Equation \eqref{deltag} is $L$-linear, the automorphism $\sigma$ must be the identity. Moreover $\Ad(g)^\prime$ is induced by $\Ad(g)$, so by taking characteristic polynomials on both sides of the equation we obtain that $\varphi$ is trivial. 
Hence Equation \eqref{deltag} gives $\Ad(g)\vert_{\tau(\cG_1)}=\Ad(\Sym^3h_g)\vert_{\tau(\cG_1)}$, so the element $g(\Sym^3h_g)^{-1}$ centralizes $\tau(\cG_1)$ and by Schur's lemma it is a scalar $\gamma_g\1_4$ for some $\gamma_g\in L$.

Choose a set of representatives $S$ for the finite group $\cG/\cG_1$. Let $L^\prime$ be the finite extension of $L$ obtained by adding the cubic roots of the elements in the set $\{\gamma_g\,\vert\, g\in\tau(S)\}$. Let $\iota\colon L\to L^\prime$ be the inclusion. For $g\in\iota\ccirc\tau(S)$ we have $\iota(\gamma_g\1_4)\in\Sym^3\GL_2(L^\prime)$ by construction of $L^\prime$, so $\iota(g)=\iota(\gamma_g\1_4\cdot\Sym^3h_g)\in\Sym^3\GL_2(L^\prime)$. For every $g\in\tau(\cG)$ we can write $g=g_1g_2$ with $g_1\in\tau(\cG_1)$ and $g_2\in\tau(S)$. Since $\tau(\cG_1)\subset\Sym^3\GL_2(L)$ we obtain $\iota(g)=\iota(g_1)\iota(g_2)\in\Sym^3\GL_2(L^\prime)$. 

For every $g\in G_\Q$, let $\tau^\prime(g)$ be the unique element of $\GL_2(L^\prime)$ that satisfies:
\begin{enumerate}
\item $\Sym^3\tau^\prime(g)=\iota\ccirc\tau(g)$;
\item $\det\tau^\prime(g)=\iota\ccirc\eta(g)$.
\end{enumerate}
Such an element exists by the result of the previous paragraph. Then the map $\tau^\prime\colon G_\Q\to\GL_2(L^\prime)$ defined by $g\mapsto\tau^\prime(g)$ is a representation satisfying $\Sym^3\tau^\prime\cong\iota\ccirc\tau$.
\end{proof}

The rest of the section is devoted to the proof of Theorem \ref{classbigim}. Let $(\im\rho_{F,p})^\prime$ be the derived subgroup of $\im\rho_{F,p}$ and let $G=(\im\rho_{F,p})\cap\Sp_4(K)$. We denote by $\ovl{G}$ the Zariski-closure of $G$ in $\Sp_4(K)$. As in \cite[Section 3]{hidatil}, we will show first that under the hypotheses of Theorem \ref{classbigim} we have $\ovl{G}=\Sp_4(K)$, and second that $G$ is $p$-adically open in $\ovl{G}$. We will replace the ordinarity assumption in \emph{loc. cit.} by that of $\Z_p$-regularity. Let $\ovl{G}^\circ$ denote the connected component of the identity in $\ovl{G}$.

Let $H$ be any connected, Zariski-closed subgroup of $\Sp_4$, defined over $K$. 
As in \cite[Section 3.4]{hidatil} we have six possibilities for the isomorphism class of $H$ over $K$:
\begin{enumerate}
\item $H\cong\Sp_4$;
\item $H\cong\SL_2\times\SL_2$;
\item $H\cong\SL_2$ embedded in a Klingen parabolic subgroup;
\item $H\cong\SL_2$ embedded in a Siegel parabolic subgroup;
\item $H\cong\SL_2$ embedded via the symmetric cube representation $\SL_2\to\Sp_4$ (in this case we write $H\cong\Sym^3\SL_2$);
\item $H\cong\{1\}$.
\end{enumerate}
We show that only (1) is possible for $H=\ovl{G}^\circ$.

\begin{lemma}\label{sym3zar}
If condition (i) or (ii) in Theorem \ref{classbigim} holds, then $\ovl{G}^\circ\cong\Sp_4$. 
\end{lemma}

\begin{proof}
Let $\fm_K$ be the maximal ideal of $\cO_K$ and let $\F_K=\cO_K/\fm_K$. 
The group $(\im\rho_{F,p})^\prime$ is contained in $\ovl{G}^\circ(\cO_K)$. 
By reducing modulo $\fm_K$ we obtain that the derived subgroup $(\im\ovl\rho_{F,p})^\prime$ of $\im\ovl\rho_{F,p}$ is contained in $\ovl{G}^\circ(\F_K)$. If $\rho_{F,p}$ is residually full, then the only choice for the isomorphism class of $\ovl{G}^\circ$ is $\ovl{G}^\circ\cong\Sp_4$. 
If $\rho_{F,p}$ is residually of symmetric cube type, then either $\ovl{G}^\circ\cong\Sp_4$ or $\ovl{G}^\circ\cong\Sym^3\SL_2$. 

Suppose that $\ovl{G}^\circ\cong\Sym^3\SL_2$. We show that there exists a $\GL_2$-eigenform $f$ such that $\rho_{F,p}\cong\Sym^3\rho_{f,p}$. This will contradict the second part of condition (ii) of Theorem \ref{classbigim}, concluding the proof of Lemma \ref{sym3zar}. 
%
%
Since $\ovl{G}^\circ(K)$ is of finite index in $\ovl{G}(K)$, Lemma \ref{subsym3} implies that $\ovl{G}(K)\subset\Sym^3\SL_2(K)$, so $\im\rho_{F,p}\subset\Sym^3\GL_2(K)$. 
Hence there exists a representation $\rho^\prime$ satisfying $\rho_{F,p}\cong\Sym^3\rho^\prime$. Since $\rho_{F,p}$ is associated with a $\GSp_4$-eigenform, Corollary \ref{sym3automcor} implies that $\rho^\prime$ is associated with a $\GL_2$-eigenform $f$. 
\end{proof}

The proof of Theorem \ref{classbigim} is completed by the following proposition.

\begin{prop}
Suppose that $\ovl{G}\cong\Sp_4(K)$. Then the group $G$ contains an open subgroup (for the $p$-adic topology) of $\Sp_4(E)$.
\end{prop}

\begin{proof}
Consider the image $G^\ad$ of $G$ under the projection $\Sp_4(K)\to\PGSp_4(K)$. It is a compact subgroup of $\PGSp_4(K)$. Since $\ovl{G}\cong\Sp_4(K)$, the group $G^\ad$ is Zariski-dense in $\PGSp_4(K)$.
By Theorem \ref{pink} there is a model $H$ of $\PGSp_4$ over $E$ such that $G^\ad$ is an open subgroup of $H(E)$.
By the assumption of $\Z_p$-regularity of $\rho$, there is a diagonal element $d$ with pairwise distinct eigenvalues. The group $H(E)$ must contain the centralizer of $d$ in $\PGSp_4(E)$, which is a split torus in $\PGSp_4(E)$. 
Since $H$ is split and $H\times_EK\cong\PGSp_{4/K}$, $H$ is a split form of $\PGSp_4$ over $E$. Then $H$ must be isomorphic to $\PGSp_4$ over $E$ by unicity of the quasi-split form of a reductive group. Hence $G^\ad$ is an open subgroup of $\PGSp_4(E)$.
Since the map $\Sp_4(K)\to\PGSp_4(K)$ has degree $2$ and $G\cap\Sp_4(E)$ surjects onto $G^\ad\cap\PGSp_4(E)$, $G$ must contain an open subgroup of $\Sp_4(E)$. In particular $G$ contains a principal congruence subgroup of $\Sp_4(\cO_E)$. 
\end{proof}
%

Theorem \ref{classbigim} states that, when $\ovl\rho_{F,p}$ is either full or of symmetric cube type, the image of $\rho_{F,p}$ is large if and only if $F$ is not a lift of an eigenform from a smaller group, the only possible such lift under these assumptions being associated with the symmetric cube representation of $\GL_2$. We think that a similar result should hold under more general assumptions on the residual representation, and that it would follow from Pink's theorem together with an analogue of Corollary \ref{sym3automcor} for the other possible Langlands lifts to $\GSp_4$.


\bigskip

\section{Finite slope families of $\GSp_{2g}$-eigenforms}\label{gspfam}

In this section we define families of finite slope $\GSp_{2g}$-eigenforms of level $\Gamma_1(N)\cap\Gamma_0(p)$, extending the definitions given in \cite[Section 3.1]{cit} for $g=1$. Our goal is to define such families integrally. In the following sections we will only use families of genus 1 or 2, but we can give the definitions for general genus with no extra effort.

Let $p$ be a prime number and let $N$ be a positive integer prime to $p$. For $g\geq 1$ let $\cD_g^{N,h}$ be the $\GSp_{2g}$-eigenvariety of tame level $\Gamma_1(N)$. Let $h\in\Q^{+,\times}$. Since the slope $\slo\colon\cD_g^{N,h}(\C_p)\to\R^{\geq 0}$ is the valuation of a rigid analytic function on $\cD_g^{N,h}$, the locus of $\C_p$-points $x\in\cD_g^N$ satisfying $\slo(x)\leq h$ admits a structure of rigid analytic subvariety of $\cD_g^N$. 
We denote it by $\cD_{g,h}^N$. 
Recall that we always identify the $g$-dimensional weight space $\cW_g$ with a disjoint union of open discs of centre $0$ and radius $1$. 
A standard way to obtain an integral structure on an admissible domain of an eigenvariety is to use the integral structure on the weight space via the weight map. The restriction of the weight map to $\cD_{g}^{N,h}$ is in general not finite if $h>0$, but it becomes finite when restricted to a sufficiently small admissible domain in $\cD_g^{N,h}$. This is assured by a result of Bella\"\i che that we recall below. 
For every affinoid subdomain $V$ of $\cW^\circ_g$, let $\cD_{g,V}^{N,h}=\cD_g^{N,h}\times_{\cW^\circ_g}V$ and let $w^h_{g,V}=w_{g,h}\vert_{\cD_{g,V}^{N,h}}\colon\cD_{g,V}^{N,h}\to V$. 

\begin{prop}(Bella\"\i che)\label{locfin}\mbox{ }
\begin{enumerate}
\item For every $\kappa\in\cW^\circ_g(\Qp)$ there exists an affinoid neighborhood $V_{h,\kappa}$ of $\kappa$ in $\cW^\circ_g$ such that the map $w^h_{g,V_{h,\kappa}}$ is finite.
\item When $h$ varies in $\Q^{+,\times}$ and $\kappa$ varies in $\cW^\circ_g$, the set $\{(w^h_{g,V_{h,\kappa}})^{-1}(V_{h,\kappa})\}_{h,\kappa}$ is an admissible affinoid covering of $\cD_g^N$.
\end{enumerate}
\end{prop}

In Bella\"\i che's terminology, a pair $(V_{h,\kappa},h)$ such that $V_{h,\kappa}$ has the property described in (1) is called an \emph{adapted pair}. 
Part (1) of Proposition \ref{locfin} follows from the fact that the characteristic power series of $U_p^{(2)}$ acting on modules of overconvergent eigenforms is strictly convergent, in particular from the calculation in \cite[Proposition II.1.12]{bellaiche} and the fact that the map from the eigenvariety to the spectral variety is finite. Part (2) follows from (1) together with the admissibility of Buzzard's covering of the spectral variety (\cite[Theorem 4.6]{buzzard}) and the construction of the eigenvariety (see \cite[Theorem II.3.3]{bellaiche}).

\begin{rem}\label{adaptrad}\mbox{ }
\begin{enumerate}
\item Every affinoid neighborhood of $\kappa\in\cW^\circ_g$ contains a wide open disc centred in $\kappa$. Proposition \ref{locfin} implies that there exists a radius $r_{h,\kappa}\in p^\Q$ such that 
\[ w^h_{g,B_g(\kappa,r_{h,\kappa}^-)}\colon\cD_{g,B_g(\kappa,r_{h,\kappa}^-)}^{N,h}\to B_g(\kappa,r_{h,\kappa}^-) \]
is a finite morphism.
\item Thanks to Hida theory for $\GSp_4$ we know that the ordinary eigenvariety $\cD_g^{M,0}$ is finite over $\cW^\circ_g$. Hence we can take $r_{0,\kappa}=1$ for every $\kappa$.
\item We would like to have an estimate for $r_{h,\kappa}$ independent of $\kappa$ and with the property that $r_{h,\kappa}\to 0$ for $h\to 0$, in order to recover the ordinary case in this limit. This is not available at the moment for the group $\GSp_{2g}$. An estimate of the analogue of this radius is known for the eigenvarieties associated with unitary groups compact at infinity by the work of Chenevier \cite[Théorème 5.3.1]{chenfam}. 
\end{enumerate}
\end{rem}

%

\subsection{Families defined over $\Z_p$}\label{Zpfam}

For our purpose of studying the images of Galois representations, we will need to have our finite slope families defined over $\Z_p$. For this reason we specialize to families over weight discs for which we can construct a $\Z_p$-model. For simplicity we only work on the connected component $\cW_g^\circ$. Recall that we defined coordinates $T_1,T_2,\ldots,T_g$ on $\cW_g^\circ$. Let $\kappa$ be a point of $\cW_g^\circ$ with coordinates $(\kappa_1,\kappa_2,\ldots,\kappa_g)$ in $\Z_p^g$; for instance we can take as $\kappa$ the arithmetic prime $P_\uk$ for some $\uk\in\Z^g$. 
Let $r_{h,\kappa}$ be the largest radius in $p^\Q$ such that the map $w_{\kappa,B_g(\kappa,r_{h,\kappa}^-)}\colon\cD_{g,B_g(\kappa,r_{h,\kappa}^-)}^{N,h}\to B_g(\kappa,r_{h,\kappa}^-)$ is finite.
Such a radius is non-zero thanks to Remark \ref{adaptrad}(1). 
Let $s_h$ be a rational number satisfying $r_h=p^{s_h}$. 
%
%
We define a model for $B_g(\kappa,r_h^-)$ over $\Q_p$ by adapting Berthelot's construction for the wide open unit disc (see \cite[Section 7]{dejong}). Write $s_h=\frac{b}{a}$ for some $a,b\in\N$. For $i\geq 1$, let $s_i=s_h+1/2^i$ and $r_i=p^{-s_i}$. 
Set
\[ A_{r_i}^\circ=\Z_p\langle t_1,t_2,\ldots,t_g,X_i\rangle/(t_j^{2^ia}-p^{a+2^ib}X_i)_{j=1,2,\ldots,g} \]
and $A_{r_i}=A_{r_i}^\circ[p^{-1}]$. 
Set $B_i=\Spm A_{r_i}$. Then $B_i$ is a $\Q_p$-model of the disc of centre $\kappa$ and radius $r_i$. 
We define morphisms $A_{r_{i+1}}^\circ\to A_{r_i}^\circ$ by
\begin{gather*}
X_{i+1}\mapsto p^aX_i^2, \\
t_j\mapsto t_j\textrm{ for }j=1,2,\ldots,g,
\end{gather*}
They induce compact maps $A_{r_{i+1}}\to A_{r_i}$ which give open immersions $B_i\into B_{i+1}$. We define $B_{g,h}=\varinjlim_i B_i$ where the limit is taken with respect to the above immersions. Let $\Lambda_{g,h}=\cO(B_{g,h})^\circ$. Then $\Lambda_{g,h}=\varprojlim_i\cO(\Spm B_i)^\circ=\varprojlim_i A_{r_i}^\circ$. We call $\Lambda_{g,h}$ the \emph{genus $g$, $h$-adapted Iwasawa algebra}; we leave its dependence on $\kappa$ implicit. We define $t_1,t_2,\ldots,t_g\in\Lambda_{2,h}$ as the projective limits of the variables $t_1,t_2,\ldots,t_g$, respectively, of $A_{r_i}^\circ$.

There is a map of $\Z_p$-algebras $\iota_{g,h}^\ast\Lambda_h\to\Lambda_{h,g}$ defined by $T_j\mapsto t_j+\kappa_j$ for $j=1,2,\ldots,g$. The inclusion $\iota_{g,h}\colon B_{g,h}\into\cW_g^\circ$ induced by $\iota_{g,h}^\ast$ makes $B_{g,h}$ into a $\Q_p$-model of $B_g(\kappa,r_h^-)$, endowed with the integral structure defined by $\Lambda_{g,h}$.

Let $\eta_h$ be an element of $\ovl{\Q}_p$ satisfying $v_p(\eta_h)=s_h$. Let $K_h=\Q_p(\eta_h)$ and let $\cO_h$ be the ring of integers of $K_h$. 
The algebra $\Lambda_{g,h}$ is not a ring of formal series over $\Z_p$, but there is an isomorphism $\Lambda_h\otimes_{\Z_p}\cO_h\cong\cO_h[[t_1,t_2,\ldots,t_g]]$.

We say that a prime of $\Lambda_{g,h}$ is arithmetic if it lies over an arithmetic prime of $\Lambda_g$. 
By an abuse of notation we will write again $P_\uk$ for an arithmetic prime of $\Lambda_{g,h}$ lying over the arithmetic prime $P_\uk$ of $\Lambda_g$. 

\begin{rem}\label{arithprimes}
Let $\uk=(k_1,k_2,\ldots,k_g)$ be a cohomological weight for $\GSp_{2g}$. 
There exists a prime $\fP$ of $\Lambda_{g,h}$ lying over the prime $P_\uk$ of $\Lambda_h$ if and only if the classical weight $\uk$ belongs to the disc $B_g(0,r_h^-)$; otherwise we have $P_\uk\Lambda_h=\Lambda_h$. This happens if and only if $v_p(k_i)>-v_p(r_h)-1$ for i=1,2,\ldots,g, as we can see via a simple calculation. 
\end{rem}

Let $\cD_{g,B_{g,h}}^{N,h}$ be the rigid analytic space and $w_{g,h}$ be the morphism fitting in the cartesian diagram
\begin{equation}\label{Dgh}
\begin{tikzcd}
\cD_{g,h}^{N}\times_{\cW^\circ_g}B_{g,h} \arrow{d}{w_{g,h}}\arrow{r}
&\cD_{g}^{N,h}\arrow{d}{w_g}\\
B_{g,h}\arrow{r}{\iota_{g,h}}
&\cW^\circ_g
\end{tikzcd}
\end{equation}
The rigid analytic space $\cD_{g,h}^N$ is a model of $\cD_{g,B_g(\kappa,r_{h,\kappa}^-)}^{N,h}$ over a $p$-adic field, but it is not necessarily defined over $\Q_p$ since the map $\iota_{g,h}$ may not be. 
We say that a $\C_p$-point of $\cD_{g,h}^N$ is classical if it is a classical point of $\cD_{g,B_g(\kappa,r_{h,\kappa}^-)}^{N,h}$. 

Let $\T_{g,h}=\cO(\cD_{g,h}^{N})^\circ$. We call $\T_{2,h}$ the \emph{genus $g$, $h$-adapted Hecke algebra}; we leave its dependence on $\kappa$ implicit again. 
The morphism $w_{g,h}$ induces $w_{g,h}^\ast\colon\Lambda_{h,g}\to\T_{g,h}$. Thanks to our choice of $r_h$, $w_{g,h}^\ast$ gives $\T_{g,h}$ a structure of finite $\Lambda_{g,h}$-algebra. The $\cD_{g,h}^{N}\to\cD_{g,B_g(\kappa,r_{h,\kappa}^-)}^{N,h}$ appearing in the diagram induces a map $\cO(\cD_{g,B_g(\kappa,r_{h,\kappa}^-)}^{N,h})^\circ\to\T_{g,h}$, that we compose with $\psi_g\colon\calH^N_g\to\cO(\cD_{g,B_g(\kappa,r_{h,\kappa}^-)}^{N,h})^\circ$ to obtain a morphism $\psi_{g,h}\colon\calH_g^N\to\T_{g,h}$. 

For a prime $\fP$ of $\T_{g,h}$ we denote by $\ev_\fP\colon\T_{g,h}\to\Zp$ the evaluation at $\fP$. We say that $\fP$ is a classical point of $\Spec\T_{g,h}$ if $\ev_\fP\ccirc\psi_{g,h}\colon\calH_g^N\to\Zp$ is the system of Hecke eigenvalues attached to a classical $\GSp_{2g}$-eigenform. These systems of eigenvalues also appear at classical points of $\cD_{g,h}^N$.

\begin{defin}
We call \emph{family of $\GSp_{2g}$-eigenforms of slope bounded by $h$} an irreducible component $I$ of $\cD_{g,h}^{N}$, equipped with the integral structure defined by $\T_{g,h}$. 
\end{defin}

We will usually refer to an $I$ as in the definition simply as a finite slope family. Let $\I^\circ=\cO(I)$. Then $\I^\circ$ is a finite $\Lambda_h$-algebra and $I$ is determined by the surjective morphism $\theta\colon\T_{g,h}\onto\I^\circ$. We sometimes refer to this morphism as a finite slope family. 
The family $I$ is equipped with maps $w_\theta\colon I\to B_{g,h}$ and $\psi_\theta\colon\calH_g^N\to\I^\circ$ induced by $w_{g,h}$ and $\psi_{g,h}$, respectively. 
The notation ${}^\circ$ denotes the fact that we are working with integral objects. 
When introducing relative Sen theory in Section \ref{sen} we will need to invert $p$ and we will drop the ${}^\circ$ from all rings.

\begin{rem}\label{conncomp}
Since $\Lambda_h$ is profinite and local and $\T_{g,h}$ is finite over $\Lambda_h$, $\T_{g,h}$ is profinite and semilocal. The connected components of $\cD_{g,B_{g,h}}^{M,h}$ are in bijection with the maximal ideals of $\T_{g,h}$. Let $I$ and $\theta$ be as above. 
Then $\ker\theta$ is contained in the unique maximal ideal $\fm_\theta$ corresponding to the connected component of $\cD_{g,h}^{N}$ containing $I$. The $\Lambda_h$-algebra $\I^\circ$ is profinite and local with maximal ideal $\fm_{\I^\circ}=\theta(\fm_\theta)$.
\end{rem}


%
%
%

Proposition \ref{siegclslopes} implies that every family $I$ contains at least a classical point. By the accumulation property of classical point and the irreducibility of $I$, the classical points are a Zariski dense subset of $I$. Hence the set of classical points of $\Spec\I^\circ$ is also Zariski dense in $\Spec\I^\circ$. Every classical point of $\Spec\I^\circ$ lies over an arithmetic prime of $\Spec\Lambda_{g,h}$.

\subsection{Non-critical points on families} 

Let $\theta\colon\T_{g,h}\to\I^{\circ}$ be a family of $\GSp_4$-eigenforms.

\begin{defin}\label{noncritgsp}
We call an arithmetic prime $P_\uk\subset\Lambda_{g,h}$ \emph{non-critical for $\I^\circ$} if:
\begin{enumerate}
\item every point of $\Spec\I^\circ$ lying over $P_\uk$ is classical;
\item the map $w_{g,B_{g,h}}^\ast\colon\Lambda_{g,h}\to\I^\circ$ is étale at every point of $\Spec\I^\circ$ lying over $P_\uk$.
\end{enumerate}
We call $P_\uk$ \emph{critical} for $\I^\circ$ if it is not non-critical. We also say that a classical weight $\uk$ is critical or non-critical for $\I^\circ$ if the arithmetic prime $P_\uk$ has that property.
\end{defin}

\begin{rem}
By Proposition \ref{siegclslopes}, if $\uk$ is a classical weight belonging to $B_{g,h}$ and $h<k_g-\frac{g(g+1)}{2}$ then $\uk$ satisfies condition (i) of Definition \ref{noncritgsp}. 
We do not know of a simple assumption on the weight that guarantees that the second condition is also satisfied. 
\end{rem}


For later use we state a simple lemma.

\begin{lemma}\label{noncritdense}
The set of non-critical arithmetic primes is Zariski-dense in $\Lambda_{h}$.
\end{lemma}

\begin{proof}
This follows from Proposition \ref{siegclslopes} and the fact that the locus of étaleness of the morphism $\Lambda_{2,h}\to\I^\circ$ is Zariski-open in $\I^\circ$. The proof is detailed in \cite[Proposition 4.1.17]{contith}.
%
\end{proof}

\subsection{The Galois representation associated with a finite slope family}\label{galrepfam}

Let $g=2$. For $h\in\Q^{+,\times}$ and $\kappa\in\cW_2^\circ$, let $r_{h,\kappa}$ be the radius chosen in the beginning of Section \ref{Zpfam}. Let $\cE$ be the set of all wide open subdomains $D$ of $\cD_2^N$ with the property that $D$ is a connected component of $(w^h_{2,B_2(\kappa,r_{h,\kappa})})^{-1}(B_2(\kappa,r_{h,\kappa}))$ for some $h$ and $\kappa$. It follows from Proposition \ref{locfin}(ii) that $\cE$ is an admissible covering of the eigenvariety $\cD_2^N$. 

Let $D\in\cE$. The pseudocharacter $T_{2}\colon G_\Q\to\cO(\cD_2^{N})^\circ$ given by Proposition \ref{biggalthm} induces a pseudocharacter $T_D\colon G_\Q\to\cO(D)^\circ$. Since $D$ is connected $\cO(D)^\circ$ is local; it is also compact by \cite[Lemma 7.2.11]{bellchen} because $D$ is wide open. Let $\ovl{T}_D$ be the reduction of $T_D$ modulo the maximal ideal of $\cO(D)^\circ$. By Theorem \ref{pseudotaylor} $\ovl{T}_D$ is the trace of a representation $\ovl\rho_D\colon G_\Q\to\GL_4(\Fp)$. 


Note that $\ovl\rho_D$ only depends on the irreducible component of $\cD_2^N$ in which $D$ is contained. Let $\cE^\irr$ be the subset of $\cE$ consisting of the domains $D$ for which $\ovl\rho_D$ is absolutely irreducible. Let $\cD_2^{M,\irr}=\bigcup_{D\in\cE^\irr}D$; it is a union of connected components of $\cD_2^M$ and it admits $\cE^\irr$ as an admissible covering. 
For $D\in\cD^\irr$, Theorem \ref{pseudolift} gives a representation $\rho_D\colon G_\Q\to\GL_4(\cO(D)^\circ)$ that has $T_D$ as its associated pseudocharacter and is uniquely determined up to isomorphism. This is actually a symplectic representation, but we only need this property in the specific case that we treat below, where $D$ is a finite slope family as defined in Section \ref{Zpfam}. The construction we presented in this paragraph will be useful in Section \ref{tripar}.


Let $\T_{2,h}$ be the genus $2$, $h$-adapted Hecke algebra. For simplicity let $\T_h=\T_{2,h}$. We implicitly replace $\T_h$ by one of its local components. Let $\theta\colon\T_h\to\I^\circ$ be a finite slope family of $\GSp_4$-eigenforms. The Galois representation associated with $\theta$ can be constructed in the same way as for the domains $D\in\cE$ of the previous paragraph, but we define it here in more detail and prove some additional properties. 
Let $\F_{\T_h}$ be the residue field of $\T_h$. 
The pseudocharacter $T_{2}\colon G_\Q\to\cO(\cD_2^{N})^\circ$ induces pseudocharacters $T_{\T_h}\colon G_\Q\to\T_h$ and $\ovl{T}_{\T_h}\colon G_\Q\to\F_{\T_h}$. By Theorem \ref{pseudotaylor} the pseudocharacter $\ovl{T}_{\T_h}$ is associated with a representation $\ovl{\rho}_{\T_h}\colon G_\Q\to\GL_4(\ovl{\F}_p)$, unique up to isomorphism. We call $\ovl{\rho}_{\T_h}$ the residual Galois representation associated with $\T_h$. 

We assume from now on that \emph{the representation $\ovl{\rho}_{\T_h}$ is absolutely irreducible.}

By the compactness of $G_\Q$ there exists a finite extension $\F^\prime$ of $\F_{\T_h}$ such that $\ovl{\rho}_{\T_h}$ is defined on $\F^\prime$. Let $W(\F_{\T_h})$ and $W(\F^\prime)$ be the rings of Witt vectors of $\F_{\T_h}$ and $\F^\prime$, respectively. Let $\T_h^\prime=\T_h\otimes_{W(\F_{\T_h})}W(\F^\prime)$. 
We consider $T_{\T_h}$ as a pseudocharacter $G_\Q\to\T_h^\prime$ via the natural inclusion $\T_h\into\T_h^\prime$. Then $T_{\T_h}$ satisfies the hypotheses of Theorem \ref{pseudolift}, so there exists a representation $\rho_{\T_h^\prime}\colon G_\Q\to\GL_4(\T_h^\prime)$ such that $\Tr\rho_{T_h^\prime}=\T_{\T_h}$. By Proposition \ref{biggalthm}, for every prime $\ell$ not dividing $Np$ we have
\begin{equation}\label{traceTh} \Tr(D_{\T_h})(\Frob_\ell)=r_{\cD_{2,B_h}^{M,h}}\ccirc\psi_2(T_{\ell,2}^{(2)}). \end{equation}
In particular $\Tr(D_{\T_h})(\Frob_\ell)$ is an element of $\T_h$. Since $\T_h$ is complete, Chebotarev's theorem implies that $\T_{\T_h}(g)$ is an element of $\T_h$ for every $g\in G_\Q$. By Theorem \ref{carayol} there exists a representation $\rho_{\T_h}\colon G_\Q\to\GL_4(\T_h)$ that is isomorphic to $\rho_{\T_h}$ over $\T_h^\prime$. 

The morphism $\theta\colon\T_h\to\I^\circ$ induces a morphism $\GL_4(\T_h)\to\GL_4(\I^\circ)$ that we still denote by $\theta$. Let $\rho_{\I^\circ}\colon G_\Q\to\GL_4(\I^\circ)$ be the representation defined by $\rho_{\I^\circ}=\theta\ccirc\rho_{\T_h}$. 
Recall that we set $\psi_\theta=\theta\ccirc r_{\cD_{2,B_h}^{M,h}}\ccirc\psi_2\colon\calH_2^M\to\I^\circ$. Let 
\[ \I^\circ_\Tr=\Lambda_h[\{\Tr(\rho_\theta(g))\}_{g\in G_\Q}]. \]
Since $\Lambda_h\subset\I_\Tr^\circ\subset\I^\circ$, the ring $\I_\Tr^\circ$ is a finite $\Lambda_h$-algebra. In particular $\I_\Tr^\circ$ is complete. 
We keep our usual notation for the reduction modulo an ideal $\fP$ of $\I_\Tr^\circ$. 
%
We say that a point $\fP$ of $\Spec\I_\Tr^\circ$ is classical if it lies under a classical point of $\Spec\I^\circ$. 

By Proposition \ref{biggalthm} we have $P_\car(\Tr(\rho_{\I^\circ})(\Frob_\ell))=\psi_\theta(P_\Min(t_{\ell,2}^{(2)};X))$, so we deduce that $\I_\Tr^\circ=\Lambda_h[\{\Tr(\rho_\theta(g))\}_{g\in G_\Q}]$. 
Since the traces of $\rho_{\I^\circ}$ belong to $\I_\Tr^\circ$, Theorem \ref{carayol} provides us with a representation 
\[ \rho_\theta\colon G_\Q\to\GL_4(\I_\Tr^\circ) \]
that is isomorphic to $\rho_{\I^\circ}$ over $\I^\circ$. 
Thanks to the following lemma we can attach to $\theta$ a \emph{symplectic} representation. 

\begin{lemma}\label{sympform}
There exists a non-degenerate symplectic bilinear form on $(\I_\Tr^\circ)^4$ that is preserved up to a scalar by the image of $\rho_\theta$. 
\end{lemma}

\begin{proof}
The argument of the proof is similar to that in \cite[Lemma 4.3.3]{gentil} and \cite[Proposition 6.4]{pillab}. We show that $\rho_\theta$ is essentially self-dual by interpolating the characters that appear in the essential self-duality conditions at the classical specializations. We deduce that $\im\rho_\theta$ preserves a bilinear form on $(\I_\Tr^\circ)^4$ up to a scalar. Such a form is non-degnerate by the irreducibility of $\rho_\theta$ and it is symplectic because its specialization at a classical point is symplectic. The details of the proof can be found in \cite[Proposition 4.1.20]{contith}.
\end{proof}

Thanks to the lemma, up to replacing it by a conjugate representation, we can suppose that $\rho_\theta$ takes values in $\GSp_4(\I_\Tr^\circ)$. We call $\rho_\theta\colon G_\Q\to\GSp_4(\I_\Tr^\circ)$ the Galois representation associated with the family $\theta\colon\T_h\to\I^\circ_\Tr$. In the following we will work mainly with this representation, so we denote it simply by $\rho$. We write $\F$ for the residue field of $\I^\circ_\Tr$ and $\ovl{\rho}\colon G_\Q\to\GSp_4(\F)$ for the residual representation associated with $\rho$. 


\begin{rem}\label{detformula}
Let $f$ be a $\GSp_4$-eigenform appearing in the family $\theta$. Let $\varepsilon_f$ be the central character, $(k_1,k_2)$ the weight and $\psi_f\colon\calH_2^M\to\Qp$ the system of Hecke eigenvalues of $f$. Let $\rho_{f,p}$ be the $p$-adic Galois representation attached to $f$ and let $\ell$ be a prime not dividing $Mp$. Then $\det\rho_{f,p}(\Frob_\ell)=\ell^6\psi_f(T_{\ell,0}^{(2)})=\varepsilon_f(\ell)\chi(\ell)^{2(k_1+k_2-3)}$. 
The determinant of $\rho(\Frob_\ell)$ interpolates the determinants of $\rho_{f,p}(\Frob_\ell)$ when $f$ varies over the forms corresponding to the classical primes of the family. Note that $\varepsilon_f$ is independent of the choice of the form $f$ in the family. Since the classical primes are Zariski-dense in $\I^\circ_\Tr$ the interpolation is unique and coincides with $\det\rho(\Frob_\ell)=\ell^6\psi_2(T_{\ell,0}^{(2)})=\varepsilon(\ell)(u^{-6}(1+T_1)(1+T_2))^{\log(\chi(\ell))/\log(u)}\in\Lambda_{2,h}$, 
where $\varepsilon$ is the central character of the family. By density of the conjugates of the Frobenius elements in $G_\Q$, we deduce that
\[ \det\rho(g)=\varepsilon(g)(u^{-6}(1+T_1)(1+T_2))^{2\log(\chi(g))/\log(u)}\in\Lambda_{2,h} \]
for every $g\in G_\Q$.
\end{rem}

\begin{rem}\label{ZarItr}\mbox{ }
\begin{enumerate}
\item Since the set of classical points of $\Spec\I^\circ$ is Zariski-dense and the map $\I_\Tr^\circ\to\I^\circ$ is injective, the set of classical points of $\Spec\I^\circ_\Tr$ is also Zariski-dense.
\item Let $\fP$ be a classical point of $\Spec\I^\circ$ lying over a point $\fP_\Tr$ of $\Spec\I_\Tr^\circ$. Then the reduction of $\rho_{\I^\circ}\colon G_\Q\to\GL_4(\I^\circ)$ modulo $\fP$ is isomorphic over $\I^\circ/\fP$ to the reduction of $\rho$ modulo $\fP_\Tr$; in particular it only depends on $\fP_\Tr$ up to isomorphism over a suitable coefficient ring. 
\end{enumerate}
\end{rem}

\bigskip

\section{Self-twists of Galois representations attached to finite slope families}\label{selftwistsec}

Given a ring $R$, we denote by $Q(R)$ its total ring of fractions and by $R^\norm$ its normalization. Now let $R$ be an integral domain. For every homomorphism $\sigma\colon R\to R$ and every $\gamma\in\GSp_4(R)$ we define $\gamma^\sigma\in\GSp_4(R)$ by applying $\sigma$ to each coefficient of the matrix $\gamma$. This way $\sigma$ induces an automorphism $[\cdot]^\sigma\colon G(R)\to G(R)$ for every algebraic subgroup $G\subset\GSp_4$ defined over $R$. For such a $G$ and any representation $\rho\colon G_\Q\to G(R)$, we define a representation $\rho^\sigma\colon G_\Q\to G(R)$ by setting $\rho^\sigma(g)=(\rho(g))^\sigma$ for every $g\in G_\Q$.

Let $S$ be a subring of $R$. We say that a homomorphism $\sigma\colon R\to R$ is a homomorphism of $R$ over $S$ if the restriction of $\sigma$ to $S$ is the identity. The following definition is inspired by \cite[Section 3]{ribetII} and \cite[Definition 2.1]{lang}.

\begin{defin}\label{selftwist}
Let $\rho\colon G_\Q\to\GSp_4(R)$ be a representation. We call \emph{self-twist for $\rho$ over $S$} an automorphism $\sigma$ of $R$ over $S$ such that there is a finite order character $\eta_\sigma\colon G_\Q\to R^\times$ and an isomorphism of representations over $R$:
\begin{equation}\label{steq} \rho^\sigma\cong\eta_\sigma\otimes\rho. \end{equation}
\end{defin}

We list some basic facts about self-twists. The proofs are straghtforward.

\begin{prop}\label{basicst}
Let $\rho\colon G_\Q\to\GSp_4(R)$ be a representation.
\begin{enumerate}
\item The self-twists for $\rho$ over $S$ form a group. 
\item If $R$ is finite over $S$ then the group of self-twists for $\rho$ over $S$ is finite.
\item Suppose that the identity of $R$ is not a self-twist for $\rho$ over $S$. Then for every self-twist $\sigma$ the character $\eta_\sigma$ satisfying the equivalence \eqref{steq} is uniquely determined.
\item Under the same hypotheses as part (3), the association $\sigma\mapsto\eta_\sigma$ defines a cocycle on the group of self-twist with values in $R^\times$.
\item 
Let $S[\Tr\Ad\rho]$ denote the ring generated over $S$ by the set $\{\Tr(\Ad(\rho)(g))\}_{g\in G_\Q}$. Then every element of $S[\Tr\Ad\rho]$ is fixed by all self-twists for $\rho$ over $S$. 
\end{enumerate}
\end{prop}

Let $\theta\colon\T_h\to\I^\circ$ be a family of $\GSp_4$-eigenforms as defined in Section \ref{gspfam}. Let $\rho\colon G_\Q\to\GSp_4(\I_\Tr^\circ)$ be the Galois representation associated with $\theta$. Recall that $\I^\circ_\Tr$ is generated over $\Lambda_h$ by the traces of $\rho$. We always work under the assumption that $\ovl{\rho}\colon G_\Q\to\GSp_4(\F)$ is absolutely irreducible. 
Let $\Gamma$ be the group of self-twists for $\rho$ over $\Lambda_h$. We omit the reference to $\Lambda_h$ from now on and we just speak of the self-twists for $\rho$. 
Let $\I_0^\circ$ be the subring $(\I^\circ_\Tr)^\Gamma$ of $\I^\circ_\Tr$ consisting of the elements fixed by every $\sigma\in\Gamma$. 
We can study the order of $\Gamma$ thanks to an argument similar to that of the proof of \cite[Proposition 7.1]{lang}.

\begin{lemma}\label{cardgamma}
The only possible prime factors of $\card(\Gamma)$ are $2$ and $3$.
\end{lemma}

\begin{proof}
Let $\ell$ be any prime not dividing $Np$. Consider the element 
\begin{equation}\label{trdet} a_\ell=\frac{(\Tr\rho(\Frob_\ell))^4}{\det\rho(\Frob_\ell)} \end{equation}
of $\I_\Tr^\circ$. For every $\sigma\in\Gamma$ and every $g\in G_\Q$ Equation \eqref{steq} gives $\Tr\rho^\sigma(g)=\eta(g)\Tr\rho(g)$ and $\det\rho^\sigma(g)=\eta(g)^4\det\rho(g)$. In particular $a_\ell^\sigma=a_\ell$ for every $\sigma\in\Gamma$, so $a_\ell\in\I^\circ_0$. By Remark \ref{detformula} we have $\det\rho(\Frob_\ell)=\varepsilon(\ell)\chi(\ell)^{2(k_1+k_2-3)}\in\Lambda_{h}$, 
where $\varepsilon$ is the central character of the family $\theta$ and $\chi\colon G_{\Q_p}\to\Z_p^\times$ denotes the cyclotomic character. 
In particular $\det\rho(\Frob_\ell)\in\I_0^\circ$. 

Consider the Galois extension of $\I_0^\circ$ defined by $\I^\prime=\I_0^\circ[a_\ell^{1/4},\det\rho(\Frob_\ell)^{1/4},\zeta_4]$, where $\zeta_4$ is a primitive fourth root of unity. 
Equation \eqref{trdet} gives an inclusion $\I_\Tr^\circ\subset\I^\prime$, hence an inclusion $\Gamma\subset\Gal(\I^\prime/\I_0^\circ)$. 
Since $\I^\prime$ is obtained from $\I_0^\circ$ by adding some fourth roots, the order of an element of $\Gal(\I^\prime/\I_0^\circ)$ cannot have prime divisors greater than $3$. This concludes the proof.
\end{proof}

\subsection{Lifting self-twists from classical points to families}\label{liftsec}

Keep the notations as above. Let $P_\uk\subset\Lambda_h$ be any non-critical arithmetic prime, as in Definition \ref{noncritgsp}. 
The representation $\rho$ reduces modulo $P_\uk\I_\Tr^\circ$ to a representation $\rho_{P_\uk}\colon G_\Q\to\GSp_4(\I_\Tr^\circ/P_\uk\I_\Tr^\circ)$. Let $\widetilde{\sigma}\in\Gamma$ and let $\widetilde{\eta}\colon G_\Q\to(\I_\Tr^\circ)^\times$ be the character associated with $\widetilde\sigma$. 
The automorphism $\widetilde{\sigma}$ fixes $\Lambda_h$ by assumption, so it induces a ring automorphism $\widetilde\sigma_{P_\uk}$ of $\I_\Tr^\circ/P_\uk\I_\Tr^\circ$. The character $\widetilde{\eta}\colon G_\Q\to\I_\Tr^\circ$ induces a character $\widetilde{\eta}_{P_\uk}\colon G_\Q\to(\I_\Tr^\circ/P_\uk\I_\Tr^\circ)^\times$. The isomorphism $\rho^{\widetilde{\sigma}}\cong\widetilde{\eta}\otimes\rho$ over $\I_\Tr^\circ$ gives an isomorphism of representations over $\I_\Tr^\circ/P_\uk\I_\Tr^\circ$:
\begin{equation}\label{k-st} \rho_{P_\uk}^{\widetilde{\sigma}_{P_\uk}}\cong\widetilde{\eta}_{P_\uk}\otimes\rho_{P_\uk}. \end{equation}

Since $P_\uk$ is non-critical $\I^\circ$ is étale over $\Lambda_h$ at $P_\uk$, hence $\I_\Tr^\circ$ is also étale over $\Lambda_h$ at $P_\uk$. In particular $P_\uk$ is a product of distinct primes in $\I_\Tr^\circ$; denote them by $\fP_1,\fP_2,\ldots,\fP_d$. Since $\widetilde{\sigma}_{P_\uk}$ is an automorphism of $\I_\Tr^\circ/P_\uk\I_\Tr^\circ\cong\prod_{i=1}^d\I_\Tr^\circ/\fP_i$, there is a permutation $s$ of the set $\{1,2,\ldots,d\}$ and isomorphisms $\widetilde{\sigma}_{\fP_i}\colon\I_\Tr^\circ/\fP_i\to\I_\Tr^\circ/\fP_{s(i)}$ for $i=1,2,\ldots,d$ such that $\widetilde{\sigma}\vert_{\I_\Tr^\circ/\fP_i}$ factors through $\widetilde{\sigma}_{\fP_i}$. The character $\widetilde{\eta}_{\widetilde{\sigma}_{P_\uk}}$ can be written as a product $\prod_{i=1}^d\widetilde{\eta}_{\fP_i}$ for some characters $\widetilde{\eta}_{\fP_i}\colon G_\Q\to(\I_\Tr^\circ/\fP_i)^\times$. From the equivalence \eqref{k-st} we deduce that
\[ \rho_{\fP_i}^{\widetilde{\sigma}_{\fP_i}}\cong\widetilde{\eta}_{\fP_{s(i)}}\otimes\rho_{\fP_{s(i)}}. \]
The goal of this subsection is to prove that if we are given, for a single value of $i$, data $s(i)$, $\widetilde{\sigma}_{\fP_i}$ and $\widetilde{\eta}_{\fP_i}$ satisfying the isomorphism above for a single value of $i$, there exists an element of $\Gamma$ giving rise to $\widetilde{\sigma}_{\fP_i}$ and $\widetilde{\eta}_{\fP_{s_i}}$ via reduction modulo $P_\uk$. This result is an analogue of \cite[Theorem 3.1]{lang}. We state it precisely in the proposition below.

\begin{prop}\label{lifttwists}
Let $i,j\in\{1,2,\ldots,d\}$. Let $\sigma\colon\I_\Tr^\circ/\fP_i\to\I_\Tr^\circ/\fP_j$ be a ring isomorphism and $\eta_\sigma\colon G_\Q\to(\I_\Tr^\circ/\fP_j)^\times$ be a character satisfying
\begin{equation}\label{eqlift} \rho_{\fP_i}^{\sigma}\cong\eta_{\sigma}\otimes\rho_{\fP_j}. \end{equation}
Then there exists $\widetilde{\sigma}\in\Gamma$ with associated character $\widetilde{\eta}\colon G_\Q\to(\I_\Tr^\circ)^\times$ such that, via the construction of the previous paragraph, $s(i)=j$, $\widetilde{\sigma}_{\fP_i}=\sigma$ and $\widetilde{\eta}_{\fP_j}=\eta_{\sigma}$.
\end{prop}

In order to prove the proposition we first lift $\sigma$ to an automorphism $\Sigma$ of a deformation ring for $\ovl{\rho}$ and then we show that $\Sigma$ descends to a self-twist for $\rho$. This strategy is the same as that of the proof of \cite[Theorem 3.1]{lang}, but there are various complications that we have to take care of. In particular, in order to descend from the deformation space to $\rho$, we show that twisting a family of $\GSp_4$-eigenforms by a Dirichet character gives another family of $\GSp_4$-eigenforms and we subsequently rely on the non-criticality of the arithmetic prime $P_k$.

Before proving Proposition \ref{lifttwists} we give a corollary. 
Let $\fP\in\{\fP_1,\fP_2,\ldots,\fP_d\}$. 
Let $\rho_{\fP}\colon G_\Q\to\GSp_4(\I_\Tr^\circ/\fP)$ be the reduction of $\rho$ modulo $\fP$ and let $\Gamma_{\fP}$ be the group of self-twists for $\rho_{\fP}$ over $\Z_p$. 
Let $\Gamma(\fP)=\{\sigma\in\Gamma\,\vert\,\sigma(\fP)=\fP\}$; it is a subgroup of $\Gamma$. Let $\widetilde{\sigma}\in\Gamma$ and let $\widetilde{\eta}\colon G_\Q\to(\I_\Tr^\circ/\fP)^\times$ be the finite order character associated with $\widetilde\sigma$. Via reduction modulo $\fP$, $\widetilde\sigma$ and $\widetilde\eta$ induce a ring automorphism $\widetilde\sigma_{\fP}$ of $\I_\Tr^\circ/\fP$ and a finite order character $\widetilde\eta_{\fP}\colon G_\Q\to(\I_\Tr^\circ/\fP)^\times$ satisfying $\rho_{\fP_i}^{\sigma_{\fP}}\cong\eta_{\sigma_{\fP}}\otimes\rho_{\fP}$. Hence $\widetilde\sigma_{\fP}$ is an element of $\Gamma_{\fP}$. The map $\Gamma(\fP)\to\Gamma_{\fP}$ defined by $\widetilde\sigma\mapsto\widetilde\sigma_{\fP}$ is a morphism of groups.

\begin{cor}\label{surjst}
The morphism $\Gamma(\fP)\to\Gamma_{\fP}$ is surjective.
\end{cor}

\begin{proof}
This results from Proposition \ref{lifttwists} by choosing $\fP_i=\fP_j=\fP$.
\end{proof}


\subsubsection{Lifting self-twists to the deformation ring}

%
We keep the notations from the beginning of the section. 
Let $\Q^{Np}$ denote the maximal extension of $\Q$ unramified outside $Np$ and set $G_\Q^{Np}=\Gal(\Q^{Np}/\Q)$. Then $\rho$ factors via $G_\Q^{Np}$ by Proposition \ref{biggalthm}. In this subsection we consider $\rho$ as a representation $G_\Q^{Np}\to\GL_4(\I^\circ_\Tr)$ via the natural inclusion $\GSp_4(\I^\circ_\Tr)\into\GL_4(\I^\circ_\Tr)$. Coherently, we consider $G_\Q^{Np}$ as the domain of all the representations induced by $\rho$ and we take as their range the points of $\GL_4$ on the corresponding coefficient ring. Note that the equivalence \eqref{eqlift} implies that $\eta_\sigma$ also factors via $G_\Q^{Np}$, so we see it as a character of this group. For simplicity we write $\eta=\eta_\sigma$.

%
Recall that we write $\fm_{\I_\Tr^\circ}$ for the maximal ideal of $\I_\Tr^\circ$ and $\F$ for the residue field $\I_\Tr^\circ/\fm_{\Tr^\circ}$. Let $W$ be the ring of Witt vectors of $\F$. The residual representation $\ovl{\rho}\colon G_\Q^{Np}\to\GL_4(\F)$ is absolutely irreducible by assumption. By the results of \cite{mazdef}, the problem of deforming $\rho$ to a representation with coefficients in a Noetherian $W$-algebra is represented by a universal couple $(R_{\ovl\rho},\ovl\rho^\univ)$ consisting of a Noetherian $W$-algebra $R_{\ovl\rho}$ and a representation $\ovl\rho^\univ\colon G_\Q^{Np}\to\GL_4(R_{\ovl\rho})$. 

By the universal property of $R_{\ovl{\rho}}$ there exists a unique morphism of $W$-algebras $\alpha_I\colon R_{\ovl{\rho}}\to\I_\Tr^\circ$ satisfying $\rho\cong\alpha_I\ccirc\ovl{\rho}^\univ$. Let $\ev_i\colon\I_\Tr^\circ\to\I_\Tr^\circ/\fP_i$ and $\ev_j\colon\I_\Tr^\circ\to\I_\Tr^\circ/\fP_j$ be the two projections. The proposition below follows from arguments completely analogous to those of \cite[Section 3.1]{lang}. The details of the proof can be found in \cite[Section 4.4]{contith}.

\begin{prop}\label{liftdefprop}\mbox{ }
\begin{enumerate}
\item The automorphism $\ovl{\sigma}$ of $\F$ is trivial.
\item There is an isomorphism $\ovl\rho\cong\ovl\eta\otimes\ovl\rho$.
\item There exists an automorphism $\Sigma$ of $R_{\ovl{\rho}}$ such that:
\begin{enumerate}[label=(\roman*)]
\item $\Sigma$ is a lift of $\sigma$ in the sense that $\sigma\ccirc\ev_i\ccirc\alpha_I=\ev_j\ccirc\alpha_I\ccirc\Sigma$;
\item $\Sigma\ccirc\ovl\rho^\univ=\eta\ccirc\ovl\rho^\univ$.
\end{enumerate}
\end{enumerate}
\end{prop}

Recall that $\rho$ is the Galois representation associated with the finite slope family $\theta$. Our next step consists in showing that the representation $\rho^\Sigma$ is associated with a family of $\GSp_4$-eigenforms of a suitable tame level and of slope bounded by $h$. Thanks to property (ii) in Proposition \ref{liftdefprop}(3) it is sufficient to show that the representation $\eta\otimes\rho$ is associated with such a family.

\subsubsection{Twisting classical eigenforms by finite order characters}

We show that the twist of a representation associated with a classical Siegel eigenform by a finite order Galois character is the Galois representation associated with a classical Siegel eigenform of the same weight but possibly of a different level. By an interpolation argument we will deduce the analogous result for the representation associated with a family of eigenforms. 

Let $f$ be a cuspidal $\GSp_4$-eigenform of weight $(k_1,k_2)$ and level $\Gamma_1(M)$ and let $\rho_{f,p}\colon G_\Q\to\GSp_4(\Qp)$ be the $p$-adic Galois representation attached to $f$. Let $\eta\colon G_\Q\to\Qp^\times$ be a character of finite order $m_0$ prime to $p$. We see $\eta$ as a Dirichlet character when convenient. 

\begin{prop}\label{siegtwist}
There exists a cuspidal Siegel eigenform $f\otimes\eta$ of weight $(k_1,k_2)$ and level $\Gamma_1(\lcm(M,m_0)^2)$ such that the $p$-adic Galois representation associated with $f\otimes\eta$ is $\eta\otimes\rho_f$.
\end{prop}

Our proof relies on the calculations made by Andrianov in \cite[Section 1]{andrtwist}. He only considers the case $k_1=k_2$, but as we will remark his work can be adapted to vector-valued forms. For $A\in\Mat_n(R)$ we write $A\geq 0$ if $A$ is positive semi-definite and $A>0$ if $A$ is positive-definite. Recall that $f$, seen as a function on a variable $Z$ in the Siegel upper half-plane $\HH^n=\left\{X+iY\,\vert\,X,Y\in\Mat_n(\R) \textrm{ and } Y>0 \right\}$, admits a Fourier expansion of the form $f(Z)=\sum_{A\in\A^n,\, A\geq 0}a_Aq^A$, where $q=e^{2\pi i\Tr(AZ)}$ and
\[ \A^n=\left\{A=(a_{jk})_{j,k}\in\Mat_n\left(\frac{1}{2}\Z\right)\,\vert\, {}^tA=A \textrm{ and } a_{jj}\in\Z \textrm{ for } 1\leq j\leq n \right\}. \]

The weight $(k_1,k_2)$ action of $\displaystyle \left(\begin{array}{cc} A&B\\C&D\end{array}\right)\in\GSp_4(\C)$ on $f$ is defined by
\begin{equation}\label{siegact}
\left(\begin{array}{cc} A&B\\C&D\end{array}\right).f=(\Sym^{k_1-k_2}(\Std)\otimes\det{}^{k_2}(\Std))(CZ+D)f\left(\frac{AZ+B}{CZ+D}\right),
\end{equation}
where $\Std$ denotes the standard representation of $\GL_2$. 
As in \cite{andrtwist}, we define the twist of $f$ by $\eta$ as
\[ f\otimes\eta=\sum_{A\in\A^n,\, A\geq 0}\eta(\Tr(A))a_Aq^A. \]
Note that Andrianov considers a family of twists by $\eta$ depending on an additional $2\times 2$ matrix $L$, but we only need the case $L=\1_2$.
 
Recall that $\mu(A)$ denotes the similitude factor of $A$.

\begin{lemma}\label{vectortwist}
Let $\eta$ be a Dirichlet character of conductor $m$ and $f$ be a cuspidal form of weight $(k_1,k_2)$ and level $\Gamma_1(M)$. Let $M^\prime=\lcm(m_0,N)^2$. 
\begin{enumerate}
\item The expansion $f\otimes\eta$ defines a cuspidal form of level $\Gamma_1(M^\prime)$ (cf. \cite[Proposition 1.4]{andrtwist}). 
\item If $A\in\GSp_4(\C)$, $[\Gamma_1(m^2)A\Gamma_1(m^2)].(f\otimes\eta)=\eta(\mu(A))([\Gamma_1(m)A\Gamma_1(m)].f)\otimes\eta$ (cf. \cite[Theorem 2.3]{andrtwist}).
\end{enumerate}
\end{lemma}

\begin{proof}
The proof relies on the same calculations as the proofs of \cite[Proposition 1.4 and Theorem 2.3]{andrtwist}, that are stated for scalar Siegel modular forms. Note first that all the steps in these proofs only involve the action of upper unipotent matrices on $f$ via formula \eqref{siegact}. The action of such matrices is clearly independent of the weight of $f$, hence all calculations are still true upon replacing the weight $(k,k)$ action with the weight $(k_1,k_2)$ action for some $k,k_1,k_2\in\N$. We deduce that the conclusions of \cite[Proposition 1.4 and Theorem 2.3]{andrtwist} hold for vector-valued Siegel modular forms. With the notations of \cite{andrtwist}, the calculations of \emph{loc. cit.} produce a form $f\otimes\eta$ of level $\widetilde\Gamma(M^\prime)$ from a form $f$ of level $\widetilde\Gamma(M)$. We can modify these levels to match those in Lemma \ref{vectortwist} by observing that $\Gamma_1(n^2)\subset\widetilde{\Gamma}(n)\subset\Gamma_1(n)$ for every $n\geq 1$.
\end{proof}

We are now ready to prove Proposition \ref{siegtwist}.

\begin{proof}
We see the form $f$ of level $\Gamma_1(M)$ as a form of level $\Gamma_1(\lcm(M,m_0))$ and the character $\eta$ of conductor $m$ as a character of conductor $\lcm(M,m_0)$. By applying Lemma \ref{vectortwist}(1) with $m=\lcm(M,m_0)$ we can construct a form $f\otimes\eta$ of level $\Gamma_1(\lcm(M,m_0)^2)$. Let $\rho_{f\otimes\eta,p}\colon G_\Q\to\GSp_4(\Qp)$ be the $p$-adic Galois representation associated with $f\otimes\eta$. We show that $\rho_{f\otimes\eta,p}\cong\eta\otimes\rho_{f,p}$.

For every congruence subgroup $\Gamma\subset\GSp_4(\C)$ and every prime $\ell$, we denote by $T_{\ell,0}$, $T_{\ell,1}$ and $T_{\ell,2}$ the Hecke operators associated with the double classes $[\Gamma\diag(\ell,\ell,\ell,\ell)\Gamma]$, $[\Gamma\diag(1,\ell,\ell,\ell^2)\Gamma]$ and $[\Gamma\diag(1,1,\ell,\ell)\Gamma]$, respectively. We do not specify the congruence subgroup with respect to which we work, since this does not create confusion in the following. 
Lemma \ref{vectortwist}(2) gives, for every prime $\ell\nmid Mm_0$, the relations $T_{\ell,0}(f\otimes\eta)=\eta(\ell^2)T_{\ell,0}(f)\otimes\eta$, $T_{\ell,1}(f\otimes\eta)=\eta(\ell^2)T_{\ell,1}(f)\otimes\eta$ and $T_{\ell,2}(f\otimes\eta)=\eta(\ell)T_{\ell,2}(f)\otimes\eta$. 

Recall that for every $\ell\nmid Mm_0p$ we have
\[ \det(1-\rho_{f,p}(\Frob_\ell)X)=\chi_f(X^4-T_{\ell,2}X^3+((T_{\ell,2})^2-T_{\ell,1}-\ell^2T_{\ell,0})X^2-\ell^3T_{\ell,2}T_{\ell,0} X+\ell^6(T_{\ell,0})^2) \]
where $\chi_f$ is the character of the Hecke algebra defining the system of eigenvalues of $f$. The equality still holds if we replace $f$ by $f\otimes\eta$. Via the relations obtained at the end of the previous paragraph we can check that $\det(1-(\eta\otimes\rho_{f,p})(\Frob_\ell)X)=\det(1-\rho_{f\otimes\eta,p}(\Frob_\ell)X)$ for every $\ell\nmid Mm_0p$. This implies that the representations $\eta\otimes\rho_{f,p}$ and $\rho_{f\otimes\eta,p}$ are equivalent.
\end{proof}

Under the hypotheses of the previous proposition we prove the following.

\begin{cor}\label{etaslope}
Let $M^\prime=\lcm(m_0,M)^2$. 
Let $x$ be a classical $p$-old point of $\cD_2^M$ having weight $(k_1,k_2)$, slope $h$ and associated Galois representation $\rho_x$. Then there exists a classical $p$-old point $x_\eta$ of $\cD_2^{M^\prime}$ having weight $(k_1,k_2)$, slope $h$ and associated Galois representation $\rho_{x_\eta}=\eta\otimes\rho_x$.
\end{cor}

\begin{proof}
Since $x$ is $p$-old, it corresponds to the $p$-stabilization of a $\GSp_4$-eigenform $f$ of level $M$ and weight $(k_1,k_2)$. Let $f\otimes\eta$ be the eigenform of weight $(k_1,k_2)$ and level $M^\prime$ given by Proposition \ref{siegtwist}. We show that it admits a $p$-stabilization of slope $h$.

We are working under the assumption that the conductor of $\eta$ is prime to $p$, so we can compute 
\begin{equation}\label{etap}\begin{gathered} \chi_{f\otimes\eta}(P_\Min(t^{(2)}_{p,2}))=\chi_f(X^4-\eta(p)T_{p,2}X^3+((\eta(p)T_{p,2})^2-\eta(p)^2T_{p,1}-p^2\eta(p)^2T_{p,0})X^2+ \\
-p^3(\eta(p)T_{p,2})(\eta(p)^2T_{p,0})X+p^6(\eta(p)^2T_{p,0})^2). \end{gathered}\end{equation}
Let $\{\alpha_i\}_{i=1,\ldots,4}$ be the four roots of $\chi_{f}(P_\Min(t^{(2)}_{p,2}))$. Then Equation \eqref{etap} shows that the roots of $\chi_{f\otimes\eta}(P_\Min(t^{(2)}_{p,2}))$ are $\{\eta(p)\alpha_i\}_{i=1,\ldots,4}$.

Suppose that $f$ is $p$-old. 
Recall that we identify $U_{p,2}^{(2)}$ with $t^{(2)}_{p,2}$ via the isomorphism $\iota_{I_{g,\ell}}^{T_g}$ of Section \ref{dilIw}. By the discussion in the proof of Prop. \ref{heckemorphstab} there are eight $p$-stabilizations of $f\otimes\eta$, one for each compatible choice of $U_{p,2}^{(2)}$ and $(U_{p,2}^{(2)})^{w_1}$ among the roots of $\chi_f(P_\Min(t^{(2)}_{p,2}))$. 
Let $f^\st$ be a $p$-stabilization of $f$ with slope $h$. 
Since $U_p^{(2)}=(U_{p,2}^{(2)})^2(U_{p,2}^{(2)})^{w_1}$, there are $i,j\in\{1,2,3,4\}$ such that $\chi_{f^\st}(U_p^{(2)})=\alpha_i^2\alpha_j$. Then by the remark of the previous paragraph there exists a $p$-stabilization $(f\otimes\eta)^\st$ of $f\otimes\eta$ such that $\chi_{(f\otimes\eta)^\st}(U_p^{(2)})=(\eta(p)\alpha_i)^2(\eta(p)\alpha_j)=\eta(p)^3\alpha_i^2\alpha_j$. 
In particular the slope of $(f\otimes\eta)^\st$ is $v_p(\chi_{(f\otimes\eta)^\st}(U_p^{(2)}))=3v_p(\eta(p))+h$. 
Since $p$ is prime to the conductor of $\eta$ we have that $\eta(p)$ is a unit, hence the slope of $(f\otimes\eta)^\st$ is $h$. The level of $(f\otimes\eta)^\st$ is $\Gamma_1(M^\prime)\cap\Gamma_0(p)$, so it defines a point of $\cD_2^{M^\prime}$.
%
\end{proof}

%

Consider the family $\theta\colon\T_h\to\I^\circ$ fixed in the beginning of the section. For every $p$-old classical point $x$ of $\theta$, let $x_\eta$ be the point of the eigenvariety $\cD_2^{M^\prime}$ provided by Corollary \ref{etaslope}. Let $r_h^\prime$ be a radius adapted to $h$ for the eigenvariety $\cD_2^{M^\prime}$. 
Let $\Lambda_h^\prime$ be the genus $2$, $h$-adapted Iwasawa algebra for $\cD_2^{M^\prime}$ and let $\T^\prime_h$ be the genus $2$, $h$-adapted Hecke algebra of level $M^\prime$.
Note that $r_h^\prime\leq r_h$, so there is a natural map $\iota_h\colon\Lambda_h\to\Lambda_h^\prime$.

\begin{lemma}\label{interptwists}
There exists a finite $\Lambda_h^\prime$-algebra $\J^\circ$, a family $\theta^\prime\colon\T^\prime_h\to\J^\circ$ and an isomorphism $\alpha\colon\I_\Tr^\circ\widehat\otimes_{\Lambda_h}\Lambda_h^\prime\to\J_\Tr^\circ$ such that the representation $\rho_{\theta^\prime}\colon G_\Q\to\GSp_4(\J_\Tr^\circ)$ associated with $\theta^\prime$ satisfies $\rho_{\theta^\prime}\cong\eta\otimes\alpha\ccirc\rho_\theta$.
\end{lemma}

\begin{proof}
Let $S$ be the set of $p$-old classical points of $\theta$. Let $S^\prime$ be the subset of $S$ consisting of the points with weight in the disc $B(0,r_h^\prime)$. We see $S^\prime$ as a subset of the set of classical points of $\cD_2^{M^\prime}$ via the natural inclusion $\cD_2^M\into\cD_2^{M^\prime}$. Thanks to the conditions on the weight and the slope we can identify $S^\prime$ with a set of classical points of $\T^\prime_h$. Note that $S^\prime$ is infinite.

Let $S^\prime_\eta=\{x_\eta\,\vert\,x\in S^\prime \}$, that is also contained in the set of classical points of $\cD_2^{M^\prime}$. 
For every $x\in S^\prime$ the weight and slope of $x_\eta$ coincide with the weight and slope of $x$. In particular $S^\prime_\eta$ can be identified with an infinite set of classical points of $\T^\prime_h$. Since $\T^\prime_h$ is a finite $\Lambda_h^\prime$-algebra, the Zariski-closure of $S_\eta^\prime$ in $\T_h^\prime$ contains an irreducible component of $\T_h^\prime$. Such a component is a family defined by a finite $\Lambda_h^\prime$-algebra $\J^\circ$ and a morphism $\theta^\prime\colon\T_h^\prime\to\J^\circ$.

Let $\rho_{\theta^\prime}\colon G_\Q\to\GSp_4(\J_\Tr^\circ)$ be the Galois representation associated with $\theta^\prime$. 
Let $S_\eta^{\theta^\prime}$ be the subset of $S_\eta^\prime$ consisting of the points that belong to $\theta^\prime$; it is Zariski-dense in $\J^\circ$ by definition of $\theta^\prime$. Let $S^{\theta^\prime}=\{x\in S^\prime\,\vert\,x_\eta\in S_\eta^{\theta^\prime}\}$. For every $x\in S^{\theta^\prime}$ let $\rho_{\theta,x}$ be the specialization of $\rho_\theta$ at $x$ and let $\rho_{\theta^\prime,x_\eta}$ be the specialization of $\rho_{\theta^\prime}$ at $x_\eta$. By the definition of the correspondence $x\mapsto x_\eta$ we have $\rho_{\theta^\prime,x_\eta}\cong\eta\otimes\rho_{\theta,x}$ over $\Qp$ for every $x\in S^{\theta^\prime}$. Hence the representation $\eta\otimes\rho_{\theta,x}$ coincides with $\iota_h\ccirc\rho_{\theta^\prime}$ on the set $S_\eta^{\theta^\prime}$. Since this set is Zariski-dense in $\J$, there exists an isomorphism $\alpha\colon\I_\Tr^\circ\widehat\otimes_{\Lambda_h}\Lambda_h^\prime\to\J_\Tr^\circ$ such that $\rho_{\theta^\prime}\cong\eta\otimes\alpha\ccirc\rho_\theta$, as desired. 
\end{proof}

\begin{rem}\label{alleta}
With the notation of the proof of Lemma \ref{interptwists}, all points of the set $S_\eta^\prime$ belong to the family $\theta^\prime$, because of the unicity of a point of $\cD_2^{M^\prime}$ given its associated Galois representation and slope.
\end{rem}

By combining Lemma \ref{interptwists} and Proposition \ref{liftdefprop}(3) we obtain the following.

\begin{cor}\label{interpSigma}
There exists a finite $\Lambda_h^\prime$-algebra $\J^\circ$, a family $\theta^\prime\colon\T^\prime_h\to\J^\circ$ and an isomorphism $\alpha\colon\I_\Tr^\circ\widehat\otimes_{\Lambda_h}\Lambda_h^\prime\to\J_\Tr^\circ$ such that the representation $\rho_{\theta^\prime}\colon G_\Q\to\GSp_4(\J_\Tr^\circ)$ associated with $\theta^\prime$ satisfies $\rho_{\theta^\prime}\cong\alpha\ccirc\rho^\Sigma$. 
\end{cor}

\subsubsection{Descending to a self-twist of the family}

We show that the automorphism $\Sigma$ of $R_{\ovl{\rho}}$ defined in the previous subsection induces a self-twist for $\rho$. This will prove Proposition \ref{lifttwists}. Our argument is an analogue for $\GSp_4$ of that in the end of the proof of \cite[Theorem 3.1]{lang}; it also appears in similar forms in \cite[Proposition 3.12]{fischman} and \cite[Proposition A.3]{ghadim}. Here the non-criticality of the prime $P_\uk$ plays an important role. 

\begin{proof}(of Proposition \ref{lifttwists}) 
Let $\ovl{\rho}\colon G_\Q\to\GSp_4(\F)$ be the residual representation associated with $\rho$. Let $R_{\ovl{\rho}}$ be the universal deformation ring associated with $\ovl{\rho}$ and let $\ovl{\rho}^\univ$ be the corresponding universal deformation. As before let $\alpha_I\colon R_{\ovl{\rho}}\to\I_\Tr^\circ$ be the unique morphism of $W$-algebras $\alpha_I\colon R_{\ovl{\rho}}\to\I_\Tr^\circ$ satisfying $\rho\cong\alpha_I\ccirc\ovl{\rho}^\univ$. 

Consider the morphism of $W$-algebras $\alpha_I^\Sigma=\alpha_I\ccirc\Sigma\colon R_{\ovl{\rho}}\to\I_\Tr^\circ$. We show that there exists an automorphism $\widetilde{\sigma}\colon\I_\Tr^\circ\to\I_\Tr^\circ$ fitting in the following commutative diagram:
\begin{equation}\label{diagdef}
\begin{tikzcd}
R_{\ovl{\rho}} \arrow{r}{\alpha_I}\arrow{d}{\Sigma}
& \I_\Tr^\circ\arrow{d}{\widetilde\sigma}\\
R_{\ovl{\rho}} \arrow{r}{\alpha_I^\Sigma}
& \I_\Tr^\circ
\end{tikzcd}
\end{equation}
We use the notations of the discussion preceding Lemma \ref{interptwists}. Consider the morphism $\theta\otimes 1\colon\T_h\widehat\otimes_{\Lambda_h}\Lambda_{h}^\prime\to\I^\circ\widehat\otimes_{\Lambda_h}\Lambda_h^{\prime}$, where the completed tensor products are taken via the map $\iota_h\colon\Lambda_h\to\Lambda_h^\prime$. For every $\Lambda_h$-algebra $A$ we denote again by $\iota_h$ the natural map $A\to A\widehat\otimes_{\Lambda_h}\Lambda_h^\prime$. The natural inclusion $\cD_2^M\into\cD_2^{M^\prime}$ induces a surjection $s_h\colon\T_h^\prime\to\T_h\widehat\otimes_{\Lambda_h}\Lambda_{h}^\prime$. We define a family of tame level $\Gamma_1(M^\prime)$ and slope bounded by $h$ by
\[ \theta^{M^\prime}=(\theta\otimes 1)\ccirc s_h\colon\T_h^\prime\to\I^\circ\widehat\otimes_{\Lambda_h}\Lambda_h^\prime. \]
The Galois representation associated with $\theta^{M^\prime}$ is $\rho_{\theta^{M^\prime}}=\iota_h\ccirc\rho\colon G_\Q\to\GSp_4(\I_\Tr^\circ\widehat\otimes_{\Lambda_h}\Lambda_h^\prime)$. Let $\theta^\prime\colon\T_h^\prime\to\J^\circ$ be the family given by Corollary \ref{interpSigma}. We identify $\I_\Tr^\circ\widehat\otimes_{\Lambda_h}\Lambda_h^\prime$ with $\J_\Tr^\circ$ via the isomorphism $\alpha$ given by the same corollary; in particular the Galois representation associated with $\theta^\prime$ is $\rho_{\theta^\prime}=\rho^\Sigma\colon G_\Q\to\GSp_4(\I_\Tr^\circ\widehat\otimes_{\Lambda_h}\Lambda_h^\prime)$. 

Recall that we are working under the assumptions of Proposition \ref{lifttwists}. In particular we are given two primes $\fP_i$ and $\fP_j$ of $\I_\Tr^\circ$, an isomorphism $\sigma\colon\I_\Tr^\circ/\fP_i\to\I_\Tr^\circ/\fP_j$ and a character $\eta_\sigma\colon G_\Q\to(\I_\Tr^\circ/\fP_j)^\times$ such that $\rho_{\fP_i}^{\sigma}\cong\eta_{\sigma}\otimes\rho_{\fP_j}$. Let $\fP_i^\prime$ be the image of $\fP_i$ via the map $\iota_h\colon\I_\Tr^\circ\widehat\otimes_{\Lambda_h}\Lambda_h^\prime$. The specialization of $\rho_{\theta^{M^\prime}}$ at $\fP_i^\prime$ is $\rho_{\fP_i}$. Let $f^\prime$ be the eigenform corresponding to $\fP_i^\prime$. By Remark \ref{alleta} there is a point of the family $\theta^\prime$ corresponding to the twist of $f$ by $\eta$; let $\fP_{i,\eta}^\prime$ be the prime of $\I_\Tr^\circ\widehat\otimes_{\Lambda_h}\Lambda_h^\prime$ defining this point. The specialization of $\rho_{\theta^\prime}$ at $\fP_{i,\eta}^\prime$ is $\eta\otimes\rho_{\fP_i}$, which is isomorphic to $\rho_{\fP_i}^{\sigma}$ by assumption. Let $f_\eta^\prime$ be the eigenform corresponding to the prime $\fP_{i,\eta}^\prime$. The forms $f^\prime$ and $f_\eta^\prime$ have the same slope by Corollary \ref{etaslope} and their associated representations are obtained from one another via Galois conjugation (given by the isomorphism $\sigma$). Hence $f^\prime$ and $f_\eta^\prime$ define the same point of the eigenvariety $\cD_2^{M^\prime}$. Such a point belongs to both the families $\theta^{M^\prime}$ and $\theta^\prime$. Since $\fP_\uk$ is non-critical, $\T_h^\prime$ is étale at every point lying over $P_\uk$, so the families $\theta^{M^\prime}$ and $\theta^\prime$ must coincide. This means that there is an isomorphism
\[ \widetilde\sigma^\prime\colon\I^\circ_\Tr\widehat\otimes_{\Lambda_h}\Lambda_{h}^\prime\to\I_\Tr^\circ\widehat\otimes_{\Lambda_h}\Lambda_h^\prime \]
such that $\rho_{\theta^\prime}=\widetilde\sigma^\prime\ccirc\rho^{M^\prime}$. The isomorphism $\widetilde\sigma^\prime$ induces by restriction an isomorphism $\Lambda_h^\prime[\Tr(\rho^{M^\prime})]\to\Lambda_h^\prime[\Tr(\rho_{\theta^\prime})]$. Note that $\Lambda_h^\prime[\Tr(\rho^{M^\prime})]=\iota_h(\I_\Tr^\circ)$ and 
\begin{gather*} 
\Lambda_h^\prime[\Tr(\rho_{\theta^\prime})]=\Lambda_h^\prime[\Tr(\widetilde\sigma^\prime\ccirc\rho^{M^\prime})]=\widetilde\sigma^\prime(\Lambda_h^\prime[\Tr(\rho^{M^\prime})])=\widetilde\sigma^\prime(\Lambda_h^\prime[\Tr(\iota_h\ccirc\rho)])=\widetilde\sigma^\prime(\iota_h(\Lambda_h[\Tr\rho]))=\widetilde\sigma^\prime(\iota_h(\I_\Tr^\circ)). 
\end{gather*}
In particular $\widetilde\sigma^\prime$ induces by restriction an isomorphism $\iota_h(\I_\Tr^\circ)\to\iota_h(\I_\Tr^\circ)$. Since $\iota_h$ is injective we can identify $\widetilde\sigma^\prime$ with an isomorphism $\widetilde\sigma\colon\I_\Tr^\circ\to\I_\Tr^\circ$. By construction $\widetilde\sigma$ fits in diagram \eqref{diagdef}.
\end{proof}

\subsection{Rings of self-twists for representations attached to classical eigenforms}\label{stclass}

Let $f$ be a classical $\GSp_4$-eigenform and $\rho_{f,p}\colon G_\Q\to\GSp_4(\Qp)$ the $p$-adic Galois representation associated with $f$. Up to replacing $\rho_{f,p}$ with a conjugate we can suppose that it has coefficients in the ring of integers $\cO_K$ of a $p$-adic field $K$. Suppose that $f$ satisfies the hypotheses of Theorem \ref{classbigim}, i.e. $\ovl{\rho}_{f,p}$ is of $\Sym^3$ type but $f$ is not the symmetric cube lift of a $\GL_2$-eigenform. Let $\Gamma_f$ be the group of self-twists for $\rho$ over $\Z_p$ and let $\cO_K^{\Gamma_f}$ be the subring of elements of $\cO_K$ fixed by $\Gamma_f$. As in in Section \ref{resbigim} we define another subring of $\cO_K$ by $\cO_E=\Z_p[\Tr(\Ad\rho)]$. 
We prove that the two subrings of $\cO_K$ we defined are actually the same.

\begin{prop}\label{sttraces}
There is an equality $\cO_K^{\Gamma_f}=\cO_E$.
\end{prop}

Before proving the proposition we recall a theorem of O'Meara about isomorphisms of congruence subgroups. 
Here $g$ is a positive integer, $F$, $F_1$ are two $p$-adic fields and $\fa$, $\fa_1$ are two non-trivial ideals in the rings of integers of $F$ and $F_1$, respectively.  
Let $V=F^{2g}$ and $V_1=F_1^{2g}$, both equipped with the bilinear alternating form defined by the matrix $J_g$ defined in the introduction. 
Let $\sigma\colon F\to F_1$ be an isomorphism. 
We say that a map $S$ of $V$ into $V_1$ is $\sigma$-semilinear if it is additive and satisfies $S(\lambda v)=\sigma(\lambda)S(v)$ for every $v\in V$ and $\lambda\in F$. 
Let $\PSp_{2g}$ and $\PGSp_{2g}$ be the projective symplectic groups of genus $g$.
%

\begin{rem}\label{symplsemilin}
Let $\sigma\colon F\to F_1$ be an isomorphism. 
Denote by $x\mapsto x^\sigma$ the isomorphism $\GSp_{2g}(F)\to\GSp_{2g}(F_1)$ obtained by applying $\sigma$ to the matrix coefficients. For every bijective, symplectic, $\sigma$-semilinear map $S\colon V\to V_1$ there exists $\gamma\in\GSp_{2g}(F_1)$ such that $SxS^{-1}=\gamma x^\sigma\gamma^{-1}$ for every $x\in\GSp_4(F)$. 
\end{rem}

By combining Remark \ref{symplsemilin} and \cite[Theorem 5.6.4]{omeara}, with the choices we made in the discussion above, we obtain the following.

\begin{thm}\label{projcongisomcor}
Let $\Delta$ and $\Delta_1$ be subgroups of $\PGSp_{2g}(F)$ and $\PGSp_{2g}(F_1)$, respectively, satisfying $\Gamma_{\PSp_{2g}(F)}(\fa)\subset\Delta$ and $\Gamma_{\PSp_{2g}(F_1)}(\fa)\subset\Delta_1$. 
Let $\Theta\colon\Delta\to\Delta_1$ be an isomorphism of groups. Then there exists an automorphism $\sigma$ of $F$ and an element $\gamma\in\PGSp_{2g}(F)$ satisfying
\[ \Theta x=\gamma x^\sigma\gamma^{-1} \]
for every $x\in\Delta$.
\end{thm}

From Theorem \ref{projcongisomcor} we deduce a result on isomorphisms of congruence subgroups of $\GSp_{2g}(F)$. 

\begin{cor}\label{congisom}\cite[Theorem 5.6.5]{omeara}
Let $\Delta$ and $\Delta_1$ be two subgroups of $\GSp_{2g}(F)$ satisfying $\Gamma_{F}(\fa)\subset\Delta$ and $\Gamma_{F_1}(\fa)\subset\Delta_1$. 
Let $\Theta\colon\Delta\to\Delta_1$ be an isomorphism of groups. Then there exists an automorphism $\sigma$ of $F$, a character $\chi\colon\Delta\to F^\times$ and an element $\gamma\in\GSp_{2g}(F)$ satisfying
\[ \Theta x=\chi(x)\gamma x^\sigma\gamma^{-1} \]
for every $x\in\Delta$.
\end{cor}


Before proving Proposition \ref{sttraces} we fix some notations. Let $\End(\fsp_4(K))$ be the $K$-vector space of $K$-linear maps $\fsp_4(K)\to\fsp_4(K)$ and let $\GL(\fsp_4(K))$ be the subgroup consisting of the bijective ones. Let $\Aut(\fgsp_4(K))$ be the subgroup of $\GL(\fsp_4(K))$ consisting of the Lie algebra automorphisms of $\fsp_4(K)$. Let $\pi_\Ad$ be the natural projection $\GSp_4(\cO_K)\to\PGSp_4(\cO_K)$ and let $\Ad\colon\PGSp_4(K)\into\GL(\fsp_4(K))$ be the injective group morphism given by the adjoint representation. 
Since $\fsp_4$ admits no outer automorphisms, $\Ad$ induces an isomorphism of $\PGSp_4(K)$ onto $\Aut(\fsp_4(K))$. For simplicity we write $\rho=\rho_{f,p}$ in the following proof (but recall that in the other sections $\rho$ is the Galois representation attached to a family).

\begin{proof}(of Proposition \ref{sttraces})
The inclusion $\cO_E\subset\cO_K^{\Gamma_f}$ follows from Proposition \ref{basicst}(5). 
We prove that $\cO_K^{\Gamma_f}\subset\cO_E$.
%
%
Since $\cO_K^{\Gamma_f}$ and $\cO_E$ are normal, 
it is sufficient to show that an automorphism of $\cO_K$ over $\cO_E$ leaves $\cO_K^{\Gamma_f}$ fixed. Consider such an automorphism $\sigma$. Since $\cO_E$ is fixed by $\sigma$ we have $(\Tr(\Ad\rho)(g))^\sigma=\Tr(\Ad\rho(g))$ for every $g\in G_\Q$, hence $\Tr(\Ad\rho^\sigma(g))=\Tr(\Ad\rho(g))$. The equality of traces induces an isomorphism $\Ad\rho^\sigma\cong\Ad\rho$ of representations of $G_\Q$ with values in $\GL(\fsp_4)$.
This means that there exists $\phi\in\GL(\fsp_4(K))$ satisfying 
\begin{equation}\label{lieaut} \Ad\rho^\sigma=\phi\ccirc\Ad\rho\ccirc\phi^{-1}. \end{equation}
We show that $\phi$ is actually an inner automorphism of $\fsp_4(K)$.

Clearly $\Ad$ induces an isomorphism $\pi_\Ad(\im\rho)\cong\im\Ad\rho$. 
For every $x\in\GL(\fsp_4(K))$ we denote by $\Theta_x$ the automorphism of $\GL(\fsp_4(K))$ given by conjugation by $x$. 
In particular we write Equation \eqref{lieaut} as $\Ad\rho^\sigma=\Theta_\phi(\Ad\rho)$.
By combining Theorems \ref{classbigim} and \ref{projcongisomcor} we show that we can replace $\phi$ by an element $\phi^\prime\in\Aut(\fsp_4(K))$ still satisfying $\Ad\rho^\sigma=\Theta_{\phi^\prime}(\Ad\rho(\phi^\prime))$. 

We identify $\PGSp_4(\cO_E)$ with a subgroup of $\PGSp_4(\cO_K^{\Gamma_f})$ via the inclusion $\cO_E\subset\cO_K^{\Gamma_f}$ given in the beginning of the proof. Consider the group $\Delta=(\pi_\Ad\im\rho)\cap\PGSp_4(\cO_E)\subset\PGSp_4(\cO_K)$ and its isomorphic image $\Ad(\Delta)\subset\GL(\fsp_4)$. 
Since $f$ satisfies the hypotheses of Theorem \ref{classbigim}, $\im\rho$ contains a congruence subgroup $\Gamma_{\cO_E}(\fa)$ of $\GSp_4(\cO_E)$ of some level $\fa\subset\cO_E$. It follows that $\pi_\Ad\im\rho$ contains the projective congruence subgroup $\PGamma_{\cO_E}(\fa)$ of $\PGSp_4(\cO_E)$, so $\Delta$ also contains $\PGamma_{\cO_E}(\fa)$. 
In particular $\Delta$ satisfies the hypotheses of Theorem \ref{projcongisomcor}. 
Since $\Ad\rho^\sigma=\Theta_\phi(\Ad\rho)$ we have an equality $(\Ad(\Delta))^\sigma=\Theta_\phi(\Ad(\Delta))$, where we identify both sides with subgroups of $\PGSp_4(\cO_E)$. 
Now $\sigma$ acts as the identity on $\PGSp_4(\cO_E)$, so the previous equality reduces to $\Ad(\Delta)=\Theta_\phi(\Ad(\Delta))$. 
Let $\Theta=\Ad\!{}^{-1}\ccirc\Theta_\phi\ccirc\Ad\!\colon\Delta\to\Delta$.
Since $\Ad$ is an isomorphism, the composition $\Theta$ is an automorphism. Moreover it satisfies 
\begin{equation}\label{thetaphi} \Theta_\phi(\Ad(\delta))=\Ad(\Theta(\delta)) \end{equation}
for every $\delta\in\Delta$. By Theorem \ref{projcongisomcor} applied to $F=F_1=K$, $\Delta_1=\Delta$ and $\Theta\colon\Delta\to\Delta$, there exists an automorphism $\tau$ of $K$ and an element $\gamma\in\GSp_4(K)$ such that $\Theta(\delta)=\gamma\delta^\tau\gamma^{-1}$ for every $\delta\in\Delta$. We see from Equation \eqref{thetaphi} that $\tau$ is trivial. It follows that $\Theta_\phi(y)=\Ad(\gamma)\ccirc y\ccirc\Ad(\gamma)^{-1}$ for all $y\in\Ad(\Delta)$. By $K$-linearity we can extend $\Theta_\phi$ and $\Theta_{\Ad(\gamma)}$ to identical automorphisms of the $K$-span of $\Ad(\Delta)$ in $\End(\fsp_4(K))$. Since $\Delta$ contains the projective congruence subgroup $\PGamma_{\cO_E}(\fa)$, its $K$-span contains $\Ad(\GSp_4(K))$; in particular it contains the image of $\Ad\rho$. Hence $\Theta_\phi$ and $\Theta_{\Ad(\gamma)}$ agree on $\Ad\rho$, which means that Equation \eqref{lieaut} implies $\Ad\rho^\sigma=\Theta_{\Ad(\gamma)}(\Ad\rho)$. 
Then by definition of $\Theta_{\Ad(\gamma)}$ we have 
$\Ad\rho^\sigma=\Ad(\gamma)\ccirc\Ad\rho\ccirc(\Ad(\gamma))^{-1}=\Ad(\gamma\rho\gamma^{-1})$. 
We deduce that there exists a character $\eta_\sigma\colon G_\Q\to\cO_K^\times$ satisfying $\rho^\sigma(g)=\eta_\sigma(g)\gamma\rho(g)\gamma^{-1}$ for every $g\in G_\Q$, hence that $\rho^\sigma\cong\eta_\sigma\otimes\rho$. We conclude that $\sigma$ is a self-twist for $\rho$. In particular $\sigma$ acts as the identity on $\cO_K^{\Gamma_f}$, as desired.
\end{proof}

\begin{rem}
Let $\rho\colon G_\Q\to\GSp_4(\I_\Tr^\circ)$ be the big Galois representation associated with a family $\theta\colon\T_h\to\I^\circ$. We can define a ring $\Lambda_h[\Tr(\Ad\rho)]$ analogous to the ring $\cO_E$ defined above. We have an inclusion $\Lambda_h[\Tr(\Ad\rho)]\subset\I_0^\circ$ given by Proposition \ref{basicst}(5). However the proof of the inclusion $\cO_K^{\Gamma_f}\subset\cO_E$ in Proposition \ref{sttraces} relied on the fact that $\im\rho_{f,p}$ contains a congruence subgroup of $\GSp_4(\cO_E)$. Since we do not know if an analogue for $\rho$ is true, we do not know whether an equality between the normalizations of $\Lambda_h[\Tr(\Ad\rho)]$ and $\I_0^\circ$ holds.
\end{rem}

Suppose that the $\GSp_4$-eigenform $f$ appears in a finite slope family $\theta\colon\T_h\to\I^\circ$. Let $\fP$ be the prime of $\I_\Tr^\circ$ associated with $f$ and suppose that $\fP\cap\Lambda_h$ is a non-critical arithmetic prime $P_\uk$. Let $\fP_0=\fP\cap\I_0^\circ$. 
We use Proposition \ref{lifttwists} to compare $\cO_K^{\Gamma_f}$ and the residue ring of $\I_0^\circ$ at $\fP_0$, as in \cite[Proposition 6.2]{lang}. 

\begin{prop}\label{stringseq}
There is an inclusion $\I_0^\circ/\fP_0\subset\cO_K^{\Gamma_f}$.
\end{prop}

\begin{proof}
Let $\sigma\in\Gamma_f$ and let $\eta_\sigma\colon G_\Q\to(\I^\circ_\Tr/\fP)^\times$ be the character associated with $\sigma$. We use the notations of Section \ref{liftsec}. By Corollary \ref{surjst} there exists a self-twist $\widetilde{\sigma}\colon\I^\circ_\Tr/\fP\to\I^\circ_\Tr/\fP$ with associated character $\eta_{\widetilde{\sigma}}\colon G_\Q\to(\I^\circ_\Tr/\fP)^\times$ such that $\fP$ is fixed under $\widetilde{\sigma}$, $\widetilde{\sigma}_\fP=\sigma$ and $\eta_{\widetilde{\sigma},\fP}=\eta_\sigma$. 
Since $\widetilde{\sigma}\in\Gamma$ and $\I_0^\circ=(\I_\Tr^\circ)^\Gamma$ we have $\I_0^\circ\subset(\I_\Tr^\circ)^{\langle\widetilde{\sigma}\rangle}$, where $\langle\widetilde\sigma\rangle$ is the cyclic group generated by $\widetilde\sigma$. 
Since $\widetilde{\sigma}$ leaves $\fP$ fixed, we can reduce modulo $\fP$ the previous inclusion to obtain $\I_0^\circ/\fP_0\subset(\I_\Tr^\circ)^{\langle\widetilde{\sigma}\rangle}/\fP$. Again since $\widetilde{\sigma}$ leaves $\fP$ fixed and $\widetilde{\sigma}$ induces $\sigma$ modulo $\fP$, we have $(\I_\Tr^\circ)^{\langle\widetilde{\sigma}\rangle}/\fP=(\I_\Tr^\circ/\fP)^{\langle{\sigma}\rangle}$, hence $\I_0^\circ/\fP_0\subset(\I_\Tr^\circ/\fP)^{\langle{\sigma}\rangle}$. This holds for every $\sigma$, so $\I_0^\circ/\fP_0\subset(\I_\Tr^\circ/\fP)^{\Gamma_f}$.
\end{proof}


The following corollary summarizes the work of this section.

\begin{cor}\label{stresfull}
Let $\rho\cong G_\Q\to\GSp_4(\I_\Tr^\circ)$ be the representation associated with the family $\theta$. Let $\fP$ be a prime of $\I_\Tr^\circ$ corresponding to a classical eigenform $f$ which is not a symmetric cube lift of a $\GL_2$-eigenform. Let $\fP_0=\fP\cap\I_0^\circ$. Then the image of $\rho_\fP\colon G_\Q\to\GSp_4(\I_\Tr^\circ/\fP)$ contains a non-trivial congruence subgroup of $\GSp_4(\I_0^\circ/\fP_0)$.
\end{cor}

\begin{proof}
As before let $\cO_E=\Z_p[\Tr\Ad\rho_\fP]$. By Theorem \ref{classbigim} the image of $\rho_\fP$ contains a congruence subgroup of $\GSp_4(\cO_E)$. By combining Propositions \ref{sttraces} and \ref{stringseq} we obtain $\I_0^\circ/\fP_0\subset\cO_E$, hence the corollary.
\end{proof}

\begin{rem}
In \cite{lang} and in \cite{cit}, where Galois images for families of $\GL_2$-eigenforms are studied, the intermediate step given by Proposition \ref{sttraces} is not necessary. Indeed the fullness result for the representation attached to a $\GL_2$-eigenform, due to Ribet and Momose \cite{momose}[Theorem 3.1]\cite{ribetII} is stated in terms of the ring fixed by the self-twists of the representation, hence an analogue of Proposition \ref{stringseq} is sufficient.
\end{rem}


\bigskip

\section{Constructing bases of lattices in unipotent subgroups}

In this section we show that the image of the Galois representation associated with a family of $\GSp_4$-eigenforms contains a ``sufficiently large'' set of unipotent elements.

\subsection{An approximation argument}

We prove a simple generalization of the approximation argument presented in the proof of \cite[Lemma 4.5]{hidatil}. We give the details of the proof since there is an imprecision in the one presented in \emph{loc. cit.}. In particular \cite[Lemma 4.6]{hidatil} does not give the inclusion (4.3) in \emph{loc. cit.}; it is replaced by Lemma \ref{communip} below. 
Let $\bG$ be a reductive group defined over $\Z$. Let $T$ and $B$ be a torus and a Borel subgroup of $\bG$, respectively. Let $\Delta$ be the set of roots associated with $(\bG,T)$.

\begin{prop}\label{approx}
Let $A$ be a profinite local ring of residual characteristic $p$ endowed with its profinite topology. Let $G$ be a compact subgroup of the level $p$ principal congruence subgroup $\Gamma_{\bG(A)}(p)$ of $\bG(A)$.
Suppose that:
\begin{enumerate}
\item the ring $A$ is complete with respect to the $p$-adic topology;
\item the group $G$ is normalized by a diagonal $\Z_p$-regular element of $\bG(A)$.
\end{enumerate}
Let $\alpha$ be a root of $\bG$. For every ideal $Q$ of $A$, let $\pi_Q\colon\bG(A)\to\bG(A/Q)$ be the natural projection, inducing a map $\pi_{Q,\alpha}\colon U^\alpha(A)\to U^\alpha(A/Q)$. Then $\pi_Q(G)\cap U^\alpha(A/Q)=\pi_Q(G\cap U^\alpha(A))$. 
\end{prop}


\begin{proof}
Let $\alpha$ be a root of $\bG$. Since the inclusion $\pi_Q(G\cap U^\alpha(A))\subset\pi_Q(G)\cap U^\alpha(A/Q)$ is trivial, it is sufficient to show that $\pi_Q\colon G\cap U^\alpha(A)\to\pi_Q(G)\cap U^\alpha(A/Q)$ is surjective. 
The unipotent subgroups $U^\alpha$ and $U^{-\alpha}$ generate a subgroup of $\bG(A)$ isomorphic to $\SL_2(A)$. We denote it by $\SL_2^\alpha(A)$. 
We identify $U^{\pm\alpha}$ with subgroups of $\SL_2^\alpha(A)$. 
Let $T^\alpha=T\cap\SL_2^\alpha$ and $B^\alpha=T^\alpha U^\alpha$. We also write $\fsl_2^\alpha, \fu^{\pm\alpha}, \ft^\alpha, \fb^{\pm\alpha}$ for the Lie algebras of the $\SL_2^\alpha,U^{\pm\alpha},T^\alpha,B^{\pm\alpha}$, respectively. 
For every positive integer $j$, we denote by $\pi_{Q^j}$ the natural projection $\bG(A)\to\bG(A/Q^j)$, as well as its restriction $\SL_2^\alpha(A)\to\SL_2^\alpha(A/Q^j)$. We define some congruence subgroups of $\SL_2^\alpha(A)$ of level $pQ^j$ by setting
\begin{gather*} 
\Gamma_A(Q^j)=\{x\in\SL_2^\alpha\cap\Gamma_A(p)\,\vert\,\pi_{Q^j}(x)=\1_{2g}\}, \\
\Gamma_{?^\alpha}(Q^j)=\{x\in\SL_2^\alpha\cap\Gamma_A(p)\,\vert\,\pi_{Q^j}(x)\in ?^\alpha(A/Q^j)\}\textrm{ for }?\in\{U,B\}.
\end{gather*} 
Note that we leave the level at $p$ implicit. 
We set $G_{?^\alpha}(Q^j)=G\cap\Gamma_{?^\alpha}(Q^j)$ for $?\in\{U,B\}$. 
Given two elements $X,Y\in\bG(A)$, we denote by $[X,Y]$ their commutator $XYX^{-1}Y^{-1}$. For every subgroup $H\subset\bG(A)$ we denote by $\mathrm{D}H$ its commutator subgroup $\{ [X,Y]\,\vert\, X,Y\in H \}$. We write $[\cdot,\cdot]_\Lie$ for the Lie bracket on $\fgsp_{2g}(A)$. 

\begin{lemma}\label{communip}
For every $j\geq 1$ we have $\DGamma_{U^\alpha}(Q)\subset\Gamma_{B^\alpha}(Q^{2j})\cap\Gamma_{U^\alpha}(Q^j)$. 
\end{lemma}

\begin{proof}
A matrix $X\in\Gamma_{U^\alpha}(Q^j)$ can be written in the form $X=UM$ where $U\in U^\alpha$ and $M\in\Gamma_A(Q^j)$. In particular its logarithm is defined, it satisfies $\exp(\log X)=X$ and it is of the form $\log X=u+m$ with $u\in\fu^\alpha(A)\subset\fsl_2^\alpha(A)$ and $m\in Q^j\fsl_2^\alpha(A)$.
Now let $X,X_1\in\Gamma_{U^\alpha}(Q^j)$ and let $\log X=u+m$ and $\log X_1=u_1+m_1$ be decompositions of the type described above.
Modulo $Q^{2j}$ we can calculate
\[ \log [X,X_1]\equiv [\log X,\log X_1]_\Lie\equiv [u,u_1]_\Lie+[m,u_1]_\Lie+[u,m_1]_\Lie+[m,m_1]_\Lie. \]
Since $u,u_1\in\fu^\alpha$ and $m,m_1\in Q^j\fsl_2^\alpha(A)$ we have $[u,u_1]_\Lie=0$ and $[m,m_1]_\Lie\in Q^{2j}\fsl_2^\alpha(A)$, so
\[ \log [X,X_1]\equiv [m,u_1]_\Lie+[u,m_1]_\Lie\pmod{Q^{2j}}. \]
Now write $m=u^{-\alpha}+b^\alpha$ and $m_1=u_1^{-\alpha}+b^\alpha_1$ with $u^{-\alpha},u^{-\alpha}_1\in Q^j\fu^{-\alpha}(A)$ and $b^{-\alpha},b^{-\alpha}_1\in Q^j\fb^\alpha(A)$. Then $[m,u_1]_\Lie=[u^{-\alpha},u_1]_\Lie+[b^\alpha,u_1]_\Lie$, which belongs to $Q^j\fb^\alpha(A)$ since $[u^{-\alpha},u_1]_\Lie\in Q^j\ft^\alpha(A)$ and $[b^\alpha,u_1]_\Lie\in Q^j\fb^\alpha(A)$. In the same way we see that $[u,m_1]_\Lie\in\fb^\alpha(A)$. We conclude that $\log [X,X_1]\in Q^j\fb^\alpha\pmod{Q^{2j}}$, so $[X,X_1]\in\Gamma_{B^\alpha}(Q^{2j})$. Trivially $[X,X_1]\in\Gamma_{U^\alpha}(Q^j)$, so this proves the lemma.
\end{proof}

Let $d\in G$ be a diagonal $\Z_p$-regular element. Since $A$ is $p$-adically complete the limit $\lim_{n\to\infty}d^{p^n}$ defines a diagonal element $\delta\in\bG(A)$. Clearly the order of $\delta$ in $\bG(A)$ is a divisor $a$ of $p-1$. 
By hypothesis $G$ is a compact subgroup of $\Gamma_{\bG(A)}(p)$, so $G$ is a pro-$p$ group and $\delta$ normalizes $G$. We denote by $\ad(\delta)$ the adjoint action of $\delta$ on $\bG(A)$.

Let $\Gamma_A(p)$ be the principal congruence subgroup of $\SL_2^\alpha(A)$ of level $p$. 
Every element of $\Gamma_A(p)$ has a unique $a$-th root in $\Gamma_A(p)$. Since $\delta$ is diagonal, it normalizes $\Gamma_A(p)$. We define a map $\Delta\colon\Gamma_A(p)\to\Gamma_A(p)$ by setting
\[ \Delta(x)=\left(x\cdot(\ad(\delta)(x))^{\alpha(\delta)^{-1}}\cdot(\ad(\delta^2)(x))^{\alpha(\delta)^{-2}}\cdots(\ad(\delta^{a-1})(x))^{\alpha(\delta)^{1-a}}\right)^{1/a} \]
for every $x\in\Gamma_A(p)$. Note that $\Delta$ is not a homomorphism, but it induces a homomorphism of abelian groups $\Delta^\ab\colon\Gamma_A(p)/\DGamma_A(p)\to\Gamma_A(p)/\DGamma_A(p)$. 
%

\begin{lemma}(cf. \cite[Lemma 4.7]{hidatil})\label{doubleexp}
If $u\in\Gamma_{U^\alpha}(Q^j)$ for some positive integer $j$, then $\pi_{Q^j}(\Delta(u))=\pi_{Q^j}(u)$ and $\Delta^2(u)\in\Gamma_{U^\alpha}(Q^{2j})$.
\end{lemma}

\begin{proof}
Let $u\in\Gamma_{U^\alpha}(Q^j)$.
We see that $\Delta$ maps $Q^j\Gamma_A(p)$ to itself, so it induces a map $\Delta_{Q^j}\colon\Gamma_A(p)/Q^j\Gamma_A(p)\to\Gamma_A(p)/Q^j\Gamma_A(p)$.
For $x\in U^\alpha(A/Q^j)$ we have $\pi_{Q^j}(\ad(\delta)(x))=\ad(\pi_{Q^j}(\delta))(x)=\pi_{Q^j}(\alpha(\delta))(x)$. From this we deduce that $\Delta_{Q^j}(x)=x$ for $x\in U^\alpha(A/Q^j)$. Since $\pi_{Q^j}(u)\in U^\alpha(A/Q^j)$ we obtain $\pi_{Q^j}(\Delta(u))=\Delta_{Q^j}(\pi_{Q^j}(u))=\pi_{Q^j}(u)$.

Consider the homomorphism $\Delta^\ab\colon\Gamma_A(p)/\DGamma_A(p)\to\Gamma_A(p)/\DGamma_A(p)$. By a direct computation we see that $\ad(\delta)(\Delta^\ab(x))=\alpha(\delta)(\Delta^\ab(x))$ for every $x\in\Gamma_A(p)/\DGamma_A(p)$, so the image of $\Delta^\ab$ lies in the $\alpha(\delta)$-eigenspace for the action of $\ad(\delta)$ on $\Gamma_A(p)/\DGamma_A(p)$. This space is $U^\alpha(A)\DGamma_A(p)/\DGamma_A(p)$, as we can see by looking at the Iwahori decomposition of $\Gamma_A(p)$. 

From the first part of the proposition it follows that $\Delta^\ab$ induces a homomorphism 
\[ \Delta^\ab_{\Gamma_{U^\alpha}}\colon\Gamma_{U^\alpha}(Q^j)/\DGamma_{U^\alpha}(Q^j)\to\Gamma_{U^\alpha}(Q^j)/\DGamma_{U^\alpha}(Q^j). \]
By the remark of the previous paragraph 
\[ \Delta^\ab_{\Gamma_{U^\alpha}}(\Gamma_{U^\alpha}(Q^j)/\DGamma_{U^\alpha}(Q^j))\subset\Gamma_{U_{\alpha}}(Q^j)\DGamma_{U^\alpha}(Q^j)/\DGamma_{U^\alpha}(Q^j). \] 
By Lemma \ref{communip} $\DGamma_{U^\alpha}(Q^j)\subset\Gamma_{B^\alpha}(Q^{2j})\cap\Gamma_{U^\alpha}(Q^j)$, so 
\[ \Delta^\ab_{\Gamma_{U^\alpha}}(\Gamma_{U^\alpha}(Q^j)/\DGamma_{U^\alpha}(Q^j))\subset\Gamma_{B^\alpha}(Q^{2j})\cap\Gamma_{U^\alpha}(Q^j)/\DGamma_{U^\alpha}(Q^j). \]
We deduce that $\Delta(u)\in\Gamma_{B^\alpha}(Q^{2j})\cap\Gamma_{U^\alpha}(Q^j)$. 

By the same reasoning as above, $\Delta$ induces a homomorphism 
\[ \Delta^\ab_{\Gamma_{B^\alpha}}\colon\Gamma_{B^\alpha}(Q^{2j})/\DGamma_{B^\alpha}(Q^{2j})\to\Gamma_{B^\alpha}(Q^{2j})/\DGamma_{B^\alpha}(Q^{2j}). \] 
The image of $\Delta^\ab_{\Gamma_{B^\alpha}}$ is in the $\alpha(\delta)$-eigenspace for the action of $\ad(\delta)$, that is $U^\alpha(Q^{2j})\DGamma_{B^\alpha}(Q^{2j})/\DGamma_{B^\alpha}(Q^{2j})$. Note that $\DGamma_{B^\alpha}(Q^{2j})\subset U^\alpha(Q^{2j})$, so
\[ \Delta^\ab_{\Gamma_{B^\alpha}}(\Gamma_{B^\alpha}(Q^{2j})/\DGamma_{B^\alpha}(Q^{2j}))\subset\Gamma_{U^\alpha}(Q^{2j})/\DGamma_{B^\alpha}(Q^{2j}). \]
Since $\Delta(u)\in\Gamma_{B^\alpha}(Q^{2j})$ we conclude that $\Delta^2(u)\in\Gamma_{U^\alpha}(Q^{2j})$.
\end{proof}

We look at $G\cap U^\alpha(A)$ and $\pi_Q(G)\cap\SL_2(A/Q)$ as subgroups of $\SL_2^\alpha(A)$ and $\SL_2^\alpha(A/Q)$, respectively. 
Let $\ovl{u}\in \pi_Q(G)\cap U^\alpha(A/U)$. Choose $u_1\in G$ and $u_2\in U^\alpha(A)$ such that $\pi_Q(u_1)=\pi_Q(u_2)=\ovl{u}$. Then $u_1u_2^{-1}\in\Gamma_A(Q)$, so $u_1\in G\cap\Gamma_{U^\alpha}(Q)$. Note that $G\cap\Gamma_{U^\alpha}(Q)$ is compact since $G$ and $\Gamma_{U^\alpha}(Q)$ are pro-$p$ groups. By Lemma \ref{doubleexp} we have $\pi_Q(\Delta^{2^m}(u_1))=\ovl{u}$ and $\Delta^{2^m}(u_1)\in\Gamma_{U^\alpha}(Q^{2m})$ for every positive integer $m$. 
Hence the limit $\lim_{m\to\infty}\Delta^{2^m}(u_1)$ defines an element $u\in\SL_2(A)$ satisfying $\pi_Q(u)=\ovl{u}$. We have $u\in G\cap\Gamma_{U^\alpha}(Q)$ because $G\cap\Gamma_{U^\alpha}(Q)$ is compact. This completes the proof of the proposition. 
\end{proof}

We give a simple corollary.

\begin{cor}\label{nontrivunip}
Let $\rho\colon G_\Q\to\GSp_4(\I^\circ_\Tr)$ be the Galois representation associated with a finite slope family $\theta\colon\T_h\to\I^\circ$. For every root $\alpha$ of $\GSp_4$ the group $\im\rho\cap U^\alpha(\I_\Tr^\circ)$ is non-trivial. 
\end{cor}

\begin{proof}
Let $\fP$ be a prime of $\I^\circ$ corresponding to a classical eigenform $f$ which is not the symmetric cube lift of a $\GL_2$-eigenform. 
Let $\cO_E=\Z_p[\Tr(\Ad\rho_\fP)]$. 
By Theorem \ref{classbigim} $\im\rho_\fP$ contains a non-trivial congruence subgroup of $\Sp_4(\cO_E)$. In particular $\im\rho_\fP\cap U^\alpha(\I_\Tr^\circ/\fP)$ is non-trivial for every root $\alpha$. 
Now we apply Proposition \ref{approx} to $\bG=\GSp_4$, $T=T_2$, $B=B_2$ $A=\I^\circ_\Tr$, $G=\im\rho$ and $Q=\fP$. We obtain that the projection $\im\rho\cap U^\alpha(\I_\Tr^\circ)\to\im\rho_\fP\cap U^\alpha(\I_\Tr^\circ/\fP)$ is surjective for every $\alpha$. In particular $\im\rho\cap U^\alpha(\I_\Tr^\circ)$ must be non-trivial for every $\alpha$.
\end{proof}

\subsection{A representation with image fixed by the self-twists}\label{repfix}

Let $\theta\colon\T_h\to\I^\circ$ be a finite slope family with associated representation $\rho\colon G_\Q\to\GSp_4(\I^\circ_\Tr)$. As before we assume $\rho$ to be residually irreducible and $\Z_p$-regular. Let $\Gamma$ be the group of self-twist of $\rho$ and let $\I_0^\circ$ be the subring of $\I_\Tr^\circ$ consisting of the elements fixed by $\Gamma$. 
By restricting the domain of $\rho$ and replacing it with a suitable conjugate representation, we obtain a $\Z_p$-regular representation with coefficients in $\I_0^\circ$. This is the main result of this section. 

We write $\eta_\sigma$ for the finite order Galois character associated with $\sigma\in\Gamma$. Let $H_0=\bigcap_{\sigma\in\Gamma}\ker\eta_\sigma$. Since $\Gamma$ is finite $H_0$ is open and normal in $G_\Q$. Note that $\Tr(\rho(H_0))\subset\GSp_4(\I_0^\circ)$. 
If $\ovl{\rho}\vert_{H_0}$ is irreducible, then by Theorem \ref{carayol} there exists $g\in\GL_4(\I_\Tr^\circ)$ such that the representation $\rho^g=g\rho g^{-1}$ satisfies $\im\rho^g\vert_{H_0}\subset\GL_4(\I_0^\circ)$. 
The hypothesis of irreducibility of $\ovl\rho\vert_{H_0}$ can probably be checked in the cases we will focus on (residually full or of symmetric cube type), but it would be too restrictive if we wanted to generalize our work to other interesting cases (for instance to lifts from $\GL_{2/F}$ with $F/\Q$ real quadratic or from $\GL_1/F$ with $F/\Q$ CM of degree $4$). For this reason we do not make the above assumption and we follow instead the approach of \cite[Proposition 4.14]{cit}, that comes in part from the proof of \cite[Theorem 7.5]{lang}. 

\begin{prop}\label{H0repr}
There exists an element $g\in\GSp_4(\I_\Tr^\circ)$ such that:
\begin{enumerate}
\item $g\rho g^{-1}(H_0)\subset\GSp_4(\I_0)$;
\item $g\rho g^{-1}(H_0)$ contains a diagonal $\Z_p$-regular element.
\end{enumerate}
\end{prop}

\begin{proof}
Let $V$ be a free, rank four $\I_\Tr^\circ$-module. The choice of a basis of $V$ determines an isomorphism $\GL_4(\I_\Tr^\circ)\cong\Aut(V)$, hence an action of $\rho$ on $V$. 
Let $d$ be a $\Z_p$-regular element contained in $\im\rho$. We denote by $\{e_i\}_{i=1,\ldots,4}$ a symplectic basis of $V$ such that $d$ is diagonal. Until further notice we work in this basis. 

By definition of self-twist, for each $\sigma\in\Gamma$ there is an equivalence $\rho^\sigma\cong\eta_\sigma\otimes\rho$. This means that there exists a matrix $C_\sigma\in\GSp_4(\I_\Tr^\circ)$ such that
\begin{equation}\label{Csigma} \rho^\sigma(g)=\eta_\sigma C_\sigma\rho(g) C_\sigma^{-1}. \end{equation} 
Recall that we write $\fm_{\I_\Tr^\circ}$ for the maximal ideal of $\I_\Tr^\circ$ and $\F$ for the residue field of $\I_\Tr^\circ$.
Let $\ovl{C}_\sigma$ be the image of $C_\sigma$ under the natural projection $\GSp_4(\I_\Tr^\circ)\to\GSp_4(\F)$. We prove the following lemma.

\begin{lemma}\label{Cdiag}
For every $\sigma\in\Gamma$ the matrix $C_\sigma$ is diagonal and the matrix $\ovl{C_\sigma}$ is scalar.
\end{lemma}

\begin{proof}
Let $\alpha$ be any root of $\GSp_4$ and $u^\alpha$ be a non-trivial element of $\im\rho\cap U^\alpha(\I_\Tr^\circ)$. 
Such a $u^\alpha$ exists thanks to Corollary \ref{nontrivunip}. Let $g^\alpha$ be an element of $G_\Q$ such that $\rho(g^\alpha)=u^\alpha$. By evaluating Equation \eqref{Csigma} at $g^\alpha$ we obtain $C_\sigma u^\alpha C_\sigma^{-1}=(u^\alpha)^\sigma$, which is again an element of $U^\alpha(\I_\Tr^\circ)$. 
We deduce that $C_\sigma$ normalizes $U^\alpha(Q(\I_\Tr^\circ))$. This holds for every root $\alpha$, so $C_\sigma$ normalizes the Borel subgroups of upper and lower triangular matrices in $\GSp_4(Q(\I_\Tr^\circ))$. Since a Borel subgroup is its own normalizer, we conclude that $C_\sigma$ is diagonal.

By Proposition \ref{liftdefprop}(1) the action of $\Gamma$ on $\I_\Tr^\circ$ induces the trivial action of $\Gamma$ on $\F$. By evaluating Equation \eqref{Csigma} at $g^\alpha$ and modulo $\fm_{\I_\Tr^\circ}$ we obtain, with the obvious notations, $\ovl{C}_\sigma\ovl{u}^\alpha(\ovl{C}_\sigma)^{-1}=(\ovl{u}^\alpha)^\sigma=\ovl{u}^\alpha$. Since $C_\sigma$ is diagonal and $\ovl{u}^\alpha\in U^\alpha(\F)$, the left hand side is equal to $\alpha(\ovl{C}_\sigma)\ovl{u}^\alpha$. We deduce that $\alpha(\ovl{C}_\sigma)=1$. Since this holds for every root $\alpha$, we conclude that $\ovl{C}_\sigma$ is scalar.
\end{proof}

We write $C$ for the map $\Gamma\to\GSp_4(\I_\Tr^\circ)$ defined by $C(\sigma)=C_\sigma$. We show that $C$ can be modified into a $1$-cocycle $C^\prime$ such that $C^\prime(\sigma)$ still satisfies Equation \eqref{Csigma}.
Define a map $c\colon\Gamma^2\to\GSp_4(\I_\Tr^\circ)$ by $c(\sigma,\tau)=C_{\sigma\tau}^{-1}C_\sigma^\tau C_\tau$ for every $\sigma,\tau\in\Gamma$. 
By using multiple times Equation \eqref{Csigma} we find that, for every $g\in G_\Q$, 
$\rho(g)=\eta_{\sigma\tau}^{-1}\eta^\tau_\sigma\eta_\tau c(\sigma,\tau)\rho(g)c(\sigma,\tau)^{-1}$. 
Recall that $\eta^\tau_\sigma\eta_\tau=\eta_{\sigma\tau}$ by Proposition \ref{basicst}(4), so the matrix $c(\sigma,\tau)$ commutes with the image of $\rho$. Since $\rho$ is irreducible, $c(\sigma,\tau)$ must be a scalar. 

For every $\sigma\in\Gamma$ and every $i\in\{1,2,3,4\}$ let $(C_\sigma)_i$ denote the $i$-th diagonal entry of $C_\sigma$. Define a map $C_i^\prime\colon\Gamma\to\GSp_4(\I_\Tr^\circ)$ by $C_i^\prime(\sigma)=(C_\sigma)_i^{-1}C_\sigma$. Let $c_i^\prime(\sigma,\tau)=C_i^\prime(\sigma\tau)^{-1}C_i^\prime(\sigma)^\tau C_i^\prime(\tau)$ for every $\sigma,\tau\in\Gamma$ and $i\in\{1,2,3,4\}$. Then
\begin{equation}\label{cocycle} c_i^\prime(\sigma,\tau)=((C_{\sigma\tau})_i^{-1}(C_\sigma)_i(C_\tau)_i)^{-1}c(\sigma,\tau). \end{equation}
Since $(C_{\sigma\tau})_i^{-1}(C_\sigma)_i^\tau(C_\tau)_i$ is the $i$-th diagonal entry of $c(\sigma,\tau)$ and $c(\sigma,\tau)$ is scalar, $(C_{\sigma\tau})_i^{-1}(C_\sigma)_i(C_\tau)_i$ is independent of $i$ and $((C_{\sigma\tau})_i^{-1}(C_\sigma)_i(C_\tau)_i)^{-1}c(\sigma,\tau)=\1_4$ for every $i$. From Equation \eqref{cocycle} we deduce that $C_i^\prime$ is a $1$-cocycle. 

Set $C^\prime_\sigma=C_i^\prime(\sigma)$. We have
\begin{equation}\label{Cprsigma} \rho^\sigma(g)=\eta_\sigma C_\sigma\rho(g) C_\sigma^{-1}=\eta_\sigma C^\prime_\sigma\rho(g) (C^\prime_\sigma)^{-1}. \end{equation}
By Lemma \ref{Cdiag} $\ovl{C}_\sigma$ is scalar, so we $\ovl{C^\prime_\sigma}=(\ovl{C}_\sigma)_i^{-1}\ovl{C}_\sigma=\1_4$ with the obvious notations.

Recall that $\{e_i\}_{i=1,\ldots,4}$ is our chosen basis of the free $\I_\Tr^\circ$-module $V$, on which $G_\Q$ acts via $\rho$. For every $v\in V$ we write as $v=\sum_{i=1}^4\lambda_i(v)e_i$ its unique decomposition in the basis $(e_i)_{i=1,\ldots,4}$, with $\lambda_i(v)\in\I_\Tr^\circ$ for $1\leq i\leq 4$.
For every $v\in V$ and every $\sigma\in\Gamma$ we set $v^{[\sigma]}=(C^\prime_\sigma)^{-1}\sum_{i=1}^4\lambda_i(v)^\sigma e_i$. This defines an action of $\Gamma$ on $V$ since $C^\prime_\sigma$ is a $1$-cocycle. Let $V^{[\Gamma]}$ denote the set of elements of $V$ fixed by $\Gamma$. The action of $\Gamma$ is clearly $\I_0^\circ$-linear, so $V^{[\Gamma]}$ has a structure of $\I_0^\circ$-submodule of $V$. 

Let $v\in V^{[\Gamma]}$ and $h\in H_0$. Then $\rho(h)v$ is also in $V^{[\Gamma]}$, as we see by a direct calculation using Equation \eqref{Cprsigma}. 
We deduce that the action of $G_\Q$ on $V$ via $\rho$ induces an action of $H_0$ on $V^{[\Gamma]}$. We will conclude the proof of the proposition after having studied the structure of $V^{[\Gamma]}$.

\begin{lemma}
The $\I_0^\circ$-submodule $V^{[\Gamma]}$ of $V$ is free of rank four and its $\I_\Tr^\circ$-span is $V$. 
\end{lemma}

\begin{proof}
Choose $i\in\{1,\ldots,4\}$. We construct a non-zero, $\Gamma$-invariant element $w_i\in\I_\Tr^\circ e_i$. 
The submodule $\I_\Tr^\circ e_i$ is stable under $\Gamma$ because $C^\prime_\sigma$ is diagonal. 
The action of $\Gamma$ on $\I_\Tr^\circ e_i$ induces an action of $\Gamma$ on the one-dimensional $\F$-vector space $\I_\Tr^\circ e_i\otimes_{\I_\Tr^\circ}\F$. Recall that the self-twists induce the identity on $\F$ by Proposition \ref{liftdefprop}(1) and that the matrix $\ovl{C^\prime_\sigma}$ is trivial for every $\sigma\in\Gamma$, so $\Gamma$ acts trivially on $\I_\Tr^\circ e_i\otimes_{\I_\Tr^\circ}\F$. 

Now choose any $v_i\in\I_\Tr^\circ e_i$. Let $w_i=\sum_{\sigma\in\Gamma}v_i^{[\sigma]}$. Clearly $w_i$ is invariant under the action of $\Gamma$. We show that $w_i\neq 0$. Let $\ovl{v}_i$, $\ovl{w}_i$ denote the images of $v_i$ and $w_i$, respectively, via the natural projection $\I_\Tr^\circ e_i\to\I_\Tr^\circ e_i\otimes_{\I_\Tr^\circ}\F$. Then $\ovl{w}_i=\sum_{\sigma\in\Gamma}\ovl{v}_i^{[\sigma]}=\sum_{\sigma\in\Gamma}\ovl{v}_i=\card(\Gamma)\cdot\ovl{v}_i$ because $\Gamma$ acts trivially on $\I_\Tr^\circ e_i\otimes_{\I_\Tr^\circ}\F$. By Lemma \ref{cardgamma} the only possible prime factors of $\card(\Gamma)$ are $2$ and $3$. Since we supposed that $p\geq 5$ we have $\card(\Gamma)\neq 0$ in $\F$. We deduce that $\ovl{w}_i=\card(\Gamma)\ovl{v}_i\neq 0$ in $\F$, so $w_i\neq 0$.

Note that $\{w_i\}_{i=1,\ldots,4}$ is an $\I_\Tr^\circ$-basis of $V$ since $\ovl{w_i}\neq 0$ for every $i$. In particular the $\I_0^\circ$-span of the set $\{w_i\}_{i=1,\ldots,4}$ is a free, rank four $\I_0^\circ$-submodule of $V$. 
Since $V^{[\Gamma]}$ has a structure of $\I_0^\circ$-module and $w_i\in V^{[\Gamma]}$ for every $i$, there is an inclusion $\sum_{i=1}^4\I_0^\circ w_i\subset V^{[\Gamma]}$. We show that this is an equality. Let $v\in V^{[\Gamma]}$. Write $v=\sum_{i=1}^4\lambda_iw_i$ for some $\lambda_i\in\I_\Tr^\circ$. Then for every $\sigma\in\Gamma$ we have $v=v^{[\sigma]}=\sum_{i=1}^4\lambda_i^\sigma w_i^{[\sigma]}=\sum_{i=1}^4\lambda_i^\sigma w_i$. Since $\{w_i\}_{i=1,\ldots,4}$ is an $\I_\Tr^\circ$-basis of $V$, we must have $\lambda_i=\lambda_i^\sigma$ for every $i$. This holds for every $\sigma$, so we obtain $\lambda_i\in\I_0^\circ$ for every $i$. Hence $v=\sum_{i=1}^4\lambda_iw_i\in\sum_{i=1}^4\I_0^\circ w_i$. 

The second assertion of the lemma follows immediately from the fact that the set $\{w_i\}_{i=1,\ldots,4}$ is contained in $V^{[\Gamma]}$ and is an $\I_\Tr^\circ$-basis of $V$.
\end{proof}

Now let $h\in H_0$. Let $\{w_i\}_{i=1,\ldots,4}$ be an $\I_0^\circ$-basis of $V^{[\Gamma]}$ satisfying $w_i\in\I_\Tr^\circ e_i$, such as that provided by the lemma. Since $\I_\Tr^\circ\cdot V^{[\Gamma]}=V$, $\{w_i\}_{i=1,\ldots,4}$ is also an $\I_\Tr^\circ$-basis of $V$. Moreover $\{w_i\}_{i=1,\ldots,4}$ is a symplectic basis of $V$, since $w_i\in\I_\Tr^\circ e_i$ for every $i$ and $\{e_i\}$ is a symplectic basis. We show that the basis $\{w_i\}_{i=1,\ldots,4}$ has the two properties we want. 
\begin{enumerate}
\item By previous remarks $V^{[\Gamma]}$ is stable under $\rho$, so $\rho(h)w_i=\sum_{i=1}^4a_{ij}w_j$ for some $a_{ij}\in\I_0^\circ$. This implies that the matrix coefficients of $\rho(h)$ in the basis $\{w_i\}_{i=1,\ldots,4}$ belong to $\I_0^\circ$. Since $\{w_i\}_{i=1,\ldots,4}$ is a symplectic basis, we have $\rho(h)\in\GSp_4(\I_0^\circ)$.
\item By our choice of $\{e_i\}_{i=1,\ldots,4}$, the $\Z_p$-regular element $d$ is diagonal in this basis. Since $w_i\in\I_\Tr^\circ e_i$, $d$ is still diagonal in the basis $\{w_i\}_{i=1,\ldots,4}$.
\end{enumerate}
\end{proof}

From now on we always work with a $\Z_p$-regular conjugate of $\rho$ satisfying $\rho(H_0)\subset\GSp_4(\I^\circ_0)$.

\subsection{Lifting unipotent elements}

We give a definition and a lemma that will be important in the following. 
Let $B\into A$ an integral extension of Noetherian integral domains. 
We call \emph{$A$-lattice in $B$} an $A$-submodule of $B$ generated by the elements of a basis of $Q(B)$ over $Q(A)$. 
The following lemma is essentially \cite[Lemma 4.10]{lang}. 

\begin{lemma}\label{lattice}
Every $A$-lattice in $B$ contains a non-zero ideal of $B$. Conversely, every non-zero ideal of $B$ contains an $A$-lattice in $B$.
\end{lemma}

Let $\theta\colon\T_h\to\I^\circ$ be a finite slope family of $\GSp_4$-eigenforms and let $\rho\colon G_\Q\to\GSp_4(\I_\Tr^\circ)$ be the representation associated with $\theta$. 
For every root $\alpha$, we identify the unipotent group $U^\alpha(\I_0^\circ)$ with $\I_0^\circ$ and $\im\rho\cap U^\alpha(\I_0^\circ)$ with a $\Z_p$-submodule of $\I_0^\circ$. The goal of this section is to show that, for every $\alpha$, $\im\rho\cap U^\alpha$ contains a basis of a $\Lambda_h$-lattice in $\I_0^\circ$. Our strategy is similar to that of \cite[Section 4.4]{cit}, which in turn is inspired by \cite{hidatil} and \cite{lang}. We proceed in two main steps, by showing that: 
\begin{enumerate}
\item there exists a non-critical arithmetic prime $P_\uk\subset\Lambda_h$ such that $\im\rho_{P_\uk\I_0^\circ}\cap U^\alpha(\I_0^\circ/P_\uk\I_0^\circ)$ contains a basis of a $\Lambda_h/P_\uk$-lattice in $\I_0^\circ/P_\uk\I_0^\circ$;
\item the natural morphism $\im\rho\cap U^\alpha(\I_0^\circ)\to\im\rho_{P_\uk\I_0^\circ}\cap U^\alpha(\I_0^\circ/P_\uk\I_0^\circ)$ is surjective, so we can lift a basis as in point (1) to a basis of a $\Lambda_h$-lattice in $\I_0^\circ$.
\end{enumerate}
Part (1) is proved via Theorem \ref{classbigim} and a result about the lifting of self-twists from $\rho_{P_\uk\I_0^\circ}$ to $\rho$ (Proposition \ref{lifttwists}).
Part (2) will result from an application of Proposition \ref{approx}. 


We start by showing that we can choose an arithmetic prime with special properties.

\begin{lemma}\label{specialprime}
Suppose that $\ovl\rho$ is either full or of symmetric cube type. 
Then there exists an arithmetic prime $P_\uk$ of $\Lambda_h$ such that:
\begin{enumerate}
\item $P_\uk$ is non-critical for $\I^\circ$ in the sense of Definition \ref{noncritgsp};
\item for every prime $\fP\subset\I^\circ$ lying above $P_\uk$, the classical eigenform corresponding to $\fP$ is not the symmetric cube lift of a $\GL_2$-eigenform.
\end{enumerate}
\end{lemma}

\begin{proof}
Let $\Sigma^\ncr$ be the set of non-critical arithmetic primes of $\Lambda_h$. By Lemma \ref{noncritdense} $\Sigma^\ncr$ is Zariski-dense in $\Lambda_h$. Let $\fP$ be a prime of $\I^\circ$ lying over a prime in $\Sigma^\ncr$ and corresponding to the symmetric cube lift of a $\GL_2$-eigenform $f$. Note that this is impossible if $\ovl\rho$ is full, so every prime of $\Sigma^\ncr$ satisfies conditions (1) and (2) in this case. Let $\rho_\fP$ and $\rho_{f,p}$ be the Galois representations associated with $\fP$ and $f$, respectively, satisfying $\rho_\fP\cong\Sym^3\rho_{f,p}$. If the weight of $f$ is $k$ then the representation $\rho_{f,p}$ is Hodge-Tate with Hodge-Tate weights $0$ and $k-1$. A simple calculation shows that $\Sym^3\rho_{f,p}$ is Hodge-Tate with Hodge-Tate weights $(0,k-1,2k-2,3k-3)$, hence $\fP$ lies over the arithmetic prime $P_{(2k-1,k+1)}$. The set $\{P_{(2k-1,k+1)}\}_{k\geq 2}$ is not Zariski-dense in $\Lambda_h$. In particular the set $\Sigma^\ncr-\{P_{(2k-1,k+1)}\}_{k\geq 2}$ is non-empty. A prime in this set satisfies conditions (1) and (2).
\end{proof}

We fix for the rest of the section an arithmetic prime $P_\uk$ of $\Lambda_h$ satisfying the conditions (1) and (2) in Lemma \ref{specialprime}.

Let $\fm_0$ denote the maximal ideal of $\I_0^\circ$. 
Let $H=\{g\in H_0\,\vert\,\rho(g)\equiv 1\pmod{\fm_0}\}$, that is a normal open subgroup of $H_0$. 
Thanks to Proposition \ref{H0repr} we can suppose that $\rho(H_0)\subset\GSp_4(\I_0^\circ)$. We define a representation $\rho_0\colon H\to\Sp_4(\I_0^\circ)$ by setting
\[ \rho_0=\rho\vert_{H}\otimes\det(\rho\vert_H)^{-1/2}. \]
Here the square root is defined via the usual power series, that converges on $\rho(H)$. 
Even though our results are all stated for the representation $\rho$, in an intermediate step we will need to work with $\rho_0$ and its reduction modulo a prime ideal of $\I_0^\circ$. In order to transfer our results to $\rho_0$ we need to relate the images of the two representations to each other. 

\subsection{Subnormal and congruence subgroups of symplectic groups}

Let $R$ be a local ring in which $2$ is a unit.
In the proof of \cite[Proposition 5.3]{lang}, the author compares the images of $\rho$ and $\rho_0$ via the classification of the subnormal subgroups of $\GL_2(R)$ by Tazhetdinov. Our technique relies on the analogous classification of the subnormal subgroups of $\Sp_4(R)$, which is also due to Tazhetdinov \cite{tazhet}. 
If $N$ and $K$ are two groups, we say that $N$ is a subnormal subgroup of $K$ if there exists $m\in\N$ and subgroups $K_i$ of $K$, for $i=0,1,2,\ldots,m$, such that 
\[ N=K_0\subset K_1\subset K_2 \subset\ldots\subset K_m=K \]
and $K_i$ is normal in $K_{i+1}$ for every $i\in\{0,1,\ldots,m-1\}$. 
We will only need the following result, that is a corollary of \cite[Theorem]{tazhet}.

\begin{thm}\label{tazhetcor}
If $N$ is a subnormal subgroup of $\Sp_4(R)$ and it is not contained in $\{\pm 1\}$, then it contains a non-trivial congruence subgroup of $\Sp_4(R)$.
\end{thm}


Let $P_\uk$ be the arithmetic prime we chose in the beginning of the section. 
By the étaleness condition in Definition \ref{noncritgsp}, $P_\uk\I^\circ$ is an intersection of distinct primes of $\I^\circ$, so $P_\uk\I^\circ_0$ is an intersection of distinct primes of $\I_0^\circ$. Let $\fQ_1,\fQ_2,\ldots,\fQ_d$ be the prime divisors of $P_\uk\I_0^\circ$.

Let $\fI$ be either $P_\uk\I_0^\circ$ or $\fQ_i$ for some $i\in\{1,2,\ldots,d\}$. Let $\rho_\fI\colon H_0\to\GSp_4(\I_0^\circ/\fI)$ and $\rho_{0,\fI}\colon H\to\Sp_4(\I_0^\circ/\fI)$ be the reductions modulo $\fI$ of $\rho$ and $\rho_0$, respectively. Let $\cG=\rho_\fI(H)$ and $\cG_0=\rho_{0,\fI}(H)$. 
Let $f\colon\GSp_4(\I^\circ_0)\to\Sp_4(\I^\circ_0)$ be the homomorphism sending $g$ to $\det(g)^{-1/2}g$.
We have $\cG=f(\cG_0)$ by definition of $\rho_0$. We show an analogue of \cite[Proposition 4.22]{cit}. 

\begin{lemma}\label{rhorho0}
The group $\cG$ contains a non-trivial congruence subgroup of $\Sp_4(\I_0^\circ/\fI)$ if and only if the group $\cG_0$ contains a non-trivial congruence subgroup of $\Sp_4(\I_0^\circ/\fI)$.
\end{lemma}

\begin{proof}
Clearly the group $\cG\cap\Sp_4(\I^\circ_0/\fI)$ is a normal subgroup of $\cG$. Then the group $f(\cG\cap\Sp_4(\I^\circ_0/\fI))$ is a normal subgroup of $f(\cG)$. Now $f(\cG)=\cG_0$ and $f(\cG\cap\Sp_4(\I^\circ_0/\fI))=\cG\cap\Sp_4(\I^\circ_0/\fI)$ since the restriction of $f$ to $\Sp_4(\I_0^\circ/\fI)$ is the identity. Hence $\cG\cap\Sp_4(\I^\circ_0/\fI)$ is a subnormal subgroup of $\Sp_4(\I^\circ_0/\fI)$ if and only if $\cG_0$ is a subnormal subgroup of $\Sp_4(\I_0^\circ/\fI)$. 
Neither $\cG\cap\Sp_4(\I^\circ_0/\fI)$ nor $\cG_0$ is contained in $\{\pm 1\}$, since the image of $\rho_{\fP_i}$ contains a non-trivial congruence subgroup of $\Sp_4(\I_0^\circ/\fP_i)$ by Theorem \ref{classbigim}. Hence Theorem \ref{tazhetcor} gives the desired equivalence.
\end{proof}

The following is a consequence of Proposition \ref{stresfull} and Lemma \ref{rhorho0}, together with our choice of $P_\uk$.

\begin{lemma}\label{resfull}
Let $\fQ$ be a prime of $\I_0^\circ$ lying over $P_\uk$. Then the image of $\rho_{0,\fQ}$ contains a non-trivial congruence subgroup of $\Sp_4(\I^\circ_0/\fQ)$. 
\end{lemma}

\begin{proof}
By Proposition \ref{stresfull} the image of $\rho_\fQ$ contains a non-trivial congruence subgroup of $\Sp_4(\I_0^\circ/\fQ)$. Since $H$ is a finite index subgroup of $G_\Q$, the same is true if we replace $\rho_\fQ$ by $\rho_\fQ\vert_H$. Now the conclusion follows from Lemma \ref{rhorho0} applied to $\fI=\fQ$.
\end{proof}

\subsection{Big image in a product}

Lifting the congruence subgroup of Proposition \ref{resfull} to $\I^\circ$ does not provide the information we need on the image of $\rho_0$. 
We need the following fullness result for $\rho_{P_\uk}$. 

\begin{prop}\label{prodresfull}
The image of the representation $\rho_{P_\uk}$ contains a non-trivial congruence subgroup of $\Sp_4(\I^\circ_0/P_\uk\I_0^\circ)$. 
\end{prop}

The strategy of the proof is similar to that of \cite[Proposition 5.1]{lang}. 
There is an injective morphism $\I_0^\circ/P_\uk\I_0^\circ\into\prod_{i=1}^d\I_0^\circ/\fQ_i$. 
Let $G$ be the image of $\im\rho_{0,P_\uk}$ in $\prod_{i=1}^d\I_0^\circ/\fQ_i$ via the previous injection. Proposition \ref{prodresfull} will follow from Lemma \ref{rhorho0}, once we prove that $G$ is an open subgroup of $\prod_{i=1}^d\I_0^\circ/\fQ_i$.
This is a consequence of a lemma of Ribet (Lemma \ref{ribetprod}) and the following.

\begin{lemma}\label{prodintwo}
Let $1\leq i<j\leq d$. Then the image of $G$ in $\I_0^\circ/\fQ_i\times\I_0^\circ/\fQ_j$ is open.
\end{lemma}

We will show that if the conclusion of the lemma is not true, then there is a self-twist $\sigma$ of $\rho$ such that $\sigma(\fQ_i)=\fQ_j$, which is a contradiction since $\I_0^\circ$ is fixed by all self-twists. 
The first part of our proof follows the strategy of \cite[Proposition 5.3]{lang}, that is inspired by \cite[Theorem 3.5]{ribetI}. 
One of its ingredients is Goursat's Lemma, that we recall here. Let $\cK_1$ and $\cK_2$ be two groups and let $\cG$ be a subgroup of $\cK_1\times\cK_2$ such that the two projections $\pi_1\colon\cG\to\cK_1$ and $\pi_2\colon\cG\to\cK_2$ are surjective. Let $\cN_1=\ker\pi_2$ and $\cN_2=\ker\pi_1$. We identify $\cN_1$ and $\cN_2$ with $\pi_1(\cN_1)$ with $\pi_2(\cN_2)$, hence with subgroups of $\cG_1$ and $\cG_2$, respectively. Clearly $\cN_1\times\cN_2\supset\cG$. The natural projections induce a map $\cG\to\cK_1/\cN_1\times\cK_2/\cN_2$. 

\begin{lemma}\label{goursat}(Goursat's Lemma, \cite[Sections 11 and 12]{goursat}\cite[Exercise 4.8, Chapter 1]{bourbaki})
The image of $\cG$ in $\cK_1/\cN_1\times\cK_2/\cN_2$ is the graph of an isomorphism $\cK_1/\cN_1\cong\cK_2/\cN_2$.
\end{lemma}

Another element of the proof of Lemma \ref{prodintwo} is the isomorphism O'Meara's theory of isomorphisms of open subgroups of $\GSp_4$ over local rings, that we recalled in Section \ref{stclass}. 

\begin{proof}(of Lemma \ref{prodintwo})
By Lemma \ref{resfull} there exist two non-zero ideals $\fl_1$ and $\fl_2$ of $\I_0^\circ/\fQ_i$ and $\I_0^\circ/\fQ_j$, respectively, such that $\Gamma_{\I_0^\circ/\fQ_i}(\fl_1)\subset\im\rho_{0,\fQ_i}$ and $\Gamma_{\I_0^\circ/\fQ_j}(\fl_2)\subset\im\rho_{0,\fQ_j}$. Recall that the domain of the representation $\rho_0$ is the open normal subgroup $H$ of $G_\Q$ defined in the beginning of this subsection. Consider the group
\[ H_1=\{h\in H\,\vert\,h\!\!\pmod{\fQ_i}\in\Gamma_{\I_0^\circ/\fQ_i}(\fl_1) \mbox{ and } h\!\!\pmod{\fQ_j}\in\Gamma_{\I_0^\circ/\fQ_j}(\fl_2)\}. \]
Since the subgroups $\Gamma_{\I_0^\circ/\fQ_i}(\fl_1)$ and $\Gamma_{\I_0^\circ/\fQ_j}(\fl_2)$ are normal and of finite index in $\Sp_4(\I_0^\circ/\fQ_i)$ and $\Sp_4(\I_0^\circ/\fQ_j)$, respectively, the subgroup $H_1$ is normal and of finite index in $H$. It is clearly closed, hence it is also open.

Let $1\leq i<j\leq d$. The couple $(i,j)$ will be fixed throughout the proof. 
Let $\cK_1=\rho_{0,\fQ_i}(H_1)$, $\cK_2=\rho_{0,\fQ_j}(H_1)$ and let $\cG_0$ be the image of $\rho_0(H_1)$ in $\cK_1\times\cK_2$. 
Note that $\cK_1$, $\cK_2$ and $\cG_0$ are profinite and closed since they are continuous images of a Galois group. 
By definition of $\fl_1$, $\fl_2$ and $H_1$ we have $\cK_1=\Gamma_{\I_0^\circ/\fQ_i}(\fl_1)$ and $\cK_2=\Gamma_{\I_0^\circ/\fQ_j}(\fl_2)$. In particular $\cK_1$ and $\cK_2$ are normal and finite index subgroups of $\Sp_4(\I_0^\circ/\fQ_i)$ and $\Sp_4(\I_0^\circ/\fQ_j)$, respectively. 
Define $\cN_1$ and $\cN_2$ as in the discussion preceding Lemma \ref{goursat}. 
They are normal closed subgroups of $\cK_1$ and $\cK_2$, respectively, since they are defined as kernels of continuous maps. In particular $\cN_1$ and $\cN_2$ are subnormal subgroups of $\Sp_4(\I_0^\circ/\fQ_i)$ and $\Sp_4(\I_0^\circ/\fQ_j)$, respectively.

Suppose that $\cN_1$ is open in $\cK_1$ and $\cN_2$ is open in $\cK_2$. Then the product $\cN_1\times\cN_2$ is open in $\cK_1\times\cK_2$. Since $\cG_0$ contains $\cN_1\times\cN_2$, it is also open in $\cK_1\times\cK_2$. The subgroup $\cK_1\times\cK_2=\Gamma_{\I_0^\circ/\fQ_i}(\fl_1)\times\Gamma_{\I_0^\circ/\fQ_j}(\fl_2)$ is an open subgroup of $\I_0^\circ/\fQ_i\times\I_0^\circ/\fQ_j$, so $\cG_0$ is open in $\I_0^\circ/\fQ_i\times\I_0^\circ/\fQ_j$. Then the conclusion of Lemma \ref{prodintwo} is true in this case. 

Now suppose that one among $\cN_1$ and $\cN_2$ is not open. Without loss of generality, let it be $\cN_1$. Since $\cN_1$ is closed in the profinite group $\cK_1$, it is not of finite index in $\cK_1$. 
By Lemma \ref{goursat} there is an isomorphism $\cK_1/\cN_1\cong\cK_2/\cN_2$, hence $\cN_2$ is not of finite index in $\cK_2$. 
In particular $\cN_1$ and $\cN_2$ are not of finite index in $\Sp_4(\I_0^\circ/\fQ_i)$ and $\Sp_4(\I_0^\circ/\fQ_j)$, respectively. 
Since $\cN_1$ is subnormal and not of finite index in $\Sp_4(\I_0^\circ/\fQ_i)$, it is contained in $\{\pm 1\}$ by Theorem \ref{tazhetcor}. The same reasoning gives that $\cN_2$ is contained in $\{\pm 1\}$. By definition of $H$ the image of $\rho_0$ lies in $\Gamma_{\I_0^\circ}(\fm_{\I_0^\circ})$; this implies that the centres of $\cK_1$ and $\cK_2$ are trivial since $p>2$. We conclude that $\cN_1\cong\cN_2\cong\{1\}$. 

By the result of the previous paragraph we have $\cK_1/\cN_1\cong\cK_1$ and $\cK_2/\cN_2\cong\cK_2$. Hence Lemma \ref{goursat} gives an isomorphism $\alpha\colon\cK_1\cong\cK_2$ such that, for every $(x,y)\in\cK_1\times\cK_2$, $(x,y)\in\cG_0$ if and only if $y=\alpha(x)$. 
By Corollary \ref{congisom}, applied to $F=Q(\I_0^\circ/\fQ_i)$, $F_1=Q(\I_0^\circ/\fQ_j)$, $\Delta=\cK_1$, $\Delta_1=\cK_2$, there exists an isomorphism $\alpha\colon Q(\I_0^\circ/\fQ_i)\to Q(\I_0^\circ/\fQ_j)$, a character $\chi\colon\cK_1\to Q(\I_0^\circ/\fQ_j)^\times$ and an element $\gamma\in\GSp_4(Q(\I_0^\circ/\fQ_i))$ such that for every $z\in\cK_1$ we have 
\begin{equation}\label{goursatK} \alpha(z)=\chi(z)\gamma\alpha(z)\gamma^{-1}, \end{equation}
where as usual we define $\alpha\colon\Sp_4(Q(\I_0^\circ/\fQ_i))\to\Sp_4(Q(\I_0^\circ/\fQ_j))$ by applying $\alpha$ to the matrix coefficients. Since the centre of $\cK_2$ is trivial, the character $\chi$ is also trivial. 
%
By recalling the definitions of $\cK_1$, $\cK_2$ and $\cG_0$ we can rewrite Equation \eqref{goursatK} as $\rho_{0,\fQ_j}(h)=\gamma_0^{-1}\alpha(\rho_{0,\fQ_i}(h))\gamma_0^{-1}$ for every $h\in H_1$. 
The last equation gives an isomorphism 
\begin{equation}\label{isomH10} \rho_{0,\fQ_i}\vert_{H_1}^\alpha\cong\rho_{0,\fQ_j}\vert_{H_1} \end{equation}
of representations of $H_1$ over $Q(\I_0^\circ/\fQ_j)$. 
Denote by $\pi_j$ the projection $\I_0^\circ\to\I_0^\circ/\fQ_j$. 
By definition of $\rho_0$ we have $\rho_0\vert_{H_1}=\rho\vert_{H_1}\otimes(\det\rho\vert_{H_1})^{-1/2}$. Define a character $\varphi\colon H_1\to Q(\I_0^\circ/\fQ_j)^\times$ by setting $\varphi(h)=\pi_j\left(\frac{\det\rho(h)}{\alpha(\det\rho(h))}\right)$ for every $h\in H_1$. Then Equation \ref{isomH10} implies that 
\begin{equation}\label{isomH1} \rho_{\fQ_i}\vert_{H_1}^\alpha\cong\varphi\otimes\rho_{\fQ_j}\vert_{H_1} \end{equation}

We will use the isomorphism \eqref{isomH1}, together with Proposition \ref{lifttwists}, to construct a self-twist for $\rho$. Let $\fP_i$ and $\fP_j$ be primes of $\I^\circ_\Tr$ that lie above $\fQ_i$ and $\fQ_j$, respectively. 

\begin{lemma}\label{extend}
The isomorphism $\alpha\colon Q(\I_0^\circ/\fQ_i)\to Q(\I_0^\circ/\fQ_j)$ and the character $\varphi\colon H_1\to Q(\I_0^\circ/\fQ_j)$ can be extended to an isomorphism $\widetilde{\alpha}\colon Q(\I_\Tr^\circ/\fP_i)\to Q(\I_\Tr^\circ/\fP_j)^\times$ and a character $\widetilde{\varphi}\colon G_\Q\to Q(\I_\Tr^\circ/\fP_j)^\times$, respectively, such that  
\begin{equation}\label{eqliftobs} \rho_{\fP_i}^{\widetilde{\alpha}}\cong\widetilde{\varphi}\otimes\rho_{\fP_j}. \end{equation}
\end{lemma}

We prove Lemma \ref{extend} by the strategy presented in \cite[Section 5]{lang}. Let $\tau\colon Q(\I_\Tr^\circ/\fP_i)\to\Qp$ be an arbitrary extension of $\alpha$ to $Q(\I_\Tr^\circ/\fP_i)$. Let $L_2=Q(\I_\Tr^\circ/\fP_j)\cdot\tau(Q(\I_\Tr^\circ/\fP_i))$. The following lemma can be proved via obstruction theory, exactly as \cite[Lemma 5.6]{lang}. 

\begin{lemma}\label{obstrext}
There exists an extension $\widetilde{\varphi}\colon G_\Q\to L_2^\times$ of $\varphi\colon H_1\to L_2^\times$ such that 
\begin{equation}\label{obstreq} \rho_1^{\widetilde{\alpha}}\cong\widetilde{\varphi}\otimes\rho_2. \end{equation}
\end{lemma}
%

In order to prove Lemma \ref{extend} it is sufficient to show that $\widetilde{\alpha}$ restricts to an isomorphism $\I_\Tr^\circ/\fP_i\to\I_\Tr^\circ/\fP_j$ and that $\widetilde{\varphi}$ takes values in $\I_\Tr^\circ/\fP_j$. We write $(\I_\Tr^\circ/\fP_j)[\widetilde{\varphi}]$ for the subring of $L_2$ generated over $\I_\Tr^\circ/\fP_j$ by the values of $\widetilde{\varphi}$. 

\begin{rem}\label{tildesigma}
Since $\widetilde{\varphi}\otimes\rho_2$ takes values in $\GL_4((\I_\Tr^\circ/\fP_j)[\widetilde{\varphi}])$, the representation $\rho_1^{\widetilde{\alpha}}$ also takes values in $\GL_4((\I_\Tr^\circ/\fP_j)[\widetilde{\varphi}])$. In particular $\widetilde{\alpha}(\Tr(\rho_1(h)))\in(\I_\Tr^\circ/\fP_j)[\widetilde{\varphi}]$ for every $h\in H_1$. Since the traces of the representation $\rho_1$ generate the ring $\I_\Tr/\fP_i$ over $\Z_p$, we conclude that $\widetilde{\alpha}$ restricts to an isomorphism $(\I_\Tr^\circ/\fP_i)[\widetilde{\varphi}]\to(\I_\Tr^\circ/\fP_j)[\widetilde{\varphi}]$. 
\end{rem}


\begin{lemma}\label{etaequal}(cf. \cite[Lemma 5.7]{lang})
We have $(\I_\Tr^\circ/\fP_i)[\widetilde{\varphi}]=\I_\Tr^\circ/\fP_i$ and $(\I_\Tr^\circ/\fP_j)[\widetilde{\varphi}]=(\I_\Tr^\circ/\fP_j)$.
\end{lemma}

\begin{proof}
As before let $\chi$ be the $p$-adic cyclotomic character. Recall that $\fP_i$ and $\fP_j$ lie over the prime $P_\uk$ of $\Lambda$, with $\uk=(k_1,k_2)$. By taking determinants in Equation \eqref{obstreq} and using Remark \ref{detformula} we obtain 
\begin{equation}\label{varphidet} \widetilde{\varphi}^4=\frac{\det(\rho_1^{\widetilde{\alpha}})}{\det(\rho_2)}=\frac{\widetilde{\alpha}(\chi^{2(k_1+k_2-3)})}{\chi^{2(k_1+k_2-3)}}. \end{equation}
Since the quantity on the right defines an element of $\I_\Tr^\circ/\fP_j$, the degree $[(\I_\Tr^\circ/\fP_j)[\widetilde{\varphi}]\colon\I_\Tr^\circ/\fP_j]$ is at most $4$. In particular the extension $(\I_\Tr^\circ/\fP_j)[\widetilde{\varphi}]$ is obtained from $\I_\Tr^\circ/\fP_j$ by adding a $2$-power root of unit, hence it is an unramified extension. The same is true for the extension $(\I_\Tr^\circ/\fP_i)[\widetilde{\varphi}]$ over $\I_\Tr^\circ/\fP_i$ thanks to the isomorphism $\widetilde{\alpha}$.

Note that the residue fields of $(\I_\Tr^\circ/\fP_i)[\widetilde{\varphi}]$ and $(\I_\Tr^\circ/\fP_j)[\widetilde{\varphi}]$ are identified by $\widetilde{\alpha}$ and those of $\I_\Tr^\circ/\fP_i$ and $\I_\Tr^\circ/\fP_j$ coincide by the non-criticality of $P_\uk$ (see the étaleness condition in Definition \ref{noncritgsp}). Let $\E$ and $\F$ be the residue fields of $(\I_\Tr^\circ/\fP_i)[\widetilde{\varphi}]$ and $\I_\Tr^\circ/\fP_i$ respectively. To conclude the proof it is sufficient to show that $\E=\F$. The isomorphism $\widetilde{\alpha}$ induces an automorphism $\ovl{\alpha}$ of the residue field $\E$ and the character $\widetilde{\varphi}$ induces a character $\ovl{\varphi}\colon G_\Q\to\E^\times$. Then $\E$ is the field $\F[\ovl{\varphi}]$ generated over $\F$ by the values of $\ovl{\varphi}$. Let $s$ be an integer such that $\ovl{\alpha}$ is the $s$-th power of the Frobenius automorphisms. By reducing Equation \ref{varphidet} modulo the maximal ideal of $(\I_\Tr^\circ/\fP_j)[\widetilde{\varphi}]$ we obtain
\[ \ovl{\varphi}^4=\frac{\ovl{\alpha}(\chi^{2(k_1+k_2-3)})}{\chi^{2(k_1+k_2-3)}}=\chi^{2(p^s-1)(k_1+k_2-3)}. \]
Since $p$ is odd, $2(p^s-1)$ is a multiple of $4$. In particular $\F[\ovl{\varphi}^4]\subset\F[\chi^4]$, that implies $\F[\ovl{\varphi}]\subset\F$. We conclude that $\E=\F$, as desired.
\end{proof}

Thanks to Remark \ref{tildesigma} and Lemma \ref{etaequal}, $\widetilde{\alpha}\colon L_1\to L_2$ restricts to an isomorphism $\widetilde{\alpha}\colon\I_\Tr^\circ/\fP_i\to\I_\Tr^\circ/\fP_j$ and $\widetilde{\varphi}$ takes values in $\I_\Tr^\circ/\fP_j$. Hence $\widetilde{\alpha}$ and $\widetilde{\varphi}$ satisfy the hypotheses of Lemma \ref{extend}.

We conclude the proof of Lemma \ref{prodintwo}. Set $\sigma=\widetilde{\alpha}\colon\I_\Tr^\circ/\fP_i\to\I_\Tr^\circ/\fP_j$ and $\eta=\widetilde{\varphi}\colon G_\Q\to\I_\Tr^\circ/\fP_j$. Thanks to Lemma \ref{extend}, $\sigma$ and $\eta$ satisfy the hypotheses of Proposition \ref{lifttwists}. Hence there exists a self-twist $\widetilde{\sigma}\colon\I_\Tr^\circ\to\I_\Tr^\circ$ for $\rho$ over $\Lambda_h$ that induces $\sigma$. In particular $\widetilde{\sigma}(\fP_i)=\fP_j$. Since $\fP_i$ and $\fP_j$ lie over different primes of $\I_0^\circ$, the self-twist $\widetilde{\sigma}$ does not fix $\I_0^\circ$, a contradiction. Recall that the assumption of this argument is that $\cN_1$ is not open in $\cK_1$ or $\cN_2$ is not open in $\cK_2$. When this is not the case we already observed that the conclusion of Lemma \ref{prodintwo} holds, so the proof of the lemma is complete. 
\end{proof}

We recall a lemma of Ribet. Let $k$ be an integer greater than $2$ and let $\cG_1,\cG_2,\ldots,\cG_k$ be profinite groups. Suppose that for every $i\in\{1,2,\ldots,k\}$ the following condition holds: 
\begin{equation}\tag{comm}
\textrm{ if }\cK\textrm{ is an open subgroup of }\cG_i\textrm{ the closure of the commutator subgroup of }\cK\textrm{ is open in }\cG_i.
\end{equation}
Let $\cG_0$ be a closed subgroup of $\cG_1\times\cG_2\times\cdots\times\cG_k$.

\begin{lemma}\label{ribetprod}\cite[Lemma 3.4]{ribetI}
Suppose that for every $i,j$ with $1\leq i<j\leq k$ the image of $\cG_0$ in $\cG_i\times\cG_j$ is an open subgroup of $\cG_i\times\cG_j$. Then $\cG_0$ is an open subgroup of $\cG_1\times\cG_2\times\cdots\times\cG_k$.
\end{lemma}

We are ready to complete the proof of Proposition \ref{resfull}.

\begin{proof}
For $1\leq i\leq k$ let $\cG_i$ be the image of $\rho_{0,\fP_i}\colon H\to\Sp_4(\I_0^\circ/\fQ_i)$. As before let $\cG_0$ be the image of $\im\rho_{0,P_\uk}$ via the inclusion $\Sp_4(\I_0^\circ/P_\uk\I_0^\circ)\into\prod_i\Sp_4(\I_0^\circ/\fQ_i\I_0^\circ)$. The groups $\cG_i$ are profinite and they satisfy condition (comm). The group $\cG_0$ is closed since it is the continuous image of $H$. By Lemma \ref{prodintwo} it is open in $\cG_i\times\cG_j$ for every $i,j$ with $1\leq i<j\leq d$. Hence Lemma \ref{ribetprod} implies that $\cG_0$ is open in $\prod_i\cG_i=\prod_i\cG_i$. 

By Proposition \ref{resfull} the group $\cG_i$ is open in $\Sp_4(\I_0^\circ/P_\uk\I_0^\circ)$ for every $i$, hence $\prod_i\cG_i$ is open in $\prod_i\Sp_4(\I_0^\circ/\fQ_i\I_0^\circ)$. We deduce that $\cG_0$ is open in $\prod_i\Sp_4(\I_0^\circ/P_\uk\I_0^\circ)$, so $\im\rho_{0,P_\uk}$ is open in $\Sp_4(\I_0^\circ/P_\uk\I_0^\circ)$. In particular $\im\rho_{0,P_\uk}$ contains a non-trivial congruence subgroup of $\Sp_4(\I_0^\circ/P_\uk\I_0^\circ)$. This remains true if we replace $\im\rho_{0,P_\uk}$ by $\im\rho_{P_\uk}$, thanks to Lemma \ref{rhorho0} applied to $\fI=P_\uk$. 
\end{proof}

\subsection{Unipotent subgroups and fullness}

Recall that for a root $\alpha$ of $\GSp_4$ we denote by $U^\alpha$ the corresponding one-parameter unipotent subgroup of $\GSp_4$ and by $\fu^\alpha$ the corresponding nilpotent subalgebra of $\fgsp_4(R)$. 
We call ``congruence subalgebra'' of $\fsp_4(R)$ a Lie algebra of the form $\fa\cdot\fsp_4(R)$ for some ideal $\fa$ of $R$. The lemma below follows from a simple computation with the Lie bracket.

\begin{lemma}\label{liestdarg}
Let $R$ be an integral domain and let $\fG$ be a Lie subalgebra of $\fsp_4(R)$. The following are equivalent:
\begin{enumerate}
\item the Lie algebra $\fG$ contains a congruence Lie subalgebra $\fa\cdot\fsp_4(R)$ of level a non-zero ideal $\fa$ of $R$;
\item for every root $\alpha$ of $\Sp_4$, the nilpotent Lie algebra $\fG\cap\fu^\alpha(R)$ contains a non-zero ideal $\fa_\alpha$ of $R$ via the identification $\fu^\alpha(R)\cong R$.
\end{enumerate}
Moreover:
\begin{enumerate}[label=(\roman*)]
\item if condition (1) is satisfied for an ideal $\fa$ then condition (2) is satisfied if we choose $\fa_\alpha=\fa$ for every $\alpha$;
\item if condition (2) is satisfied for a set of ideals $\{\fa_\alpha\}_\alpha$ then condition (1) is satisfied for the ideal $\fa=\prod_\alpha{\fa^\alpha}$, where the product is over all roots $\alpha$ of $\Sp_4$.
\end{enumerate}
\end{lemma}

A computation with commutators gives an analogue of Lemma \ref{liestdarg} dealing with unipotent and congruence subgroups rather than Lie algebras.

\begin{lemma}\label{stdarg}
Let $R$ be an integral domain and let $G$ be a subgroup of $\GSp_4(R)$. The following are equivalent:
\begin{enumerate}
\item the group $G$ contains a principal congruence subgroup $\Gamma_R(\fa)$ of level a non-zero ideal $\fa$ of $R$;
\item for every root $\alpha$ of $\Sp_4$, the unipotent group $G\cap U^\alpha(R)$ contains a non-zero ideal $\fa_\alpha$ of $R$ via the identification $U^\alpha(R)\cong R$.
\end{enumerate}
Moreover:
\begin{enumerate}[label=(\roman*)]
\item if condition (1) is satisfied for an ideal $\fa$ then condition (2) is satisfied if we choose $\fa_\alpha=\fa$ for every $\alpha$;
\item if condition (2) is satisfied for a set of ideals $\{\fa_\alpha\}_\alpha$ then condition (1) is satisfied for the ideal $\fa=\prod_\alpha{\fa_\alpha}$, where the product is taken over all roots of $\Sp_4$.
\end{enumerate}
\end{lemma}

\begin{rem}\label{fl2}
In both Lemma \ref{liestdarg} and Lemma \ref{stdarg}, if there is an ideal $\fa^\prime$ of $R$ such that the choice $\fa_\alpha=\fa^\prime$ for every $\alpha$ satisfies condition (2), then the choice $\fa=(\fa^\prime)^2$ satisfies condition (1). 
\end{rem}

By applying Proposition \ref{prodresfull} and Lemma \ref{stdarg} to $R=\I^\circ_0/P_\uk\I_0^\circ$ and $G=\im\rho_{0,P_\uk}$ we obtain the following corollary.

\begin{cor}\label{corresfull}
For every root $\alpha$ of $\GSp_4$ the group $\im\rho_{P_\uk}\cap U^\alpha(\I^\circ_0/P_\uk\I_0^\circ)$ contains the image of an ideal of $\I^\circ_0/P_\uk\I_0^\circ$. 
\end{cor}

\subsection{Lifting the congruence subgroup}

If $\alpha$ is a root of $\GSp_4$, $G$ is a group, $R$ is a ring and $\tau\colon G\to\GSp_4(R)$ is a representation, let $U^\alpha(\tau)=\tau(G)\cap U^\alpha(R)$. We always identify $U^\alpha(R)$ with $R$, hence $U^\alpha(\tau)$ with an additive subgroup of $R$.

Recall that $\rho\colon H_0\to\GSp_4(\I_0^\circ)$ is the representation associated with a finite slope family $\theta\colon\T_h\to\I^\circ$ and that $\rho_{P_\uk}$ is the reduction of $\rho$ modulo $P_\uk\I_0^\circ$. We use Corollary \ref{corresfull} together with Proposition \ref{approx} to obtain a result on the unipotent subgroups of the image of $\rho$. 


\begin{prop}\label{uniplatt}
For every root $\alpha$ of $\GSp_4$, the group $U^\alpha(\rho)$ contains a basis of a $\Lambda_h$-lattice in $\I_0^\circ$.
\end{prop}

\begin{proof}
Let $\pi_\uk\colon\I_0^\circ\to\I_0^\circ/P_\uk\I_0^\circ$ be the natural projection. We denote also by $\pi_\uk$ the induced map $\GSp_4(\I_0^\circ)\to\GSp_4(\I_0^\circ/P_\uk\I_0^\circ)$. For a root $\alpha$ of $\GSp_4$, let $\pi^\alpha_\uk\colon U^\alpha(\I_0^\circ)\to U^\alpha(\I_0^\circ/P_\uk\I_0^\circ)$ be the projection induced by $\pi_\uk$. 

Let $G=\im\rho\cap\Gamma_{\GSp_4(\I_0^\circ)}(p)$ and $G_{P_\uk}=\pi_\uk(G)$. 
We check that the choices $A=\I_0^\circ$, $\bG=\GSp_4$, $T=T_2$, $B=B_2$, $G=\im\rho\cap\Gamma_{\GSp_4(\I_0^\circ)}(p)$ and $Q=P_\uk$ satisfy the hypotheses of Proposition \ref{approx}:
\begin{itemize}[label={--}]
\item the group $G$ is compact since $\im\rho$ is the continuous image of a Galois group and $\Gamma_{\GSp_4(\I_0^\circ)}(p)$ is a pro-$p$ group;
\item by assumption $\im\rho$ contains a diagonal $\Z_p$-regular element $d$, and since $\Gamma_{\GSp_4(\I_0^\circ)}(p)$ is a normal subgroup of $\GSp_4(A)$ the element $d$ normalizes $\im\rho\cap\Gamma_{\GSp_4(\I_0^\circ)}(p)$.
\end{itemize}
Hence by Proposition \ref{approx} $\pi^\alpha_\uk$ induces a surjection $G\cap U^\alpha(\I_0^\circ)\to G_\uk\cap U^\alpha(\I_0^\circ/P_\uk\I_0^\circ)$. 
Let $G^\alpha=G\cap U^\alpha(\I_0^\circ)$ and $G_\uk^\alpha=G_\uk\cap U^\alpha(\I_0^\circ/P_\uk\I_0^\circ)$. As usual we identify them with $\Z_p$-submodules of $\I_0^\circ$ and $\I_0^\circ/P_\uk\I_0^\circ$, respectively.

By Corollary \ref{corresfull} there exists a non-zero ideal $\fa_\uk$ of $\I_0^\circ/P_\uk\I_0^\circ$ such that $\fa_\uk\subset\im\rho_{P_\uk}\cap U^\alpha(\I_0^\circ/P_\uk\I_0^\circ)$. 
Set $\fb_\uk=p\fa_\uk$. Then $\fb_\uk\subset G_\uk^\alpha$. 
%
%
By the result of the previous paragraph the map $G^\alpha\to G_\uk^\alpha$ induced by $\pi_\uk^\alpha$ is surjective, so we can choose a subset $A$ of $G^\alpha$ that surjects onto $\fb_\uk$. Let $M$ be the $\Lambda_h$-span of $A$ in $\I_0^\circ$. Let $\fb$ be the pre-image of $\fb_\uk$ via $\pi_\uk^{\alpha}\colon\I_0^\circ\to\I_0^\circ/P_\uk\I_0^\circ$. Clearly $A\subset\fb$, so $M$ is a $\Lambda_h$-submodule of $\fb$. Moreover $M/P_\uk M=\fb_\uk$ by the definition of $A$. Since $\Lambda$ is local Nakayama's lemma implies that the inclusion $M\into\fb$ is an equality. In particular the $\Lambda_h$-span of $G^\alpha$ contains an ideal of $\I_0^\circ$. By Lemma \ref{lattice} this implies that $G^\alpha$ contains a basis of a $\Lambda_h$-lattice in $\I_0^\circ$.
%
\end{proof}

\bigskip

\section{Relative Sen theory}\label{sen}

Let $\theta\colon\T_h\to\I^\circ$ be a finite slope family. We keep the notations of the previous sections. Recall that the image of the family in the connected component of unity of the weight space is a disc $B_2(\kappa,r_{h,\kappa})$ adapted to the slope $h$. We make from now on the following assumption:
\begin{equation*}\tag{exp}
B_2(\kappa,r_{h,\kappa})\subset B_2(0,p^{-1/{p-1}}).
\end{equation*}
The only purpose of this assumption is to assure the convergence of an exponential series (see Section \ref{senexp}).

In Section \ref{gspfam} we defined a family of radii $\{r_i\}_{i\geq 1}$ and we let $A_{r_i}$ be the ring of rigid analytic functions bounded by $1$ on $B(0,r_i)$. For every $i\geq 1$ there is a natural injection $\iota_{r_i}\colon\Lambda_h\to A_{r_i}$. 
Set $\I_{r_i,0}^\circ=\I_0^\circ\widehat{\otimes}_{\Lambda_h}A_{r_i}^\circ$. 
We endow $\I_{r_i,0}^\circ$ with its $p$-adic topology.

\begin{rem}\label{topolo}\mbox{ }
\begin{enumerate}
\item The ring $\I_0^\circ$ admits two inequivalent topologies: the profinite one and the $p$-adic one. The representation $\rho$ is continuous with respect to the profinite topology on $\I_0^\circ$ but it is not necessarily continuous with respect to the $p$-adic one.
\item Since $\I_0^\circ$ is a finite $\Lambda_h$-algebra, $\I_{r_i,0}^\circ$ is a finite $A_{r_i}^\circ$-algebra. There is an injective ring morphism $\iota^\prime_{r_i}\colon\I_0^\circ\into\I_{r_i,0}^\circ$ sending $f$ to $f\otimes 1$. This map is continuous with respect to the profinite topology on $\I_0^\circ$ and the $p$-adic topology on $\I_{r_i,0}^\circ$.
\end{enumerate}
\end{rem}

\noindent We will still write $\iota^\prime_{r_i}$ for the map $\GSp_4(\I_0^\circ)\into\GSp_4(\I_{r_i,0}^\circ)$ induced by $\iota^\prime_{r_i}$. 

We associated with $\theta$ a representation $\rho\vert_{H_0}\colon H_0\to\GSp_4(\I^\circ_0)$ that is continuous with respect to the profinite topologies on both its domain and target. 
By Remark \ref{topolo}(1) $\rho\vert_{H_0}$ needs not be continuous with respect to the $p$-adic topology on $\GSp_4(\I_0^\circ)$. This poses a problem when trying to apply Sen theory. For this reason we introduce for every $i$ the representation $\rho_{r_i}\colon H_0\to\GSp_4(\I^\circ_{r_i,0})$ defined by $\rho_{r_i}=\iota^\prime_{r_i}\ccirc\rho\vert_{H_0}$. We deduce from the continuity of $\iota^\prime_{r_i}$ that $\rho_{r_i}$ is continuous with respect to the profinite topology on $H_0$ and the $p$-adic one on $\I_{r_i,0}^\circ$. It is clear from the definition that the image of $\rho_{r_i}$ is independent of $i$ as a topological group. 

There is a good notion of Lie algebra for a pro-$p$ group that is topologically of finite type. For this reason we further restrict $H_0$ so that the image of $\rho_{r_i}$ is a pro-$p$ group. Let $H_{r_1}=\{g\in H_0 \,\vert\, \rho_{r_1}(g)\cong\1_4 \pmod{p}\}$ and set $H_{r_i}=H_{r_1}$ for every $i\geq 1$.
The subgroup $\{M\in\GSp_4(\I_{r_1,0}^\circ)\,\vert\,M\cong\1_4\pmod{p}\}$ is of finite index in $\GSp_4(\I_{r_1,0}^\circ)$. Note that this depends on the fact that we extended the coefficients to $\I_{r_1,0}$, since $\{M\in\GSp_4(\I_{0}^\circ)\,\vert\,M\cong\1_4\pmod{p}\}$ is \emph{not} of finite index in $\GSp_4(\I_0)$. We deduce that $H_{r_1}$ is a normal open subgroup of $G_\Q$. Let $K_{H_{r_i}}$ be the subfield of $\ovl{\Q}$ fixed by $H_{r_i}$. It is a finite Galois extension of $\Q$.

Recall that we fixed an embedding $G_{\Q_p}\into G_\Q$, identifying $G_{\Q_p}$ with a decomposition subgroup of $G_\Q$ at $p$. Let $H_{r_i,p}=H_{r_i}\cap G_{\Q_p}$.  
Let $K_{H_{r_i},p}$ be the subfield of $\Qp$ fixed by $H_{r_i,p}$. The field $K_{H_{r_i},p}$ will play a role when we apply Sen theory. 
For every $i$, let $G_{r_i}=\rho_{r_i}(H_{r_i})$ and $G_{r_i}^\loc=\rho_{r_i}(H_{r_i,p})$. 

\begin{rem}\label{Grind}
The topological Lie groups $G_{r}$ and $G_{r}^\loc$ are independent of $r$, in the following sense. For positive integers $i,j$ with $i\leq j$ let $\iota_{r_j}^{r_i}\colon\I_{r_j,0}\to\I_{r_i,0}$ be the natural morphism induced by the restriction of analytic functions $A_{r_j}\to A_{r_i}$. Since $H_{r_i}=H_{r_j}=H_{r_1}$ by definition, $\iota_{r_j}^{r_i}$ induces isomorphisms $\iota_{r_j}^{r_i}\colon G_{r_j}\xto{\sim} G_{r_i}$ and $\iota_{r_j}^{r_i}\colon G_{r_j}^\loc\xto{\sim} G_{r_i}^\loc$.
\end{rem}



\subsection{Big Lie algebras}\label{biglie}

As before let $r$ be a radius among the $r_i$, $i\in\N^{>0}$. We will associate with $\rho_r(H_{r_i})$ a Lie algebra that will give the context in which to apply Sen's results. Our methods require that we work with $\Q_p$-Lie algebras, so we define the rings $A_{r}=A_{r}^\circ[p^{-1}]$ and $\I_{r,0}=\I_{r,0}^\circ[p^{-1}]$.

Let $\fa$ be a height two ideal of $\I_{r,0}$. The quotient $\I_{r,0}/\fa$ is a finite-dimensional $\Q_p$-algebra. 
Let $\pi_\fa\colon\I_{r,0}\to\I_{r,0}/\fa$ be the natural projection. We still denote by $\pi_\fa$ the induced map $\GSp_4(\I_{r,0})\to\GSp_4(\I_{r,0}/\fa)$. Consider the subgroups $G_{r,\fa}=\pi_\fa(G_r)$ and $G_{r,\fa}^\loc=\pi_\fa(G_r^\loc)$ of $\GSp_4(\I_{r,0}/\fa)$. They are both pro-$p$ groups and they are topologically of finite type since $\GSp_4(\I_{r,0}/\fa)$ is. Note that it makes sense to consider the logarithm of an element of $G_{r,\fa}$ since this group is contained in $\{M\in\GSp_4(\I_{r,0}/fa)\,\vert\,M\cong\1_4\pmod{p}\}$.

We attach to $G_{r,\fa}$ and $G_{r,\fa}^\loc$ the $\Q_p$-vector subspaces $\fG_{r,\fa}$ and $\fG_{r,\fa}^\loc$ of $\fgsp_4(\I_{r,0}/\fa)$ defined by 
\[ \fG_{r,\fa}=\Q_p\cdot\log G_{r,\fa}\quad \textrm{and}\quad \fG_{r,\fa}^\loc=\Q_p\cdot\log G_{r,\fa}^\loc. \]
The $\Q_p$-Lie algebra structure of $\fgsp_4(\I_{r,0}/\fa)$ restricts to a $\Q_p$-Lie algebra structure on $\fG_{r,\fa}$ and $\fG_{r,\fa}^\loc$. These two Lie algebras are finite-dimensional over $\Q_p$ since $\fgsp_4(\I_{r,0}/\fa)$ is.

\begin{rem}\label{lieind}
The Lie algebras $\fG_{r,\fa}$ and $\fG_{r,\fa}^\loc$ are independent of $r$, in the following sense. For positive integers $i,j$ with $i\leq j$ let $\iota_{r_j}^{r_i}\colon\I_{r_j,0}\to\I_{r_i,0}$ be the natural morphism. By Remark \ref{Grind} $\iota_{r_j}^{r_i}$ induces isomorphisms $\iota_{r_j}^{r_i}\colon\fG_{r_j,\fa}\xto{\sim}\fG_{r_i,\fa}$ and $\iota_{r_j}^{r_i}\colon\fG_{r_j,\fa}^\loc\xto{\sim}\fG_{r_i,\fa}^\loc$.
\end{rem}

\begin{rem}
The definitions of $\fG_{r,\fa}$ and $\fG_{r,\fa}^\loc$ do not make sense if $\fa$ is not a height two ideal. In this case $\I_{r,0}/\fa$ is not a finite extension of $\Q_p$ and $G_r$ and $G_r^\loc$ need not be topologically of finite type. We can define subsets $\fG_{r,\fa}$ and $\fG_{r,\fa}^\loc$ of $\fgsp_4(\I_{r,0}/\fa)$ as above but they do not have in general a Lie algebra structure. In particular the choice $\fa=0$ does not give Lie algebras for $G_r$ and $G_r^\loc$. 
\end{rem}

Recall that there is a natural injection $\Lambda_2\into\Lambda_h$, hence an injection $\Lambda_2[p^{-1}]\into\Lambda_h[p^{-1}]$. For every $\uk=(k_1,k_2)$ the ideal $P_\uk\Lambda_h[p^{-1}]$ is either prime in $\Lambda_h[p^{-1}]$ or equal to $\Lambda_h[p^{-1}]$. We define the set of ``bad'' ideals $S_\Lambda^\bad$ of $\Lambda_2[p^{-1}]$ as
\[ S_\Lambda^\bad=\{(1+T_1-u),(1+T_2-u^2),(1+T_2-u(1+T_1)),((1+T_1)(1+T_2)-u^3)\}. \]
Then we define the set of bad prime ideals of $\Lambda_h[p^{-1}]$ as 
\[ S^\bad=\{P\mbox{ prime of } \Lambda_h[p^{-1}]\,\vert\, P\cap\Lambda_2[p^{-1}]\in S_\Lambda^\bad \}. \] 
We will take care to define rings where the images of the ideals in $S^\bad$ consist of invertible elements. The reason for this will be clear in Section \ref{senexp}. Let $S_2$ be the set of ideals $\fa$ of $\I_{r,0}$ of height two such that $\fa$ is prime to $P$ for every $P\in S^\bad$. Let $S_2^\prime$ be the subset of prime ideals in $S_2$. We define the ring 
\[ \B_r=\varprojlim_{\fa\in S_2}\I_{r,0}/\fa, \]
where the limit of finite-dimensional $\Q_p$-Banach spaces is taken with respect to the natural transition maps $\I_{r,0}/\fa_1\to\I_{r,0}/\fa_2$ defined for every inclusion of ideals $\fa_1\subset\fa_2$. We equip $\I_{r,0}/\fa$ with the $p$-adic topology for every $\fa$ and $\B_r$ with the projective limit topology. There is a natural injection $\iota_{\B_r}\colon\I_{r,0}\into\B_r$ with dense image. 
There is an isomorphism of \emph{rings} 
\begin{equation}\label{BrprodP} \B_r\cong\prod_{P\in S_2^\prime}\widehat{(\I_{r,0})}_P, \end{equation}
where $\widehat{(\I_{r,0})}_P=\varprojlim_i\I_{r,0}/P^i$ with respect to the natural transition maps, but \eqref{BrprodP} is \emph{not} an isomorphism of topological rings if we equip $\widehat{(\I_{r,0})}_P$ with the $P$-adic topology for every $P$. In this case the resulting product topology is not the topology on $\B_r$, which is the $p$-adic one. 

Now consider the sets
\begin{gather*}
S^{\bad}_A=\{P\cap A_{r}\,\vert\, P\in S^\bad\}, \quad S_{2,A}=\{\fa\cap A_{r}\,\vert\, \fa\in S_2\}, \quad S_{2,A}^\prime=\{\fa\cap A_{r,}\,\vert\, \fa\in S_2^\prime\}.
\end{gather*}
For later use we define a ring
\[ B_r=\varprojlim_{\fa\in S_{2,A}}A_{r}/\fa, \]
where the limit of finite-dimensional $\Q_p$-Banach spaces is taken with respect to the natural transition maps $A_{r}/\fa_1\to A_{r}/\fa_2$ defined for every inclusion of ideals $\fa_1\subset\fa_2$. We equip $A_{r}/\fa$ with the $p$-adic topology for every $\fa$ and $B_r$ with the projective limit topology. There is a natural injection $\iota_{B_r}\colon A_{r}\into B_r$ with dense image. 
There is an isomorphism of \emph{rings}
\begin{equation}\label{brprodP} B_r\cong\prod_{P\in S_{2,A}^\prime}\widehat{(A_{r})}_P \end{equation}
where $\widehat{(A_{r})}_P=\varprojlim_i A_{r}/P^i$ with respect to the natural transition maps, but \eqref{brprodP} is \emph{not} an isomorphism of topological rings if we equip $\widehat{(A_{r})}_P$ with the $P$-adic topology for every $P$. In this case the resulting product topology is not the topology on $B_r$, which is the $p$-adic one.

\begin{rem}\label{badBr}
For every $P\in S^\bad$ we have $P\cdot\B_r=\B_r$, since the limit defining $\B_r$ is over ideals prime to $P$. In the same way we have $P\cdot B_r=B_r$ for every $P\in S_A^{\bad}$.
\end{rem}

Recall that $\I_{r,0}$ is a finite $A_{r}$-algebra. Then $\I_{r,0}/\fa$ is a finite $A_{r}/(\fa\cap A_{r})$-algebra for every $\fa\in S_2$, so the ring $\B_r$ has a natural structure of topological $B_r$-algebra. For every $\fa\in S_2$ the degree of the extension $\I_{r,0}/\fa$ over $A_{r}/(\fa\cap A_{r})$ is bounded by that of $\I_{r,0}$ over $A_{r}$. We deduce that $\B_r$ is a finite $B_r$-algebra. 

We proceed to define the Lie algebras of $G_r$ and $G_r^\loc$ as subalgebras of $\fgsp_4(\B_r)$. Let
\begin{gather*} \fG_r=\varprojlim_{\fa\in S_2}\fG_{r,\fa}\quad\textrm{ and }\quad \fG^\loc_r=\varprojlim_{\fa\in S_2}\fG^\loc_{r,\fa}, \end{gather*}
where $\fG_{r,\fa}$ and $\fG_{r,\fa}^\loc$ are the Lie algebras we attached to $G_{r,\fa}$ and $G_{r,\fa}^\loc$. The $\Q_p$-Lie algebra structures on $\fG_{r,\fa}$ and $\fG_{r,\fa}^\loc$ induce $\Q_p$-Lie algebra structures on $\fG_r$ and $\fG_r^\loc$. We endow $\fG_r$ and $\fG_r^\loc$ with the $p$-adic topology induced by that on $\fgsp_4(\B_r)$.

When we introduce the Sen operators we will have to extend the scalars of the various rings and Lie algebras to $\C_p$. We denote this operation by adding a lower index $\C_p$ to the objects previously defined. 
We still endow all the rings with their $p$-adic topology. Clearly $\I_{r,0,\C_p}$ has a structure of finite $A_{r,\C_p}$-algebra and $\B_{r,\C_p}$ has a structure of finite $B_{r,\C_p}$-algebra. The injections $\iota_{\B_r}$ and $\iota_{B_r}$ induce injections with dense image $\iota_{\B_r,\C_p}\colon\I_{r,0,\C_p}\into\B_{r,\C_p}$ and $\iota_{B_r,\C_p}\colon A_{r,\C_p}\into\B_{r,\C_p}$. The Lie algebras $\fG_{r,\C_p}$ and $\fG_{r,\C_p}^\loc$ are $\C_p$-Lie subalgebras of $\fgsp_4(\B_{r,\C_p})$.

\begin{rem}\label{noBrstr}
The $\Q_p$-Lie algebras $\fG_r$ and $\fG_r^\loc$ do not have a priori any $\B_r$ or $B_r$-module structure. As a crucial step in our arguments we will use Sen theory to induce a $B_{r,\C_p}$-vector space (hence a $B_{r,\C_p}$-Lie algebra) structure on suitable subalgebras of $\fG_{r,\C_p}$.
\end{rem}

\subsection{The Sen operator associated with a $p$-adic Galois representation}\label{subsendef}

Let $L$ be a $p$-adic field and let $\ccR$ be a Banach $L$-algebra. Let $K$ be another $p$-adic field, $m$ be a positive integer and $\tau\colon\Gal(\ovl{K}/K)\to\GL_m(\ccR)$ be a continuous representation. We recall the construction of the Sen operator associated with $\tau$, following \cite{seninf}. 

We fix embeddings of $K$ and $L$ in $\Qp$. The constructions that follow will depend on these choices. We suppose that the Galois closure $L^\Gal$ of $L$ over $\Q_p$ is contained in $K$. If this is not the case we simply restrict $\tau$ to the open subgroup $\Gal(\ovl{K}/KL^\Gal)\subset\Gal(\ovl{K}/K)$. 
We denote by $\chi\colon\Gal(\ovl{L}/L)\to\Z_p^\times$ the $p$-adic cyclotomic character. 
Let $L_{\infty}$ be a totally ramified $\Z_p$-extension of $L$. Let $\gamma$ be a topological generator of $\Gamma=\Gal(L_{\infty}/L)$. For a positive integer $n$, let $\Gamma_n\subset\Gamma$ be the subgroup generated by $\gamma^{p^n}$ and $L_n=L_\infty^{\langle\gamma^{p^n}\rangle}$ be the subfield of $L_\infty$ fixed by $\Gamma_n$. We have $L_{\infty}=\cup_n L_n$. Let $L_n^\prime=L_nK$ and $G_n^\prime=\Gal(\ovl{L}/L_n^\prime)$.

Write $\ccR^m$ for the $\ccR$-module over which $\Gal(\ovl{K}/K)$ acts via $\tau$. We define an action of $\Gal(\ovl{K}/K)$ on $\ccR^m\widehat{\otimes}_L \C_p$ by letting $\sigma\in\Gal(\ovl{K}/K)$ send $x\otimes y$ to $\tau(\sigma)(x)\otimes\sigma(y)$. Then by \cite{seninf} there exists a matrix $M\in \GL_m\left(\ccR\widehat{\otimes}_L \C_p\right)$, an integer $n\ge 0$ and a representation $\delta\colon\Gamma_n\to \GL_m(\ccR\otimes_L L_{n}^\prime)$ such that for all $\sigma\in G_n^\prime$ we have 
\begin{equation}\label{eqdefsen} M^{-1}\tau(\sigma)\sigma(M)=\delta(\sigma). \end{equation}

\begin{defin}\label{senop}
The Sen operator associated with $\tau$ is the element
\[ \phi=\lim_{\sigma\to 1}\frac{\log(\delta(\sigma))}{\log(\chi(\sigma))} \]
of $\Mat_m(\ccR\widehat{\otimes}_L \C_p)$.
\end{defin}

The limit in the definitions always exists and is independent of the choice of $\delta$ and $M$.


Now suppose that $\ccR=L$ and that $\tau$ is a Hodge-Tate representation with Hodge-Tate weights $h_1,h_2,\ldots,h_m$. Let $\phi$ be the Sen operator associated with $\tau$; it is an element of $\Mat_m(\C_p)$. The following theorem is a consequence of the results of \cite{sencont}. 

\begin{thm}\label{charpolsen}
The characteristic polynomial of $\phi$ is $\prod_{i=1}^m(X-h_i)$.
\end{thm}

We restrict now to the case $L=\ccR=\Q_p$, so that $\tau$ is a continuous representation $\Gal(\ovl{K}/K)\to\GL_m(\Q_p)$. 
Define a $\Q_p$-Lie algebra $\fg\subset\Mat_m(\Q_p)$ by $\fg=\Q_p\cdot\log(\tau(\Gal(\ovl{K}/K)))$. We say that $\fg$ is the Lie algebra of $\tau(\Gal(\ovl{K}/K))$. 
Let $\phi$ be the Sen operator associated with $\tau$.

\begin{thm}\label{liealgsen}\cite[Theorem 1]{senlie}
The Sen operator $\phi$ is an element of $\fg\widehat\otimes_{\Q_p}\C_p$.
\end{thm}

\begin{rem}\label{doubt}
The proof of Theorem \ref{liealgsen} relies on the fact that $\tau(\Gal(\ovl{K}/K))$ is a finite dimensional Lie group. It is doubtful that this proof can be generalized to the relative case. 
\end{rem}

\subsection{The relative Sen operator associated with $\rho_{r}$} 

Fix a radius $r$ in the set $\{r_i\}_{i\in\N^{>0}}$. 
Consider as usual the representation $\rho_{r}\colon H_0\to\GSp_4(\I_{r,0})$. We defined earlier a $p$-adic field $K_{H_r,p}$. Write $G_{K_{H_r,p}}$ for its absolute Galois group.
We look at the restriction $\rho_{r}\vert_{G_{K_{H_r,p}}}\colon G_{K_{H_r,p}}\to\GSp_4(\I_{r,0})$ as a representation with values in $\GL_4(\I_{r,0})$. Recall that $\fG_r^\loc$ is the Lie algebra associated with the image of $\rho_{r}\vert_{G_{K_{H_r,p}}}$. The goal of this section is to prove an analogue of Theorem \ref{liealgsen} for this representation, i.e. to attach to $\rho_{r}\vert_{G_{K_{H_r,p}}}$ a ``$\B_r$-Sen operator'' belonging to $\fG_{r,\C_p}^\loc$. 
We start by constructing various Sen operators via Definition \ref{senop}.

\begin{enumerate}
\item The $\Q_p$-algebra $\I_{r,0}$ is complete for the $p$-adic topology. We associate with $\rho_{r}\vert_{G_{K_{H_r,p}}}$ a Sen operator $\phi_{r}\in\Mat_4(\I_{r,0,\C_p})$. 
\item Let $\fa\in S_2$. Then $\I_{r,0}/\fa$ is a finite-dimensional $\Q_p$-algebra. As usual write $\pi_\fa\colon\I_{r,0}\to\I_{r,0}/\fa$ for the natural projection. Denote by $\rho_{r,\fa}$ the representation $\pi_\fa\ccirc\rho_{r}\vert_{G_{K_{H_r,p}}}\colon G_{K_{H_r,p}}\to\GL_4(\I_{r,0}/\fa)$. We associate with $\rho_{r,\fa}$ a Sen operator $\phi_{r,\fa}\in\Mat_4((\I_{r,0}/\fa)\widehat{\otimes}_{\Q_p}\C_p)$. 
\item Let $\fa\in S_2$. Let $d$ be the $\Q_p$-dimension of $\I_{r,0}/\fa$. Let $k$ be a positive integer. An $\I_{r,0}/\fa$-linear endomorphism of $(\I_{r,0}/\fa)^k$ defines a $\Q_p$-linear endomorphism of the underlying $\Q_p$-vector space $\Q_p^{kd}$. This gives natural maps $\alpha_{\Q_p}\colon\Mat_k(\I_{r,0}/\fa)\to\Mat_{kd}(\Q_p)$ and $\alpha_{\Q_p}^\times\colon\GL_k(\I_{r,0}/\fa)\to\GL_{kd}(\Q_p)$ (we leave the dependence of these morphisms on $k$ implicit).
Choose $k=4$ and consider the representation $\rho^{\Q_p}_{r,\fa}=\alpha^\times_{\Q_p}\ccirc\rho_{r,\fa}\colon G_\Q\to\GL_{4d}(\Q_p)$. We associate with $\rho_{r,\fa}^{\Q_p}$ a Sen operator $\phi_{r,\fa}^{\Q_p}\in\Mat_{4d}(\C_p)$.
\end{enumerate}

Note that Theorem \ref{liealgsen} can be applied only to representations with coefficients in $\Q_p$, hence to construction (3) above. We will prove that the operators constructed in (1), (2) and (3) are related, so that it is possible to transfer information from one to the others. We write $\pi_{\fa,\C_p}=\pi_\fa\otimes 1\colon\I_{r,0,\C_p}\to\I_{r,0,\C_p}/\fa\I_{r,0,\C_p}$. We still write $\pi_{\fa,\C_p}$ for the maps $\Mat_4(\I_{r,0,\C_p})\to\Mat_4(\I_{r,0,\C_p}/\fa\I_{r,0,\C_p})$ and $\GL_4(\I_{r,0,\C_p})\to\GL_4(\I_{r,0,\C_p}/\fa\I_{r,0,\C_p})$ obtained by applying $\pi_{\fa,\C_p}$ to the matrix coefficients. As before we let $d$ be the $\Q_p$-dimension of $\I_{r,0}/\fa$. For every positive integer $k$, we set $\alpha_{\C_p}=\alpha_{\Q_p}\otimes 1\colon\Mat_k(\I_{r,0,\C_p}/\fa\I_{r,0,\C_p})\to\Mat_{kd}(\C_p)$ and $\alpha_{\C_p}^\times=\alpha_{\Q_p}^\times\otimes 1\colon\GL_k(\I_{r,0,\C_p}/\fa\I_{r,0,\C_p})\to\GL_{kd}(\C_p)$. 

\begin{prop}\label{senequal}
For every $\fa\in S_2$ the following relations hold:
\begin{enumerate}[label=(\roman*)]
\item $\phi_{r,\fa}=\pi_{\fa,\C_p}(\phi_{r})$;
\item $\phi_{r,\fa}^{\Q_p}=\alpha_{\C_p}(\phi_{r,\fa})$.
\end{enumerate}
\end{prop}

\begin{proof}
We deduce this result from the construction of the Sen operator presented in Section \ref{subsendef}. We first specialize it to the representation $\rho_{r}\vert_{G_{K_{H_r,p}}}\colon G_{K_{H_r,p}}\to\GL_4(\I_{r,0})$; in particular we choose $m=4$, $K=K_{H_r,p}$ and $L=\Q_p$. By the discussion preceding Definition \ref{senop}, there exists a matrix $M_0\in\GL_4(\I_{r,0,\C_p})$, an integer $n_0\ge 0$ and a representation $\delta_0\colon\Gamma_{n_0}\to\GL_4(\I_{r,0}\widehat{\otimes}_{\Q_p} (\Q_p)_{n_0}^\prime)$ such that for all $\sigma\in\Gal(\Qp/(\Q_p)_{n_0}^\prime)$ we have
\begin{equation}\label{senr0} M_0^{-1}\rho_{r}(\sigma)\sigma(M_0)=\delta_0(\sigma). \end{equation}
Let $M_{0,\fa}=\pi_{\fa,\C_p}(M_0)\in\Mat_4(\I_{r,0,\C_p}/\fa\I_{r,0,\C_p})$ and $\delta_{0,\fa}=\pi_{\fa,\C_p}\ccirc\delta_0\colon\Gamma_{n_0}\to\GL_4((\I_{r,0}/\fa)\widehat{\otimes}_{\Q_p} (\Q_p)_{n_0}^\prime)$. By applying $\pi_{\fa,\C_p}$ to both sides of Equation \eqref{senr0} we obtain
\begin{equation}\label{senr0fa} M_{0,\fa}^{-1}\rho_{r,\fa}(\sigma)\sigma(M_{0,\fa})=\delta_{0,\fa}(\sigma) \end{equation}
for every $\sigma\in\Gal(\Qp/(\Q_p)_{n_0}^\prime)$.
Hence the choices $M=M_{0,\fa}$, $n=n_0$ and $\delta=\delta_{0,\fa}$ satisfy Equation \eqref{eqdefsen} specialized to the representation $\rho_{r,\fa}$. Then, by definition, the Sen operator associated with $\rho_{r,\fa}$ is 
\[ \phi_{r,\fa}=\lim_{\sigma\to 1}\frac{\log(\delta_{0,\fa}(\sigma))}{\log(\chi(\sigma))}, \]
that coincides with $\pi_{\fa,\C_p}(\phi_{r})$.

For (ii), keep notations as in the previous paragraph. Let $M_{0,\fa}^{\Q_p}=\alpha_{\C_p}^\times(M_{0,\fa})$ and $\delta_{0,\fa}^{\Q_p}=\alpha_{\C_p}^\times\ccirc\delta_{0,\fa}$. By applying $\alpha_{\C_p}^\times$ to both sides of Equation \eqref{senr0fa} we obtain
\begin{equation} (M_{0,\fa}^{\Q_p})^{-1}\rho^{\Q_p}_{r,0,\fa}(\sigma)\sigma(M^{\Q_p}_{0,\fa})=\delta^{\Q_p}_{0,\fa}(\sigma) \end{equation}
for every $\sigma\in\Gal(\Qp/(\Q_p)_{n_0}^\prime)$. Then the choices $M=M^{\Q_p}_{0,\fa}$, $n=n_0$ and $\delta=\delta^{\Q_p}_{0,\fa}$ satisfy Equation \eqref{eqdefsen} specialized to the representation $\rho^{\Q_p}_{r,0,\fa}$, so by definition the Sen operator associated with $\rho^{\Q_p}_{r,0,\fa}$ is
\[ \phi^{\Q_p}_{r,0,\fa}=\lim_{\sigma\to 1}\frac{\log(\delta^{\Q_p}_{0,\fa}(\sigma))}{\log(\chi(\sigma))}. \]
A simple check shows that the right hand side is equal to $\alpha_{\C_p}(\phi_{r,\fa})$.
\end{proof}

Recall that there is a natural inclusion $\iota^\prime_{\B_r,\C_p}\colon\I_{r,0,\C_p}\into\B_{r,\C_p}$. It induces an injection $\Mat_4(\I_{r,0,\C_p})\into\Mat_4(\B_{r,\C_p})$ that we still denote by $\iota_{\B_r,\C_p}$. We define the $\B_r$-Sen operator attached to $\rho_{r}\vert_{G_{K_{H_r,p}}}$ as
\[ \phi_{\B_r}=\iota^\prime_{\B_r,\C_p}(\phi_{r}). \]
By definition $\phi_{\B_r}$ is an element of $\Mat_4(\B_{r,\C_p})$. Since $\B_{r,\C_p}=\varprojlim_{\fa\in S_2}\I_{r,0}/\fa$, it is clear that $\phi_{\B_r}=\varprojlim_{\fa\in S_2}\pi_{\fa,\C_p}(\phi_{r})$. Then Proposition \ref{senequal}(i) implies that
\begin{equation}\label{senlimitfa} \phi_{\B_r}=\varprojlim_{\fa\in S_2}\phi_{r,\fa}. \end{equation}
We use Proposition \ref{senequal}(ii) to show the following. 

\begin{prop}\label{seninalg}
The operator $\phi_{\B_r}$ belongs to the Lie algebra $\fG^\loc_{r,\C_p}$. In particular it belongs to $\fG_{r,\C_p}$.
\end{prop}

\begin{proof}
For every $\fa\in S_2$, let $d_\fa$ be the degree of the extension $\I_{r,0}/\fa$ over $\Q_p$. Let $\fG_{r,\fa}^{\loc,\Q_p}$ be the Lie subalgebra of $\Mat_{4d_\fa}$ associated with the image of $\rho_{r,\fa}^{\Q_p}$, defined by $\fG_{r,\fa}^{\loc,\Q_p}=\Q_p\cdot\log(\im\rho_{r,\fa}^{\Q_p})$. 
Let $\fG_{r,\fa,\C_p}^{\loc,\Q_p}=\fG_{r,\fa}^{\loc,\Q_p}\widehat{\otimes}_{\Q_p}\C_p$. Since $\im\rho_{r,\fa}^{\Q_p}=\alpha_{\Q_p}^\times(\im\rho_{r,\fa})$ we can write
\begin{equation}\label{alphaalg} \fG_{r,\fa,\C_p}^{\loc,\Q_p}=\alpha_{\C_p}(\fG^\loc_{r,\fa,\C_p}). \end{equation}

The representation $\rho_{r,\fa}^{\Q_p}$ satisfies the assumptions of Theorem \ref{liealgsen}, so the Sen operator $\phi^{\Q_p}_{r,0,\fa}$ belongs to $\fG_{r,\fa,\C_p}^{\loc,\Q_p}$. By Proposition \ref{senequal}(ii) $\phi_{r,\fa}^{\Q_p}=\alpha_{\C_p}(\phi_{r,\fa})$. Then Equation \eqref{alphaalg} and the injectivity of $\alpha_{\C_p}$ give
\begin{equation}\label{seninfa} \phi_{r,\fa}\in\fG^\loc_{r,\fa,\C_p}. \end{equation}
Since $\fG^\loc_{r,\C_p}=\varprojlim_{\fa\in S_2}\fG^\loc_{r,\fa,\C_p}$, Equations \eqref{senlimitfa} and \eqref{seninfa} imply that $\phi_{\B_r}\in\fG^\loc_{r,\C_p}$.
\end{proof}

\subsection{The exponential of the Sen operator}\label{senexp}

We use the work of the previous section to construct an element of $\GL_4(\B_r)$ that has some specific eigenvalues and normalizes the Lie algebra $\fG_{r,\C_p}^\loc$. Such an element will be used in Section \ref{exlevel} to induce a $B_{r,\C_p}$-module structure on some subalgebra of $\fG_{r,\C_p}$, thus replacing the matrix ``$\rho(\sigma)$'' of \cite{hidatil} that is not available in the non-ordinary setting. 

Let $\phi_{r}\in\Mat_4(\I_{r,0,\C_p})$ be the Sen operator defined in the previous section. We rescale it to define an element $\phi_{r}^\prime=\log(u)\phi_{r}$, where $u=1+p$. 
Let $(T_1,T_2)$ be the images in $A_r$ of the coordinate functions on the weight space. The logarithms and the exponentials in the following proposition are defined via the usual power series, that converge because of the assumption (exp) we made in the beginning of Section \ref{sen}. 

\begin{prop}\label{eigensen}\mbox{}
\begin{enumerate}
\item The eigenvalues of $\phi^\prime_{r}$ are $0$, $\log(u^{-2}(1+T_2))$, $\log(u^{-1}(1+T_1))$ and $\log(u^{-3}(1+T_1)(1+T_2))$. 
\item The operator $\phi^\prime_{r}$ belongs to $\Mat_4(\I_{r,0,\C_p})^{\geq\frac{1}{p-1}}$. In particular the exponential series defines an element $\exp(\phi^\prime_{r})\in\GL_4(\I_{r,0,\C_p})$. 
\item The eigenvalues of $\exp(\phi^\prime_{r,0})$ are $1$, $u^{-2}(1+T_2)$, $u^{-1}(1+T_1)$ and $u^{-3}(1+T_1)(1+T_2)$.
\end{enumerate}
\end{prop}

\begin{proof}(of Proposition \ref{eigensen})
We prove part (1). 
The $p$-adic Galois representation $\rho_f$ associated with a classical eigenform $f$ of weight $(k_1,k_2)$ is Hodge-Tate with Hodge-Tate weights $(0,k_2-2,k_1-1,k_1+k_2-3)$. By Theorem \ref{charpolsen} these weights are the eigenvalues of the Sen operator $\phi_f$ associated with $\rho_f$. By Proposition \ref{senequal}(i) the eigenvalues of $\phi_r$ interpolate those of the operators $\phi_f$ when $f$ varies in the set of classical points of $A_r$. 
Since such points form a Zariski-dense subset of $\Spec A_r$, the interpolation is unique. A simple check shows that it is given by the function $F\colon A_r\to\C_p^4$ defined by $F(T_1,T_2)=(0, \log(u^{-2}(1+T_2))/\log(u), \log(u^{-1}(1+T_1))/\log(u), \log(u^{-3}(1+T_1)(1+T_2))/\log(u))$. By normalizing we obtain the eigenvalues given in the proposition. 

Statements (2) and (3) follow immediately from (1).
\end{proof}

Let $\Phi_{\B_r}=\iota_{\B_{r,\C_p}}(\exp(\phi^\prime_{r,0}))$. By definition $\Phi_{\B_r}$ is an element of $\GL_4(\B_{r,\C_p})$. We show that it has the two properties we need. 
We define a matrix $C_{T_1,T_2}\in\GSp_4(\B_{r,\C_p})$ by
\[ C_{T_1,T_2}=\diag\!\left(u^{-3}(1+T_1)(1+T_2),u^{-1}(1+T_1),u^{-2}(1+T_2),1\right). \]

\begin{prop}\label{expsen}\mbox{}
\begin{enumerate} 
\item There exists $\gamma\in\GSp_4(\B_{r,\C_p})$ satisfying $\Phi_{\B_r}=\gamma C_{T_1,T_2}\gamma^{-1}$.
\item The element $\Phi_{\B_r}$ normalizes the Lie algebra $\fG_{r,\C_p}$.
\end{enumerate}
\end{prop}

\begin{proof}
The matrices $\Phi_{\B_r}$ and $C_{T_1,T_2}$ have the same eigenvalues by Proposition \ref{eigensen}(3). Hence there exists $\gamma\in\GL_4(\B_{r,\C_p})$ satisfying the equality of part (1) if and only if the difference between any two of the eigenvalues of $\Phi_{\B_r}$ is invertible in $\B_r$. We check by a direct calculation that each one of these differences belongs to an ideal of the form $P\cdot\B_r$ with $P\in S^\bad$, hence it is invertible in $\B_r$ by Remark \ref{badBr}. Since both $\Phi_{\B_r}$ and $C_{T_1,T_2}$ are elements of $\GSp_4(\B_{r,\C_p})$, we can take $\gamma\in\GSp_4(\B_{r,\C_p})$.

Part (2) follows from Proposition \ref{seninalg}.
\end{proof}

\bigskip

\section{Existence of a Galois level in the residual symmetric cube and full cases}\label{exlevel}

We have all the ingredients we need to state and prove our first main theorem. 
Let $h\in\Q^{+,\times}$. Let $\T_h$ be a local component of the $h$-adapted Hecke algebra of genus $2$ and level $\Gamma_1(M)\cap\Gamma_0(p)$. Suppose that condition (exp) of Section \ref{sen} is satisfied and that the residual Galois representation $\ovl\rho_{\T_h}$ associated with $\T_h$ is either full or of symmetric cube type in the sense of Definition \ref{sctype}. 
Let $\theta\colon\T_h\to\I^\circ$ be a family, i.e. the morphism of finite $\Lambda_h$ algebras describing an irreducible component of $\T_h$. 
Let $\rho\colon G_\Q\to\GSp_4(\I^\circ_\Tr)$ be the Galois representation associated with $\theta$. Suppose that $\rho$ is $\Z_p$-regular in the sense of Definition \ref{Zpreg}. 
For every radius $r$ in the set $\{r_i\}_{i\in\N^{>0}}$ defined in Section \ref{gspfam}, let $\fG_r$ be the Lie algebra that we attached to $\im\rho$ in Section \ref{biglie}. 

\begin{thm}\label{thexlevel}
There exists a non-zero ideal $\fl$ of $\I_0$ such that 
\begin{equation}\label{levincl} \fl\cdot\fsp_4(\B_r)\subset\fG_r \end{equation}
for every $r\in\{r_i\}_{i\in\N^{>0}}$.
\end{thm}

Let $\Delta$ be the set of roots of $\GSp_4$ with respect to our choice of maximal torus. 
Recall that for $\alpha\in\Delta$ we denote by $\fu^\alpha$ the nilpotent subalgebra of $\fgsp_4$ corresponding to $\alpha$. 
Let $r$ be a radius in the set $\{r_i\}_{i\geq 1}$. 
We set $\fU_r^\alpha=\fG_r\cap\fu^\alpha(\B_r)$ and $\fU_{r,\C_p}^\alpha=\fG_{r,\C_p}\cap\fu^\alpha(\B_{r,\C_p})$, which coincides with $\fU_r^\alpha\widehat{\otimes}_{\Q_p}\C_p$.
Via the isomorphisms $\fu^\alpha(\B_r)\cong\B_r$ and $\fu^\alpha(\B_{r,\C_p})\cong\B_{r,\C_p}$ we see $\fU_r^\alpha$ as a $\Q_p$-vector subspace of $\B_r$ and $\fU_{r,\C_p}^\alpha$ as a $\C_p$-vector subspace of $\B_{r,\C_p}$.

Recall that $U^\alpha$ denotes the one-parameter unipotent subgroup of $\GSp_4$ associated with the root $\alpha$. 
Let $H_{r}$ be the normal open subgroup of $G_\Q$ defined in the beginning of Section \ref{sen}. Note that Proposition \ref{uniplatt} holds with $\rho\vert_{H_0}$ replaced by $\rho\vert_{H_r}$ since $H_r$ is open in $G_\Q$. Let $U^\alpha(\rho\vert_{H_r})=U^\alpha(\I_0^\circ)\cap\rho(H_r)$ and $U^\alpha(\rho_{r})=U^\alpha\cap\rho_r(H_r)$. 
Via the isomorphisms $U^\alpha(\I_0)\cong\I_0$ and $U^\alpha(\I_{r,0})\cong\I_{r,0}$ we identify $U^\alpha(\rho\vert_{H_0})$ and $U^\alpha(\rho_{r})$ with $\Z_p$-submodules of $\I_0$ and $\I_{r,0}$, respectively. Note that the injection $\I^\circ_0\into\I_{r,0}^\circ$ induces an isomorphism of $\Z_p$-modules $U^\alpha(\rho\vert_{H_0})\cong U^\alpha(\rho_{r})$. 

We define a nilpotent subalgebra of $\fgsp_4(\I_{r,0})$ by $\fU_{\I_{r,0}}^\alpha=\Q_p\cdot\log(U^\alpha(\rho_{r}))$. 
We identify $\fU_{\I_{r,0}}^\alpha$ with a $\Q_p$-vector subspace of $\I_{r,0}$.
Note that the natural injection $\iota_{\B_r}\colon\I_{r,0}\into\B_r$ induces an injection $\fU_{\I_{r,0}}^\alpha\into\fU_r^\alpha$ for every $\alpha$.

\begin{lemma}\label{alglatt}
For every $\alpha\in\Delta$ and every $r$ there exists a non-zero ideal $\fl^\alpha$ of $\I_0$, independent of $r$, such that the $B_r$-span of $\fU_r^\alpha$ contains $\fl^\alpha\B_r$.
\end{lemma}

\begin{proof}
Let $d$ be the dimension of $Q(\I^\circ_0)$ over $Q(\Lambda_h)$. Let $\alpha\in\Delta$. By Proposition \ref{uniplatt} the unipotent subgroup $U^\alpha(\rho\vert_{H_r})$ contains a basis $E=\{e_i\}_{i=1,\ldots,d}$ of a $\Lambda_h$-lattice in $\I_0^\circ$. Lemma \ref{lattice} implies that the $\Lambda_h[p^{-1}]$-span of $E$ contains a non-zero ideal $\fl^\alpha$ of $\I_0$. 
Consider the map $\iota^\alpha\colon U^\alpha(\I_0)\to\fu^\alpha(\B_r)$ given by the composition
\[ U^\alpha(\I_0)\into U^\alpha(\I_{r,0})\xto{\log{}}\fu^\alpha(\I_{r,0})\into\fu^\alpha(\B_r), \]
where all the maps have been introduced above. Note that $\iota^\alpha(U^\alpha(\rho\vert_{H_0}))\subset\fU_r^\alpha$. Let $E_{\B_r}=\iota^\alpha(E)$. Since $\iota^\alpha$ is a morphism of $\I_0$-modules we have 
\[ B_r\cdot\fU_r^\alpha\supset B_r\cdot E_{\B_r}=B_r\cdot(\Lambda_h[p^{-1}]\cdot E_{\B_r})=B_r\cdot\iota^\alpha(\Lambda_h[p^{-1}]\cdot E)\supset\B_r\cdot\iota^\alpha(\fl^\alpha)=\fl^\alpha\B_r. \]
By construction and by Remark \ref{Grind} the ideal $\fl^\alpha$ can be chosen independently of $r$. 
\end{proof}

Let $\gamma$ be an element of $\GSp_4(\B_{r,\C_p})$ such that $\Phi_{\B_r}=\gamma C_{T_1,T_2}\gamma^{-1}$; it exists by Proposition \ref{expsen}(1). 
Let $\fG_{r,\C_p}^\gamma=\gamma^{-1}\fG_{r,\C_p}\gamma$. For each $\alpha\in\Delta$ let $\fU_{r,\C_p}^{\gamma,\alpha}=\fu^\alpha(\B_{r,\C_p})\cap\fG_{r,\C_p}^\gamma$. 
We prove the following lemma by an argument similar to that of \cite[Theorem 4.8]{hidatil}. 

\begin{lemma}\label{Brstr}
For every $\alpha\in\Delta$ the Lie algebra $\fU_{r,\C_p}^{\gamma,\alpha}$ is a $B_{r,\C_p}$-submodule of $\B_{r,\C_p}$. 
\end{lemma}

\begin{proof}
By Proposition \ref{expsen}(2) the operator $\Phi_{\B_r}$ normalizes $\fG_{r,\C_p}$, hence $C_{T_1,T_2}$ normalizes $\fG_{r,\C_p}^\gamma$. Since $C_{T_1,T_2}$ is diagonal it also normalizes $\fU_{r,\C_p}^{\gamma,\alpha}$. Moreover $\Ad(C_{T_1,T_2})u^\alpha=\alpha(C_{T_1,T_2})u^\alpha$ for every $u^\alpha\in\fU_{r,\C_p}^{\gamma,\alpha}$. 
Let $\alpha_1$ and $\alpha_2$ be the roots sending $\diag(t_1,t_2,\nu t_2^{-1},\nu t_1^{-1})\in T_2$ to $t_1/t_2$ and $\nu^{-1}t_2^2$, respectively. With respect to our choice of Borel subgroup, the set of positive roots of $\GSp_4$ is $\{\alpha_1,\alpha_2,\alpha_1+\alpha_2,2\alpha_1+\alpha_2\}$. The Lie bracket gives an identification $[\fU_{r,\C_p}^{\gamma,\alpha_1},\fU_{r,\C_p}^{\gamma,\alpha_2}]=\fU_{r,\C_p}^{\gamma,\alpha_1+\alpha_2}$. 
Conjugation by $C_{T_1,T_2}$ on the $\C_p$-vector space $\fu^{\alpha_1}(\B_{r,\C_p})$ induces multiplication by $\alpha_1(C_{T_1,T_2})=u^{-2}(1+T_2)$. Since $u^{-2}\in\Z_p^\times$ and $\fU_{r,\C_p}^{\gamma,\alpha_1}$ is stable under $\Ad(C_{T_1,T_2})$, multiplication by $1+T_2$ on $\fu^{\alpha_1}(\B_{r,\C_p})$ leaves $\fU_{r,\C_p}^{\gamma,\alpha_1}$ stable.  
Now we compute 
\begin{gather*} (1+T_2)\cdot\fU_{r,\C_p}^{\gamma,\alpha_1+\alpha_2}=(1+T_2)\cdot [\fU_{r,\C_p}^{\gamma,\alpha_1},\fU_{r,\C_p}^{\gamma,\alpha_2}]=[(1+T_2)\cdot \fU_{r,\C_p}^{\gamma,\alpha_1},\fU_{r,\C_p}^{\gamma,\alpha_2}]\subset [\fU_{r,\C_p}^{\gamma,\alpha_1},\fU_{r,\C_p}^{\gamma,\alpha_2}]=\fU_{r,\C_p}^{\gamma,\alpha_1+\alpha_2}, \end{gather*}
where the inclusion $(1+T_2)\cdot \fU_{r,\C_p}^{\gamma,\alpha_1}\subset\fU_{r,\C_p}^{\gamma,\alpha_1}$ is the result of the previous sentence. We deduce that multiplication by $1+T_2$ on $\fu^{\alpha_1+\alpha_2}(\B_{r,\C_p})$ leaves $\fU_{r,\C_p}^{\gamma,\alpha_1+\alpha_2}$ stable.

Similarly, conjugation by $C_{T_1,T_2}$ on the $\C_p$-vector space $\fu^{\alpha_2}(\B_{r,\C_p})$ induces multiplication by $\alpha_2(C_{T_1,T_2})=u\cdot\frac{1+T_1}{1+T_2}$. Since $u\in\Z_p^\times$ and $\fU_{r,\C_p}^{\gamma,\alpha_2}$ is stable under $\Ad(C_{T_1,T_2})$, multiplication by $1+T_2$ on $\fu^{\alpha_2}(\B_{r,\C_p})$ leaves $\fU_{r,\C_p}^{\gamma,\alpha_2}$ stable. The same calculation as above shows that multiplication by $\frac{1+T_1}{1+T_2}$ on $\fu^{\alpha_1+\alpha_2}(\B_{r,\C_p})$ leaves $\fU_{r,\C_p}^{\gamma,\alpha_1+\alpha_2}$ stable.

Having proved that multiplication by both $1+T_2$ and $\frac{1+T_1}{1+T_2}$ leaves $\fU_{r,\C_p}^{\gamma,\alpha_1+\alpha_2}$ stable, we deduce that multiplication by $(1+T_2)\cdot\frac{1+T_1}{1+T_2}=1+T_1$ also leaves $\fU_{r,\C_p}^{\gamma,\alpha_1+\alpha_2}$ stable. Since $\fU_{r,\C_p}^{\gamma,\alpha_1+\alpha_2}$ is a $\C_p$-vector space, we obtain that the $\C_p[T_1,T_2]$-module structure on $\fu^{\alpha_1+\alpha_2}(\B_{r,\C_p})$ induces a $\C_p[T_1,T_2]$-module structure on $\fU_{r,\C_p}^{\gamma,\alpha_1+\alpha_2}$. With respect to the $p$-adic topology $\fU_{r,\C_p}^{\gamma,\alpha_1+\alpha_2}$ is complete and $\C_p[T_1,T_2]$ is dense in $B_{r,\C_p}$, so the $B_{r,\C_p}$-module structure on $\fu^{\alpha_1+\alpha_2}(\B_{r,\C_p})$ induces a $B_{r,\C_p}$-module structure on $\fU_{r,\C_p}^{\gamma,\alpha_1+\alpha_2}$.

If $\beta$ is another root, we can write
\begin{gather*} B_{r,\C_p}\cdot\fU_{r,\C_p}^{\gamma,\beta}=B_{r,\C_p}\cdot[\fU_{r,\C_p}^{\gamma,\alpha_1+\alpha_2},\fU_{r,\C_p}^{\gamma,\beta-\alpha_1-\alpha_2}]\subset \\
\subset [B_{r,\C_p}\cdot\fU_{r,\C_p}^{\gamma,\alpha_1+\alpha_2},\fU_{r,\C_p}^{\gamma,\beta-\alpha_1-\alpha_2}]\subset  [\fU_{r,\C_p}^{\gamma,\alpha_1+\alpha_2},\fU_{r,\C_p}^{\gamma,\beta-\alpha_1-\alpha_2}]=\fU_{r,\C_p}^{\gamma,\beta}, \end{gather*}
where the inclusion $B_{r,\C_p}\cdot\fU_{r,\C_p}^{\gamma,\alpha_1+\alpha_2}\subset\fU_{r,\C_p}^{\gamma,\alpha_1+\alpha_2}$ is the result of the previous paragraph.
\end{proof}

\begin{proof}(of Theorem \ref{thexlevel})
Let $E_{\B_r}\subset\fU_r^\alpha$ be the set defined in the proof of Lemma \ref{alglatt}. Let $E_{\B_r,\C_p}=\{e\otimes 1\,\vert\,e\in E_{\B_r}\}\subset\fU_{r,\C_p}^\alpha$. Consider the Lie subalgebra $B_{r,\C_p}\cdot\fG_{r,\C_p}$ of $\fgsp_4(\B_{r,\C_p})$. For every $\alpha\in\Delta$ we have $B_{r,\C_p}\cdot\fG_{r,\C_p}\cap\fu^\alpha(B_{r,\C_p})=B_{r,\C_p}\cdot\fU^\alpha_r$. By Lemma \ref{alglatt} there exists an ideal $\fl^\alpha$ of $\I_0$, independent of $r$, such that $\fl^\alpha\cdot\B_{r,\C_p}\subset B_{r,\C_p}\cdot\fU^\alpha_r$. Let $\fl_0=\prod_{\alpha\in\Delta}\fl^\alpha$. Then Lemma \ref{liestdarg} gives an inclusion
\begin{equation}\label{flBr} \fl_0\cdot\fsp_4(\B_{r,\C_p})\subset B_{r,\C_p}\cdot\fG_{r,\C_p}. \end{equation}

As before let $\gamma$ be an element of $\GSp_4(\B_{r,\C_p})$ satisfying $\Phi_{\B_r}=\gamma C_{T_1,T_2}\gamma^{-1}$. 
The Lie algebra $\fl_0\cdot\fsp_4(\B_{r,\C_p})$ is stable under $\Ad(\gamma^{-1})$, so Equation \ref{flBr} implies that $\fl_0\cdot\fsp_4(\B_{r,\C_p})=\gamma^{-1}(\fl_0\cdot\fsp_4(\B_{r,\C_p}))\gamma\subset \gamma^{-1}(B_{r,\C_p}\cdot\fG_r)\gamma=B_{r,\C_p}\cdot\gamma^{-1}\fG_r\gamma=B_{r,\C_p}\cdot\fG_r^{\gamma}$. 
We deduce that, for every $\alpha\in\Delta$,
\begin{equation}\begin{gathered}\label{flBrgamma} \fl_0\cdot\fu^\alpha(\B_{r,\C_p})=\fu^\alpha(\B_{r,\C_p})\cap\fl_0\cdot\fsp_4(\B_{r,\C_p})\subset\fu^\alpha(\B_{r,\C_p})\cap B_{r,\C_p}\cdot\fG_{r,\C_p}^{\gamma}= \\
=B_{r,\C_p}\cdot(\fu^\alpha(\B_{r,\C_p})\cap\fG_{r,\C_p}^{\gamma})=B_{r,\C_p}\cdot\fU_{r,\C_p}^{\gamma,\alpha}. \end{gathered}\end{equation}
By Lemma \ref{Brstr} $\fU_{r,\C_p}^{\alpha,\gamma}$ is a $B_{r,\C_p}$-submodule of $\fu_r(\B_{r,\C_p})$, so $B_{r,\C_p}\cdot\fU_{r,\C_p}^{\gamma,\alpha}=\fU_{r,\C_p}^{\gamma,\alpha}$. Hence Equation \eqref{flBrgamma} gives
\begin{equation}\label{flUBrgamma} \fl_0\cdot\fu^\alpha(\B_{r,\C_p})\subset\fU_{r,\C_p}^{\gamma,\alpha} \end{equation}
for every $\alpha$. 
Set $\fl_1=\fl_0^2$. By Lemma \ref{liestdarg} and Remark \ref{fl2}, applied to the Lie algebra $\fG_{r,\C_p}$ and the set of ideals $\{\fl_1\B_r\}_{\alpha\in\Delta}$, Equation \eqref{flUBrgamma} implies that 
$\fl_1\cdot\fsp_4(\B_{r,\C_p})\subset\fG_{r,\C_p}^\gamma$. 
Observe that the left hand side of the last equation is stable under $\Ad(\gamma)$, so we can write
\begin{equation}\label{flGBrgamma} \fl_1\cdot\fsp_4(\B_{r,\C_p})=\gamma(\fl_1\cdot\fsp_4(\B_{r,\C_p}))\gamma^{-1}\subset\gamma\fG_{r,\C_p}^\gamma\gamma^{-1}=\fG_{r,\C_p}. \end{equation}

To complete the proof we show that the extension of scalars to $\C_p$ in Equation \ref{flGBrgamma} is unnecessary, up to restricting the ideal $\fl_1$. By Equation \ref{flGBrgamma} we have, for every $\alpha$,
\begin{equation}\label{cpincl} \fl_1\cdot\B_{r,\C_p}\subset\fU_{r,\C_p}^\alpha. \end{equation}
We prove that the above inclusion of $\C_p$-vector spaces descends to an inclusion $\fl_1\cdot\B_r\subset\fU_r^\alpha$ of $\Q_p$-vector spaces. Let $I$ be some index set and let $\{f_i\}_{i\in I}$ be an orthonormal basis of $\C_p$ as a $\Q_p$-Banach space, satisfying $1\in\{f_i\}_{i\in I}$. Let $\fa$ be any ideal of $\I_{r,0}$ belonging to the set $S_2$. Recall that the $\Q_p$-vector space $\B_r/\fa\B_r\cong\I_{r,0}/\fa$ is finite-dimensional. We write $\pi_\fa$ for the projection $\B_r\to\I_{r,0}/\fa$ and also for its restriction $\I_{r,0}\to\I_{r,0}/\fa$. Let $n$ and $d$ be the $\Q_p$-dimensions of $\I_{r,0}/\fa$ and $\pi_\fa(\fU_r^\alpha)$, respectively. Choose a $\Q_p$-basis $\{v_j\}_{j=1,\ldots,n}$ of $\I_{r,0}/\fa$ such that $\{v_j\}_{j=1,\ldots,d}$ is a $\Q_p$-basis of $\fU_r^\alpha$.

Let $v$ be any element of $\pi_\fa(\fl_1)$. Then $v\otimes 1\in\pi_\fa(\fl_1)\widehat{\otimes}_{\Q_p}\C_p$ and by Equation \eqref{cpincl} we have $v\otimes 1\in\pi_\fa(\fU_r^\alpha)\widehat{\otimes}_{\Q_p}\C_p$. Now $\{v_j\otimes f_i\}_{1\leq j\leq n;\, i\in I}$ and $\{v_j\otimes f_i \}_{1\leq j\leq d;\, i\in I}$ are orthonormal $\Q_p$-basis of $\B_r/\fa\widehat{\otimes}_{\Q_p}\C_p$ and $\pi_\fa(\fU_r^\alpha)\widehat{\otimes}_{\Q_p}\C_p$, respectively. Hence there exists a set $\{\lambda_{j,i}\}_{1\leq j\leq d;\, i\in I}\subset\Q_p$ converging to $0$ in the filter of complements of finite subsets of $\{1,2,\ldots,d\}\times I$ such that $v\otimes 1=\sum_{j=1,\ldots,d;\, i\in I}\lambda_{j,i}(v_j\otimes f_i)$. By setting $\lambda_{j,i}=0$ for $d<j\leq n$ we obtain a representation $v\otimes 1=\sum_{j=1,\ldots,n;\, i\in I}\lambda_{j,i}(v_j\otimes f_i)$ with respect to the basis $\{v_j\otimes f_i\}_{1\leq j\leq n;\, i\in I}$ of $(\B_r/\fa)\widehat{\otimes}_{\Q_p}\C_p$. 
On the other hand there exist $a_j\in\Q_p$, $j=1,2,\ldots,n$, such that $v=\sum_{j=1}^n a_jv_j$, so $v\otimes 1=\sum_{j=1}^n a_j(v_j\otimes 1)$ is another representation of $v\otimes 1$ with respect to the basis $\{v_j\otimes f_i\}_{1\leq j\leq n;\, i\in I}$. 
By the uniqueness of the representation of an element in a $\Q_p$-Banach space in terms of a given orthonormal basis we must have $a_j=\lambda_{j,i}$ if $f_i=1$. In particular $a_j=0$ for $d<j\leq n$, so $v=\sum_{j=1}^da_jv_j$ is an element of $\pi_\fa(\fU_r^\alpha)$.

The discussion above proves that $\pi_\fa(\fl_1)\subset\pi_\fa(\fU_r^\alpha)$ for every $\fa\in S_2$. By taking a projective limit over $\fa$ with respect to the natural maps we obtain $\fl_1\cdot\B_r\subset\fU_r^\alpha$. 
Let $\fl=\fl_1^2$. From Lemma \ref{liestdarg} and Corollary \ref{fl2}, applied to the Lie algebra $\fG_{r,\C_p}$ and the set of ideals $\{\fl_1\B_r\}_{\alpha\in\Delta}$, we deduce that
\[ \fl\cdot\fsp_4(\B_r)\subset\fG_{r}. \]
By definition we have $\fl=\fl_1^2=\fl_0^{4}=\left(\prod_{\alpha\in\Delta}\fl_\alpha\right)^{4}$. 
For every $\alpha$ the ideal $\fl^\alpha$ provided by Lemma \ref{alglatt} is independent of $r$, so $\fl$ is also independent of $r$. This concludes the proof of Theorem \ref{thexlevel}. 
\end{proof}

\begin{defin}\label{deflevel}
We call \emph{Galois level} of $\theta$ and denote by $\fl_\theta$ the largest ideal of $\I_0$ satisfying the inclusion \eqref{levincl}.
\end{defin}

\subsection{The Galois level of ordinary families}\label{exlevelord}

We explain how our arguments can be applied to an ordinary family of $\GSp_4$-eigenforms in order to show a stronger result than Theorem \ref{thexlevel}.
Let $M$ be a positive integer. 
Let $\T^\ord$ be a local component of the big ordinary cuspidal Hecke algebra of level $\Gamma_1(M)\cap\Gamma_0(p)$ for $\GSp_4$; it is a finite and flat $\Lambda_2$-algebra. With the terminology of Section \ref{gspfam} we consider $\T^\ord$ as the genus $2$, $0$-adapted Hecke algebra of the given level. 
Suppose that the residual representation $\ovl\rho_{\T^\ord}$ associated with $\T^\ord$ is absolutely irreducible and of $\Sym^3$ type in the sense of Definition \ref{sctype}. 
Let $\theta\colon\T^\ord\to\I^\circ$ be a family, i.e. the morphism of finite $\Lambda_2$ algebras describing an irreducible component of $\T^\ord$. 
Note that the algebra $\T^\ord$ may different from the one given by the construction in Section \ref{gspfam} for the choices $h=0$ and $r_h=1$; however all of our arguments and contructions are equally valid for the algebra $\T^\ord$. None of them relied on the fact that the slope of the family was positive. 

We keep all the notations we introduced for the family $\theta$. 
Let $\rho\colon G_\Q\to\GSp_4(\I^\circ_\Tr)$ be the Galois representation associated with $\theta$. Suppose that $\rho$ is $\Z_p$-regular in the sense of Definition \ref{Zpreg}. Then we have the following.

\begin{thm}\label{thexlevelord}
There exists a non-zero ideal $\fl$ of $\I_0^\circ$ and an element $g$ of $\GSp_4(\I_0^\circ)$ such that 
\begin{equation}\label{levinclord} g\Gamma_{\I_0^\circ}(\fl)g^{-1}\subset\im\rho. \end{equation}
\end{thm}

The main difference with respect to the proof of Theorem \ref{thexlevel} is that relative Sen theory is not necessary anymore, since the exponential of the Sen operator defined in Section \ref{senexp} is replaced by an element provided by the ordinarity of $\rho$. This is the reason why we do not need the Lie-theoretic constructions and we obtain a group-theoretic result. Note that this also makes the inversion of $p$ unnecessary. Theorem \ref{thexlevelord} is an analogue of \cite[Theorem 2.4]{lang}, which deals with ordinary families of $\GL_2$-eigenforms, and a generalization to the case where $\I^\circ\neq\Lambda_2$ of \cite[Theorem 4.8]{hidatil} for $n=2$ and families of residual symmetric cube type. 

We only sketch the proof of the theorem, pointing out the differences with respect to that of Theorem \ref{thexlevel}. 

\begin{proof}
Let $u=1+p$, let $\chi$ be the $p$-adic cyclotomic character and, for $\sigma\in\I_0^{\circ,\times}$, let $\ur(\sigma)\colon G_{\Q_p}\to\I_0^{\circ,\times}$ be the unramified character sending a lift of the Frobenius automorphism to $\sigma$. 
By Hida theory the ordinarity of $\theta$ implies the ordinarity of the Galois representation $\rho$, in the sense that the restriction of $\rho$ to a decomposition group at $p$ is a conjugate of an upper triangular representation with diagonal entries given by 
\[ \left(\chi^{-3}\cdot((1+T_1)(1+T_2))^{\frac{\log(\chi)}{\log(u)}}\ur(\alpha),\chi^{-1}\cdot(1+T_1)^{\frac{\log(\chi)}{\log(u)}}\ur(\beta),\chi^{-2}\cdot(1+T_2)^{\frac{\log(\chi)}{\log(u)}}\ur(\gamma),\ur(\delta)\right) \]
for some $\alpha,\beta,\gamma,\delta\in\I_0^{\circ,\times}$. Consider a conjugate of $\rho$ that has the form displayed above. Up to conjugation by an upper triangular matrix we can suppose that $\im\rho$ contains a diagonal $\Z_p$-regular element. By Proposition \ref{H0repr} we can further replace the representation with a conjugate by a diagonal matrix such that $\rho(H_0)\subset\GSp_4(\I_0^\circ)$. This is true because the basis we start with in the proof of Proposition \ref{H0repr} is replaced by a collinear one. 

We work from now on with the last one of the conjugates of the original $\rho$ mentioned in the previous paragraph; this choice gives the element $g$ appearing in Theorem \ref{thexlevelord}. It is clear from the form of $\rho$ that there exists an element $\sigma$ in the inertia subgroup at $p$ such that $\rho(\sigma)=C_{T_1,T_2}$, where $C_{T_1,T_2}$ is the matrix defined in Section \ref{senexp}. Hence $\im\rho$ is stable under $\Ad C_{T_1,T_2}$. The same argument as in Lemma \ref{Brstr}, with the nilpotent algebra $\fU_{r,\C_p}^{\gamma,\alpha}$ replaced by the unipotent subgroup $U^\alpha(\im\rho)$ and the extension of rings $B_r\subset\B_r$ replaced by $\Lambda_2\subset\I_0^\circ$, gives $U^\alpha(\im\rho)$ a structure of $\Lambda_2$-module for every root $\alpha$ of $\Sp_4$. By Proposition \ref{uniplatt} $U^\alpha(\im\rho)$ contains a basis of a $\Lambda_2$-lattice in $\I_0^\circ$ for every $\alpha$. Hence, by Lemma \ref{lattice}, $U^\alpha(\im\rho)$ contains a non-zero ideal of $\I_0^\circ$ for every $\alpha$. By Proposition \ref{stdarg} the group $\im\rho$ contains a non-trivial congruence subgroup of $\Sp_4(\I_0^\circ)$.
\end{proof}

\bigskip

\section{The symmetric cube morphisms of Hecke algebras}

Let $\Sym^3\colon\GL_2\to\GSp_4$ be the morphism of group schemes over $\Z$ defined by the symmetric cube representation of $\GL_2$. It fits in an exact sequence $0\to\mu_3\to\GL_2\to\GSp_4$ of group schemes over $\Z$. 
If $R$ is a ring we still denote by $\Sym^3$ the morphism $\GL_2(R)\to\GSp_4(R)$ induced by the morphism of group schemes. 
For every representation $\rho$ of a group with values in $\GL_2(R)$ we set $\Sym^3\rho=\Sym^3\ccirc\rho$. 

Kim and Shahidi proved the existence of a Langlands functoriality transfer from $\GL_2$ to $\GL_4$ associated with $\Sym^3\colon\GL_2(\C)\to\GL_4(\C)$ \cite[Theorem B]{kimsha}. Thanks to an unpublished result by Jacquet, Piatetski-Shapiro and Shalika \cite[Theorem 9.1]{kimsha}, this transfer descends to $\GSp_4$. We briefly recall these results.

Let $\pi=\bigotimes_v\pi_v$ be a cuspidal automorphic representation of $\GL_2(\A_\Q)$, where $v$ varies over the places of $\Q$. Let $\rho_v$ be the two-dimensional representation of the Weil-Deligne group of $\Q_v$ attached to $\pi_v$. Consider the four-dimensional representation $\Sym^3\rho_v=\Sym^3\circ\rho_v$ of the same group. By the local Langlands correspondence for $\GL_4$, $\Sym^3\rho_v$ is attached to an automorphic representation $\Sym^3\pi_v$ of $\GL_4(\Q_v)$. Define a representation of $\GL_4(\A_\Q)$ as $\Sym^3\pi=\bigotimes_v\Sym^3\pi_v$. Then we have the following theorems.


\begin{thm}\label{lintransf}\cite[Theorem B]{kimsha}
The representation $\Sym^3\pi$ is an automorphic representation of $\GL_4(\A_\Q)$. If $\pi$ is attached to a non-CM eigenform of weight $k\geq 2$, then $\Sym^3\pi$ is cuspidal.
\end{thm}

\begin{thm}\label{gentransf}\cite[after Theorem 9.1]{kimsha}
If $\pi$ is attached to a non-CM eigenform of weight $k\geq 2$, then there exists a globally generic cuspidal automorphic representation $\Pi$ of $\GSp_4(\A_\Q)$ such that $\Sym^3\pi$ is the functorial lift of $\Pi$ under the embedding $\GSp_4(\C)\into\GL_4(\C)$.
\end{thm}

\subsection{Compatible levels for the classical symmetric cube transfer}

If $K$ is a compact open subgroup of $\GSp_4(\widehat{\Z})$, we call \emph{level of $K$} the smallest integer $M$ such that $K$ contains the principal congruence subgroup of $\GSp_4(\widehat{\Z})$ of level $M$. Given an automorphic representation $\Pi$ of $\GSp_4(\A_\Q)$, we call level of $\Pi$ the smallest integer $M$ such that the finite component of $\Pi$ admits an invariant vector by a compact open subgroup of $\GSp_4(\widehat{\Z})$ of level $M$.

Recall that we fixed for every prime $\ell$ an embedding $G_{\Q_\ell}\into G_\Q$. If $\sigma\colon G_\Q\to\GL_n(\ovl{\Q}_p)$ is a representation and $\ell\neq p$ is a prime, set $\sigma_\ell=\sigma\vert_{G_{\Q_\ell}}$. We denote by $N(\sigma,\ell)$ the conductor of $\sigma_\ell$, defined in \cite{serre}. The prime-to-$p$ conductor of $\sigma$ is defined as $N(\sigma)=\prod_{\ell\neq p}N(\sigma,\ell)$.
We recall a standard formula giving $N(\sigma,\ell)$ for every $\ell$ prime to $p$ (see \cite[Proposition 1.1]{livne}).
Let $I\subset G_{\Q_\ell}$ be an inertia subgroup and for $k\geq 1$ let $I_k$ be its higher inertia subgroups. 
Let $V$ be the two-dimensional $\Qp$-vector space on which $G_\Q$ acts via $\sigma$. For every subgroup $H\subset G_\Q$ let $d_{H,\sigma}$ be the codimension of the subspace of $V$ fixed by $\sigma(H)$. 
Then $N(\sigma,\ell)=\ell^{n_{\sigma,\ell}}$, where 
\begin{equation}\label{condform} 
n_{\sigma,\ell}=d_{I,\sigma_\ell}+\sum_{k\geq 1}\frac{d_{I_k,\ovl{\sigma_\ell}}}{[I\colon I_k]}.
\end{equation}
%
Write $\Pi_f$ for the component of $\Pi$ at the finite places and $\Pi_\infty$ for the component of $\Pi$ at $\infty$. 
Since the representation $\Pi$ given by the above theorem is globally generic, it does not correspond to a holomorphic modular form for $\GSp_4$. However Ramakrishnan and Shahidi showed that the generic representation $\Pi_\infty$ can be replaced by a holomorphic representation $\Pi_\infty^\hol$ such that $\Pi_f\otimes\Pi_\infty^\hol$ belongs to the $L$-packet of $\Pi$. This is the content of \cite[Theorem A$^\prime$]{ramsha}, that we recall below. Note that in \emph{loc. cit.} the theorem is stated only for $\pi$ associated with a form $f$ of level $\Gamma_0(N)$ and even weight $k\geq 2$, but Ramakrishnan pointed out that the proof also works when $f$ has level $\Gamma_1(N)$ and arbitrary weight $k\geq 2$. The theorem also gives an information on the level of the representation produced by the lift.

Let $\pi$ be the automorphic representation of $\GL_2(\A_\Q)$ associated with a cuspidal, non-CM eigenform $f$ of weight $k\geq 2$ and level $\Gamma_1(N)$ for some $N\geq 1$. Let $p$ be a prime not dividing $N$ and let $\rho_{f,p}$ be the $p$-adic Galois representation attached to $f$.

\begin{thm}(see \cite[Theorem A$^\prime$]{ramsha})\label{classtransf}
There exists a cuspidal automorphic representation $\Pi^\hol=\bigotimes_v\Pi^{\hol}_v$ of $\GSp_4(\A_\Q)$, satisfying:
\begin{enumerate}
\item $\Pi^{\hol}_\infty$ is in the holomorphic discrete series;
\item $L(s,\Pi^{\hol})=L(s,\pi,\Sym^3)$;
\item $\Pi^{\hol}$ is unramified at primes not dividing $N$;
\item $\Pi^{\hol}$ admits an invariant vector by a compact open subgroup $K$ of $\GSp_4(\A_\Q)$ of level $N(\Sym^3\rho_{f,p})$.
\end{enumerate}
\end{thm}

We deduce the following corollary.

\begin{cor}\label{formtransf}
Let $f$ be a cuspidal, non-CM $\GL_2$-eigenform of weight $k\geq 2$. For every prime $\ell$ let $\rho_{f,\ell}$ be the $\ell$-adic Galois representation associated with $f$. There exists a cuspidal $\GSp_4$-eigenform $F$ of weight $(2k-1,k+1)$ with associated $\ell$-adic Galois representation $\Sym^3\rho_{f,\ell}$ for every prime $\ell$. For every prime $p$ not dividing $N$, the level of $F$ is a divisor of the prime-to-$p$ conductor of $\Sym^3\rho_{f,p}$.
\end{cor}

Note that the weight $(2k-1,k+1)$ is cohomological since $k\geq 2$.

\begin{proof}
Everything follows immediately from Theorem \ref{classtransf} except for the weight of $F$, that can be found by writing the Hodge-Tate weights of $\Sym^3\rho_{f,p}$ in terms of those of $\rho_{f,p}$. 
\end{proof}

We denote by $\Sym^3f$ the cuspidal Siegel eigenform given by the corollary. Let $N(f)$ and $N(\Sym^3f)$ be the levels of $f$ and $\Sym^3f$, respectively. Thanks to the property (4) in Theorem \ref{classtransf} we can give an upper bound for $N(\Sym^3f)$ in terms of $N(f)$ by comparing $N(\Sym^3\rho_{f,p})$ and $N(\rho_{f,p})$ for a prime $p$ not dividing $N(f)$.

As before let $\sigma\colon G_\Q\to\GL_n(\ovl{\Q}_p)$ be a representation and let $\sigma_\ell=\sigma\vert_{G_{\Q_\ell}}$ for every prime $\ell$. 

\begin{lemma}\label{conductors}
For every prime $\ell\neq p$ we have $N(\Sym^3\sigma_\ell)\mid N(\sigma_\ell)^3$. In particular $N(\Sym^3\sigma)\mid N(\sigma)^3$.
\end{lemma}

\begin{proof}
We use the notations of formula \eqref{condform}. We check that $d_{H,\Sym^3\sigma}\leq 3d_{H,\sigma}$ for every subgroup $H$ of $G_\Q$, so formula \eqref{condform} gives $N(\Sym^3\sigma,\ell)\mid N(\sigma,\ell)^3$. 
Since the prime-to-$p$ conductor is defined as the product of the conductors at the primes $\ell\neq p$, we obtain that $N(\Sym^3\sigma)\mid N(\sigma)^3$.
\end{proof}

\begin{defin}\label{sym3leveldef}
Let $N$ be a positive integer and let $N=\prod_{i=1}^d\ell_i^{a_i}$ be its decomposition in prime factors, with $\ell_i\neq \ell_j$ if $i\neq j$. For every $i\in\{1,2,\ldots,d\}$ set $a_i^\prime=1$ if $a_i=1$ and $a_i^\prime=3a_i$ if $a_i>1$. 
We define an integer $M$, depending on $N$, by $M=\prod_{i=1}^d\ell_i^{a_i^\prime}$.
\end{defin}

\begin{cor}\label{sym3level}
Let $N=N(f)$ and let $M=M(N)$ be the integer given by Definition \ref{sym3leveldef}. Then $N(\Sym^3f)\mid M$.
\end{cor}

\begin{proof}
Let $\pi_f=\bigotimes_\ell\pi_{f,\ell}$ be the automorphic representation of $\GL_2(\A_\Q)$ associated with $f$. Let $\pi_{\Sym^3f}=\bigotimes_\ell\pi_{\Sym^3f,\ell}$ be the automorphic representation of $\GSp_4(\A_\Q)$ associated with $\Sym^3f$. For every prime $\ell$ the Galois representations associated with the local components $\pi_{f,\ell}$ and $\pi_{\Sym^3f,\ell}$ are $\rho_{f,\ell}$ and $\Sym^3\rho_{f,\ell}$, respectively. 
As before let $N=\prod_{i=1}^d\ell_i^{a_i}$ be the decomposition of $N$ in prime factors. If $\ell\nmid N$ the representation $\pi_{f,\ell}$ is unramified, so $\pi_{\Sym^3f,\ell}$ is also unramified.

Let $i\in\{1,2,\ldots,d\}$. If $a_i=1$ the local component $\pi_{f,\ell_i}$ is Iwahori-spherical, hence Steinberg. Then the image of the inertia subgroup at $\ell_i$ via $\rho_{f,\ell_i}$ contains a regular unipotent element $u$. The image of the inertia subgroup at $\ell_i$ via $\Sym^3\rho_{f,\ell_i}$ contains the regular unipotent element $\Sym^3u$, so the automorphic representation $\pi_{\Sym^3f,\ell_i}$ is Iwahori-spherical. Hence the factor $\ell_i$ appears with exponent one in $N(\Sym^3f)$.

Now suppose that $a_i>1$. 
Let $p$ be a prime not dividing $N$. 
By Corollary \ref{formtransf} the power of $\ell_i$ appearing in $N(\Sym^3f)$ is a divisor of $N(\Sym^3\rho_{f,p},\ell_i)$, that is a divisor of $N(\rho_{f,p},\ell_i)^3$ by Lemma \ref{conductors}. By a classical result of Carayol $N(\rho_{f,\ell_i})$ is a divisor of $\ell_i^{a_i}$, hence the conclusion. 
\end{proof}

Borrowing the terminology of \cite[Section 4.3]{ludwigunit}, we say that $\Gamma_1^{(1)}(N)$ and $\Gamma^{(2)}_1(M)$ are compatible levels for the symmetric cube transfer.

\subsection{Constructing the morphisms of Hecke algebras}\label{heckesym3}

As usual we fix an integer $N\geq 1$ and a prime $p$ not dividing $N$. 
We work with the abstract Hecke algebras $\calH_1^N$, $\calH_2^N$ spherical outside $N$ and Iwahoric dilating at $p$. 
Let $M$ be the integer given by Definition \ref{sym3leveldef}, depending on $N$. If $f$ is a non-CM $\GL_2$-eigenform of level $\Gamma_1(N)$, we denote by $\Sym^3f$ the classical, cuspidal $\GSp_4$-eigenform of level $\Gamma_1(M)$ given by Corollary \ref{formtransf}. Our goal for this section is to determine the systems of Hecke eigenvalues of the $p$-stabilizations of $\Sym^3f$ in terms of that of a $p$-stabilization of $f$.

If $\chi$ is a system of Hecke eigenvalues, we write $\chi_\ell$ for its local component at the prime $\ell$. 

\begin{rem}\label{sym3polcalc}
We will need multiple times the following simple computation. 
Let $R$ be a ring and let $g\in\GL_2(R)$. Let $g$ act on $R^2$ via the standard representation and let $P(g;X)=\det(\1-X\cdot g)=X^2-TX+D$ be the characteristic polynomial of $g$. 
Then the characteristic polynomial of $\Sym^3g$ is $P(\Sym^3g;X)=X^4-(T^3-2TD)X^3+(T^4-3DT^2+2D^2)X^2-D^3(T^3+2TD)X+D^6$. 
\end{rem}

If $T,D\in R$ are arbitrary and $P(X)=X^2-TX+D$, we define the \emph{symmetric cube of $P(X)$} as 
\[ \Sym^3P(X)=X^4-(T^3-2TD)X^3+(T^4-3DT^2+2D^2)X^2-D^3(T^3+2TD)X+D^6. \]

\subsubsection{The morphism of unramified Hecke algebras}

We define a morphism of unramified abstract Hecke algebras and show that it has the desired property with respect to the system of eigenvalues of $f$ and $\Sym^3f$ outside $Np$.

\begin{defin}\label{heckemorphunrdef}
For every prime $\ell\nmid Np$, let 
\[ \lambda_\ell\colon\calH(\GSp_4(\Q_\ell),\GSp_4(\Z_\ell))\to\calH(\GL_2(\Q_\ell),\GL_2(\Z_\ell)) \]
be the morphism defined by
\begin{gather*}
T^{(2)}_{\ell,0}\mapsto (T^{(1)}_{\ell,0})^3, \\
T^{(2)}_{\ell,1}\mapsto -(T^{(1)}_{\ell,1})^6+(4\ell -2)T^{(1)}_{\ell,0}(T^{(1)}_{\ell,1})^4+(6\ell-4\ell^2)(T^{(1)}_{\ell,0})^2(T^{(1)}_{\ell,1})^2-3\ell^2(T^{(1)}_{\ell,0})^3, \\
T^{(2)}_{\ell,2}\mapsto (T^{(1)}_{\ell,1})^3-2\ell T^{(1)}_{\ell,1} T^{(1)}_{\ell,0}. 
\end{gather*}
Let $\lambda^{Np}\colon\calH_2^{Np}\to\calH_1^{Np}$ be the morphism defined by $\lambda^{Np}=\bigotimes_{\ell\nmid Np}\lambda_\ell$.
\end{defin}

\begin{prop}\label{heckemorphunr}
Let $R$ be a ring. Let $\chi^{Np}_1\colon\calH^{Np}_1\to R$, $\chi^{Np}_2\colon\calH^{Np}_2\to R$ be two morphisms and let $\rho_1\colon G_\Q\to\GL_2(R)$, $\rho_2\colon G_\Q\to\GSp_4(R)$ be two representations satisfying:
\begin{enumerate}
\item for $g=1,2$ $\rho_g$ is unramified outside $Np$;
\item for $g=1,2$, every prime $\ell\nmid Np$ and a lift $\Frob_\ell\in G_\Q$ of the Frobenius at $\ell$,
\[ \det(1-X\rho_i(\Frob_\ell))=\chi_i^{Np}(P_\Min(t^{(g)}_{\ell,g};X)); \]
\item there is an isomorphism $\rho_2\cong\Sym^3\rho_1$.
\end{enumerate}
Then $\lambda^{Np}$ is the only morphism $\calH_2^{Np}\to\calH_1^{Np}$ such that $\chi^{Np}_2=\chi_1^{Np}\ccirc\lambda^{Np}$.
\end{prop}


\begin{proof}
Let $\ell$ be a prime not dividing $Np$. 
%
By Equation \eqref{minpol1} we have $P_\Min(t^{(1)}_{\ell,1};X)=X^2-T^{(1)}_{\ell,1}(f)X+\ell T_{\ell,0}^{(1)}$. Hence hypothesis (2) with $g=1$ gives  
\begin{equation} \det(1-X\rho_i(\Frob_\ell))=\chi_1^{Np}(X^2-T^{(1)}_{\ell,1}(f)X+\ell T_{\ell,0}^{(1)}). \end{equation}
By the calculation in Remark \eqref{sym3polcalc} we can write 
\begin{equation}\label{unr1}\begin{gathered}
\det(1-X\Sym^3\rho(\Frob_\ell))=X^4-(T^{(1)}_{\ell,1}-2\ell T^{(1)}_{\ell,1} T_{\ell,0}^{(1)})X^3+ \\
+((T^{(1)}_{\ell,1})^4-3\ell T_{\ell,0}^{(1)}(T^{(1)}_{\ell,1})^2+2\ell^2 (T_{\ell,0}^{(1)})^2)X^2-\ell^3(T_{\ell,0}^{(1)})^3((T^{(1)}_{\ell,1})^3+2\ell T^{(1)}_{\ell,1} T_{\ell,0}^{(1)})X+\ell^6(T_{\ell,0}^{(1)})^6. \end{gathered}\end{equation}
By Equation \eqref{minpol2} we have $P_\Min(t^{(2)}_{\ell,2};X)=X^4-T^{(2)}_{\ell,2}X^3+((T^{(2)}_{\ell,2})^2-T^{(2)}_{\ell,1}-\ell^2T^{(2)}_{\ell,0})X^2-\ell^3T^{(2)}_{\ell,2}T^{(2)}_{\ell,0} X+\ell^6(T^{(2)}_{\ell,0})^2$, so hypothesis (2) with $g=2$ gives
\begin{equation}\label{unr2}\begin{gathered}
\det(1-X\Sym^3\rho(\Frob_\ell))=\chi_2^{Np}(X^4-T^{(2)}_{\ell,2}X^3+ \\ +((T^{(2)}_{\ell,2})^2-T^{(2)}_{\ell,1}-\ell^2T^{(2)}_{\ell,0})X^2-\ell^3T^{(2)}_{\ell,2}T^{(2)}_{\ell,0} X+\ell^6(T^{(2)}_{\ell,0})^2).
\end{gathered}\end{equation}
By comparing the coefficients of the right hand sides of Equations \eqref{unr1} and \eqref{unr2} we obtain the relations
\begin{gather*} 
\chi_2^{Np}(T^{(2)}_{\ell,1})=\chi_1^{Np}(-(T^{(1)}_{\ell,1})^6+(4\ell -2)T^{(1)}_{\ell,0} (T^{(1)}_{\ell,1})^4+(6\ell-4\ell^2)(T^{(1)}_{\ell,1})^2(T^{(1)}_\ell)^2-3\ell^2(T^{(1)}_{\ell,0})^3), \\ 
\chi_2^{Np}(T^{(2)}_{\ell,2})=\chi_1^{Np}((T^{(1)}_\ell)^3-2\ell T^{(1)}_{\ell,1} T^{(1)}_{\ell,0}), \quad \chi_2^{Np}(T^{(2)}_{\ell,0})=\chi_1^{Np}((T^{(1)}_{\ell,0})^3).
\end{gather*}
We deduce that $\lambda_\ell$ is the only morphism $\calH(\GSp_4(\Q_\ell),\GSp_4(\Z_\ell))\to\calH(\GL_2(\Q_\ell),\GL_2(\Z_\ell))$ satisfying $\chi_2^{Np}(\Sym^3f)=\chi_1^{Np}\ccirc\lambda_\ell$. Since this is true for every $\ell\nmid Np$, we conclude that $\lambda^{Np}$ is the only morphism $\calH_2^{Np}\to\calH_1^{Np}$ satisfying $\chi^{Np}_2=\chi_1^{Np}\ccirc\lambda^{Np}$.
\end{proof}

As a special case of Proposition \ref{heckemorphunr} we obtain the following corollary.

\begin{cor}\label{heckemorphunrf}
Let $f$ be a classical, non-CM $\GL_2$-eigenform $f$ of level $\Gamma_1(N)$ and system of eigenvalues $\chi_1^{Np}\colon\calH_1^{Np}\to\Qp$ outside $Np$. Let $\Sym^3f$ be the symmetric cube lift of $f$ given by Corollary \ref{formtransf}. Then the system of eigenvalues $\chi_2^{Np}$ of $\Sym^3f$ outside $Np$ is $\chi_1^{Np}\ccirc\lambda^{Np}\colon\calH_2^{Np}\to\Qp$. 
\end{cor}

\begin{proof}
The corollary follows from Proposition \ref{heckemorphunr} applied to $R=\Qp$, $\chi_1^{Np}$ and $\chi_2^{Np}$ as in the statement, $\rho_1=\rho_{f,p}$ and $\rho_2=\rho_{\Sym^3f,p}$.
\end{proof}

\subsubsection{The morphisms of Iwahori-Hecke algebras}

We study the systems of Hecke eigenvalues of the $p$-stabilizations of $\Sym^3f$.

\begin{defin}\label{heckemorphdef}
For $i\in\{1,2,\ldots,8\}$ we define morphisms 
\[ \lambda_{i,p}\colon\calH(T_2(\Q_p),T_2(\Z_p))^-\to\calH(T_1(\Q_p),T_1(\Z_p)). \]
For $i\in\{1,2,3,4\}$ the morphism $\lambda_{i,p}$ is defined on a set of generators of $\calH(T_2(\Q_p),T_2(\Z_p))^-$ as follows: 
\begin{enumerate}
\item $\lambda_{1,p}$ maps $\quad t^{(2)}_{p,0}\mapsto (t^{(1)}_{p,0})^3$, $\quad t^{(2)}_{p,1}\mapsto t^{(1)}_{p,0}(t^{(1)}_{p,1})^4$, $\quad t^{(2)}_{p,2}\mapsto (t^{(1)}_{p,1})^3$; 

\item $\lambda_{2,p}$ maps $\quad t^{(2)}_{p,0}\mapsto (t^{(1)}_{p,0})^3$, $\quad t^{(2)}_{p,1}\mapsto (t^{(1)}_{p,0})^2(t^{(1)}_{p,1})^2$, $\quad t^{(2)}_{p,2}\mapsto (t^{(1)}_{p,1})^3$; 

\item $\lambda_{3,p}$ maps $\quad t^{(2)}_{p,0}\mapsto (t^{(1)}_{p,0})^3$, $\quad t^{(2)}_{p,1}\mapsto t^{(1)}_{p,0}(t^{(1)}_{p,1})^4$, $\quad t^{(2)}_{p,2}\mapsto t^{(1)}_{p,0}t^{(1)}_{p,1}$; 

\item $\lambda_{4,p}$ maps $\quad t^{(2)}_{p,0}\mapsto (t^{(1)}_{p,0})^3$, $\quad t^{(2)}_{p,1}\mapsto (t^{(1)}_{p,0})^4(t^{(1)}_{p,1})^{-2}$, $\quad t^{(2)}_{p,2}\mapsto t^{(1)}_{p,0}t^{(1)}_{p,1}$. 
\end{enumerate}

For $i\in\{5,6,7,8\}$ the morphism $\lambda_{i,p}\colon\calH(T_2(\Q_p),T_2(\Z_p))\to\calH(T_1(\Q_p),T_1(\Z_p))$ is given by
\[ \lambda_{i,p}=\delta\ccirc\lambda_{i-4,p} \]
where $\delta$ is the automorphism of $\calH(T_1(\Q_p),T_1(\Z_p))$ defined on a set of generators of the subalgebra $\calH(T_1(\Q_p),T_1(\Z_p))^-$ by
\begin{equation}\label{autdelta}\begin{gathered}
\delta(t^{(1)}_{p,0})=t^{(1)}_{p,0}, \quad \delta(t^{(1)}_{p,1})=t^{(1)}_{p,0}(t^{(1)}_{p,1})^{-1}
\end{gathered}\end{equation}
and extended in the unique way.
\end{defin}
%

Let $f^\st$ be a $p$-stabilization of a classical, cuspidal, non-CM $\GL_2$-eigenform $f$ of level $\Gamma_1(N)$. Let $\chi_{1,p}\colon\calH(\GL_2(\Q_p),\GL_2(\Z_p))\to\Qp$ and $\chi_{1,p}^\st\colon\calH(\GL_2(\Q_p),I_{1,p})^-\to\Qp$ be the systems of Hecke eigenvalues at $p$ of $f$ and $f^\st$, respectively. Note that $\chi_{1,p}$ is the restriction of $\chi_{1,p}^\st$ to the abstract spherical Hecke algebra at $p$. Let $(\Sym^3f)^\st$ be a $p$-stabilization of $\Sym^3f$. Let $\chi_{2,p}\colon\calH(\GSp_4(\Q_p),\GSp_4(\Z_p))\to\Qp$ and $\chi_{2,p}^\st\colon\calH(\GSp_4(\Q_p),I_{2,p})^-\to\Qp$ be the systems of Hecke eigenvalues at $p$ of $\Sym^3f$ and $(\Sym^3f)^\st$, respectively. Again $\chi_{2,p}$ is the restriction of $\chi_{2,p}^\st$ to the abstract spherical Hecke algebra at $p$. 

Recall from Section \ref{dilIw} that for $g=1,2$ there is an isomorphism of $\Q$-algebras $\iota_{I_{2,p}}^{T_2}\colon\calH(\GSp_{2g}(\Q_p),I_{g,p})^-\to\calH(T_g(\Q_p),T_g(\Z_p))^-$. 
Let $\iota^{I_{2,p}}_{T_2}\colon\calH(T_g(\Q_p),T_g(\Z_p))^-\to\calH(\GSp_{2g}(\Q_p),I_{g,p})^-$ be its inverse. In particular $\chi^\st_g\ccirc\iota^{I_{g,p}}_{T_g}$ is a character $\calH(T_g(\Q_p),T_g(\Z_p))^-\to\Qp$. 
By Remark \ref{extchar} the character $\chi_{g,p}^\st\ccirc\iota^{I_{g,p}}_{T_g}$ can be extended uniquely to a character $(\chi_{g,p}^\st\ccirc\iota^{I_{g,p}}_{T_g})^\ext\colon\calH(T_g(\Q_p),T_g(\Z_p))\to\Qp$.
%

\begin{prop}\label{heckemorphstab}
There exists $i\in\{1,2,\ldots,8\}$ such that 
\[ \chi^\st_2\ccirc\iota_{I_{2,p}}^{T_2}=(\chi^\st_1\ccirc\iota_{I_{1,p}}^{T_1})^\ext\ccirc\lambda_{i,p}. \]
Moreover, if $\lambda_p\colon\calH(T_2(\Q_p),T_2(\Z_p))\to\calH(T_1(\Q_p),T_1(\Z_p))$ is another morphism satisfying $\chi^\st_2\ccirc\iota_{I_{2,p}}^{T_2}=(\chi^\st_1\ccirc\iota_{I_{1,p}}^{T_1})^\ext\ccirc\lambda_p$, then there exists $i\in\{1,2,\ldots,8\}$ such that $\lambda_p=\lambda_{i,p}$.
\end{prop}


\begin{proof}
In this proof we leave the composition with the isomorphism $\iota_{I_{1,p}}^{T_1}$ and $\iota_{I_{2,p}}^{T_2}$ implicit and we consider $\chi_{1,p}^\st$ and $\chi_{2,p}^\st$ as characters respectively of $\calH(T_1(\Q_p),T_1(\Z_p))^-$ and $\calH(T_2(\Q_p),T_2(\Z_p))^-$ for notational ease.
Let $\rho_{f,p}\colon G_\Q\to\GL_2(\ovl{\Q}_p)$ be the $p$-adic Galois representation associated with $f$, so that the $p$-adic Galois representation associated with $\Sym^3f$ is $\Sym^3\rho_{f,p}$. 
%
%
%
Via $p$-adic Hodge theory we attach to $\rho_{f,p}$ a two-dimensional $\Qp$-vector space $\bD_\cris(\rho_{f,p})$ endowed with a $\ovl{\Q}_p$-linear Frobenius endomorphism $\varphi_\cris(\rho_{f,p})$ satisfying $\det(1-X\varphi_\cris(\rho_{f,p}))=\chi_{1,p}(P_\Min(t^{(2)}_{p,2};X))$. 

We will use the notations of Section \ref{weyl} for the elements of the Weyl groups of $\GL_2$ and $\GSp_4$. Let $\alpha_p$ and $\beta_p$ be the two roots of $\chi_{1,p}(P_\Min(t^{(2)}_{p,2};X))$, ordered so that $\chi_{1,p}^\st(t^{(1)}_{p,1})=\alpha_p$ and $\beta_p=\chi_{1,p}^\st((t^{(1)}_{p,1})^w)$. 


Let $\bD_\cris(\rho_{\Sym^3f,p})$ be the $4$-dimensional $\Qp$-vector space attached to $\rho_{\Sym^3f,p}$ by $p$-adic Hodge theory. Denote by $\varphi_{\cris}(\rho_{\Sym^3f,p})$ the Frobenius endomorphism acting on $\bD_\cris(\rho_{\Sym^3f,p})$. It satisfies $\det(1-X\varphi_\cris(\rho_{\Sym^3f,p}))=\chi_{2,p}(P_\Min(t^{(2)}_{p,2};X))$ by \cite[Théorème 1]{urban}. 
The coefficients of $P_\Min(t^{(2)}_{p,2};X)$ belong to the spherical Hecke algebra at $p$, so we have $\chi_{2,p}^\st(P_\Min(t^{(2)}_{p,2};X))=\chi_{2,p}(P_\Min(t^{(2)}_{p,2};X))$. 
From $\rho_{\Sym^3f,p}=\Sym^3\rho_{f,p}$ we deduce that
\begin{equation}\label{polsymcube}
\chi_{2,p}^\st(P_\Min(t^{(2)}_{p,2};X))=\det(1-X\varphi_\cris(\rho_{\Sym^3f,p}))=(X-\alpha_p^3)(X-\alpha_p^2\beta_p)(X-\alpha_p\beta_p^2)(X-\beta_p^3).
\end{equation}
%
By developing the left hand side via the first equality of Equation \eqref{minpol2} and the right hand side via Equation \eqref{polsymcube} we obtain
\begin{gather*}
(X-\chi_{2,p}^\st(t^{(2)}_{\ell,2}))(X-\chi_{2,p}^\st((t^{(2)}_{\ell,2})^{w_1}))\cdot (X-\chi_{2,p}^\st((t^{(2)}_{\ell,2})^{w_2}))(X-\chi_{2,p}^\st((t^{(2)}_{\ell,2})^{w_1w_2}))=\\
=(X-\alpha_p^3)(X-\alpha_p^2\beta_p)(X-\alpha_p\beta_p^2)(X-\beta_p^3). 
\end{gather*}
In particular the sets of roots of the two sides must coincide.
Since $t^{(2)}_{\ell,2}(t^{(2)}_{\ell,2})^{w_1w_2}=(t^{(2)}_{\ell,2})^{w_1}(t^{(2)}_{\ell,2})^{w_2}$ we have eight possible choices. 
Four choices for the $4$-tuple $\chi_{2,p}^\st(t^{(2)}_{\ell,2}),\chi_{2,p}^\st((t^{(2)}_{\ell,2})^{w_1}),\chi_{2,p}^\st((t^{(2)}_{\ell,2})^{w_2}),\chi_{2,p}^\st((t^{(2)}_{\ell,2})^{w_1w_2})$ are 
\begin{gather*}
(\alpha_p^3,\alpha_p^2\beta_p,\alpha_p\beta_p^2,\beta_p^3), (\alpha_p^3,\alpha_p\beta_p^2,\alpha_p^2\beta_p,\beta_p^3), (\alpha_p^2\beta_p,\alpha_p^3,\beta_p^3,\alpha_p\beta_p^2), (\alpha_p^2\beta_p,\beta_p^3,\alpha_p^3,\alpha_p\beta_p^2).
\end{gather*}
The other four choices are obtained by exchanging $\alpha_p$ with $\beta_p$ in the ones above. 

Since $t^{(2)}_{p,1}=t^{(2)}_{\ell,2}(t^{(2)}_{\ell,2})^{w_1}$ and $t^{(2)}_{p,0}=t^{(2)}_{\ell,2}(t^{(2)}_{\ell,2})^{w_1w_2}$, the displayed $4$-tuples give for $(\chi_{2,p}^\st(t^{(2)}_{\ell,0}),\chi_{2,p}^\st(t^{(2)}_{\ell,1}),\chi_{2,p}^\st(t^{(2)}_{\ell,2}))$ the choices
\begin{gather*}
(\alpha_p^3\beta_p^3,\alpha_p^5\beta_p,\alpha_p^3), (\alpha_p^3\beta_p^3,\alpha_p^4\beta_p^2,\alpha_p^3), (\alpha_p^3\beta_p^3,\alpha_p^5\beta_p,\alpha_p^2\beta_p), (\alpha_p^3\beta_p^3,\alpha_p^2\beta_p^4,\alpha_p^2\beta_p).
\end{gather*}
By writing $\alpha_p=\chi_{1,p}^\st(t^{(1)}_{p,1})$, $\beta_p=\chi_{1,p}^\st((t^{(1)}_{p,1})^w)$ and recalling that $t^{(1)}_{p,0}=t^{(1)}_{p,1}(t^{(1)}_{p,1})^w$, the previous triples take the form
\begin{equation}\label{list}\begin{gathered}
(\chi_{1,p}^\st(t^{(1)}_{p,0})^3,\chi_{1,p}^\st(t^{(1)}_{p,0}(t^{(1)}_{p,1})^4),\chi_{1,p}^\st((t^{(1)}_{p,1})^3)), (\chi_{1,p}^\st(t^{(1)}_{p,0})^3,\chi_{1,p}^\st((t^{(1)}_{p,0})^2(t^{(1)}_{p,1})^2),\chi_{1,p}^\st((t^{(1)}_{p,1})^3)), \\
(\chi_{1,p}^\st(t^{(1)}_{p,0})^3,\chi_{1,p}^\st(t^{(1)}_{p,0}(t^{(1)}_{p,1})^4),\chi_{1,p}^\st(t^{(1)}_{p,0}t^{(1)}_{p,1})), (\chi_{1,p}^\st(t^{(1)}_{p,0})^3,\chi_{1,p}^\st((t^{(1)}_{p,0})^4(t^{(1)}_{p,1})^{-2}),\chi_{1,p}^\st(t^{(1)}_{p,0}t^{(1)}_{p,1})).
\end{gathered}\end{equation}
The triples corresponding to the other four possibilities are obtained by replacing $t^{(1)}_{p,0}$ and $t^{(1)}_{p,1}$ in the triples above by their images via the automorphism $\delta$ of $\calH(T_1(\Q_p),T_1(\Z_p))$ defined by Equation \eqref{autdelta}. 

Let $\lambda_p\colon\calH(T_2(\Q_p),T_2(\Z_p))^-\to\calH(T_1(\Q_p),T_1(\Z_p))$ be a morphism satisfying $\chi^\st_2=(\chi^\st_1)^\ext\ccirc\lambda_p\ccirc\iota_{2,p}^-$ (recall that we leave the maps $\iota_{I_{g,p}}^{T_g}$ implicit). By the arguments of the previous paragraph this happens if and only if the triple $(\lambda_{i,p}(t^{(2)}_{p,0}),\lambda_{i,p}(t^{(2)}_{p,1}),\lambda_{i,p}(t^{(2)}_{p,2}))$ coincides with one of the four listed in \eqref{list} or the four derived from those by applying $\delta$. A simple check shows that these triples correspond to the choices $\lambda_p=\lambda_{i,p}$ for $i\in\{1,2,\ldots,8\}$.
\end{proof}

\begin{rem}
Since all the Hecke actions we consider are for the algebras $\calH_g^N$, $g=1,2$, that are dilating Iwahoric at $p$, we want to know whether the morphisms $\lambda_{i,p}$, $i\in\{1,2,\ldots,8\}$, can be replaced by morphisms $\lambda_{i,p}^-$ of dilating Hecke algebras that satisfy $\chi^\st_{2,p}=\chi^\st_{1,p}\ccirc\lambda_{i,p}^-$. Equivalently, we look for the values of $i$ such that there exists a morphism $\lambda_{i,p}^-\colon\calH(\GSp_4(\Q_p),I_{2,p})^-\to\calH(\GL_2(\Q_p),I_{1,p})^-$ making the following diagram commute:
\begin{center}\label{iotadiag}
\begin{tikzcd}[baseline=(current bounding box.center)]
\calH(\GSp_4(\Q_p),I_{2,p})^-\arrow{d}{\lambda_{i,p}^-}\arrow{r}{\iota^{I_{2,p}}_{T_2}}
&\calH(T_2(\Q_p),T_2(\Z_p))^- \arrow{rd}{\lambda_{i,p}}
&{}\\
\calH(\GL_2(\Q_p),I_{1,p})^- \arrow{r}{\iota^{I_{1,p}}_{T_1}}
&\calH(T_1(\Q_p),T_1(\Z_p))^- \arrow{r}{\iota_{1,p}^-}
&\calH(T_1(\Q_p),T_1(\Z_p)).
\end{tikzcd}
\end{center}
Clearly $\lambda_{i,p}^-$ exists if and only if the image of $\lambda_{i,p}^-$ lies in $\calH(T_1(\Q_p),T_1(\Z_p))^-$. A simple check shows that this is true only for $i\in\{1,2,3\}$.
\end{rem}

\subsubsection{The product morphisms}

Now we can combine the results for the unramified and Iwahori-Hecke algebras.

\begin{defin}\label{heckemorphprod}
Let $i\in\{1,2,3\}$. Let $\lambda_{i,p}^-\colon\calH(\GSp_4(\Q_p),I_{2,p})^-\to\calH(\GL_2(\Q_p),I_{1,p})^-$ be the morphisms making diagram \eqref{iotadiag} commute. Let $\lambda_i\colon\calH^N_2\to\calH^N_1$ be the morphism defined by $\lambda_i=\lambda^{Np}\otimes\lambda_{i,p}^-$.
\end{defin}

Keep the notations as before. Let $\chi_{2}^{\st,i}\colon\calH^N_2\to\Qp$ be the character defined by
\begin{enumerate}
\item $\chi_{2,\ell}^{\st,i}=\chi_{1,\ell}^\st\ccirc\lambda_i$ for every prime $\ell\nmid Np$;
\item $\chi^{\st,i}_{2,p}=(\chi^\st_{1,p}\ccirc\iota_{I_{1,p}}^{T_1})^\ext\ccirc\lambda_{i,p}\ccirc\iota^{I_{2,p}}_{T_2}$.
\end{enumerate}
From Propositions \ref{heckemorphunr} and \ref{heckemorphstab} we deduce the following.

\begin{cor}\label{heckemorphprodcor}
For every $i\in\{1,2,\ldots,8\}$, the form $\Sym^3f$ has a $p$-stabilization $(\Sym^3f)^{\st}_i$ with associated system of Hecke eigenvalues $\chi_{2}^{\st,i}$.
Conversely, if $(\Sym^3f)^\st$ is a $p$-stabilization of $\Sym^3f$ with associated system of Hecke eigenvalues $\chi_2^\st$, then there exists $i\in\{1,2,\ldots,8\}$ such that $\chi_2^\st=\chi_2^{\st,i}$.
\end{cor}

\begin{rem}\label{fourforms}
If $\chi_1^{\st,1}$ and $\chi_1^{\st,2}$ are the systems of Hecke eigenvalues of the two $p$-stabilizations of $f$, then $\chi^\st_1\ccirc\iota_{I_{1,p}}^{T_1}=(\chi^\st_1\ccirc\iota_{I_{1,p}}^{T_1})^\ext\ccirc\delta$, where the superscript $\ext$ denotes extension of characters from $\calH(T_g(\Q_p),T_g(\Z_p))^-$ to $\calH(T_g(\Q_p),T_g(\Z_p))$ and $\delta$ is defined by Equation \ref{autdelta}. For this reason the eight forms $F_i$, $1\le i\le 8$, can be constructed via the four maps $\lambda_i$, $1\le i\le 4$, starting with the two $p$-stabilizations of $f$. It will be useful to think of every $\GL_2$-eigenform of Iwahoric level at $p$ as having four symmetric cube lifts on the $\GSp_4$-eigencurve, rather than of a form of trivial level at $p$ having eight lifts.
\end{rem}

Let $f_\alpha^\st$ be a $p$-stabilization of a classical, cuspidal, non-CM $\GL_2$-eigenform $f$. Let $h$ be the slope of $f$. 
For $i\in\{1,2,3,4\}$, denote by $\Sym^3(f_\alpha^\st)_i$ the $\GSp_4$-eigenform $(\Sym^3f)^\st_i$ given by Corollary \ref{heckemorphprodcor}. Thanks to Remark \ref{fourforms} the forms $(\Sym^3f)^\st_i$ with $5\le i\le 8$ coincide with $(\Sym^3f_\beta^\st)_i$, $1\le i\le 4$, where $f_\beta^\st$ is the $p$-stabilization of $f$ different from $f_\alpha^\st$.

Recall that $U_p^{(1)}=U_{p,1}^{(1)}$ and $U_p^{(2)}=U_{p,1}^{(2)}U_{p,2}^{(2)}$. We defined the slope of a $\GSp_{2g}$-eigenform of Iwahoric level at $p$ as the $p$-adic valuations of the normalized eigenvalue of $U_p^{(g)}$ acting on the form.
Let $k$ and $h$ be the weight and slope, respectively, of $f_\alpha^\st$. 
The following derives from Proposition \ref{heckemorphstab} via some simple calculations. 

\begin{cor}\label{liftslopes}
The slopes of the forms $\Sym^3(f_\alpha^\st)_i$, with $1\le i\le 4$, are: 
\begin{gather*}
\slo(\Sym^3(f_\alpha^\st)_1)=7h, \slo(\Sym^3(f_\alpha^\st)_2)=\slo(\Sym^3(f_\alpha^\st)_3)=k-1+5h, \slo(\Sym^3(f_\alpha^\st)_4)=4(k-1)-h.
\end{gather*}
\end{cor}

If $f^\st$ is a $p$-old $\GL_2$-eigenform of level $\Gamma_1(N)\cap\Gamma_0(p)$, we write $\chi_{2,f^\st}^{i}$ for the system of Hecke eigenvalues of $\Sym^3(f^\st)_i$, $1\le i\le 4$. 
For a $\Qp$-point $x$ of $\cD_2^M$ let $\chi_x\colon\calH_2^N\to\Qp$ be the system of Hecke eigenvalues associated with $x$. 
For $1\le i\le 4$, let $S^{\Sym^3}_i$ be the set of $\Qp$-points $x$ of $\cD_2^M$ defined by the condition
\begin{center}
$x\in S^{\Sym^3}_i$ $\iff$ $\exists$ a $p$-old $\GL_2$-eigenform $f^\st$ of level $\Gamma_1(N)\cap\Gamma_0(p)$ such that $\chi_x=\chi_{2,f^\st}^i$.
\end{center}
Then we have the following.

\begin{cor}\label{stabdense}
If $i\neq 1$ then the set $S^{\Sym^3}_i$ is discrete in $\cD_2^M$.
\end{cor}

\begin{proof}
Let $i\in\{1,2,3,4\}$. 
Let $D_2$ be an affinoid domain on $\cD_2^M$. 
For $x\in D_2(\C_p)\cap S^{\Sym^3}_i$, there exists a $\GL_2$-eigenform $f$ of level $\Gamma_1(N)$ such that $\chi_x=\chi_{2,f}^{\st,i}$. Let $f^\st$ be a $p$-stabilization of $f$ and let $x_f^\st$ be the corresponding point on $\cD_1^N$. Since the systems of Hecke eigenvalues vary analytically on $\cD_2^M$ and $\cD_1^N$, the set $\{x_f^\st\,\vert\,x\in D_2(\C_p)\cap S^{\Sym^3}_i\}$ must be contained in $D_1(\C_p)$ for an affinoid domain $D_1$ on $\cD_1^N$. For $j=1,2$, Remark \ref{affslope} gives that the slope $v_p(\psi_j(U_p^{(j)}))$ is bounded on $D_j$ by a constant $c_j$. By imposing that $h\leq c_1$ and $\slo(f_\alpha^\st)_i)\leq c_2$ in Corollary \ref{liftslopes}, we obtain an upper bound for $k$ if $i\neq 1$. Since there is only a finite number of classical $\GL_2$-eigenforms of given weight and level, the set $S_{\Sym^3,i}\cap D_2(\C_p)$ is finite if $i\neq 1$.
\end{proof}

\begin{rem}\label{interpslopes}
As a consequence of Corollary \ref{stabdense} the only symmetric cube lifts that we can hope to interpolate $p$-adically are those in the set $S^{\Sym^3}_1$. We will prove in Section \ref{sym3locus} that the Zariski closure of this set is a $1$-dimensional subvariety of $\cD_2^M$. 
\end{rem}

\subsubsection{An assumption on the residual Galois representation}\label{cond3twist}

Let $\ovl\rho\colon G_\Q\to\GL_2(\ovl\F_p)$ be a representation. Let $\cD_{1,\ovl\rho}^N$ be the union of the connected components of $\cD_1^N$ having $\ovl\rho$ as associated residual representation. From now on we replace $\cD_1^N$ by a subspace of the form $\cD_{1,\ovl\rho}^N$ for some $\ovl\rho$; we do it implicitly, so we still write $\cD_1^N$ for $\cD_{1,\ovl\rho}^N$. 
The only purpose of this choice is to assure that the symmetric cube morphism of eigenvarieties we construct in Section \ref{morpheigen} is a closed immersion; this will be a consequence of Lemma \ref{no3twist} below. 

There is a map $\Sym^3_1$ from the set of classical, non-CM, p-old points of $\cD_{1,\rho}^N$ to the set $S_1^{\Sym^3}$ of Corollary \ref{stabdense}; it maps a point $x$ corresponding to an eigenform $f$ to the point of $S_1^{\Sym^3}$ corresponding to $(\Sym^3f_x)^\st_1$. 

\begin{lemma}\label{no3twist}
The map $\Sym^3_1$ is injective.
\end{lemma}

\begin{proof}
If $x_1$ and $x_2$ are two points of $\cD_{1,\ovl\rho}^N$ satisfying $\Sym^3_1(x_1)=\Sym^3_1(x_2)$, then $\Sym^3\rho_{x_1}\cong\Sym^3\rho_{x_2}$. This implies that $\rho_{x_1}\cong\rho_{x_2}\otimes\chi$ for a character $\chi\colon G_\Q\to\Qp^\times$ of order $3$. Since $\ovl\rho_1=\ovl\rho_2=\ovl\rho$ and $p>3$, the character $\chi$ is trivial and $\rho_{x_1}\cong\rho_{x_2}$. We deduce that $x_1$ and $x_2$ are two $p$-stabilizations of the same form of trivial level at $p$. If they are distinct then $\Sym^3_1(x_1)\neq\Sym^3_1(x_2)$ by construction, a contradiction.
We conclude that $x_1=x_2$.
\end{proof}


\bigskip

\section{Morphisms of BC-eigenvarieties}\label{BCsec}

We recall Bella\"\i che and Chenevier's definition of eigenvarieties and some of their results, following \cite[Section 7.2.3]{bellchen}. We refer to their eigenvarieties as \emph{BC-eigenvarieties}, in order to distinguish this notion from the definition of eigenvariety we gave in Section \ref{eigenmac} (a product of Buzzard's eigenvariety machine). We will use these results to interpolate the classical symmetric cube lifts given by Corollary \ref{formtransf} into a morphism of eigenvarieties. We remark that Ludwig \cite{ludwigll,ludwigunit} also relies on the results of \cite[Section 7.2.3]{bellchen}. We think that our approach may be more systematic.

As usual fix a prime $p\geq 5$. We call ``BC-datum'' a $4$-tuple $(g,\calH,\eta,\cS^\cl)$ where:
\begin{itemize}[label={--}]
\item $g$ is a positive integer;
\item $\calH$ is a commutative ring;
\item $\eta$ is a distinguished element of $\calH$;
\item $\cS^\cl$ is a subset of $\Hom(\calH,\Qp)\times\Z^g$. 
\end{itemize}
The superscript ``$\cl$'' stands for ``classical''. In our applications $\calH$ will be a Hecke algebra and $\cS^\cl$ will be a set of couples $(\psi,\uk)$ each consisting of the system of eigenvalues $\psi$ and the weight $\uk$ of a classical eigenform. In the proposition below $\cW^\circ_g$ is the connected component of unity in the $g$-dimensional weight space. Recall that we identify $\Z^g$ with the set of classical weights in $\cW_G$. 
Also recall that for an extension $L$ of $\Qp$ and an $L$-point $x$ of a rigid analytic space $X$ we denote by $\ev_x\colon\cO(X)\to L$ the evaluation morphism at $x$.

\begin{defin}\label{BCdef}\cite[Definition 7.2.5]{bellchen}
A \emph{BC-eigenvariety} for the datum $(g,\calH,\eta,\cS^\cl)$ is a $4$-tuple $(\cD,\psi,w,S^\cl)$ consisting of  
\begin{itemize}[label={--}]
\item a reduced rigid analytic space $\cD$ over $\Q_p$,
\item a ring morphism $\psi\colon\calH\to\cO(\cD)$ such that $\psi(\eta)$ is invertible, 
\item a morphism $w\colon\cD\to\cW^\circ_g$ of rigid analytic spaces over $\Q_p$,
\item an accumulation and Zariski-dense subset $S^\cl\subset\cD(\Qp)$ such that $w(S^\cl)\subset\Z^g$, 
\end{itemize}
satisfying the following conditions:
\begin{enumerate}
\item the map 
\begin{equation}\label{defnu} \widetilde{\nu}=(w,\psi(\eta)^{-1})\colon\cD\to\cW^\circ_g\times\G_m \end{equation}
induces a finite morphism $\cD\to\widetilde{\nu}(\cD)$;
\item there exists an admissible affinoid covering $\cC$ of $\widetilde{\nu}(\cD)$ such that, for every $V\in\cC$, the map
\[ \psi\otimes\widetilde{\nu}^\ast\colon\calH\otimes_\Z\cO(V)\to\cO(\widetilde{\nu}^{-1}(V)) \] 
is surjective; 
\item the evaluation map
\begin{equation}\begin{gathered}
\widetilde{\ev}\colon S^\cl\to\Hom(\calH,\Qp)\times\Z^g, \\
x\mapsto(\psi_x,w(x)), 
\end{gathered}\end{equation}
where $\psi_x=\ev_x\ccirc\psi$, induces a bijection $S^\cl\to\cS^\cl$.
\end{enumerate}
\end{defin}

We often refer to $\cD$ as the BC-eigenvariety for the given BC-datum and leave the other elements of the BC-eigenvariety implicit. 

%

If a BC-eigenvariety for the given BC-datum exists then it is unique in the sense of the proposition below. 

\begin{prop}\label{BCunique}\cite[Proposition 7.2.8]{bellchen}
Let $(\cD_1,\psi_1,w_1,S_1^\cl)$ and $(\cD_2,\psi_2,w_2,S_2^\cl)$ be two BC-eigenvarieties for the same BC-datum $(g,\calH,\eta,\cS^\cl)$. Then there is a unique isomorphism $\zeta\colon\cD_1\to\cD_2$ of rigid analytic spaces over $\Q_p$ such that $\psi_1=\zeta^\ast\ccirc\psi_2$, $w_1=w_2\ccirc\zeta$ and $\zeta(S_1^\cl)=S_2^\cl$.
\end{prop}

In the previous sections we defined various rigid analytic spaces via Buzzard's eigenvariety machine. 
We check that these spaces are BC-eigenvarieties for a suitable choice of BC-datum. As a first step we prove the lemma below. Consider an eigenvariety datum $(\cW^\circ,\calH,(M(A,w))_{A,w},(\phi_{A,w})_{A,w},\eta)$ and let $(\cD,\psi,w)$ be the eigenvariety produced from this datum by Theorem \ref{theigenmac}. 

\begin{lemma}\label{BCcond}
The triple $(\cD,\psi,w)$ satisfies conditions (1) and (2) of Definition \ref{BCdef}. 
\end{lemma}

\begin{proof}
We refer to Buzzard's construction (see \cite[Sections 4-5]{buzzard}). 
Let $\cZ$ be the spectral variety for the given datum. Let $\widetilde{\nu}$ be the map defined by Equation \eqref{defnu}. By construction of $\cD$ we have $\widetilde{\nu}(\cD)=\cZ$ and the map $\widetilde{\nu}\colon\cD\to\cZ$ is finite, so condition (1) of Definition \ref{BCdef} holds. 

Let $\cC$ be the admissible affinoid covering of $\cZ$ defined in \cite[Section 4]{buzzard}. For $V\in\cC$ let $A=\Spm R=w_\cZ(V)$ be its image in $\cW^\circ$. Let $w\in\Q$ be sufficiently large, so that the module $M(A,w)$ is defined. Let $M(A,w)=N\oplus F$ be the decomposition associated with $V$ by Riesz theory, following the discussion and the notations in \cite[Section 5]{buzzard}. Then $\cO(\widetilde{\nu}^{-1}(V))$ is the $R$-span of the image of $\calH$ in $\End_{R,\cont}N$. Since $\cO(V)$ is an $R$-module, the map $\psi\colon\calH\otimes\cO(V)\to\cO(\widetilde{\nu}^{-1}V)$ is surjective, hence condition (2) is also satisfied.
\end{proof}

Suppose that there exists an accumulation and Zariski-dense subset $S^\cl$ of $\cD$ such that the set 
\[ \cS^\cl=\{(\psi_x,w(x))\,\vert\,x\in S^\cl\} \]
is contained in $\Hom(\calH,\Qp)\times\Z^g$.
Then $(\cD,\psi,w,S^\cl)$ clearly satisfies condition (3) of Definition \ref{BCdef} with respect to the set $\cS^\cl$, hence the following. 

\begin{cor}\label{buzzBC}
The $4$-tuple $(\cD,\psi,w,S^\cl)$ is a BC-eigenvariety for the datum $(g,\calH,\eta,\cS^\cl)$.
\end{cor}

\subsection{Changing the BC-datum}\label{changeBC}

Let $(\cD,\psi,w,S^\cl)$ be a BC-eigenvariety for the datum $(g,\calH,\eta,\cS^\cl)$. Let $S_0^\cl$ be an accumulation subset of $S^\cl$ and let $\cD_0$ be the Zariski closure of $S_0^\cl$ in $\cD$. 
Let $\cS_0^\cl$ be the image of $S_0^\cl$ via the bijection $S^\cl\to\cS^\cl$. Let $\psi_0\colon\calH\to\cO(\cD_0)$ be the composition of $\psi\colon\calH\to\cO(\cD)$ with the restriction $\cO(\cD)\to\cO(\cD_0)$. Let $w_0=w\vert_{\cD_0}$.

\begin{lemma}\label{subBC}
The $4$-tuple $(\cD_0,\psi_0,w_0,S_0^\cl)$ is a BC-eigenvariety for the datum $(g,\calH,\eta,\cS_0^\cl)$.
\end{lemma}

\begin{proof}
We check that the conditions of Definition \ref{BCdef} are satisfied by $(\cD_0,\psi_0,w_0,S_0^\cl)$, knowing that they are satisfied by $(\cD,\psi,w,S^\cl)$.  Let $\widetilde{\nu}=(w,\psi(\eta)^{-1})\colon\cD\to\cW^\circ\times\G_m$ and let $\cZ=\widetilde{\nu}(\cD)$. Let $\cZ_0=\widetilde{\nu}(\cD_0)$. 
Since $\widetilde{\nu}\colon\cD\to\cZ$ is finite and $\cD_0$ is Zariski-closed in $\cD$, the map $\widetilde{\nu}\vert_{\cD_0}\colon\cD_0\to\cZ_0$ is also finite, hence (1) holds.

Consider an admissible covering $\cC$ of $\cZ$ satisfying condition (2). Then $\{V\cap\cZ_0\}_{V\in\cC}$ is an admissible covering of $\cZ_0$. Let $V\in\cC$ and $V_0=V\cap\cZ_0$. Consider the diagram
\begin{center}
\begin{tikzcd}[baseline=(current bounding box.center)]
\calH\otimes\cO(V)\arrow{r}{\psi\otimes\widetilde{\nu}^\ast}\arrow{d}
&\cO(\widetilde{\nu}^{-1}(V))\arrow{d}\\
\calH\otimes\cO(V_0)\arrow{r}{\psi_0\otimes\widetilde{\nu}_0^\ast}
&\cO(\widetilde{\nu}_0^{-1}(V_0))
\end{tikzcd}
\end{center}
The horizontal arrows are given by the restriction of analytic functions. Since the left vertical arrow is surjective, the right one is also surjective, giving (2). 

By definition of $S_0^\cl$ the map $\widetilde{\ev}$ induces a bijection $S_0^\cl\to\cS_0^\cl$, so (3) is also true.
%
%
\end{proof}

We prove some relations between BC-eigenvarieties associated with different BC-data.

\begin{lemma}\label{changeweight}
Let $g_1$ and $g_2$ be two positive integers with $g_1\leq g_2$. Let $\Theta\colon\cW^\circ_{g_1}\to\cW^\circ_{g_2}$ be an immersion of rigid analytic spaces that maps classical points of $\cW^\circ_{g_1}$ to classical points of $\cW^\circ_{g_2}$. Let $(g_1,\calH,\eta,\cS_1^\cl)$ and $(g_2,\calH,\eta,\cS_2^\cl)$ be two BC-data satisfying
\[ \{(\psi,\Theta(\uk))\in\Hom(\calH,\Qp)\times\Z^{g_2}\,\vert\,(\psi,\uk)\in\cS_1^\cl\}\subset\cS_2^\cl \]
Let $(\cD_1,\psi_1,w_1,S_1^\cl)$ and $(\cD_2,\psi_2,w_2,S_2^\cl)$ be the BC-eigenvarieties for the two data. Then there exists a closed immersion of rigid analytic spaces $\xi_\Theta\colon\cD_1\to\cD_2$ such that $\psi_{1}=\xi_\Theta^\ast\ccirc\psi_2$, $w_{1}=w_{2}\ccirc\xi_\Theta$ and $\xi_\Theta(S_1^\cl)\subset S_2^\cl$.
\end{lemma}

\begin{proof}
Let $\cD_1^\Theta=\cD_2\times_{\cW^\circ_{g_2}}\cW^\circ_{g_1}$, where the map $\cW^\circ_{g_1}\to\cW^\circ_{g_1}$ is $\Theta$. Let $\zeta^\Theta\colon\cD_1^\Theta\to\cD_2$ and $w_1^\Theta\colon\cD_1^\Theta\to\cW^\circ_{g_1}$ be the natural maps fitting into the cartesian diagram
\begin{center}
\begin{tikzcd}[baseline=(current bounding box.center)]
\cD_1^\Theta \arrow[hook]{r}{\zeta^\Theta}\arrow{d}{w_1^\Theta}
&\cD_2\arrow{d}{w_2}\\
\cW^\circ_{g_1} \arrow[hook]{r}{\Theta}
&\cW^\circ_{g_2}
\end{tikzcd}
\end{center}
Then $\zeta^\Theta$ induces a ring morphism $\zeta^{\Theta,\ast}\colon\cO(\cD_2)\to\cO(\cD_1^\Theta)$. Let $\psi_1^\Theta=\zeta^{\Theta,\ast}\ccirc\psi_2$. 
Note that $\zeta^\Theta$ is a closed immersion. 

Let $\cS_1^\Theta=\{(\psi,\uk)\in\Hom(\calH,\Qp)\times\Z^{g_1}\,\vert\,(\psi,\Theta(\uk))\in\cS_2^\cl\}$. 
Then the $4$-tuple $(\cD_1^\Theta,\zeta_\Theta^\ast\ccirc\psi_2,w_1^\Theta,\zeta_\Theta^{-1}(S_2^\cl))$ is a BC-eigenvariety for the datum $(g_1,\calH,\eta,\cS_1^\Theta)$. By assumption $\cS_1^\cl\subset\cS_1^\Theta$. Consider the Zariski-closure $\cD_1^\prime$ of $\widetilde{\ev}^{-1}(\cS_1^\cl)$ in $\cD_1^\Theta$. Let $\iota^\prime\colon\cD_1^\prime\to\cD_1^\Theta$ be the natural closed immersion and let $w_1^\prime=w_1^\Theta\vert_{\cD_1^\prime}$, $\psi_1^\prime=(\iota^\prime)^\ast\ccirc\psi_1^\Theta$. By Lemma \ref{subBC} the $4$-tuple $(\cD_1^\prime,\psi_1^\prime,w_1^\prime,\widetilde{\ev}^{-1}(\cS_1^\cl))$ is a BC-eigenvariety for the BC-datum $(g_1,\calH,\eta,\cS_1^\cl)$. Since $(\cD_1,\psi_1,w_1,\cS_1^\cl)$ is a BC-eigenvariety for the same datum, Proposition \ref{BCunique} gives an isomorphism of rigid analytic spaces $\zeta\colon\cD_1\to\cD_1^\prime$ compatible with all the extra structures. 
The composition $\xi_\Theta=\zeta_\Theta\ccirc\iota^\prime\ccirc\zeta\colon\cD_1\to\cD_2$ is a closed immersion with the desired properties.
\end{proof}

Let $(g,\calH,\eta_1,\cS^\cl)$ and $(g,\calH,\eta_2,\cS^\cl)$ be two BC-data that differ only by the choice of the distinguished elements of $\calH$. Let $(\cD_1,\psi_1,w_1,S_1^\cl)$ and $(\cD_2,\psi_2,w_2,S_2^\cl)$ be BC-eigenvarieties for the two data. 
We say that condition (Fin) is satisfied if the following holds:

\smallskip

\noindent (Fin) the map
\begin{gather*} 
\widetilde{\nu}_{1,2}\colon\cD_1\to\cW^\circ_g\times\G_m, \\
x\mapsto (w_1(x),\ev_x\ccirc\psi_1(\eta_2)^{-1})
\end{gather*}
induces a finite morphism $\cD_1\to\widetilde{\nu}_{1,2}(\cD_1)$.

\smallskip

\begin{lemma}\label{changecpt}
Under assumption \textnormal{(Fin)}, there exists an isomorphism of rigid analytic spaces $\xi_\eta\colon\cD_1\to\cD_2$ such that $\psi_{1}=\xi_\eta^\ast\ccirc\psi_2$, $w_{1}=w_{2}\ccirc\xi_\eta$ and $\xi_2(S_1^\cl)=S_2^\cl$.
\end{lemma}

\begin{proof}
We check that the $4$-tuple $(\cD_1,\psi_1,w_1,S_1^\cl)$ is a BC-eigenvariety for the datum $(g,\calH,\eta_2,\cS^\cl)$. All properties of Definition \ref{BCdef} except (1) are satisfied because $(\cD_1,\psi_1,w_1,S_1^\cl)$ is a BC-eigenvariety for the datum $(g,\calH,\eta_1,\cS^\cl)$. Property (1) is satisfied thanks to hypothesis (Fin). Then $(\cD_1,\psi_1,w_1,S_1^\cl)$ and $(\cD_2,\psi_2,w_2,S_2^\cl)$ are BC-eigenvarieties for the same datum, and Proposition \ref{BCunique} gives an isomorphism of rigid analytic spaces $\cD_1\to\cD_2$ with the desired properties.
\end{proof}

\begin{lemma}\label{changealg}
Let $\calH_1$ and $\calH_2$ be two commutative rings and let $\lambda\colon\calH_2\to\calH_1$ be a ring morphism. Let $(g,\calH_1,\eta_1,\cS_1^\cl)$ and $(g,\calH_2,\eta_2,\cS_2^\cl)$ be two BC-data that satisfy $\eta_1=\lambda(\eta_2)$ and
\begin{equation}\label{S1S2} \cS_1^\cl=\{(\psi\ccirc\lambda,\uk)\,\vert\,(\psi,\uk)\in\cS_2^\cl\}. \end{equation}
Let $(\cD_1,\psi_1,w_1,S_1^\cl)$ and $(\cD_2,\psi_2,w_2,S_2^\cl)$ be BC-eigenvarieties for the two data. Suppose that the map $\cS_2^\cl\to\cS_1^\cl$ defined by $(\psi,\uk)\mapsto(\psi\ccirc\lambda,\uk)$ is a bijection. Then there exists an isomorphism of rigid analytic spaces $\xi_\lambda\colon\cD_1\to\cD_2$ such that $\psi_{1}\ccirc\lambda=\xi_\lambda^\ast\ccirc\psi_2$, $w_{1}=w_{2}\ccirc\xi_\lambda$ and $\xi_\lambda(S_1^\cl)=S_2^\cl$.
\end{lemma}

\begin{proof}
Consider the $4$-tuple $(\cD_1,\psi_1\ccirc\lambda,w_1,S_1^\cl)$. We show that it defines a BC-eigenvariety for the datum $(g,\calH_2,\eta_2,\cS_2^\cl)$. Property (1) of Definition \ref{BCdef} is satisfied since $\psi_1\ccirc\lambda(\eta_2)=\psi_1(\eta_1)$ and the map $(w,\psi_1(\eta_1)^{-1})$ is finite by property (1) relative to the datum $(g,\calH_1,\eta_1,\cS_1^\cl)$. Property (2) is a consequence of equality \eqref{S1S2} together with the fact that $S_1^\cl$ is Zariski-dense in $\cD_1$. Property (3) follows immediately from equality \eqref{S1S2}.

Now the $4$-tuples $(\cD_1,\psi_1\ccirc\lambda,w_1,S_1^\cl)$ and $(\cD_2,\psi_2,w_2,S_2^\cl)$ define two BC-eigenvarieties for the datum $(g,\calH_2,\eta_2,\cS_2^\cl)$, so Proposition \ref{BCunique} gives a morphism $\xi_\lambda\colon\cD_1\to\cD_2$ of rigid analytic spaces such that $\psi_1\ccirc\lambda=\xi_\lambda^\ast\ccirc\psi_2$, $w_1=w_2\ccirc\xi_\lambda$ and $\xi_\lambda(S_1^\cl)=S_2^\cl$, as desired.
\end{proof}

\bigskip

\section{The symmetric cube morphism of eigenvarieties}

Fix a prime $p$ and an integer $N\geq 1$ prime to $p$. Let $M$ be the integer given as a function of $N$ by Definition \ref{sym3leveldef}. Set $\lambda=\lambda_1$, where $\lambda_1\colon\calH_2^N\to\calH_1^N$ is the morphism given by Definition \ref{heckemorphprod}.

We work from now on with the curves $\cD_1^N\times_{\cW_1}\cW_1^\circ$ and $\cD_2^M\times_{\cW_2}\cW_2^\circ$. We still denote them by $\cD_1^N$ and $\cD_2^M$ in order not to complicate notations. Our aim is to construct a closed immersion $\cD_1^N\to\cD_2^{M}$ interpolating the map defined by the symmetric cube transfer on the classical points. As in \cite{ludwigll} we define two auxiliary eigenvarieties.


\subsection{The first auxiliary eigenvariety}\label{aux1}

Recall that for every affinoid subdomain $A=\Spm R$ of $\cW_1$ and for every sufficiently large rational number $w$ there is a Banach $R$-module $M_1(A,w)$ of $w$-overconvergent modular forms of weight $\kappa_A$ and level $N$, carrying an action $\phi^1_{A,w}\colon\calH^N_1\to\End_{R,\cont}M_1(A,w)$. We let $\calH^N_2$ act on $M_1(A,w)$ through the map 
\[ \phi^{1,\aux}_{A,w}=\phi^1_{A,w}\ccirc\lambda\colon\calH^N_2\to\End_{R,\cont}M_1(A,w). \]
We have $\phi^{1,\aux}_{A,w}(U_p^{(2)})=\phi^{1,\aux}_{A,w}(U^{(2)}_{p,1}U^{(2)}_{p,2})=\phi^1_{A,w}(\lambda(U^{(2)}_{p,1}U^{(2)}_{p,2}))=\phi^1_{A,w}(U^{(1)}_{p,0}(U^{(1)}_{p,1})^7)$. 
This operator is compact on $M_1(A,w)$ since it is the composition of the compact operator $\phi_{A,w}^{1,\aux}(U^{(1)}_{p,1})$ with a continuous operator.

\begin{defin}\label{Dlambda}
Let $(\cD_{1,\lambda}^N,\psi_{1,\lambda},w_{1,\lambda})$ be the eigenvariety associated with the datum
\[ (\cW_1^\circ,\calH_2^N,(M_1(A,w))_{A,w},(\phi^{1,\aux}_{A,w})_{A,w},U^{(2)}_p) \]
by the eigenvariety machine.
\end{defin}

Since $\cW_1^\circ$ is equidimensional of dimension $1$, the eigenvariety $\cD_{1,\lambda}^N$ is also equidimensional of dimension $1$.

We denote by $S_1^\cl$ the set of classical points of $\cD_1^N$ and by $S_1^{\cl,\cG}$ the set of classical non-CM points of $\cD_1^N$. 
Recall that we defined a non-CM eigencurve $\cD_1^{N,\cG}$ as the Zariski-closure of $S_1^{\cl,\cG}$. By Remark \ref{nonCMdense} the set $S_1^{\cl,\cG}$ is an accumulation subset of $\cD_{1}^{N,\cG}$ and the weight map $w_1^\cG\colon\cD_{1}^{N,\cG}\to\cW_1^\circ$ is surjective.

We define two subsets of $\cD_{1,\lambda}^N$ by
\begin{gather*} 
S_{1,\lambda}^\cl=\{x\in\cD_{1,\lambda}^N\,\vert\,\psi_x=\chi_f\ccirc\lambda\textrm{ for a classical, $p$-old }\GL_2\textrm{ eigenform }f \}, \\ 
S_{1,\aux}^\cl=\{x\in\cD_{1,\lambda}^N\,\vert\,\psi_x=\chi_f\ccirc\lambda\textrm{ for a classical, $p$-old, non-CM }\GL_2\textrm{ eigenform }f \}.
\end{gather*}

\begin{defin}
Let $\cD_{1,\aux}^N$ be the Zariski-closure of the set $S_{1,\aux}^\cl$ in $\cD_{1,\lambda}^N$.
\end{defin}

We denote by $\psi_{1,\aux}\colon\calH_2^N\to\cO(\cD_{1,\aux}^N)$ and $w_{1,\aux}\colon\cD_{1,\aux}^N\to\cW_1^\circ$ the morphisms obtained from the corresponding morphisms for $\cD_{1,\lambda}^N$.

\subsection{The second auxiliary eigenvariety}\label{aux2}

We identify $\cW_1^\circ$ with $B_1(0,1^-)$ and $\cW_2^\circ$ with $B_2(0,1^-)$ via the isomorphisms $\eta_1$ and $\eta_2$ of Section \ref{weightsp}. 
This way we obtain coordinates $T$ on $\cW_1^\circ$ and $(T_1,T_2)$ on $\cW_2^\circ$. 

Let $k\geq 2$ be an integer. Let $f$ be a cuspidal $\GL_2$-eigenform of weight $k$ and level $\Gamma_1(N)$ and let $f^\st$ be a $p$-stabilization of $f$. Let $F=(\Sym^3f)^{\st}_i$ be one of the $p$-stabilizations of $\Sym^3f$ defined in Corollary \ref{heckemorphprodcor}. By Corollary \ref{classtransf} $(\Sym^3f)^{\st}_i$ has weight $(2k-1,k+1)$. In particular $f^\st$ defines a point of the fibre of $\cD_1^N$ at $T=u^k-1$, and $(\Sym^3f)^{\st}_i$ defines a point of the fibre of $\cD_2^{M}$ at $(T_1,T_2)=(u^{2k-1}-1,u^{k+1}-1)$.
%
%
The map $u^k-1\mapsto (u^{2k-1}-1,u^{k+1}-1)$ is interpolated by the morphism of rigid analytic spaces
\begin{gather*}
\iota\colon\cW_1^\circ\into\cW_2^\circ, \\
T\mapsto (u^{-1}(1+T)^2-1,u(1+T)-1).
\end{gather*}
The map $\iota$ induces an isomorphism of $\cW_1^\circ$ onto its image, that is the rigid analytic curve in $\cW_2^\circ$ defined by the equation $u^{-3}(1+T_2)^2-(1+T_1)=0$. 
By construction $\iota$ induces a bijection between the classical weights of $\cW_1^\circ$ and the classical weights of $\cW_2^\circ$ belonging to $\iota(\cW_1^\circ)$. Since the classical weights form an accumulation and Zariski-dense subset of $\cW_1^\circ$, they also form an accumulation and Zariski-dense subset of $\iota(\cW_1^\circ)$.

After Corollary \ref{liftslopes} we defined for $i\in\{1,2,3,4\}$ a set $S_i^{\Sym^3}\subset\cD_2^M(\Qp)$. By construction of $\iota$, for every $i$ the weight of every point in $S_i^{\Sym^3}$ is a classical weight belonging to $\iota(\cW_1^\circ)$. 
Since $\iota(\cW_1^\circ)$ is a one-dimensional Zariski-closed subvariety of $\cW_2^\circ$, the image of the Zariski-closure in $\cD_2^M$ of $S_i^{\Sym^3}$ under the weight map is contained in $\iota(\cW_1^\circ)$. By Remark \ref{interpslopes} the set $S_i^{\Sym^3}$ is discrete in $\cD_2^M(\Qp)$ if $i\geq 2$, so the only interesting Zariski-closure is that of $S_1^{\Sym^3}$. 

\begin{defin}\label{Daux2}
Let $\cD_{2,\aux}^{M}$ be the Zariski closure of $S_1^{\Sym^3}$ in $\cD_2^{M}$ and let $\iota_{2,\aux}\colon\cD_{2,\aux}^{M}\to\cD_2^M$ be the natural closed immersion. Define $w_{2,\aux}\colon\cD_{2,\aux}^{M}\to\cW_1^\circ$ and $\psi_{2,\aux}\colon\calH_2^N\to\cO(\cD_{2,\aux}^M)$ as $w_{2,\aux}=\iota^{-1}\ccirc w_2\vert_{\cD_{2,\aux}^{M}}$ and $\psi_{2,\aux}=\iota_{2,\aux}^\ast\ccirc\psi_2$. 
\end{defin}

\subsection{Constructing the symmetric cube morphism}\label{morpheigen}

We construct morphisms of rigid analytic spaces
\begin{gather*}
\xi_1\colon\cD_1^{N,\cG}\to\cD_{1,\aux}^N, \quad \xi_2\colon\cD_{1,\aux}^N\to\cD_{2,\aux}^{M}, \quad \xi_3\colon\cD_{2,\aux}^{M}\to\cD_2^M
\end{gather*}
making the following diagrams commute:
\begin{equation}\label{xi123}\begin{gathered}
\begin{tikzcd}
\cD_1^{N,\cG} \arrow{r}{\xi_1}\arrow{d}
&\cD_{1,\aux}^N \arrow{r}{\xi_2}\arrow{d}
&\cD_{2,\aux}^{M} \arrow{d}\arrow{r}{\xi_3}
&\cD_{2}^{M} \arrow{d}\\
\cW_1^\circ \arrow{r}{=}
&\cW_1^\circ \arrow{r}{=}
&\cW_1^\circ \arrow{r}{\iota}
&\cW_2^\circ
\end{tikzcd}
\\
\begin{tikzcd}
\calH_2^N \arrow{r}{=}\arrow{d}
&\calH_2^N \arrow{r}{=}\arrow{d}
&\calH_2^N \arrow{r}{\lambda}\arrow{d}
&\calH_1^N \arrow{d}\\
\cO(\cD_2^{M})\arrow{r}{\xi_3^\ast}
&\cO(\cD_{2,\aux}^{M})\arrow{r}{\xi_2^\ast}
&\cO(\cD_{1,\aux}^N) \arrow{r}{\xi_1^\ast}
&\cO(\cD_1^{N,\cG})
\end{tikzcd}
\end{gathered}
\end{equation}

In order to construct $\xi_1$, $\xi_2$ and $\xi_3$ we interpret the eigenvarieties appearing in the diagrams as BC-eigenvarieties for suitably chosen BC-data and we rely on the results of Section \ref{BCsec}.
%

We define two subsets $\cS_{1}^\cl$ and $\cS_{1}^{\cl,\cG}$ of $\Hom(\calH_1^{N},\Qp)\times\Z$ by
\begin{gather*}
\cS_{1}^\cl=\{(\psi,k)\in\Hom(\calH_1^{N},\Qp)\times\Z\,\vert\,\psi=\chi_f \\
\textrm{ for a cuspidal, classical, $p$-old }\GL_2\textrm{-eigenform }f\textrm{ of weight }k\}, \\
\cS_{1}^{\cl,\cG}=\{(\psi,k)\in\Hom(\calH_1^{N},\Qp)\times\Z\,\vert\,\psi=\chi_f \\
\textrm{ for a cuspidal, classical, $p$-old, non-CM } \GL_2\textrm{-eigenform }f\textrm{ of weight }k\}.
\end{gather*}
We define two subsets $\cS_{1,\lambda}$ and $\cS_{1,\aux}$ of $\Hom(\calH_2^{N},\Qp)\times\Z$ by
\begin{gather*}
\cS_{1,\lambda}^\cl=\{(\psi,k)\in\Hom(\calH_2^{N},\Qp)\times\Z\,\vert\,\psi=\chi_f\ccirc\lambda \\
\textrm{ for a cuspidal, classical, $p$-old }\GL_2\textrm{-eigenform }f\textrm{ of weight }k\}, \\
\cS_{1,\aux}^\cl=\{(\psi,k)\in\Hom(\calH_2^{N},\Qp)\times\Z\,\vert\,\psi=\chi_f\ccirc\lambda \\
\textrm{ for a cuspidal, classical, $p$-old, non-CM }\GL_2\textrm{-eigenform }f\textrm{ of weight }k\}.
\end{gather*}

\begin{lemma}\label{eigenBClambda}\mbox{ }
\begin{enumerate}
\item The $4$-tuple $(\cD_{1}^N,\psi_1,w_1,S_1^\cl)$ is a BC-eigenvariety for the datum $(1,\calH_1^N,U_p^{(1)},\cS_1^\cl)$.
\item The $4$-tuple $(\cD_{1}^N,\psi_1,w_1,S_1^\cl)$ is a BC-eigenvariety for the datum $(1,\calH_1^N,\lambda(U_p^{(2)}),\cS_1^\cl)$.
\item The $4$-tuple $(\cD_{1}^{N,\cG},\psi_1,w_1,S_1^{\cl,\cG})$ is a BC-eigenvariety for the datum $(1,\calH_1^N,\lambda(U_p^{(2)}),\cS_{1}^{\cl,\cG})$.
\item The $4$-tuple $(\cD_{1,\lambda}^N,\psi_{1,\lambda},w_{1,\lambda},S^\cl_{1,\lambda})$ is a BC-eigenvariety for the datum $(1,\calH_2^N,U_p^{(2)},\cS_{1,\lambda}^\cl)$.
\item The $4$-tuple $(\cD_{1,\aux}^N,\psi_{1,\aux},w_{1,\aux},S^\cl_{1,\aux})$ is a BC-eigenvariety for the datum $(1,\calH_2^N,U_p^{(2)},\cS_{1,\aux}^\cl)$.
\end{enumerate}
\end{lemma}

\begin{proof}
Part (1) follows from Lemma \ref{buzzBC}.

For part (2), observe that the couple $(w_1,\psi_1(U_p^{(1)}))$ satisfies condition (Fin) since $\lambda(U_p^{(2)})=U_{p,0}^{(1)}(U_{p,1}^{(1)})^7$. Hence Lemma \ref{changecpt} gives an isomorphism between the eigenvarieties for the data $(1,\calH_1^N,\lambda(U_p^{(2)}),\cS_1^\cl)$ and $(1,\calH_1^N,U_p^{(1)},\cS_1^\cl)$, as desired.

We prove part (3). Let $\widetilde{\ev}\colon S_1^\cl\to\cS_1^\cl$ be the evaluation map given in property (3) of Definition \ref{BCdef}. By definition the eigenvariety $\cD_1^{N,\cG}$ is the Zariski-closure in $\cD_1^N$ of the set $S_1^{\cl,\cG}$. The image of $S_1^{\cl,\cG}$ in $\cS_1^\cl$ via $\widetilde{\ev}$ is $\cS_1^{\cl,\cG}$, so our statement follows from Lemma \ref{subBC} applied to $S^\cl=S_1^\cl$ and $S^\cl_0=S_1^{\cl,\cG}$.

Part (4) follows from Definition \ref{Dlambda} and Corollary \ref{buzzBC}.

The proof of part (5) is analogous to that of part (3). Let $\widetilde{\ev}\colon S_{1,\lambda}^\cl\to\cS_{1,\lambda}^\cl$ be the evaluation map. By definition the eigenvariety $\cD_{1,\aux}^N$ is the Zariski-closure in $\cD_1^N$ of the set $S_{1,\aux}^\cl$. The image of $S_{1,\aux}^\cl$ in $\cS_{1,\lambda}^\cl$ via $\widetilde{\ev}$ is $\cS_{1,\aux}^\cl$, so the desired conclusion follows from Lemma \ref{subBC} applied to $S^\cl=S_{1,\lambda}^\cl$ and $S^\cl_0=S_{1,\aux}^\cl$.
%
\end{proof}

Now consider the second auxiliary eigenvariety $\cD_{2,\aux}^M$. It is equipped with maps $\psi_{2,\aux}\colon\calH_2^N\to\cO(\cD^M_{2,\aux})$ and $w_{2,\aux}\colon\cD_{2,\aux}^M\to\cW_2^\circ$. 
Recall that $\cD_{2,\aux}^M$ is defined as the Zariski-closure in $\cD_2^M$ of the set $S_1^{\Sym^3}$.
Define a subset $\cS_{2,\aux}^\cl$ of $\Hom(\calH_2^{N},\Qp)\times\Z$ by
\begin{gather*}
\cS_{2,\aux}^\cl=\{(\psi,k)\in\Hom(\calH_2^{N},\Qp)\times\Z\,\vert\,\psi=\chi_F \\
\textrm{ where }F=(\Sym^3f)^\st_1\textrm{ for a cuspidal classical non-CM } \GL_2\textrm{-eigenform }f\textrm{ of weight }k\}. 
\end{gather*}

\begin{lemma}\label{eigenBC2}
The $4$-tuple $(\cD_{2,\aux}^M,\psi_2,w_2,S_1^{\Sym^3})$ defines a BC-eigenvariety for the datum $(1,\calH_2^N,U_p^{(2)},\cS_{2,\aux}^\cl)$. 
\end{lemma}

\begin{proof}
It is clear from the definitions of $S_2^\cl$ and $\cS_2^\cl$ that the evaluation of $(\psi_{2,\aux},w_{2,\aux})$ at a point $x\in S_1^{\Sym^3}$ induces a bijection $S_1^{\Sym^3}\to\cS_2^\cl$. Then the lemma follows from Corollary \ref{subBC} applied to the choices $\cD=\cD_2^M$, $S_0^\cl=S_1^{\Sym^3}$, $g_0=1$ and $\iota_0=\iota$. 
\end{proof}

\begin{rem}\label{BC12}
The sets $\cS_{1,\lambda}^\cl$ and $\cS_{2,\aux}^\cl$ coincide. Indeed $(\Sym^3f)^\st_1$ is well-defined for every cuspidal non-CM $\GL_2$-eigenform $f$, and a $\GSp_4$-eigenform $F$ satisfies $\chi_F=\chi_f\ccirc\lambda$ if and only if $F=(\Sym^3f)^\st_1$. 
\end{rem}

Let $S_2^\cl$ be the set of classical points of $\cD_2^M$. Define a subset $\cS^\cl_2$ of $\Hom(\calH_2^{N},\Qp)\times\Z^2$ by 
\begin{gather*}
\cS^\cl_2=\{(\psi,\uk)\in\Hom(\calH_2^{N},\Qp)\times\Z^2\,\vert\,\psi=\chi_F\textrm{ for a cuspidal classical }\GSp_4\textrm{-eigenform }F\textrm{ of weight }\uk\}.
\end{gather*}

\begin{lemma}\label{eigenBC3}
The $4$-tuple $(\cD_2^M,\psi_2,w_2,S^\cl_2)$ is a BC-eigenvariety for the datum $(2,\calH_2^N,U_p^{(2)},\cS^\cl_2)$.
\end{lemma}

\begin{proof}
This is an immediate consequence of Corollary \ref{buzzBC}.
\end{proof}

We are ready to prove the existence of the morphisms fitting into diagram \eqref{xi123}.

\begin{prop}\label{exxi1}
There exists an isomorphism $\xi_1\colon\cD_1^{N,\cG}\to\cD_{1,\aux}^N$ of rigid analytic spaces over $\Q_p$ making the leftmost squares in the diagrams \eqref{xi123} commute.
\end{prop}

\begin{proof}
Note that the map $\cS_1^{\cl,\cG}\to\cS_{1,\aux}^\cl$ defined by $(\psi,k)\mapsto(\psi\ccirc\lambda,k)$ is a bijection by Remark \ref{no3twist}. 
Thanks to Lemma \ref{eigenBClambda}(3) and (5) we know that the $4$-tuples $(\cD_{1}^{N,\cG},\psi_1,w_1,S_1^{\cl,\cG})$ and $(\cD_{1,\aux}^N,\psi_{1,\aux},w_{1,\aux},S_{1,\aux}^\cl)$ are BC-eigenvarieties for the data $(1,\calH_1^N,\lambda(U_p^{(2)}),\cS_1^{\cl,\cG})$ and $(1,\calH_2^N,U_p^{(2)},\cS_{1,\aux}^\cl)$, respectively.  
Hence Lemma \ref{changealg} applied to the morphism $\lambda\colon\calH_2^N\to\calH_1^N$ and the two data above gives the desired isomorphism $\xi_1\colon\cD_{1}^{N,\cG}\to\cD_{1,\aux}^N$.
\end{proof}

\begin{prop}\label{exxi2}
There exists an isomorphism $\xi_2\colon\cD_{1,\aux}^N\to\cD_{2,\aux}^{M}$ of rigid analytic spaces over $\Q_p$ making the central squares in the diagrams \eqref{xi123} commute.
\end{prop}


\begin{proof}
Lemmas \ref{eigenBClambda}(5) and \ref{eigenBC2} together with Remark \ref{BC12} imply that the $4$-tuples $(\cD_{1,\aux}^N,\psi_1,w_1,S_{1,\aux}^\cl)$ and $(\cD_{2,\aux}^M,\psi_2,w_2,S_{2,\aux}^\cl)$ are both BC-eigenvarieties for the datum $g=1$, $\calH=\calH_2^N$, $\eta=U_p^{(2)}$ and $\cS^\cl=\cS_{1,\aux}^\cl=\cS_{2,\aux}^\cl$. Now the proposition follows from Proposition \ref{BCunique}.
\end{proof}

\begin{prop}\label{exxi3}
There exists a closed immersion $\xi_3\colon\cD_{2,\aux}^N\to\cD_{2}^{M}$ of rigid analytic spaces over $\Q_p$ making the rightmost squares in the diagrams \eqref{xi123} commute.
\end{prop}

\begin{proof}
This is a consequence of Lemma \ref{changeweight} applied to the BC-data $(2,\calH_2^N,U_p^{(2)},\cS_2)$ and $(1,\calH_2^N,U_p^{(2)},\cS_2^\cl)$, with the morphism $\cW_1\to\cW_2$ being $\iota$.
\end{proof}

Finally, we can define the desired $p$-adic interpolation of the symmetric cube transfer.

\begin{defin}\label{defxi} 
We define a morphism $\xi\colon\cD_1^{N,\cG}\to\cD_2^M$ of rigid analytic spaces over $\Q_p$ by $\xi=\xi_3\ccirc\xi_2\ccirc\xi_1$.
\end{defin}


\begin{prop}\label{sym3closure}
\begin{enumerate}
\item The morphism $\xi$ is a closed immersion of eigenvarieties. 
\item The image of $\xi$ is equidimensional of dimension $1$.
\item The Zariski-closure of the set $S_1^{\Sym^3}$ in $\cD_2^M$ is equidimensional of dimension $1$.
\end{enumerate}
\end{prop}

\begin{proof}
Since the diagrams \ref{xi123} are commmutative, $\xi$ is a morphism of eigenvarieties. It is a closed immersion because $\xi_1$ and $\xi_2$ are isomorphisms and $\xi_3$ is a closed immersion, hence (1). Statement (2) follows from (1) and the fact that the non-CM eigencurve $\cD_1^{N,\cG}$ is equidimensional of dimension $1$ (see Remark \ref{nonCMdense}). By construction the image of $\xi$ is the Zariski-closure of the set $S_1^{\Sym^3}$, so we have (3).
\end{proof}

\begin{rem}
Let $f$ be a classical, cuspidal, CM $\GL_2$-eigenform of level $\Gamma_1(N)$. 
Since $f$ is CM, the $\GSp_4$-eigenform $\Sym^3f$ provided by Corollary \ref{formtransf} may not be cuspidal. Suppose that it is not. Let $x$ be a point of $\cD_1^N$ corresponding to a positive slope $p$-stabilization of $f$. By \cite[Corollary 3.6]{cit}, $x$ is a CM point of a non-CM component $I$ of $\cD_1^N$. Let $\xi(x)$ be the image of $x$ via the morphism of Definition \ref{defxi}. Then $\xi(x)$ belongs to the cuspidal eigenvariety $\cD_2^{M}$, but it is not cuspidal since $\Sym^3f$ is not. This means that $\xi(x)$ is a non-cuspidal specialization of a cuspidal family of $\GSp_4$-eigenforms. Brasca and Rosso \cite{braros} constructed an eigenvariety for $\GSp_4$ parametrizing the systems of Hecke eigenvalues associated with the non-cuspidal overconvergent eigenforms and they glued it with $\cD_2^M$. It should be possible to show that a cuspidal and a non-cuspidal component of this glued eigenvariety cross at $\xi(x)$.
\end{rem}

\begin{rem}
We defined the morphism $\xi$ on the union of connected components of a fixed residual Galois representation, in order to obtain a closed immersion (see Lemma \ref{no3twist}). When the residual representation varies, the morphisms obtained this way glue into a morphism of eigenvarieties $\cD_1^N\to\cD_2^M$ that is $3:1$ on its image. Working with this morphism is not of any interest to our purposes, since it is never a problem to fix a residual Galois representation on $\cD_2^M$.
\end{rem}

\bigskip

\section{If $\Sym^3\rho$ is modular then $\rho$ is modular}\label{triang}

The goal of this section is to show that if the symmetric cube of a continuous representation $\rho\colon G_\Q\to\GL_2(\Qp)$ is associated with a classical or overconvergent $\GSp_4$-eigenform, then $\rho$ is associated with a classical or overconvergent, respectively, $\GL_2$-eigenform. 

We refer to \cite{berger,colmez} for the definitions and results that we need from the theory of $(\varphi,\Gamma)$-modules. As before $E$ is a finite extension of $\Q_p$, fixed throughout the section. Let $\Gamma$ be the Galois group over $E$ of a $\Z_p$-extension of $E$ and let $H_E=G_E/\Gamma$. Let $\cR$ be the Robba ring over $E$. 
A $(\varphi,\Gamma)$-module over $\cE^\dagger$ or $\cR$ is a free module $D$ of finite type carrying commuting actions of $\Gamma$ and $\varphi$ and such that $\varphi(D)$ generates $D$ as a $\cR$-module. There is a functor $\bD_\rig$ that from the category of finite-dimensional $E$-representations of $G_{\Q_p}$ and that of $(\varphi,\Gamma)$-modules of slope $0$ on $\cR$. This functor induces an equivalence between the two categories.

We say that a $(\varphi,\Gamma)$-module $D$ over $\cR$ is triangulable if it is obtained via successive extensions of $(\varphi,\Gamma)$-modules of rank one. We say that the representation $V$ is trianguline if $\bD_\rig(V)$ is triangulable.

\subsection{Trianguline parameters of overconvergent $\GSp_4$-eigenforms}\label{tripar}

Let $g=1$ or $2$. 
Let $F$ be an overconvergent, finite slope $\GSp_{2g}$-eigenform and let $\rho_{F,p}\colon G_\Q\to\GSp_{2g}(\Qp)$ be the $p$-adic Galois representation associated with $F$. 
As Berger observed in \cite[Section 4.3]{bergertri}, the following result is a combination of \cite[Theorem 6.3]{kisin} and \cite[Proposition 4.3]{colmez}.

\begin{thm}\label{modtri}
If $g=1$, the representation $\rho_{f,p}\vert_{G_{\Q_p}}$ is trianguline.
\end{thm}

If $g=2$, an analogue of Theorem \ref{modtri} for $\rho_{F,p}$ can be deduced from the work of Kedlaya, Pottharst and Xiao \cite{kedpotxia}. Moreover the results of \emph{loc. cit.} allow us to write the parameters of the triangulation of $\rho_{F,p}$ in terms of a Hecke polynomial, as for classical points.


With the notations of Section \ref{gspfam}, consider the locus $\cD_2^{M,\irr}$ where the residual Galois representation is irreducible on $\cD_2^M$ and its admissible covering $\cE^\irr$. 
Let $D\in\cE^\irr$ and let $\rho_D\colon G_Q\to\GL_4(\cO(D))$ be the representation constructed in Section \ref{gspfam}. Keep the notations of \cite{kedpotxia} for Robba rings. Let $M_D$ be the $(\varphi,\Gamma)$-module over $\cR_D(\pi)$ attached to $\rho_D$. Then $\{M_D\}_{D\in\cE^\irr}$ is a family of $(\varphi,\Gamma)$-modules over $\cD_2^{M,\irr}$ in the sense of \cite[Section 2.1]{kedpotxia}. 
For $x\in\cD_2^{M,\irr}(\C_p)$, let $\rho_x\colon G_\Q\to\GL_4(\Qp)$ and $\psi_x\colon\calH_2^N\to\Qp$ be the Galois representation and the system of Hecke eigenvalues, respectively, attached to $x$. Let $M_x$ be the $(\varphi,\Gamma)$-module over $\cR$ attached to $\rho_x$. Denote by $\ev_x$ the evaluation of rigid analytic functions on $\cD_2^{M,\irr}$ at $x$. We identify the weight of $x$ with a character $(\kappa_1(x),\kappa_2(x))\colon (\Z_p^\times)^2\to\C_p^\times$. Let $\id\colon\Z_p^\times\to\Z_p^\times$ be the identity. 
We still write $\psi_2$ for the morphism of $\Q$-algebras $\calH_2^M\to\cO(\cD_2^{M,\irr})$ induced by $\psi_2\colon\calH_2^M\to\cO(\cD_2^{M})$. 
Let $\delta_i$, $1\le i\le 4$, be the characters $\Q_p^\times\to\cO(\cD_2^{M,\irr})^\times$ defined by
\begin{equation*}\begin{gathered}
\delta_1\vert_{\Z_p^\times}=1, \quad \delta_1(p)=\psi_2(U_{p,2}^{(2)}); \\
\delta_2\vert_{\Z_p^\times}=\kappa_1/\id, \quad \delta_2(p)=\psi_2((U_{p,2}^{(2)})^{w_1}); \\
\delta_3\vert_{\Z_p^\times}=\kappa_2/\id^2, \quad \delta_3(p)=\psi_2((U_{p,2}^{(2)})^{w_2}); \\
\delta_4\vert_{\Z_p^\times}=\kappa_1\kappa_2(p)/\id^3, \quad \delta_4(p)=\psi_2((U_{p,2}^{(2)})^{w_1w_2}).
\end{gathered}\end{equation*}
For $x\in\cD_2^{M,\irr}(\C_p)$, let $\delta_{i,x}=\ev_x\ccirc\delta_i\colon\Q_p^\times\to\Qp$. 

\begin{rem}\label{frobtrirem}
There is an equality
\begin{equation}\label{frobtri} \prod_{i=1}^4(X-\delta_{i}(p))=\psi_2(P_\Min(t_{2,p}^{(2)};X)) \end{equation}
in $\cO(\cD_2^{M,\irr})[X]$. 
This is true when we specialize at a classical point $x$, since in this case the polynomial $\prod_{i=1}^4(X-\delta_{i,x}(p))$ coincides with the characteristic polynomial of the crystalline Frobenius acting on $\bD_\cris(\rho_x)$ by a result of Berger (see \cite[Proposition 1.8]{colmez}). This polynomial coincides with $\psi_x(P_\Min(t_{2,p}^{(2)};X))$ by \cite[Théorème 1]{urban}. Since the polynomials in Equation \eqref{frobtri} have analytic coefficients and coincide on the Zariski-dense subset of classical points of $\cD_2^{M,\irr}$, they must be equal.

By specializing Equation \eqref{frobtri} at any $x\in\cD_2^{M,\irr}(\C_p)$, we obtain an equality $\prod_{i=1}^4(X-\delta_{i,x}(p))=\psi_x(P_\Min(t_{2,p}^{(2)};X))$ in $\Qp[X]$.
\end{rem}

The following is a consequence of \cite[Theorem 6.3.13]{kedpotxia}.

\begin{thm}\label{famtri}\mbox{ }
\begin{enumerate}
\item For every $x\in\cD_2^{M,\irr}(\C_p)$, the $(\varphi,\Gamma)$-module $M_x$ is trianguline.
\item There exist a Zariski open rigid analytic subspace $\widetilde\cD_2^{M,\irr}$ of $\cD_2^{M,\irr}$ such that for every $x\in\cD_2^{M,\irr}(\C_p)$ the $(\varphi,\Gamma)$-module $M_x$ is triangulable with parameters $\delta_{i,x}\colon\Q_p^\times\to\Qp$.
\end{enumerate}
\end{thm}

\begin{proof}
Let $x\in\cD_2^{M,\irr}(\C_p)$. Let $X$ be a union of irreducible components of $\cD_2^{M,\irr}(\C_p)$ containing $x$ and a classical point. By Coleman's classicality criterion it will contain a Zariski-dense subset of classical points. Consider the sheaf $M_X$ of $(\varphi,\Gamma)$-modules on $X$ obtained by restriction from that on $\cD_2^{M,\irr}$. If $z$ is a classical point of $X$, then the eigenvalues of the crystalline Frobenius acting on $D_\cris(\rho_z)$ are $\delta_{i}(p)$, $1\le i\le 4$, by the discussion in Remark \ref{frobtrirem}. Then, with the terminology of \cite[Definition 6.3.2]{kedpotxia}, $M_X$ is a densely pointwise strictly trianguline $(\varphi,\Gamma)$-module over $\cR_X(\pi)$ with respect to the parameters $\delta_i$, $1\le i\le 4$, and the Zariski-dense set given by the classical points of $X$. Now \cite[Corollary 6.3.13]{kedpotxia} gives that $M_x$ is trianguline, hence part (1) of the theorem. Next \cite[Corollary 6.3.10]{kedpotxia} gives a Zariski-open subspace $\widetilde{X}$ of $X$ such that, for $y\in\widetilde{X}(\C_p)$, the parameters of the triangulation of $M_y$ are exactly $\delta_{i,y}$, $1\le i\le 4$. Repeating this argument for every choice of $x$ and $X$ gives part (2) of the theorem. 
\end{proof}


Now let $N$ be a positive integer prime to $p$ and let $M=M(N)$ be as in Definition \ref{sym3leveldef}. 
Let $F$ be an overconvergent $\GSp_4$-eigenform corresponding to a point of $\widetilde\cD_2^{M,\irr}$. Suppose that there is a $\GL_2$-eigenform $f$ of level $N$ such that $\rho_{F,p}\cong\Sym^3\rho_{f,p}$, with the usual notations. Let $\chi_F\colon\calH_2^N\to\Qp$ and $\chi_f\colon\calH_1^N\to\Qp$ be the systems of Hecke eigenvalues of the two forms. Write $\chi_F=\chi_F^{Np}\otimes\chi_{F,p}$ and $\chi_{f,p}=\chi_f^{Np}\otimes\chi_{f,p}$. Proposition \ref{heckemorphunr} describes $\chi_F^{Np}$ in terms of those of $\chi_f^{Np}$. Thanks to Theorem \ref{famtri}(2) and Remark \ref{frobtrirem} we can describe $\chi_{F,p}$ in terms of $\chi_{f,p}$. Here the notations are the same as for Proposition \ref{heckemorphstab}.

\begin{prop}\label{heckemorphtri}
There exists $i\in\{1,2,3,4\}$ such that 
\begin{equation}\label{heckemorphtrieq} \chi_{F,p}\ccirc\iota_{I_{2,p}}^{T_2}=(\chi_{f,p}\ccirc\iota_{I_{1,p}}^{T_1})^\ext\ccirc\lambda_{i,p}. \end{equation}
Moreover, if $\lambda_p\colon\calH(T_2(\Q_p),T_2(\Z_p))\to\calH(T_1(\Q_p),T_1(\Z_p))$ is another morphism satisfying $\chi_{F,p}\ccirc\iota_{I_{2,p}}^{T_2}=(\chi_{f,p}\ccirc\iota_{I_{1,p}}^{T_1})^\ext\ccirc\lambda_p$, then there exists $i\in\{1,2,3,4\}$ such that $\lambda_p=\lambda_{i,p}$.
\end{prop}

\begin{proof}
The proof is completely analogous to that of Proposition \ref{heckemorphstab}, once we replace $\bD_\cris(\rho_{f,p})$ and $\bD_\cris(\rho_{f,p})$ by the $(\varphi,\Gamma)$-modules $\bD_\rig(\rho_{f,p})$ and $\bD_\rig(\rho_{F,p})$, respectively. We use Theorem \ref{famtri}(2) and Remark \ref{frobtrirem} to describe the parameters of the triangulations of the two $(\varphi,\Gamma)$-modules in terms of Hecke polynomials.
\end{proof}


Let $F$ be a finite slope overconvergent $\GSp_4$-eigenform of tame level $M$.

\begin{prop}\label{dimcrys}
There are at most $2\cdot\Dim_{\Qp}\bD_\cris(\rho_{F,p})$ overconvergent $\GSp_4$-eigenforms $F^\prime$ satisfying $\rho_{F^\prime,p}\cong\rho_{F,p}$. 
\end{prop}

\begin{proof}
The accumulation and Zariski-dense set $Z$ of classical points of $\cD_2^{M,\irr}$ satisfies the assumptions (CRYS) and (HT) of \cite[Section 3.3.2]{bellchen}. For every $x\in Z$, $\psi_x(U_{p,2}^{(2)})$ is an eigenvalue of the crystalline Frobenius acting on $\bD_\cris(\rho_x)$. 
Then \cite[Theorem 3.3.3]{bellchen} implies that, for every $\C_p$-point $x$ of $\cD_2^{M,\irr}$, $\psi_x(U_{p,2}^{(2)})$ is an eigenvalue of the crystalline Frobenius acting on $\bD_\cris(\rho_x)$. Hence $\psi_x(U_{p,2}^{(2)})$ can take at most $\Dim_{\Qp}{\bD_\cris(\rho_{F,p})}$ disctinct values. There are exactly two characters of the Iwahori-Hecke algebra giving the same value for $\psi_x(U_{p,2}^{(2)})$, hence $2\cdot\Dim_{\Qp}{\bD_\cris(\rho_{F,p})}$ choices for the system of Hecke eigenvalues of $F^\prime$ at $p$. Since the system of Hecke eigenvalues of $F^\prime$ outside $Np$ is determined by the associated Galois representation, we obtain the desired result.
\end{proof}

Now let $f$ be a finite slope overconvergent $\GL_2$-eigenform of tame level $N$. Proposition \ref{dimcrys} implies the following.

\begin{cor}\label{dimcryssym}
There are at most $2\cdot\dim_{\Qp}\bD_\cris(\rho)$ finite slope overconvergent $\GSp_4$-eigenforms $F^\prime$ of level $M$ satisfying $\rho_{f,p}\cong\Sym^3\rho_{f,p}$. 
\end{cor}

\subsection{Non-abelian cohomology and semilinear group actions}

We recall a few results from the theory of non-abelian cohomology. Let $S$ and $T$ be two pointed sets with distinguished elements $s$ and $t$, respectively. Let $f\colon S\to T$ be a map of pointed sets. We define the kernel of $f$ by $\ker f=\{s\in S\,\vert\, f(s)=t\}$. Thanks to this notion we can speak of exact sequences of pointed sets.

Let $G$ be a topological group. Let $A$ be a topological group endowed with a continuous action of $G$, compatible with the group structure. For $i\in\{0,1\}$ let $H^i(G,A)$ be the continuous cohomology of $G$ with values in $A$. Then $H^i_\cont(G,A)$ has the structure of a pointed set with distinguished element given by the class of the trivial cocycle. For $i=0$ we have $H^0(G,A)=A^G$, the pointed set of $G$-invariant elements in $A$; its distinguished point is the identity. Since $A$ is not necessarily abelian, we have no notion of continuous cohomology in degree greater than $1$. 
Let $B$, $C$ be two other topological groups with the same additional structures as $A$, and let
\begin{equation}\label{exseqgroups} 1\to A\xto{\alpha} B\xto{\beta} C\to 1 \end{equation}
be a $G$-equivariant short exact sequence of topological groups. Then there is an exact sequence of pointed sets 
\begin{equation}\label{longexseqgroups} 1\to A^G\to B^G\to C^G\xto{\delta} H^1(G,A)\to H^1(G,B)\to H^1(G,C). \end{equation}
The connecting map $\delta$ is defined as follows. Let $c\in C^G$ and let $b\in B$ such that $\beta(b)=c$. Then $\delta(c)$ is the map given by $g\mapsto \alpha^{-1}(b^{-1}\cdot g.b)$ for every $g\in G$. 
We call \eqref{longexseqgroups} the long exact sequence in cohomology associated with \eqref{exseqgroups}. 

Now suppose that $A$ and $B$ are topological groups with the same structures as before, but $C$ is just a topological pointed set with a continuous action of $G$ that fixes the distinguished element of $C$. Since $C$ is not a group we cannot define $H^1(G,C)$. However the pointed set $H^0(G,C)=C^G$ of $G$-invariant elements of $C$ is well-defined; its distinguished element is the distinguished element of $C$. 

\begin{prop}\label{longexseqpointed}
Let $A$, $B$, $C$ be as in the discussion above. Suppose that
\[ 1\to A\to B\to C\to 1 \]
is an exact sequence of topological pointed sets. Then there is an exact sequence of pointed sets
\[ 1\to A^G\to B^G\to C^G\xto{\delta} H^1(G,A)\to H^1(G,B). \]
\end{prop}

The connecting map $\delta$ is defined as in the case of an exact sequence of groups. This definition does not rely on the group structure of $C$.

\begin{proof}
We check exactness at every term as in the case of an exact sequence of groups. None of these checks relies on the group structure of $C$.
\end{proof}


%
Let $G$ be a topological group. 
Let $\bB$ be a topological ring equipped with a continuous action of $G$, compatible with the ring structure. 
Let $n$ be a positive integer and let $M$ be a free $\bB$-module of rank $n$, endowed with the topology induced by that on $\bB$. 
We say that two semilinear actions of $G$ on $M$ are equivalent if they can be obtained by one another via a change of basis. 
We choose a basis $(e_1,e_2,\ldots,e_n)$ of $M$, hence an isomorphism $\GL(M)\cong\GL_n(\bB)$. We let $G$ act on $\GL_n(\bB)$ via its action on $\bB$. Two semilinear actions $g(\cdot)_1$ and $g(\cdot)_2$ of $G$ on $M$ are equivalent if and only if there exists $A\in\GL(M)$ such that $g(x)_1=M\cdot g(x)_2\cdot (g(A))^{-1}$ for every $g\in G$ and $x\in M$. 
There is a bijection
\begin{equation}\label{semilinbij}
\{\mbox{Equivalence classes of semilinear and continuous actions of } G \mbox{ on } M\}\leftrightarrow H^1(G,\GL_n(\bB)). 
\end{equation}
Given a semilinear action of $G$ on $M$, we define $a\in H^1(G,\GL_n(\bB))$ as the class of the cocycle that maps $g\in G_{\Q_p}$ to the matrix $(a^g_{ij})_{i,j}\in\GL_2(\bB)$ satisfying $g(e_i)=\sum_ja^g_{ij}e_j$ for every $i\in\{1,2,\ldots,n\}$.

We say that $G$ acts trivially on $M$ if there exists a basis $(e_1^\prime,e_2^\prime,\ldots,e_n^\prime)$ such that $g.e_i^\prime=e_i^\prime$ for every $g\in G$ and every $i\in\{1,2,\ldots,n\}$. The action of $G$ is trivial if and only if the corresponding class in $H^1(G,\GL_n(\bB))$ is trivial. We say that the action of $G$ is triangular if there exists a basis with respect to which the matrix $(a^g_{ij})_{i,j}$ is upper triangular for every $g\in G$.

\subsection{Representations with a de Rham symmetric cube}

Now suppose that $\bB$ is a $\C_p$-algebra equipped with a continuous action of $G_{\Q_p}$, compatible with the ring structure and with the natural action of $G_{\Q_p}$ on $\C_p$. Suppose that the subring of $G_{\Q_p}$-invariant elements in $\bB$ is $\Q_p$.

Recall that there is an exact sequence of algebraic groups over $\Z$:
\begin{equation}\label{exseq} 
1\to\mu_3\to\GL_2\xto{\Sym^3}\GL_4, 
\end{equation}
where $\mu_3\to\GL_2$ sends $\zeta$ to $\zeta\cdot\1_2$. 
Consider the exact sequence induced by \eqref{exseq} on the $\bB$-points:
\begin{equation}
1\to\mu_3(\bB)\to\GL_2(\bB)\to\GL_4(\bB). 
\end{equation}
Let $G_{\Q_p}$ act on each term via its action on $\bB$; this action is clearly continuous and compatible with the group structure on each term. The above sequence is $G_{\Q_p}$-equivariant. We split it into the short exact sequence
\begin{equation}\label{exseq1}
1\to\mu_3(\bB)\xto{\iota}\GL_2(\bB)\xto{\pi}(\GL_2/\mu_3)(\bB)\to 1
\end{equation}
and the injection
\begin{equation}\label{exseq2}
1\to(\GL_2/\mu_3)(\bB)\xto{\Sym^3}\GL_4(\bB).
\end{equation}
Both this sequences are $G_{\Q_p}$-equivariant. Since $\Sym^3\GL_2(\bB)$ is not normal in $\GL_4(\bB)$ we cannot complete \eqref{exseq2} to a short exact sequence of groups. However we can complete it to an exact sequence of pointed sets. Let $H$ be the algebraic group $\Sym^3\GL_2$. Let $[\GL_4,H](\bB)$ be the set of right classes $\{M\cdot H(\bB)\,\vert\, M\in\GL_4(\bB)\}$. We equip $[\GL_4,H]$ with a structure of topological pointed set by giving it the quotient topology and letting the class $H(\bB)$ be the distinguished point. Let $G_{\Q_p}$ act on $[\GL_4,H](\bB)$ by $g.(M\cdot H(\bB))=(g.M)\cdot H(\bB)$; this action is continuous and it leaves the distinguished point fixed. Then there is a $G_{\Q_p}$-equivariant exact sequence of topological pointed sets
\begin{equation}\label{exseqpointed}
1\to(\GL_2/\mu_3)(\bB)\to\GL_4(\bB)\to[\GL_4,H](\bB)\to 1, 
\end{equation}
where the first two non-trivial terms also have a group structure compatible with the action of $G_{\Q_p}$. 
Thanks to Proposition \ref{longexseqpointed} there is an exact sequence of pointed sets
\begin{equation}\label{longexsym3}\begin{gathered}
1\to((\GL_2/\mu_3)(\bB))^{G_{\Q_p}}\to(\GL_4(\bB))^{G_{\Q_p}}\to([\GL_4,H](\bB))^{G_{\Q_p}}\to \\
\to H^1(G_{\Q_p},\GL_2/\mu_3(\bB))\xto{H^1(\Sym^3)} H^1(G_{\Q_p},\GL_4(\bB)).
\end{gathered}\end{equation}

\begin{rem}\label{sym3inj}
Let $[\GL_4,H](\Q_p)$ be the subset of $[\GL_4,H](\bB)$ consisting of right classes $\{M\cdot H(\bB)\,\vert\, M\in\GL_4(\Q_p)\}$. Since $G_{\Q_p}$ acts on each term of \eqref{exseqpointed} via its action on $\bB$, we have 
\begin{gather*}
((\GL_2/\mu_3)(\bB))^{G_{\Q_p}}=(\GL_2/\mu_3)(\Q_p), \\
(\GL_4(\bB))^{G_{\Q_p}}=\GL_4(\Q_p), \\
([\GL_4,H](\bB))^{G_{\Q_p}}=[\GL_4,H](\Q_p).
\end{gather*}
In particular the map $(\GL_4(\bB))^{G_{\Q_p}}\to([\GL_4,H](\bB))^{G_{\Q_p}}$ that appears in the exact sequence \eqref{longexsym3} is surjective. Hence the kernel of the map $H^1(\Sym^3)$ is trivial, i.e. it contains only the distinguished point of $H^1(G_{\Q_p},\GL_2/\mu_3(\bB))$. 
\end{rem}

Now consider the short exact sequence of topological groups \eqref{exseq1}:
\[ 1\to\mu_3(\bB)\xto{\iota}\GL_2(\bB)\to(\GL_2/\mu_3)(\bB)\to 1. \] 
The associated long exact sequence of pointed sets is
\begin{equation}\begin{gathered}\label{longexmu3}
1\to(\mu_3(\bB))^{G_{\Q_p}}\to(\GL_2(\bB))^{G_{\Q_p}}\to((\GL_2/\mu_3)(\bB))^{G_{\Q_p}}\to \\
\to H^1(G_{\Q_p},\mu_3(\bB))\xto{H^1(\iota)} H^1(G_{\Q_p},\GL_2(\bB)\xto{H^1(\pi)} H^1(G_{\Q_p},\GL_2/\mu_3)(\bB)).
\end{gathered}\end{equation}

Let $M$ be a free $\bB$-module of rank $2$, endowed with the topology induced by $\bB$. Suppose that $G_{\Q_p}$ acts continuously on $M$. Then $\Sym^3M$ is a free $\bB$-module of rank $4$ endowed with the natural semilinear action of $G_{\Q_p}$ induced by that on $M$. 
We use the exact sequences we constructed, together with the bijection \eqref{semilinbij}, to prove the second part of the following proposition.

\begin{prop}\label{semilinsym3}\mbox{ }
\begin{enumerate} 
\item If the action of $G_{\Q_p}$ on $M$ is trivial then the action of $G_{\Q_p}$ on $\Sym^3M$ is trivial.
\item If the action of $G_{\Q_p}$ on $\Sym^3M$ is trivial then there exists a subgroup $H$ of $G_{\Q_p}$ of index $3$ that acts trivially on $M$.
\end{enumerate}
\end{prop}

\begin{proof}
The first statement is trivial. 
We prove the second one. The bijection \eqref{semilinbij} associates with the action of $G_{\Q_p}$ on $M$ a class $\sigma\in H^1(G_{\Q_p},\GL_2(\bB))$. Recall the maps $H^1(\pi)$ and $H^1(\Sym^3)$ 
that appear in the sequences \eqref{longexmu3} and \eqref{longexsym3}. The class in $H^1(G_{\Q_p},\GL_4(\bB))$ associated with the action of $G_{\Q_p}$ on $\Sym^3M$ is $(H^1(\Sym^3)\ccirc H^1(\pi))(\sigma)$; by assumption it is trivial. By Remark \ref{sym3inj} the kernel of $H^1(\Sym^3)$ is trivial, hence $(H^1(\pi))(\sigma)$ is trivial. 
By the exactness of \eqref{longexmu3} the class $\sigma$ belongs to the image of $H^1(\iota)\colon H^1(G_{\Q_p},\mu_3(\bB))\to H^1(G_{\Q_p},\GL_2(\bB))$. Let $\tau$ be an element of $H^1(G_{\Q_p},\mu_3(\bB))$ satisfying $(H^1(\iota))(\tau)=\sigma$. 
Since $\C_p\subset\bB$, $\mu_3(\bB)$ is the group of cubic roots of $1$, that we simply denote by $\mu_3$. Let $y$ be the image of $\tau$ via an isomorphism $H^1(G_{\Q_p},\mu_3)\cong\Q_p/\Q_p^3$. 
Let $x\in\Q_p$ be a representative of $y$. The cocycle $\tau$ is trivial on the subgroup $H=\Gal(\Qp/\Q_p[x^{1/3}])$ of $G_{\Q_p}$. Since $\sigma=(H^1(\iota))(\tau)$, $\sigma$ is also trivial on $H$. By definition of the bijection \eqref{semilinbij}, the above implies that the action of $H$ on $\Sym^3M$ is trivial. The group $H$ has index $1$ or $3$ in $G_{\Q_p}$, as desired.
\end{proof}

Until the end of the section $E$ is a $p$-adic field and $V$ is a finite-dimensional $E$-vector space, endowed with the $p$-adic topology and with a continuous action of $G_{\Q_p}$. 

By definition $V$ is de Rham if and only if the semilinear action of $G_{\Q_p}$ on $\bB_\dR\otimes V$ is trivial, and the analogous statement is true for $\Sym^3V$. Since a representation of $G_{\Q_p}$ is potentially de Rham if and only if it is de Rham, Proposition \ref{semilinsym3} implies the following.

\begin{cor}\label{sym3dR}
The representation $V$ of $G_{\Q_p}$ is de Rham if and only if $\Sym^3V$ is de Rham.
\end{cor}

\subsection{Representations with a trianguline symmetric cube}

Let $D$ be a $(\varphi,\Gamma)$-module over $\cR$. We define a $(\varphi,\Gamma)$-module $\Sym^3D$ over $\cR$ as follows by taking the underlying $\cR$-module to be the symmetric cube of $D$ as a $\cR$-module and letting $\Gamma$ and $\varphi$ act in the natural way.
We can check that $\varphi_{\Sym^3D}(\Sym^3D)$ generates $D$ as an $\cR$-module. 

As before $E$ is a $p$-adic field and $V$ is an $E$-vector space carrying an $E$-linear action of $G_{\Q_p}$. We use the standard notations for twists of representations of $G_{\Q_p}$ by characters.

\begin{rem}\label{sym3rig}
There is an isomorphism $\Sym^3(\bD_\rig(V))\cong\bD_\rig(\Sym^3V)$ of $(\varphi,\Gamma)$-modules over $\cR$. 
\end{rem}

We study the case where $\Sym^3V$ is trianguline. The goal of this subsection is to prove the following. 

\begin{prop}\label{sym3tri}
Suppose that $V$ is irreducible.
\begin{enumerate}[label=(\roman*)]
\item If the representation $V$ is trianguline then $\Sym^3V$ is trianguline.
\item If the representation $\Sym^3V$ is trianguline then either $V$ is trianguline or $V$ is a twist of a de Rham representation. In particular $V$ is a twist of a trianguline representation. 
\end{enumerate}
\end{prop}

The first statement is immediate. The proof of the second one relies on a technique used by Di Matteo in \cite{dimat}, together with the classification of two-dimensional potentially trianguline representations carried on by Berger and Chenevier in \cite{bergchen}. Di Matteo considers two representations $V$ and $W$ such that the tensor product representation $V\otimes W$ is trianguline, and proves that in this case $V$ and $W$ are potentially trianguline. We will adapt his arguments to our situation. 

Let $K$ be a $p$-adic field. 
Let $\bB$ be a topological field equipped with a continuous action of $G_K$. 
Let $\cC_\bB^K$ be the category of semilinear $\bB$-representations of $G_K$. 
The $\bB$-linear dual of an object of $\cC_\bB^K$ and the tensor product over $\bB$ of two objects of $\cC_\bB^K$ define new objects in the usual way. In this section all duals and tensor products are in $\cC_\bB^K$ except when stated otherwise. 

Let $\eta\colon G_K\to\bB^\times$ be a cocycle. Let $\bB(\eta)$ be a one-dimensional $\bB$-vector space with a generator $e$, equipped with the semilinear action of $G_K$ defined by $g.e=\eta(g)e$ for every $g\in G_K$. We simply write $\bB$ when $\eta$ is the trivial cocycle. Clearly every one-dimensional object in $\cC_\bB^K$ is isomorphic to $\bB(\eta)$ for some cocycle $\eta$. Note that if $\eta$ takes values in $\bB^{G_K}$ then $\eta$ is a character. For every object $M$ of $\cC_\bB^K$ we set $M(\eta)=M\otimes\bB(\eta)$.

For every object $M$ of $\cC_\bB^K$ and every finite extension $K^\prime$ of $K$, we consider $M$ as an object of $\cC_\bB^{K^\prime}$ with the action induced by the inclusion $G_{K^\prime}\subset G_K$.

We say that an $n$-dimensional object $M$ of $\cC_\bB^K$ is \emph{triangulable} if there exists a filtration
\[ M=M_0\supset M_1\supset M_2\supset\ldots\supset M_{n-1}\supset M_n=0 \]
where, for every $i\in\{1,2,\ldots,n\}$, $M_i$ is a $G_K$-stable subspace of $M$ and $M_{i-1}/M_{i}$ is one-dimensional. If there exists such a filtration that satisfies $M_{i-1}/M_{i}\cong\bB(\eta_i)$ for some characters $\eta_1,\eta_2,\ldots,\eta_n\colon G_K\to\bB^{G_K}$, then we say that $M$ is \emph{triangulable by characters}. 
These definitions are analoguous to those in the beginning of \cite[Section 3]{dimat}, but we omit the specification ``split'' since we use Colmez's terminology for trianguline representations rather than Berger's.

From now on $M$ is a two-dimensional \emph{irreducible} object in $\cC_\bB^K$.

\begin{lemma}\label{samedim}
Let $X$ and $X^\prime$ be two irreducible objects in $\cC_\bB^K$. If $X\otimes X^\prime$ has a one-dimensional quotient in $\cC_\bB^K$, then $\dim_\bB X=\dim_\bB X^\prime$.
\end{lemma}

\begin{proof}
The one-dimensional quotient of $X\otimes X^\prime$ is isomorphic to $\bB(\eta)$ for a cocycle $\eta\colon G_K\to\bB$. Consider the following tautological exact sequence in $\cC_\bB^K$:
\[ 0\to\ker\phi\to X\otimes X^\prime\xto{\phi}\bB(\eta)\to 0. \]
There is a $G_K$-equivariant map $\phi^\prime\colon X\to (X^\prime)^\ast(\eta)$ sending $x\in X$ to the function $\phi^\prime(x)\in (X^\prime)^\ast(\eta)$ defined by $x^\prime\mapsto\phi(x\otimes x^\prime)$ for every $x^\prime\in X^\prime$. Since $\phi$ is non-zero, $\phi^\prime$ is also non-zero. The representations $X$ and $(X^\prime)^\ast(\eta)$ are irreducible, hence the non-zero $G_K$-equivariant map $\phi^\prime$ is an isomorphism. We conclude that $\dim_\bB X=\dim_\bB (X^\prime)^\ast(\eta)=\dim_\bB X^\prime$.
\end{proof}

\begin{lemma}\label{tritensor}
Suppose that $\Sym^3M$ is triangulable by characters. Let $\eta_1,\eta_2,\eta_3,\eta_4\colon G_K\to\bB^{G_K}$ be the characters appearing in the triangulation of $\Sym^3M$. 
Then:
\begin{enumerate}[label=(\roman*)]
\item there exists an irreducible object $M_1$ of $\cC_\bB^K$ such that $\Sym^3M\cong M_1\otimes M$;
\item there is a decomposition $\Sym^3M\cong\bigoplus_{i=1}^4\bB(\eta_i)$ in $\cC_\bB^K$. 
\end{enumerate}
\end{lemma}

The central ingredients in the proof are \cite[Lemma 3.1.3]{dimat} and the proof of \cite[Corollary 3.1.4]{dimat}.

\begin{proof}
Let $\Sym^3M=Y\supset Y_1\supset Y_2\supset Y_3\supset Y_4=0$ be a filtration of $\Sym^3M$ satisfying $Y_{i-1}/Y_{i}\cong\bB(\eta_i)$ for $1\leq i\leq 4$. In particular $\bB(\eta_1)$ is a quotient of $\Sym^3M$ and $\bB(\eta_4)$ is a subobject of $\Sym^3M$. Let $\pi_{\eta_1}\colon\Sym^3M\to\bB(\eta_1)$ and $\pi\colon\Sym^2M\otimes M\to\Sym^3M$ be the natural projections.

Consider the following exact sequence in $\cC_\bB^K$:
\[ 0\to\ker\pi\to\Sym^2M\otimes M\xto{\pi}\Sym^3M\to 0 \]
The surjection $\pi_{\eta_1}\ccirc\pi\colon\Sym^2M\otimes M\to\bB(\eta_1)$ defines a one-dimensional quotient of $\Sym^2M\otimes M$. If $\Sym^2M$ is irreducible then Lemma \ref{samedim} implies that $\dim_\bB\Sym^2M=\dim_\bB M$, which is a contradiction since $\Sym^2M$ is three-dimensional. Then $\Sym^2M$ is reducible; this means that it admits a non-trivial filtration in $\cC_\bB^K$ (i.e. a filtration in $G_K$-stable subspaces). For simplicity, set $X=\Sym^2M$. All the maps and the filtrations we write are in $\cC_\bB^K$. There are three possibilities:
\begin{enumerate}
\item there is a filtration $X=X_0\supset X_1\supset X_2\supset X_3=0$ with $\dim_\bB (X_{i-1}/X_i)=1$ for $i=1,2,3$;
\item there is a filtration $X=X_0\supset X_1\supset X_2=0$ with $\dim_\bB (X/X_1)=1$, $\dim_\bB X_1=2$ and $X_1$ irreducible;
\item there is a filtration $X=X_0\supset X_1\supset X_2=0$ with $\dim_\bB (X/X_1)=2$, $\dim_\bB X_1=1$ and $X/X_1$ irreducible;
\end{enumerate}
Suppose that (1) holds. Since $X$ is obtained from $X/X_1$, $X_1/X_2$ and $X_2$ by successive extensions, $X\otimes M$ is obtained by successive extensions of $(X/X_1)\otimes M$, $(X_1/X_2)\otimes M$ and $X_2\otimes M$. Hence there exists $i\in\{1,2,3\}$ such that the surjection $X\otimes M\to\bB(\eta_1)$ induces a surjection $X_{i-1}/X_i\otimes M\to\bB(\eta_1)$. Since $X_{i-1}/X_i$ and $M$ are irreducible, Lemma \ref{samedim} implies that $\dim_\bB (X_{i-1}/X_i)=\dim_\bB M=2$, a contradiction since $\dim_\bB (X_{i-1}/X_i)=1$ for every $i$. 

Suppose that we are in case (2). As before, there exists $i\in\{1,2\}$ such that $X\otimes M\to\bB(\eta_1)$ induces a surjection $\pi_{\eta_1}^\prime(X_{i-1}/X_i)\otimes M\to\bB(\eta_1)$. If $i=1$ Lemma \ref{samedim} implies that $\dim_\bB (X/X_1)=\dim_\bB M$, a contradiction. Hence there is an exact sequence 
\[ 0\to\ker\pi_{\eta_1}^\prime\to X_1\otimes M\xto{\pi_{\eta_1}^\prime}\bB(\eta_1). \]
Since $X_1$ and $M$ are irreducible, this sequence splits by \cite[Lemma 3.1.3]{dimat}. 
In particular there is a section $\bB(\eta_1)\into X_1\otimes M$. By composing this section with the inclusion $X_1\otimes M\into X\otimes M$ and the projection $X\otimes M\to\Sym^3M$ we obtain a section of the map $\pi_{\eta_1}$, hence a splitting of the exact sequence
\[ 0\to\ker\pi_{\eta_1}\to\Sym^3M\xto{\pi_{\eta_1}}\bB(\eta_1)\to 0. \]
By definition of $\pi_{\eta_1}$ we have $Y_1=\ker\pi_{\eta_1}$, so $\Sym^3M\cong Y_1\oplus\bB(\eta_1)$. 
Now $Y_2$ is a subobject of $Y_1$, hence $Y_2\oplus\bB(\eta_1)$ is a subobject of $\Sym^3M$. There is an isomorphism $\Sym^3M/(Y_2\oplus\bB(\eta_1))\cong Y_1/Y_2\cong\bB(\eta_2)$, giving a projection $\pi_{\eta_2}\colon\Sym^3M\to\bB(\eta_2)$. By replacing $\pi_{\eta_1}$ with $\pi_{\eta_2}$ in the above argument, we obtain that the sequence 
\[ 0\to\ker\pi_{\eta_2}\to\Sym^3M\xto{\pi_{\eta_2}}\bB(\eta_2)\to 0 \]
splits. Then $\Sym^3M\cong\ker\pi_{\eta_2}\oplus\bB(\eta_2)$. Since $\ker\pi_{\eta_2}\cong Y_2\oplus\bB(\eta_1)$ we obtain $\Sym^3M\cong Y_2\oplus\bB(\eta_1)\oplus\bB(\eta_2)$. We repeat the argument for the projection to $\bB(\eta_3)$ and we obtain a decomposition $\Sym^3M\cong\bigoplus_{i=1}^4\bB(\eta_i)$, together with maps $\pi_{\eta_i}\colon X_1\otimes M\to\bB(\eta_i)$. 

Now consider the map $\psi\colon X_1\otimes M\to\Sym^3M$ obtained by composing the inclusion $X_1\otimes M\into X\otimes M$ with $\pi\colon X\otimes M\to\Sym^3M$. By the results of the previous paragraph, $\Sym^3M\cong\bigoplus_{i=1}^4\bB(\eta_i)$ and for every $i\in\{1,2,3,4\}$ there is a map $\pi_{\eta_i}\colon X_1\otimes M\to\bB(\eta_i)$. Hence $\psi$ is surjective. Since $X_1\otimes M$ and $\Sym^3M$ are both $4$-dimensional, $\psi$ is an isomorphism. Moreover $X_1$ is irreducible, so part (1) of the lemma is true with $M_1=X_1$. 

Suppose that we are in case (3). Consider the map $\psi\colon X_1\otimes M\to\Sym^3M$ obtained by composing the inclusion $X_1\otimes M\to\Sym^2M\otimes M$ with the projection $\pi\colon\Sym^2M\otimes M\to\Sym^3M$. Since $X_1$ is one-dimensional and $M$ is irreducible, $X_1\otimes M$ is irreducible. Hence the kernel of $\psi$ is either $0$ or $X_1\otimes M$. In the first case the image of $\psi$ defines a two-dimensional irreducible subobject of $\Sym^3M$, contradicting the fact that $\Sym^3M$ is triangulable. In the second case $\pi$ factors via a surjective map $\pi_1\colon (X/X_1)\otimes M\to\Sym^3M$. Since $\dim_\bB((X/X_1)\otimes M)=\dim_\bB\Sym^3M$, $\pi_1$ is an isomorphism. Now $X/X_1$ is irreducible, so part (1) of the lemma is true with $M_1=X/X_1$. 

The decomposition of $\Sym^3M$ given in part (2) of the lemma follows from part (1) and \cite[Corollary 3.1.4]{dimat}.
\end{proof}

We recall another result of \cite{dimat}.

\begin{lemma}\label{trifinord}\cite[Lemma 3.2.1]{dimat}
Let $N$ and $N^\prime$ be two objects of $\cC_\bB^K$ such that $N\otimes N^\prime$ is triangulable by characters. Let $\{\eta_i\}_{i=1}^d$ be the set of characters $G_K\to\bB^{G_K}$ appearing in the triangulation of $N\otimes N^\prime$. Then $\eta_1^{-1}\eta_i$ is a finite order character for every $i\in\{1,2,\ldots,d\}$.
\end{lemma}

The following lemma is proved in the same way as \cite[Theorem 3.2.2]{dimat}, with the difference that we work in the language of $(\varphi,\Gamma)$-modules rather than in that of $B$-pairs. Recall that $E$ is a $p$-adic field and $V$ is a two-dimensional $E$-representation of $G_{\Q_p}$. 

\begin{lemma}\label{sym3pottri}
Suppose that $V$ is irreducible. If $\Sym^3V$ is trianguline, then $V$ is potentially trianguline.
\end{lemma}

\begin{proof} 
Consider the $(\varphi,\Gamma)$-modules $\bD_\rig(V)$ and $\bD_\rig(\Sym^3V)$. They are free $\cR$-modules carrying a semilinear action of $G_{\Q_p}$. By Remark \ref{sym3rig} there is an isomorphism of $(\varphi,\Gamma)$-modules $\bD_\rig(\Sym^3V)\cong\Sym^3\bD_\rig(V)$. In particular this is an isomorphism of semilinear representations of $G_{\Q_p}$, where we let $G_{\Q_p}$ act via $\G_{\Q_p}\onto\Gamma$.

Since $\Sym^3V$ is trianguline, $\bD_\rig(\Sym^3V)$ is obtained by successive extensions of rank one $(\varphi,\Gamma)$-modules $D_i$, $1\leq i\leq 4$. By \cite[Théorème 0.2(i)]{colmez}, for every $i\in\{1,2,3,4\}$ there exists a character $\eta_i\colon\Q_p^\times\to E^\times$ such that $D_i\cong\cR(\eta_i)$. 
Note that $E^\times=\cR^{G_E}$, so $\eta_i\vert_{G_E}$ takes values in $\cR^{G_E}$. 

Since $V$ is irreducible, \cite[Corollary 2.2.2]{dimat} implies that $\bD_\rig(V)$ is irreducible as a semilinear $\cR$-representation of $G_{\Q_p}$. In particular the choice $M=\bD_\rig(V)$ satisfies the assumptions of Lemma \ref{tritensor}, hence part (2) of that lemma gives a $G_{\Q_p}$-equivariant decomposition $\bD_\rig(\Sym^3V)\cong\bigoplus_{i=1}^4\cR(\eta_i)$. 

Now by Lemma \ref{trifinord} there exists a finite extension $L$ of $E$ such that $\eta_1^{-1}\eta_i\vert_{G_{L}}$ is trivial for every $i$. 
Hence there is an isomorphism $\bD_\rig(\Sym^3V)(\eta_1^{-1})\cong\bigoplus_{i=1}^4\cR$ of $\cR$-representations of $G_L$. This means that $\bD_\rig(\Sym^3V)(\eta_1^{-1})$ is a trivial $\cR$-representation of $G_L$. Let $\eta^\prime\colon G_\Q\to\Qp^\times$ be a character satisfying $\bD_\rig(\mu)=\cR(\eta_1)$. Then $\bD_\rig((\Sym^3V)(\mu^{-1}))=(\bD_\rig(\Sym^3V))(\eta_1^{-1})$. 
By \cite[Theorem 0.2]{berger} (see the formulation in \cite[Proposition 1.8]{colmez}) there is an isomorphism
\[ \bD_\st(\Sym^3V(\mu^{-1}))=(\cR[1/t,T]\otimes_\cR\bD_\rig(\Sym^3V))^{\Gamma_L} \]
of filtered $(\varphi,N)$-modules. We know that $G_L$ acts trivially on $\bD_\rig((\Sym^3V)(\eta_1^{-1}))$, so the module $\bD_\st((\Sym^3V)(\eta_1^{-1}))$ is four-dimensional. This means that $(\Sym^3V)(\mu^{-1})$ is a semi-stable representation of $G_L$. In particular it is a de Rham representation of $G_L$. 

Let $\mu^\prime(x)=\mu(x)/\vert\mu(x)\vert\colon\Q_p^\times\to\cO_E^\times$. Let $E_1$ be a finite extension of $E$ that contains $p^{1/6}$ and let $L_1$ be a finite extension of $L$ such that $\mu^\prime\vert_{G_{L_1}}$ is trivial modulo the maximal ideal of $\cO_E$. 
Then there exists a character $\mu^{-1/6}\colon\Q_p^\times\to E_1^\times$ such that $(\mu^{-1/6})^6=\mu^{-1}$. Since $\Sym^3(V(\mu^{-1/6}))\cong(\Sym^3V)(\mu^{-1})$ and $(\Sym^3V)(\mu^{-1})$ is de Rham, $V(\mu^{-1/6})$ is also de Rham by Proposition \ref{sym3dR}. In particular $V(\mu^{-1/6})$ is potentially trianguline, so its twist $V$ is still potentially trianguline by \cite[Proposition 4.3]{colmez}. 
\end{proof}

In order to deduce Proposition \ref{sym3tri} from Lemma \ref{sym3pottri} we need the following result by Berger and Chenevier, who classified the two-dimensional potentially trianguline representations of $G_{\Q_p}$. 
Here we do not suppose that $V$ is irreducible. 

\begin{thm}\label{berchenpot}\cite[Théorème A]{bergchen}
If $V$ is potentially trianguline, then it satisfies at least one of the following conditions:
\begin{enumerate}
\item $V$ is trianguline;
\item $V$ is the direct sum of two characters or an induced representation; 
\item $V$ is a twist of a de Rham representation by a character.
\end{enumerate}
\end{thm}

With this final ingredient we can prove Proposition \ref{sym3tri}.

\begin{proof}
The proof of (i) is straightforward. We show (ii). Since $\Sym^3V$ is trianguline, $V$ is potentially trianguline by Lemma \ref{sym3pottri}. Then $V$ satisfies one of the three conditions listed in Theorem \ref{berchenpot}. By assumption $V$ is irreducible, so it cannot satisfy (2). Hence (1) or (3) must hold, as desired. 
\end{proof}

\subsection{Representations with symmetric cube of automorphic origin}\label{subsecaut}

Let $\rho_1\colon G_{\Q}\to\GL_2(\Qp)$ and $\rho_2\colon G_{\Q}\to\GSp_4(\Qp)$ be two continuous representations.

\begin{thm}\label{sym3autom}
Suppose that:
\begin{enumerate}
\item $\rho_2$ is odd and it is unramified outside a finite set of primes;
\item the residual representation $\ovl{\rho}_2$ associated with $\rho_2$ is absolutely irreducible; 
\item $\rho_2\cong\Sym^3\rho_1$.
\end{enumerate}
Then the following conclusions hold.
\begin{enumerate}[label=(\roman*)]
\item If $\rho_2$ is associated with an overconvergent cuspidal $\GSp_4$-eigenform, then $\rho_1$ is associated with an overconvergent cuspidal $\GL_2$-eigenform.
\item If $\rho_2$ is associated with a classical cuspidal $\GSp_4$-eigenform, then $\rho_1$ is associated with a classical cuspidal $\GL_2$-eigenform.
\end{enumerate}
\end{thm}

\begin{proof}
Note that assumption (1) implies that the residual representation $\ovl{\rho}_1$ is absolutely irreducible.

We prove part (i). The representation $\rho_2$ is associated with an overconvergent cuspidal $\GSp_4$-eigenform $F$, so it is trianguline by Theorem \ref{modtri}. By Proposition \ref{sym3tri} the representation $\rho_1$ is a twist of a trianguline representation. Then Theorem \cite[Theorem 1.2.4(2)]{emerton} implies that $\rho_1$ is the twist by a character of a representations associated with an overconvergent cuspidal $\GL_2$-eigenform. We show that the character occurring here can be taken to be trivial. 

Let $V$ be a two-dimensional $E$-vector space carrying an action of $G_{\Q_p}$ via $\rho_1$ and let $\ovl{V}$ be the associated residual representation. Let $\alpha\colon G_\Q\to\Qp^\times$ be a character and $N$ be a positive integer such that $V(\alpha)$ is associated with an overconvergent cuspidal $\GL_2$-eigenform $f$ of level $\Gamma_1(N)\cap\Gamma_0(p)$. Let $x$ be the point of the eigencurve $\cD_1^N$ corresponding to $f$. Let $M$ be the positive integer associated with $N$ by Definition \ref{sym3leveldef}. Let $\xi\colon\cD_1^{N,\cG}\to\cD_2^M$ be the morphism of Definition \ref{defxi}. Let $\Sym^3f$ be the overconvergent $\GSp_4$-eigenform corresponding to the point $\xi(x)$. The Galois representation associated with $\Sym^3f$ is $\Sym^3(V(\alpha))$. 

For a continuous representation $W$ of $G_{\Q_p}$, we denote by $\phi_W$ the generalized Sen operator associated with $W$ (see \cite[Section 2.2]{kisin} for the construction). 
Let $(\kappa_1,\kappa_2)$ be the eigenvalues of $\phi_V$. A calculation shows that $\phi_{\Sym^3V}$ has eigenvalues $(3\kappa_1,\kappa_1+2\kappa_2,2\kappa_1+\kappa_2,3\kappa_2)$. Since $\Sym^3V$ is attached to an overconvergent $\GSp_4$-eigenform we must have $3\kappa_1=0$, hence $\kappa_1=0$. Set $\kappa=\kappa_2$, so that the eigenvalues of $\phi_V$ are $(0,\kappa)$. Recall that the weight of the character $\alpha$ is defined by $w(\alpha)=\log(\alpha(u))/\log(u)$, where $u$ is a generator of $\Z_p^\times$. The eigenvalues of $\phi_{V(\alpha)}$ are $(w(\alpha),\kappa+w(\alpha))$. Since $V$ comes from an overconvergent $\GL_2$-eigenform we must have $w(\alpha)=0$. In particular the eigenvalues of $\phi_{\Sym^3V}$ and $\phi_{\Sym^3(V(\alpha))}$ are the same. This means that $\Sym^3V$ and $\Sym^3(V(\alpha))$ are associated with two overconvergent $\GSp_4$-eigenforms $F$ and $\Sym^3f$ of the same weight, given in our usual coordinates by $(\kappa+1,2\kappa-1)$. Let $\chi_{\kappa_1,\kappa_2}$ be the specialization at $(\kappa+1,2\kappa-1)$ of the $p$-adic deformation of the cyclotomic character. The determinants of $\Sym^3V$ and $\Sym^3(V(\alpha))$ are given by the product of $\chi_{\kappa_1,\kappa_2}$ with the central characters of $F$ and $\Sym^3f$, respectively. In particular the two determinants differ by a finite order character. We deduce that $\alpha^6$, hence $\alpha$, is a finite order character. By twisting the overconvergent $\GL_2$-eigenform $f$ by the finite order character $\alpha^{-1}$ we obtain an overconvergent $\GL_2$-eigenform with associated Galois representation $V$. 

We prove part (ii). Since $\rho_2$ is associated with a classical cuspidal $\GSp_4$-eigenform, it is a de Rham representation. Then Proposition \ref{sym3dR} implies that $\rho_1$ is also a de Rham representation. 
The representation $\rho_2$ is trianguline because it is de Rham, so part (i) of the theorem implies that $\rho_1$ is attached to an overconvergent $\GL_2$-eigenform $f$. Since $\rho_1$ is de Rham, the form $f$ is classical. 
\end{proof}

\begin{cor}\label{sym3kimsha}
If $\rho_1$, $\rho_2$ satisfy the assumptions of Theorem \ref{sym3autom} and $\rho_2$ is associated with a classical cuspidal $\GSp_4$-eigenform $F$, then there exists a $\GL_2$-eigenform $f$ such that $F$ is the symmetric cube lift $\Sym^3f$ given by Corollary \ref{formtransf}.
\end{cor}

\begin{proof}
The representation $\rho_1$ is attached to a classical cuspidal $\GL_2$-eigenform $f$ by Theorem \ref{sym3autom}(ii). Then $\rho_2$ is the $p$-adic Galois representation attached to the form $\Sym^3f$. We conclude that $F=\Sym^3f$.
\end{proof}

\bigskip

\section{The symmetric cube locus on the $\mathrm{GSp}_4$-eigenvariety}\label{sym3locus}

In this section $p$ is a prime number, $N$ is a positive integer prime to $p$ and $M$ is the integer, depending on $N$, given by Definition \ref{sym3leveldef}. Let $T_1\colon G_\Q\to\cO(\cD_1^N)$ and $T_2\colon G_\Q\to\cO(\cD_2^M)$ be the pseudocharacters provided by Proposition \ref{biggalthm}. By an abuse of notation, if $\cV_1$ and $\cV_2$ are subvarieties of $\cD_1^N$ and $\cD_2^M$, respectively, we still write $\psi_1\colon\calH_1^N\to\cO(\cV_1)$ and $\psi_2\colon\calH_2^N\to\cO(\cV_2)$ for the compositions of $\psi_1$ and $\psi_2$ with the restrictions of analytic functions to $\cV_1$ and $\cV_2$, respectively. We also write $T_{\cV_1}\colon G_\Q\to\cO(\cV_1)$ and $T_{\cV_2}\colon G_\Q\to\cO(\cV_2)$ for the compositions of $T_1$ and $T_2$ with the restrictions of analytic functions to $\cV_1$ and $\cV_2$, respectively.

\begin{thm}\label{sym3type}
Let $\cV_2$ be a rigid analytic subvariety of $\cD_2^M$. Consider the following four conditions. 
\begin{enumerate}
\item[(1a)] There exists a morphism of rings $\psi_2^{(1)}\colon\calH_1^{Np}\to\cO(\cV_2)$ such that the following diagram commutes:
\begin{equation}\label{sym3typediag}
\begin{tikzcd}[baseline=(current bounding box.center)]
\calH_2^{Np} \arrow{r}{\lambda^{Np}}\arrow{d}{\psi_2}
&\calH_1^{Np} \arrow{dl}{\psi_2^{(1)}}\\
\cO(\cV_2)
&{} 
\end{tikzcd}
\end{equation}
\item[(1b)] There exists a pseudocharacter $T_{\cV_2,1}\colon G_\Q\to\cO(\cV_2)$ of dimension $2$ such that 
\begin{equation}\label{sym3typeeq} T_{\cV_2}=\Sym^3T_{\cV_2,1}. \end{equation}
\item[(2a)] There exists a rigid analytic subvariety $\cV_1$ of $\cD^N_1$ and a morphism of rings $\phi\colon\cO(\cV_1)\to\cO(\cV_2)$ such that the following diagram commutes:
\begin{equation}\label{sym3auttypediag}
\begin{tikzcd}
\calH_2^{Np} \arrow{r}{\lambda^{Np}}\arrow[bend left]{rrr}{\psi_2}
&\calH_1^{Np} \arrow{r}{\psi_1}
&\cO(\cV_1)\arrow{r}{\phi} 
&\cO(\cV_2)
\end{tikzcd}
\end{equation}
\item[(2b)] There exists a rigid analytic subvariety $\cV_1$ of $\cD^N_1$ and a morphism of rings $\phi\colon\cO(\cV_1)\to\cO(\cV_2)$ such that 
\begin{equation}\label{sym3typeauteq} T_{\cV_2}=\Sym^3(\phi\ccirc T_{\cV_1}). \end{equation}
\end{enumerate}
Then:
\begin{enumerate}[label=(\roman*)]
\item (1a) and (1b) are equivalent;
\item (2a) and (2b) are equivalent;
\item (2b) implies (1b);
\item when $\cV_2$ is a point, the four conditions are equivalent.
\end{enumerate}
\end{thm}

\begin{proof}
We prove (i), (ii), (iii) for an arbitrary rigid analytic subvariety $\cV_2$ of $\cD_2^M$.

\medskip

\textbf{(1a) $\implies$ (1b).} Let $\psi_2^{(1)}\colon\calH_1^{Np}\to\cO(\cV_2)$ be a morphism of rings making diagram \eqref{sym3typediag} commute. 
By the argument in the proof of Proposition \ref{heckemorphunr}, the commutativity of diagram \eqref{sym3typediag} gives an equality
\begin{equation}\label{existscube} \psi_2(P_\Min(t_{\ell,2}^{(2)};X))=\Sym^3(\psi_2^{(1)}(P_\Min(t_{\ell,1}^{(1)};X))). \end{equation}

Choose a character $\varepsilon_1$ satisfying $\varepsilon_1^6=\varepsilon$. 
For every $\ell$ not dividing $Np$, let $P_\ell$ be a polynomial in $\calH_2^{Np}[X]^{\deg=2}$ satisfying:
\begin{equation}\label{existscube1} \Sym^3P_\ell(X)=\psi_2(P_\Min(t_{\ell,2}^{(2)};X)); \end{equation}
and
\begin{equation} P_\ell(0)=\varepsilon_1\cdot(1+T)^{\log(\chi(g))/\log(u)}. \end{equation}
Such a polynomial exists thanks to Equation \eqref{existscube} and to Remark \ref{detformula}, and it is clearly unique. The roots of $P_\ell$ differ from those of $\psi_2(P_\Min(t_{\ell,2}^{(2)};X))$ by a factor equal to a cubic root of $1$. 

By Chebotarev's theorem the set $\{\gamma\Frob_\ell\gamma^{-1}\}_{\ell\nmid Np;\,\gamma\in G_\Q}$ is dense in $G_\Q$. The map
\begin{gather*}
P\colon\{\gamma\Frob_\ell\gamma^{-1}\}_{\ell\nmid Np;\,\gamma\in G_\Q}\to\cO(\cV_2)[X]^{\deg=2}, \\
\gamma\Frob_\ell\gamma^{-1}\mapsto P_\ell,
\end{gather*}
is continuous with respect to the restriction of the profinite topology on $G_\Q$. This follows from the fact that the maps 
\begin{gather*}
\{\gamma\Frob_\ell\gamma^{-1}\}_{\ell\nmid Np;\,\gamma\in G_\Q}\to\cO(\cV_2)[X]^{\deg=4} \\
\gamma\Frob_\ell\gamma^{-1}\mapsto\psi_2(P_\Min(t_{\ell,2}^{(2)};X))=\Sym^3P(\gamma\Frob_\ell\gamma^{-1})(X)
\end{gather*}
and 
\begin{gather*}
\{\gamma\Frob_\ell\gamma^{-1}\}_{\ell\nmid Np;\,\gamma\in G_\Q}\to\cO(\cV_2)^\times \\
\gamma\Frob_\ell\gamma^{-1}\mapsto P(\gamma\Frob_\ell\gamma^{-1})(0)=\varepsilon_1\cdot(1+T)^{\log(\chi(g))/\log(u)}
\end{gather*}
are continuous on $\{\gamma\Frob_\ell\gamma^{-1}\}_{\ell\nmid Np;\,\gamma\in G_\Q}$. Hence $P$ can be extended to a continuous map $P\colon G_\Q\to\cO(\cV_2)[X]^{\deg=2}$. 
Now define a map $T_{\cV_2,1}\colon G_\Q\to\cO(\cV_2)$ by $T_{\cV_2,1}(g)=(P(g)(1)+P(g)(-1))/2$. 
We can check that $T_{\cV_2,1}$ is a pseudocharacter of dimension $2$. Its characteristic polynomial is $P$, so the fact that $T_{\cV_2}=\Sym^3T_{\cV_2,1}$ follows from Equation \eqref{existscube1}.

\medskip

\textbf{(1b) $\implies$ (1a).} Suppose that there exists a pseudocharacter $T_{\cV_2,1}\colon G_\Q\to\cO_{\cV_2}$ such that $T_{\cV_2}=\Sym^3 T_{\cV_2,1}$. 
Then $P_\car(T_{\cV_2})=\Sym^3P_\car(T_{\cV_2,1})$. By evaluating the two polynomials at $\Frob_\ell$ we obtain 
\begin{equation}\label{T21}\begin{gathered}
\psi_2(P_\Min(t_{\ell,2}^{(2)};X))=P_\car(T_{\cV_2})(\Frob_\ell)=\Sym^3P_\car(T_{\cV_2,1})(\Frob_\ell)= \\
=\Sym^3\left(X^2-T_{\cV_2,1}(\Frob_\ell)X+\frac{T_{\cV_2,1}(\Frob_\ell)^2-T_{\cV_2,1}(\Frob_\ell^2)}{2}\right), 
\end{gathered}\end{equation}
where the first equality is given by Proposition \ref{biggalthm} and the last one comes from Equation \eqref{detpseudo}. 
Let $\psi_2^{(1)}\colon\calH_1^{Np}\colon\cO(\cV_2)$ be a morphism of rings satisfying  
\begin{equation}\label{psi21}
X^2-T_{\cV_2,1}(\Frob_\ell)X+\frac{T_{\cV_2,1}(\Frob_\ell)^2-T_{\cV_2,1}(\Frob_\ell^2)}{2}=X^2-\psi_2^{(1)}(T_{\ell,1}^{(1)})X+\ell\psi_2^{(1)}(T_{\ell,0}^{(1)}) 
\end{equation}
for every $\ell\nmid Np$. It is clear that such a morphism exists and is unique. Note that the right hand side of Equation \eqref{psi21} is $\psi_2^{(1)}(P_\Min(t_{\ell,1}^{(1)};X))$. Then Equation \eqref{T21} gives
\[ \psi_2(P_\Min(t_{\ell,2}^{(2)};X))=\Sym^3(\psi_2^{(1)}(P_\Min(t_{\ell,1}^{(1)};X))). \]
Exactly as in the proof of Proposition \ref{heckemorphunr}, by developing the two polynomials and comparing their coefficients we obtain that $\psi_2=\psi_2^{(1)}\ccirc\lambda^{Np}$. Hence $\psi_2^{(1)}$ fits into diagram \eqref{sym3typediag}.

\medskip

\textbf{(2a) $\iff$ (2b).} Let $\cV_1$ be a subvariety of $\cD^N_1$ and let $\phi\colon\cO(\cV_1)\to\cO(\cV_2)$ be a morphism of rings. We show that the couple $(\cV_1,\phi)$ satisfies (2a) if and only if it satisfies (2b). 
For $g=1,2$ and every prime $\ell\nmid Np$ Proposition \ref{biggalthm} gives
\begin{equation}\label{frobdet} P_\car(T_{\cV_g})(\Frob_\ell)=\psi_g(P_\Min(t_{\ell,g}^{(g)};X)). \end{equation}
The argument in the proof of Proposition \ref{heckemorphunr} gives an equality
\begin{equation}\label{lambdaNp} \lambda^{Np}(P_\Min(t_{\ell,2}^{(2)};X))=\Sym^3(P_\Min(t_{\ell,1}^{(1)};X)). \end{equation}
Since the set $\{\gamma\Frob_\ell\gamma^{-1}\}_{\ell\nmid Np;\,\gamma\in G_\Q}$ is dense in $G_\Q$, the pseudocharacters $\Sym^3(\phi\ccirc T_{\cV_1})$ and $T_{\cV_2}$ coincide if and only if their characteristic polynomials coincide on $\Frob_\ell$ for every $\ell\nmid Np$. 
By Equation \eqref{frobdet} the condition above is equivalent to
\[ \Sym^3(\phi\ccirc\psi_1(P_\Min(t_{\ell,1}^{(1)};X)))=\psi_2(P_\Min(t_{\ell,2}^{(2)};X)) \]
for every $\ell\nmid Np$. 
Thanks to Equation \eqref{lambdaNp} the left hand side can be rewritten as 
\[ \Sym^3(\phi\ccirc\psi_1(P_\Min(t_{\ell,1}^{(1)};X)))=\phi\ccirc\psi_1(\Sym^3(P_\Min(t_{\ell,1}^{(1)};X)))=\phi\ccirc\psi_1\ccirc\lambda^{Np}(P_\Min(t_{\ell,2}^{(2)};X)). \]
When $\ell$ varies over the primes not dividing $Np$ the coefficients of the polynomials $P_\Min(t_{\ell,2}^{(2)};X)$ generate the Hecke algebra $\calH_2^{Np}$. Hence the equality of the right hand sides of the last two equations holds if and only if $\phi\ccirc\psi_1\ccirc\lambda^{Np}=\psi_2$.

\medskip

\textbf{(2b) $\implies$ (1b).} Suppose that condition (2b) is satisfied by some closed subvariety $\cV_1$ of $\cD_1^N$ and some morphism of rings $\phi\colon\cO(\cV_1)\to\cO(\cV_2)$. Consider the pseudocharacter $T_{\cV_2,1}=\phi\ccirc T_{\cV_1}\colon G_\Q\to\cO(\cV_2)$. Clearly $T_{\cV_2,1}$ satisfies condition (1b).

\medskip

It remains to prove that (1b) $\implies$ (2b) when $\cV_2$ is a $\Qp$-point of $\cD_2^M$. For this step we will need the results we recalled in Section \ref{triang}. 
Write $x_2$ for the point $\cV_2$; the system of eigenvalues $\psi_{x_2}$ is that of a classical $\GSp_4$-eigenform. 
By Remark \ref{nonclassrep}(1) $T_{x_2}$ is the pseudocharacter associated with a representation $\rho_{x_2}\colon G_\Q\to\GL_4(\Qp)$. Let $E$ be a finite extension of $\Q_p$ over which $\rho_{x_2}$ is defined. 
Suppose that $x_2$ satisfies condition (1b). Let $T_{x_2,1}\colon G_\Q\to\Qp$ be a pseudocharacter such that $T_{x_2}\cong\Sym^3T_{x_2,1}$. By Theorem \ref{pseudotaylor} there exists a representation $\rho_{x_2,1}\colon G_\Q\to\GL_2(\Qp)$ such that $T_{x_2,1}=\Tr(\rho_{x_2,1})$. Then Remark \ref{sym3pseudorem} implies that $\rho_{x_2}\cong\Sym^3\rho_{x_2,1}$. 
Since $\rho_{x_2}$ is attached to an overconvergent $\GSp_4$-eigenform, Theorem \ref{sym3autom}(ii) implies that $\rho_{x_2,1}$ is the $p$-adic Galois representation attached to an overconvergent $\GL_2$-eigenform $f$. 
Such a form defines a point $x_1$ of the eigencurve $\cD_1^N$.
Thus the subvariety $\cV_1=x_1$ satisfies condition (2b). 
\end{proof}

\begin{rem}\label{sym3subvar}
The four properties stated in Theorem \ref{sym3type} are stable when passing to a subvariety, in the following sense. Let $\cV_2$ and $\cV_2^\prime$ be two rigid analytic subvarieties of $\cD_2^M$ satisfying $\cV_2^\prime\subset\cV_2$. Let $(\ast)$ denote one of the conditions of Theorem \ref{sym3type}. If $(\ast)$ holds for $\cV_2$ then it holds for $\cV_2^\prime$. Thanks to the theorem it is sufficient to prove this statement for $\ast=1b$ and $\ast=2b$, in which cases it is trivial.
\end{rem}

In light of Theorem \ref{sym3type} we give the following definitions.

\begin{defin}\label{sym3locdef}
\begin{enumerate}
\item We say that a subvariety $\cV_2$ of $\cD_2^M$ is of \emph{$\Sym^3$ type} if it satisfies the equivalent conditions (2a) and (2b) of Theorem \ref{sym3type}.
\item The \emph{$\Sym^3$-locus} of $\cD_2^M$ is the set of points of $\cD_2^M$ of $\Sym^3$ type. 
\end{enumerate}
\end{defin}

\begin{rem}\label{sym3defrem}
A variety $\cV_2$ of $\Sym^3$ type also satisfies conditions (1a) and (1b) of Theorem \ref{sym3type} thanks to the implication (2b) $\implies$ (1b).
\end{rem}

Let $\iota\colon\cW_1^\circ\to\cW_2^\circ$ is the closed immersion constructed in Section \ref{aux2}. Let $\cD_{2,\iota}^M$ be the one-dimensional subvariety of $\cD_2^M$ fitting in the cartesian diagram
\begin{center}
\begin{tikzcd}[baseline=(current bounding box.center)]
\cD_{2,\aux}^{M} \arrow{d}\arrow{r}
&\cD_{2}^{M} \arrow{d}{w_2}\\
\iota(\cW_1^\circ)\arrow{r}{\iota}
&\cW_2^\circ
\end{tikzcd}
\end{center}

The following lemma follows from a simple computation involving the generalized Hodge-Tate weights of a point of $\Sym^3$ type.

\begin{lemma}\label{sym3iota}
The $\Sym^3$-locus of $\cD_2^M$ is contained in the one-dimensional subvariety $\cD_{2,\iota}^M$.
\end{lemma}
%

The $\Sym^3$-locus of $\cD_2^M$ admits a Hecke-theoretic definition thanks to condition (2b) of Theorem \ref{sym3type}. We elaborate on this.
Consider the following maps:
\begin{center}
\begin{tikzcd}[baseline=(current bounding box.center)]\label{diagIsym3}
\calH_2^{Np} \arrow{d}{\lambda^{Np}}\arrow{r}{\psi_2}
&\cO(\cD^M_2)
\\
\calH_1^{Np} \arrow{r}{\psi_1}
&\cO(\cD^N_1)
\end{tikzcd}
\end{center}
We define an ideal $\cI_{\Sym^3}$ of $\cO(\cD^M_2)$ by
\[ \cI_{\Sym^3}=\psi_1(\ker(\psi_2\ccirc\lambda^{Np}))\cdot\cO(\cD^M_2). \] 
We denote by $\cD_{2,\Sym^3}^M$ the analytic Zariski subvariety of $\cD_2^M$ defined as the zero locus of the ideal $\cI_{\Sym^3}$. 

\begin{prop}\label{sym3id}\mbox{ }
\begin{enumerate}[label=(\roman*)]
\item The $\Sym^3$-locus of $\cD_2^M$ is the set of points underlying $\cD_{2,\Sym^3}^M$.
\item The variety $\cD_{2,\Sym^3}^M$ is of $\Sym^3$ type.
\item A rigid analytic subvariety $\cV_2$ of $\cD_2^M$ is of $\Sym^3$ type if and only if it is a subvariety of $\cD_{2,\Sym^3}^M$. 
\item A rigid analytic subvariety $\cV_2$ of $\cD_2^M$ satisfies conditions (1a) and (1b) of Theorem \ref{sym3type} if and only if it is a subvariety of $\cD_{2,\Sym^3}^M$.
\end{enumerate}
\end{prop}

\begin{proof}
We prove (i). Let $x_2$ be any $\Qp$-point of $\cD_2^M$ and let $\ev_{x_2}\colon\cO(\cD_2^M)\to\Qp$ be the evaluation at $x_2$. 
The system of eigenvalues corresponding to $x_2$ is $\psi_{x_2}=\ev_{x_2}\ccirc\psi_2\colon\calH_2^{Np}\to\Qp$. By definition $x_2$ is of $\Sym^3$ type if and only if there exists a morphism of rings $\ev_{x_1}\colon\cO(\cD_1^N)\to\Qp$ such that the following diagram commutes:
\begin{center}
\begin{tikzcd}[baseline=(current bounding box.center)]
\calH_2^{Np} \arrow{d}{\lambda^{Np}}\arrow{r}{\psi_2}
&\cO(\cD^M_2)\arrow{r}{\ev_{x_2}}
&\Qp
\\
\calH_1^{Np} \arrow{r}{\psi_1}
&\cO(\cD^N_1)\arrow{ru}{\ev_{x_1}}
&{}
\end{tikzcd}
\end{center}
By elementary algebra the map $\ev_{x_1}$ exists if and only if $\ev_{x_2}(\ker(\psi_2\ccirc\lambda^{Np}))=0$. This is equivalent to the fact that the point $x_2$ is in the zero locus of the ideal $\cI_{\Sym^3}$.

For (ii) it is sufficient to observe that there exists a morphism of rings $\Xi^\ast_{\Sym^3}\colon\cO(\cD_1^N)\to\cO(\cD_{2,\Sym^3}^M)$ fitting into the commutative diagram
\begin{equation}\label{arrowsym3}
\begin{tikzcd}[column sep=large]
\calH_2^{Np} \arrow{r}{\lambda^{Np}}\arrow[bend left]{rrr}{r_{\Sym^3}\ccirc\psi_2}
&\calH_1^{Np} \arrow{r}{\psi_1}
&{\cO(\cD_1^N)}\arrow{r}{\Xi^\ast_{\Sym^3}}
&\cO(\cD_{2,\Sym^3}^M)
\end{tikzcd} 
\end{equation}
Such a $\Xi^\ast_{\Sym^3}$ exists since by definition of $\cD_{2,\Sym^3}^M$ we have $r_{\cD_{2,\Sym^3}^M}\ccirc\psi_2(\ker(\lambda^{Np}\ccirc\phi_{\Sym^3}))=0$.

Note that the ``if'' implications of (iii) and (iv) follow from Lemma \ref{sym3subvar}, together with Remark \ref{sym3defrem} for (iv). 

To prove the other direction of (iii) we look again at diagram \eqref{sym3auttypediag} for a subvariety $\cV_2$ of $\cD_2^M$. In order for $\cV_2$ to satisfy condition (2a) of Theorem \ref{sym3type} we must have $r_{\cV_2}(\ker(\lambda^{Np}\ccirc\Xi^\ast_{\Sym^3}))=0$, so $\cV_2$ is contained in $\cD_{2,\Sym^3}^M$. 

Finally, let $\cV_2$ be a rigid analytic subvariety of $\cD_2^M$ satisfying conditions (1a) and (1b) of Theorem \ref{sym3type}. Let $x_2$ by a point of $\cV_2$. By Lemma \ref{sym3subvar} $x_2$ satisfies conditions (1a) and (1b). By Theorem \ref{sym3type}, $x_2$ also satisfies conditions (2a,\, 2b), so it is a point of $\cD_{2,\Sym^3}^M$. We conclude that $\cV_2$ is a subvariety of $\cD_{2,\Sym^3}^M$. 
\end{proof}

\begin{rem}
By Proposition \ref{sym3id} the $\Sym^3$-locus in $\cD_2^M$ can be given the structure of a Zariski-closed rigid analytic subspace. From now on we will always consider the $\Sym^3$-locus as equipped with this structure and we will identify it with the subvariety $\cD_{2,\Sym^3}^M$ of $\cD_2^M$.
\end{rem}

Proposition \ref{sym3id}(i) and Lemma \ref{sym3iota} give the following.

\begin{cor}
The $\Sym^3$-locus intersects each irreducible component of $\cD_2^M$ in a proper analytic Zariski subvariety of dimension at most $1$.
\end{cor}



Propositions \ref{sym3id}(iii) and (iv) allows us to improve the result of Theorem \ref{sym3type}.

\begin{cor}\label{sym3proper}
For every rigid analytic subvariety $\cV_2$ of $\cD_2^M$ the conditions (1a), (1b), (2a), (2b) of Theorem \ref{sym3type} are equivalent.
\end{cor}
%

\subsection{Equations for the symmetric cube locus}

By using the description of $\lambda^{Np}$ given in Definition \ref{heckemorphunrdef} we can write ``equations'' for the $\Sym^3$ locus in terms of the rigid analytic functions $\psi_2(T_{\ell,i}^{(2)})$ for $i\in\{0,1,2\}$, but they are not very illuminating. We find that the ideal $\cI_{\Sym^3}$ of $\cO(\cD_2^M)$ is generated by the elements of the set
\begin{gather*} \{\psi_2((T_{\ell,2}^{(2)})^6+(\ell^2+4\ell+8)T_{\ell,0}^{(2)}(T_{\ell,2}^{(2)})^4+3T_{\ell,1}^{(2)}(T_{\ell,2}^{(2)})^4-(5\ell^4+12\ell^3)(T_{\ell,0}^{(2)})^2(T_{\ell,2}^{(2)})^2+ \\
-(2\ell^2-4\ell)T_{\ell,0}^{(2)}T_{\ell,1}^{(2)}(T_{\ell,2}^{(2)})^4-3(T_{\ell,1}^{(2)})^2(T_{\ell,2}^{(2)})^2+(3\ell^2T_{\ell,0}^{(2)}+T_{\ell,1}^{(2)})(\ell^2T_{\ell,0}^{(2)}-T_{\ell,1}^{(2)})^2\}_{\ell\nmid Np}. \end{gather*}

By the result of Proposition \ref{heckemorphtri}, the intersection $\widetilde\cD_{2,\Sym^3}^{M,\irr}$ of the $\Sym^3$ locus with $\widetilde\cD_2^{M,\irr}$ can be decomposed as a union $\bigcup_{j=1}^4\widetilde\cD_{2,\Sym^3}^{M,j}$, where $\widetilde\cD_{2,\Sym^3}^{M,j}$ is defined by the property 
\[ x\in\widetilde\cD_{2,\Sym^3}^{M,j}(\C_p) \iff x\in\widetilde\cD_2^{M,\irr}\cap\widetilde\cD_{2,\Sym^3}^M\textrm{ and the choice }i=j\textrm{ satisfies Equation }\eqref{heckemorphtrieq} \]
Using the description of $\lambda_{i,p}$, $1\le i\le 4$, given in Definition \ref{heckemorphdef} we can write $\widetilde\cD_{2,\Sym^3}^{M,j}(\C_p)=\{x\in\widetilde\cD_{2,\Sym^3}^{M,\irr}(\C_p)\,\vert\,E_j=0\}$ where
\begin{gather*}
E_1=(U_{p,0}^{(2)})^3(U_{p,1}^{(2)})^3-(U_{p,2}^{(2)})^4, \quad E_2=(U_{p,0}^{(2)})^2(U_{p,2}^{(2)})^2-(U_{p,1}^{(2)})^3, \\
E_3=U_{p,0}^{(2)}U_{p,1}^{(2)}-(U_{p,2}^{(2)})^3, \quad E_4=(U_{p,0}^{(2)})^2-U_{p,1}^{(2)}(U_{p,2}^{(2)})^2. 
\end{gather*}

Define $\cD_{2,\Sym^3}^{M,j}$ as the Zariski-closure of $\widetilde\cD_{2,\Sym^3}^{M,j}$ in $\cD_{2,\Sym^3}^{M}$. Note that by taking this closure we are just adding a discrete set of points.

\subsection{An inverse to the symmetric cube morphism of eigenvarieties}

Consider the map $\Xi^\ast_{\Sym^3}\colon\cO(\cD_1^N)\to\cO(\cD_{2,\Sym^3}^M)$ appearing in the commutative diagram \eqref{arrowsym3}; it induces a map of rigid analytic spaces $\Xi_{\Sym^3}\colon\cD_{2,\Sym^3}^M\to\cD_1^N$. 

\begin{rem}\label{dimsym}
By Corollary \ref{dimcryssym}, the map $\Xi_{\Sym^3}$ is quasi-finite and its degree at a point $x\in\cD_1^N(\C_p)$ is at most $\dim_{\Qp}(\bD_\cris(\rho_x\vert_{G_{\Q_p}}))$, where $\rho_x\colon G_\Q\to\GSp_4(\Qp)$ is the Galois representation attached to $x$.
\end{rem}


Let $\ovl\rho_1\colon G_\Q\to\GL_4(\Fp)$ be a representation. 
Let $\ovl\rho_2=\Sym^3\ovl\rho_1$. 
Consider the union $\cD_{2,\ovl\rho_2}^M$ of the connected components of $\cD_2^M$ of residual Galois representation isomorphic to $\ovl\rho_2$. We constructed in Section \ref{morpheigen} a map $\xi$ from the eigencurve $\cD_{1,\ovl\rho}^N$ to the eigenvariety $\cD_{2,\ovl\rho_2}^M$. Let $\cD_{2,\ovl\rho_2,\Sym^3}^M=\cD_{2,\Sym^3}^M\cap\cD_{2,\ovl\rho_2}^M$ and $\cD_{2,\ovl\rho_2,\Sym^3}^{M,j}=\cD_{2,\Sym^3}^{M,j}\cap\cD_{2,\ovl\rho_2}^M$.

\begin{prop}
\begin{enumerate}
\item The rigid analytic space $\cD_{2,\ovl\rho_2,\Sym^3}^{M,j}$ is equidimensional of dimension $1$ if $j=1$ and it is $0$-dimensional otherwise.
\item The image of $\xi$ in $\cD_{2,\ovl\rho_2}^M$ is $\cD_{2,\ovl\rho_2,\Sym^3}^{M,j}$. 
\end{enumerate}
\end{prop}

\begin{proof}
We prove the two statements together. We replace implicitly the sets $S_i^{\Sym^3}$, $1\le i\le 4$, defined before Corollary \ref{stabdense}, by their intersections with $\cD_{2,\ovl\rho_2}^M$.
By construction of $\xi$, the image of this map is the Zariski-closure in $\cD_{2,\ovl\rho_2}^M$ of the set $S^{\Sym^3}_1$, that consists of points of $\Sym^3$ type by definition. Since $\cD_{2,\ovl\rho_2,\Sym^3}^{M,1}$ is Zariski-closed in $\cD_{2,\ovl\rho_2}^M$ and contains $S^{\Sym^3}_1$, the Zariski-closure of $S^{\Sym^3}_1$ is contained in $\cD_{2,\ovl\rho_2,\Sym^3}^{M,1}$.
This proves one inclusion of (2).

Let be a $1$-dimensional irreducible component of $\cD_{2,\ovl\rho_2,\Sym^3}^{M,j}$. The set $S$ of classical, $p$-old points of $C$ is Zariski-dense in $C$ because of Theorem \ref{siegclslopes}. Every point in $S$ is of $\Sym^3$ type, so it is the symmetric cube lift of a classical $p$-old point of $\cD_{1,\ovl\rho_1}^N$ by Theorem \ref{sym3type}. This means that $S$ is contained in $\bigcup_{i=1}^4S_i^{\Sym^3}$. Since $E_j$ vanishes on $\cD_{2,\ovl\rho_2,\Sym^3}^{M,j}$, we must have $S\subset S_j^{\Sym^3}$. If $j\neq 1$ then $S_j^{\Sym^3}$ is discrete by Corollary \ref{stabdense}, so we conclude that $j=1$. This proves that every 1-dimensional component of $\cD_{2,\ovl\rho_2,\Sym^3}^{M}$ is contained in $\cD_{2,\ovl\rho_2,\Sym^3}^{M,1}$ and is not contained in $\cD_{2,\ovl\rho_2,\Sym^3}^{M,j}$ for any $j\neq 1$. In particular $\cD_{2,\ovl\rho_2,\Sym^3}^{M,j}$ is $0$-dimensional if $j\neq 1$. We also obtain that $\cD_{2,\ovl\rho_2,\Sym^3}^{M,1}$ is the Zariski-closure of $S_j^{\Sym^3}$.

Let $x$ be a point of $\cD_{2,\ovl\rho_2,\Sym^3}^{M,1}$, not necessarily classical. Proposition \ref{heckemorphtri} implies that there exists a point $x_1$ of $\cD_{1,\ovl\rho_1}^N$ such that $\chi_{x}\ccirc\iota_{I_{2,p}}^{T_2}=(\chi_{x_1}\ccirc\iota_{I_{1,p}}^{T_1})^\ext\ccirc\lambda_{1,p}$, with the usual notations for systems of Hecke eigenvalues. This means that $x_1$ is the image of $x$ via $\xi$, hence the remaining inclusion of (2) holds. Let $I$ be an irreducible component of $\cD_{1,\ovl\rho_1}^N$ containing $x_1$. Since $\xi$ is a closed immersion and its image is contained in $\cD_{2,\ovl\rho_2,\Sym^3}^{M,1}$, $\xi(I)$ is an irreducible component of $\cD_{2,\ovl\rho_2}^{M,1,\Sym^3}$ containing $x_1$. We deduce that $\cD_{2,\ovl\rho_2,\Sym^3}^{M,1}$ is equidimensional of dimension $1$. 
\end{proof}

\begin{cor}
The morphisms of rigid analytic spaces $\xi$ and $\Xi\vert_{\cD_{2,\ovl\rho_2}^{M,1}}$ are inverses to one another. In particular $\Xi$ is an isomorphism on the Zariski-open subspace $\cD_{2,\ovl\rho_2,\Sym^3}^{M,1}$ of $\cD_{2,\ovl\rho_2,\Sym^3}^{M}$.
\end{cor}

\bigskip

\section{The fortuitous $\Sym^3$-congruence ideal of a finite slope family}\label{fortcong}

Let $\theta\colon\T_h\onto\I^\circ$ be a finite slope family and let $\rho\colon G_\Q\to\GSp_4(\I^\circ_\Tr)$ be the representation associated with $\theta$ in the previous section. Recall that $\ovl{\rho}$ is absolutely irreducible by assumption. We also assume that $\rho$ is $\Z_p$-regular and of residual $\Sym^3$ type, as in Definitions \ref{Zpreg} and \ref{sctype}. 
In this section we define a ``fortuitous congruence ideal'' for the family $\theta$. It is the ideal describing the intersection of the $\Sym^3$-locus of $\cD_2^M$ with the family $\theta$. Recall that the $\Sym^3$-locus is the zero locus of the ideal $\cI_{\Sym^3}$ of $\cO(\cD_2^{M})^\circ$ defined in Section \ref{sym3locus} and that $r_{\cD_{2,B_h}^{M,h}}\colon\cO(\cD_2^{M})^\circ\to\T_h$ denotes the restriction of analytic functions. 

\begin{defin}\label{sym3cong}
The \emph{fortuitous $\Sym^3$-congruence ideal} for the family $\theta\colon\T_h\to\I^\circ$ is the ideal of $\I^\circ$ defined by
\[ \fc_\theta=(\theta\ccirc r_{\cD_{2,B_h}^{M,h}})(\cI_{\Sym^3})\cdot\I^\circ. \]
\end{defin}

In most cases we will simply refer to $\fc_\theta$ as the ``congruence ideal''. 
The next proposition describes its main properties.
Let $\fI$ be an ideal of $\I^\circ$ and let $\fI_\Tr=\fI\cap\I^\circ_\Tr$. Let $\rho_\fI\colon G_\Q\to\GSp_4(\I_\Tr^\circ/\fI_\Tr)$ be the reduction of $\rho$ modulo $\fI$. If $\theta_1\colon\T_{h,1}\to\J$ is a finite slope family of $\GL_2$-eigenforms we denote by $\rho_{\theta_1}\colon G_\Q\to\GL_2(\J)$ the associated Galois representation. For an ideal $\cJ$ of $\J$ we let $\rho_{\theta_1,\cJ}\colon G_\Q\to\GL_2(\J/\cJ)$ be the reduction of $\rho_{\theta_1}$ modulo $\cJ$. 

\begin{prop}\label{fcchar}
The following are equivalent:
\begin{enumerate}[label=(\roman*)]
\item $\fI\supset\fc_\theta$;
\item there exists a finite extension $\I^\prime$ of $\I_\Tr^\circ/\fI_\Tr$ and a representation $\rho_{\fI,1}\colon G_\Q\to\GL_2(\I^\prime)$ such that $\rho_\fI\cong\Sym^3\rho_{\fI,1}$ over $\I^\prime$; 
\item there exists a finite slope family of $\GL_2$-eigenforms $\theta_1\colon\T_{h/7,1}\to\J^\circ$, an ideal $\fJ$ of $\J^\circ$ and a map $\phi\colon\J^\circ/\fJ\to\I_\Tr^\circ$ such that $\rho_\fI\cong\phi\ccirc\Sym^3\rho_{\theta_1,\fJ}$ over $\I_\Tr^\circ$. 
\end{enumerate}
\end{prop}

Note that we did not specify the image in the weight space of the admissible subdomain of $\cD_1^N$ associated with the family $\theta_1$. It is the preimage in $\cW_1^\circ$ of the disc $B_{2,h}$ via the immersion $\iota\colon\cW_1^\circ\to\cW_2^\circ$ defined in Section \ref{aux2}.

\begin{proof}
Since all the coefficient rings are local and all the residual representations are absolutely irreducible, we can apply the results of Section \ref{sym3locus} by replacing the pseudocharacters everywhere with the associated representations, that exist by Theorem \ref{pseudolift} and are defined over the ring of coefficients of the pseudocharacter by Theorem \ref{carayol} (see the argument in the beginning of Section \ref{galrepfam}). 

Now the equivalence (i) $\iff$ (ii) follows from Proposition \ref{sym3id}(iv) applied to the rigid analytic variety $\cV_2=I$. The equivalence (ii) $\iff$ (iii) follows from Proposition \ref{sym3id}(iii) by checking that the slopes satisfy the required inequality: this is a consequence of Corollary \ref{liftslopes} and Remark \ref{interpslopes}. 
\end{proof}

\begin{cor}\label{trivcong}
If there is no representation $\ovl\rho_1\colon G_\Q\to\GL_2(\Fp)$ satisfying $\ovl\rho\cong\Sym^3\ovl\rho_1$, then $\fc_{\theta}=\I^\circ$.
\end{cor}

\begin{prop}\label{fcInonzero}
The ideal $\fc_{\theta}$ is non-zero.
\end{prop}

\begin{proof}
Suppose by contradiction that $\fc_{\theta}=0$. 
Since $\fc_\theta=(\theta\ccirc r_{\cD_{2,B_h}^{M,h}})(\cI_{\Sym^3})\cdot\I^\circ$ we must have $(\theta\ccirc r_{\cD_{2,B_h}^{M,h}})(\cI_{\Sym^3})=0$. This means that the $2$-dimensional family $I$ is contained in the zero locus $\cD_{2,\Sym^3}^{M}$ of $\cI_{\Sym^3}$. This is impossible by Lemma \ref{sym3iota}. 
\end{proof}


The fortuitous $\Sym^3$-congruence ideal is an analogue of the congruence ideal of \cite[Definition 3.10]{cit}. 
There is an important difference between the situation studied here and in \cite{cit} and those treated in \cite{hida,hidatil}. In \cite{hida,hidatil} the congruence ideal describes the locus of intersection between a fixed ``general'' family (i.e. such that its specializations are not lifts of forms from a smaller group) and the ``non-general'' families. Such non-general families are obtained as the $p$-adic lift of families of overconvergent eigenforms for smaller groups (e.g. $\GL_{1/K}$ for an imaginary quadratic field $K$ in the case of CM families of $\GL_2$-eigenforms, as in \cite{hida}, and $\GL_{2/F}$ for a real quadratic field $F$ in the case of ``twisted Yoshida type'' families of $\GSp_4$-eigenforms, as in \cite{hidatil}). In our setting there are no non-general families: the overconvergent $\GSp_4$-eigenforms that are lifts of overconvergent eigenforms for smaller groups must be of $\Sym^3$ type by Lemma \ref{sym3zar} and Theorem \ref{sym3autom}, and we know that the $\Sym^3$-locus on the $\GSp_4$-eigenvariety does not contain any two-dimensional irreducible component by Proposition \ref{fcInonzero}. Hence the ideal $\fc_\theta$ measures the locus of points that are of $\Sym^3$ type, without belonging to a two-dimensional family of $\Sym^3$ type. For this reason we call it the ``fortuitous'' $\Sym^3$-congruence ideal. 
This is a higher-dimensional analogue of the situation of \cite{cit}, where it is shown that the positive slope CM points do not form one-dimensional families but appear as isolated points on the irreducible components of the eigencurve (see \cite[Corollary 3.6]{cit}). 


Note that conditions (ii) and (iii) in Proposition \ref{fcchar} only depend on the ideal $\fI\cap\I^\circ_\Tr$, so we expect $\fc_\theta$ to be generated by elements of $\I_\Tr^\circ$. We prove this in the following.

\begin{prop}
Let $\fc_{\theta,\Tr}=\fc_\theta\cap\I_\Tr^\circ$. Then $\fc_\theta=\fc_{\theta,\Tr}\cdot\I^\circ$.
\end{prop}

\begin{proof}
By definition $\fc_{\theta,\Tr}=\theta\ccirc r_I(\cI_{\Sym^3})\cdot\I^\circ$. By definition $\cI_{\Sym^3}=\psi_2(\ker(\psi_1\ccirc\lambda^{Mp}))$, where the notations are as in diagram \eqref{diagIsym3}. Since $\ker(\psi_1\ccirc\lambda^{Mp})\subset\calH_2^{Mp}$ we have 
\[ \theta\ccirc r_I(\cI_{\Sym^3})=\theta\ccirc r_I\ccirc\psi_2(\ker(\psi_1\ccirc\lambda^{Mp}))\subset\theta\ccirc r_I\ccirc\psi_2(\calH_2^{Mp}). \] 
By the remarks of Section \ref{galrepfam} the ring $\I_\Tr^\circ$ contains $\theta\ccirc r_I\ccirc\psi_2(\calH_2^{Mp})$ in $\I^\circ$, so $\theta\ccirc r_I(\cI_{\Sym^3})$ is a subset of $\I_\Tr^\circ$ and the ideal $\fc_{\theta,\Tr}=\theta\ccirc r_I(\cI_{\Sym^3})\cdot\I^\circ_\Tr$ satisfies $\fc_\theta=\fc_{\theta,\Tr}\cdot\I^\circ$.
\end{proof}

Proposition \ref{fcchar} can be translated into a characterization of the ideal $\fc_{\theta,\Tr}$. For an ideal $\fI$ of $\I_\Tr^\circ$ let $\rho_\fI\colon G_\Q\to\GSp_4(\I_\Tr^\circ/\fI)$ be the reduction of $\rho$ modulo $\fI$.

\begin{cor}\label{fctrchar}
Let $\fI$ be an ideal of $\I_\Tr^\circ$. 
The following are equivalent:
\begin{enumerate}[label=(\roman*)]
\item $\fI\supset\fc_{\theta,\Tr}$;
\item there exists a finite extension $\I^\prime$ of $\I^\circ/\fI$ and a representation $\rho_{\fI,1}\colon G_\Q\to\GL_2(\I^\prime)$ such that $\rho_\fI\cong\Sym^3\rho_{\fI,1}$ over $\I^\prime$;
\item there exists a finite slope family of $\GL_2$-eigenforms $\theta_1\colon\T_{h/7,1}\to\J^\circ$, an ideal $\fJ$ of $\J^\circ$ and a map $\phi\colon\J^\circ/\fJ\to\I_\Tr^\circ$ such that $\rho_\fI\cong\phi\ccirc\Sym^3\rho_{\theta_1,\fJ}$.
\end{enumerate}
\end{cor}

We use the results of Section \ref{morpheigen} to obtain some information on the height of the prime divisors of $\fc_{\theta}$. Here $\iota\colon\cW_1^\circ\to\cW_2^\circ$ is the inclusion defined in Section \ref{aux2}. For a classical weight $k$ in $\cW_1^\circ$ we have $\iota(k)=(k+1,2k-1)$, with the obvious abuse of notation.

\begin{prop}
Suppose that there exists a non-CM classical point $x\in\cD^N_1$ of weight $k$ such that $\slo(x)\leq h/7$ and $\iota(k)\in B_{2,h}$ and $k>h-4$. Then the ideal $\fc_\theta$ has a prime divisor of height $1$.
\end{prop}

\begin{proof}
Let $x$ be a point satisfying the assumptions of the proposition and let $f$ be the corresponding classical $\GL_2$-eigenform. Let $\Sym^3x$ be the point of $\cD_2^M$ that corresponds to the form $(\Sym^3f)^\st_1$ defined in Corollary \ref{heckemorphprodcor}. Let $\xi\colon\cD_1^{N,\cG}\to\cD_2^M$ be the map of rigid analytic spaces given by Definition \ref{defxi}. The image of an irreducible component $J$ of $\cD_1^{N,\cG}$ containing $x$ is an irreducible component $\xi(J)$ of $\cD_2^M$ that contains $\Sym^3x$. By Corollary \ref{liftslopes} we have $\slo(\Sym^3x)\leq h$. Since $k+1>h-3$ the weight map is étale at the point $\Sym^3x$, so there exists only one finite slope family of $\GSp_4$-eigenforms containing $\Sym^3x$. This means that $\xi(J)$ intersects the admissible domain $I$ in a one-dimensional subspace. The ideal of $\I^\circ=\cO(I)^\circ$ consisting of elements that vanish on $\xi(J)$ is a height one ideal of $I$ that divides the congruence ideal $\fc_\theta$. In particular $\fc_\theta$ admits a height one prime divisor.
\end{proof}

\subsection{The $\I_0^\circ$-congruence ideal}

Starting with Corollary \ref{fctrchar} we can descend further and prove that $\fc_\theta$ is generated by elements invariant under the action of the group of self-twists.

\begin{prop}
Let $\fc_{\theta,0}=\fc_{\theta,\Tr}\cap\I_0^\circ$. Then $\fc_{\theta,\Tr}=\fc_{\theta,0}\cdot\I_\Tr^\circ$.
\end{prop}

\begin{proof}
Let $\sigma$ be a self-twist and let $\eta_\sigma\colon G_\Q\to(\I_\Tr^\circ)^\times$ be the associated finite order character. Let $\fc_{\theta,\Tr}^\sigma=\sigma(\fc_\Tr^\theta)$. Since $\sigma$ is an automorphism of $\I_\Tr^\circ$, it induces an isomorphism $\I_\Tr^\circ/\fc_{\theta,\Tr}\cong\I^\circ_\Tr/\fc^\sigma_{\theta,\Tr}$. In particular we can consider the two representations $\rho_{\fc_{\theta,\Tr},1}\colon G_\Q\to\GSp_4(\I^\circ_\Tr/\fc^\sigma_{\theta,\Tr})$ and $\rho_{\fc_{\theta,\Tr},1}^\sigma=\sigma\ccirc\rho_{\fc_{\theta,\Tr},1}\colon G_\Q\to\GSp_4(\I^\circ_\Tr/\fc^\sigma_{\theta,\Tr})$. By Corollary \ref{fctrchar} applied to the ideal $\fI=\fc_{\theta,\Tr}$ there exists a representation $\rho_{\fc_{\theta,\Tr},1}\colon G_\Q\to\GL_2(\I^\circ/\fc_{\theta,\Tr})$ such that $\rho_{\fc_\theta,\Tr}\cong\Sym^3\rho_{\fc_{\theta,\Tr},1}$. We apply $\sigma$ to both sides of this equivalence and we obtain $\rho_{\fc_{\theta,\Tr}}^\sigma\cong\Sym^3\rho_{\fc_{\theta,\Tr},1}^\sigma$. 
By definition of self-twist $\rho^\sigma\cong\eta_\sigma\otimes\rho$. By reducing modulo $\fc_{\theta,\Tr}^\sigma$ we obtain, with the obvious notations, $(\rho^\sigma)_{\fc_{\theta,\Tr}^\sigma}\cong\eta_{\sigma,\fc_{\theta,\Tr}^\sigma}\otimes\rho_{\fc_{\theta,\Tr}^\sigma}$. 
Now $(\rho^\sigma)_{\fc_{\theta,\Tr}^\sigma}=(\rho_{\fc_{\theta,\Tr}})^\sigma$, so by combining the two previous equations we deduce that $(\rho^\sigma)_{\fc_{\theta,\Tr}^\sigma}\cong\eta_{\sigma,\fc_{\theta,\Tr}^\sigma}\otimes\Sym^3\rho_{\fc_{\theta,\Tr},1}^\sigma$. 
Since $\eta_{\sigma,\fc_{\theta,\Tr}^\sigma}$ is a finite order character, there exists an extension $\I_1$ of $\I^\circ_\Tr/\fc^\sigma_{\theta,\Tr}$ of degree at most $3$ and a character $\eta_{\sigma,\fc_{\theta,\Tr}^\sigma,1}$ satisfying $(\eta_{\sigma,\fc_{\theta,\Tr}^\sigma,1})^3=\eta_{\sigma,\fc_{\theta,\Tr}^\sigma}$. Then
\[ (\rho^\sigma)_{\fc_{\theta,\Tr}^\sigma}\cong\Sym^3(\eta_{\sigma,\fc_{\theta,\Tr}^\sigma,1}\otimes\rho_{\fc_{\theta,\Tr},1}^\sigma), \]
so the implication (ii) $\implies$ (i) of Corollary \ref{fctrchar} gives $\fc_{\theta,\Tr}^\sigma\supset\fc_{\theta,\Tr}$. This holds for every $\sigma\in\Gamma$, hence $\bigcap_{\sigma\in\Gamma}\fc_{\theta,\Tr}^\sigma\supset\fc_{\theta,\Tr}$. This is an equality because the inclusion in the other direction is trivial. We conclude that $\fc_{\theta,\Tr}$ is $\Gamma$-stable, so the ideal $\fc_{\theta,\Tr}\cap\I_0^\circ$ of $\I_0^\circ$ satisfies $(\fc_{\theta,\Tr}\cap\I_0^\circ)\cdot\I_\Tr^\circ=\fc_{\theta,\Tr}$. 
\end{proof}

\begin{defin}\label{I0congr}
We call $\fc_{\theta,0}$ the \emph{fortuitous $(\Sym^3,\I^\circ_0)$-congruence ideal} for the family $\theta\colon\T_h\to\I^\circ$.
\end{defin}
%

For an ideal $\fI$ of $\I_0^\circ$ we denote by $\rho_{\fI}\colon H_0\to\GSp_4(\I^\circ_0/\fI)$ the reduction of $\rho\vert_{H_0}$ modulo $\fI$. 
The ideal $\fc_{\theta,0}$ admits a characterization similar to that of $\fc_\theta$ and $\fc_{\theta,\Tr}$. 

\begin{prop}\label{cong0equiv}
Let $P_0$ be a prime ideal of $\I^\circ_0$. The following are equivalent.
\begin{enumerate}[label=(\roman*)]
\item $P_0\supset\fc_{\theta,0}$;
\item there exists a finite extension $\I^\prime$ of $\I^\circ_\Tr/P_0\I^\circ_\Tr$ and a representation $\rho_{P_0\I^\circ_\Tr,1}\colon G_\Q\to\GL_2(\I^\prime)$ such that $\rho_{P_0\I_\Tr^\circ}\cong\Sym^3\rho_{\I^\prime}$ over $\I^\prime$;
\item for one prime $P$ of $\I^\circ_\Tr$ lying above $P_0$ there exists a finite extension $I^\prime$ of $\I^\circ_\Tr/P$ and a representation $\rho_{P,1}\colon G_\Q\to\GL_2(\I^\prime)$ such that $\rho_P\cong\Sym^3\rho_{P,1}$ over $\I^\prime$; 
\item there exists a representation $\rho_{P_0,1}\colon H_0\to\GL_2(\I^\circ_0/\fI)$ such that $\rho_{P_0}\cong\Sym^3\rho_{P_0,1}$ over $\I^\circ_0/\fI$.
\end{enumerate}
\end{prop}

\begin{proof}
We prove the chain of implications (i) $\implies$ (ii) $\implies$ (iii) $\implies$ (iv). If $P_0\supset\fc_{\theta,0}$ then $P_0\cdot\I_\Tr^\circ\supset\fc_{\theta,0}\cdot\I^\circ_\Tr=\fc_{\theta,\Tr}$. Now (ii) follows from Corollary \ref{fctrchar}. 

If (ii) holds for some $\I^\prime$ and $\rho_{P_0\I^\circ_\Tr,1}$ and if $P$ is a prime of $\I^\circ_\Tr$ lying above $P_0$ then $P\supset P_0\I^\circ_\Tr$, so it makes sense to reduce $\rho_{P_0\I_\Tr^\circ,1}$ modulo $P\I^\prime$. The resulting representation $\rho_{P,1}\colon G_\Q\to\GL_2(\I^\circ_\Tr/P)$ satisfies (iii).

If (iii) is satisfied by some $\rho_{P_0,1}$ then $\rho_{P_0,1}=\rho_{P,1}\vert_{H_0}$ satisfies (iv). 

We complete the proof by showing that (iv) $\implies$ (ii) and (iii) $\implies$ (i). 
If (iv) holds then the image of $\rho_{P_0}$ is contained in $\Sym^3\GL_2(\I^\circ_0/\fI)$. 
Since $\rho_{P_0}=\rho_{P_0\I_\Tr^\circ}\vert_{H_0}$ Lemma \ref{subsym3} implies that, after extending the coefficients to a finite extension $\I_0^\prime$ of $\I^\circ_\Tr/P_0\I^\circ_\Tr$ the image of $\rho_{P_0\I_\Tr^\circ}$ is contained in $\Sym^3\GL_2(\I_0^\prime)$. This proves (ii). 

Suppose that (iii) holds. By Corollary \ref{fctrchar} $P\supset\fc_{\theta,\Tr}$, so $P_0=P\cap\I_0^\circ\supset\fc_{\theta,0}$, which is (i).
\end{proof}


The following is a corollary of Proposition \ref{fcInonzero}.

\begin{cor}\label{fcnonzero}
The ideal $\fc_{\theta,0}$ is non-zero.
\end{cor}

\bigskip

\section{An automorphic description of the Galois level}\label{levcong}

In Section \ref{exlevel} we attached a Galois level to a family of finite slope $\GSp_4$-eigenforms. 
The goal of this section is to compare this Galois theoretic objects with the congruence ideal introduced in Section \ref{fortcong}, that is an object defined in terms of congruences of overconvergent automorphic forms.

We work in the setting of Theorem \ref{thexlevel}. In particular $h$ is a positive rational number, $\theta\colon\T_h\to\I^\circ$ is a family of $\GSp_4$-eigenforms of slope bounded by $h$ and $\rho\colon G_\Q\to\GSp_4(\I_\Tr^\circ)$ is the Galois representation associated with $\theta$. We make the same assumptions on $\theta$ and $\rho$ as in Theorem \ref{thexlevel}; in particular $\rho$ is $\Z_p$-regular and the residual representation $\ovl\rho$ is either full or of symmetric cube type. 
With the family $\theta$ we associate two ideals of $\I_0$:
\begin{itemize}[label={--}]
 \item the ideal $\fc_{\theta,0}\cdot\I_0$, where $\fc_{\theta,0}$ is the fortuitous $(\Sym^3,\I^\circ_0)$-congruence ideal (see Definition \ref{I0congr}); 
 \item the Galois level $\fl_\theta$ (see Definition \ref{deflevel}).
\end{itemize}
%

To simplify notations we write $\fc_{\theta,0}$ for $\fc_{\theta,0}\cdot\I_0$. 
For every ring $R$ and every ideal $\fI$ of $R$ we denote by $V_R(\fI)$ the set of primes of $R$ containing $\fI$. 
The theorem below is an analogue of \cite[Theorem 6.2]{cit}. The set $S^\bad$ of ``bad'' primes of $\I_0^\circ$ was defined in Section \ref{biglie}. Note that $V_{\I_0}(\fl_\theta)\cap S^\bad$ is empty because the property defining the Galois level only involves $\fl_\theta\cdot\B_r$, and the primes in $S^\bad$ are invertible in $\B_r$. 

\begin{thm}\label{comparison}
There is an equality of sets $V_{\I_0}(\fc_{\theta,0})-S^\bad=V_{\I_0}(\fl_\theta)$.
\end{thm}

Recall that there is a natural inclusion $\iota_r\colon\I_0\into\I_{r,0}$.

\begin{proof}
First we prove that $V_{\I_0}(\fc_{\theta,0})-S^\bad\subset V_{\I_0}(\fl_\theta)-S^\bad$. Choose a radius $r$ in the set $\{r_i\}_{i\in\N^{>0}}$ defined in Section \ref{gspfam}. Let $P\in V_{\I_0}(\fc_{\theta,0})-S^\bad$ and let $\rho_P$ be the reduction of $\rho\vert_{H_0}\colon H_0\to\GSp_4(\I_0)$ modulo $P$. By Proposition \ref{cong0equiv} there exists a representation $\rho_{P,1}\colon H_0\to\GL_2(\I_0/P)$ such that $\rho_{P}\cong\Sym^3\rho_{P,1}$. 
Let $\rho_{r,P}=\iota_r\ccirc\rho_P$ and $\rho_{r,P,1}=\iota_r\ccirc\rho_{P,1}$. The isomorphism above gives $\rho_{r,P}\cong\Sym^3\rho_{r,P,1}$. 

Suppose by contradiction that $\fl_\theta\not\subset P$. By definition of $\fl_\theta$ we have $\fG_r\supset\fl_\theta\cdot\fsp_4(\B_r)$. Recall that $\B_r/P=\I_{r,0}/P$ by the construction of $\B_r$. By looking at the previous inclusion modulo $P$ we obtain
\begin{equation}\label{congsubalg} \fG_{r,P}\supset(\fl_\theta/(P\cap\fl_\theta))\cdot\fsp_4(\I_{r,0}/P). \end{equation}
Since $\fl_\theta\not\subset P$ we have $\fl_\theta/(P\cap\fl_\theta)\neq 0$. By definition $\fG_{r,P}=\Q_p\cdot\log\im\rho_{r,P}$. By our previous argument $\im\rho_{r,P}\subset\Sym^3\GL_2(\I_{r,0}/P\I_{r,0})$, so $\log\im\rho_{r,P}$ cannot contain a subalgebra of the form $\fI\cdot\fsp_4(\I_{r,0}/P\I_{r,0})$ for a non-zero ideal $\fI$ of $\I_{r,0}/P\I_{r,0}$. This contradicts Equation \eqref{congsubalg}. 

We prove the inclusion $V_{\I_0}(\fl_\theta)-S^\bad\subset V_{\I_0}(\fc_{\theta,0})-S^\bad$. Let $P$ be a prime of $\I_0$. We have to show that if $P\notin S^\bad$ and $\fl_\theta\subset P$ then $\fc_{\theta,0}\subset P$. Every prime of $\I_0$ is the intersection of the maximal ideals that contain it, so it is sufficient to show the previous implication when $P$ is a maximal ideal. 

Let $P$ be a maximal ideal of $\I_0$ such that $P\notin S^\bad$ and $\fl_\theta\subset P$. 
Let $\kappa_P$ be the residue field $\I_0/P$. We define two ideals of $\I_{r,0}$ by $\fl_{\theta,r}=\iota_r(\fl_\theta)\I_{r,0}$ and $P_r=\iota_r(P)\I_{r,0}$. Note that $\iota_r$ induces an isomorphism $\I_0/P\cong\I_{r,0}/P_r$. In particular $P_r$ is maximal in $\I_{r,0}$ and $\I_{r,0}/P_r\cong\kappa_P$, which is a local field.


As before let $\rho_{r,P}=\iota_r\ccirc\rho_P$. 
The residual representation $\ovl{\rho}_{r,P}\colon H_0\to\GSp_4(\I_{r,0}^\circ/\fm_{\I_{r,0}^\circ})$ associated with $\rho_{r,P}$ coincides with $\ovl{\rho}\vert_{H_0}$. In particular $\rho_{r,P}$ is of residual $\Sym^3$ type in the sense of Definition \ref{sctype}. Let $G_{r,P}=\im\rho_{r,P}$ and $G_{r,P}^\circ$ be the connected component of the identity in $G_{r,P}$. Let $\ovl{G_{r,P}^\circ}^\Zar$ be the Zariski closure of $G_{r,P}^\circ$ in $\GSp_4(\I_{r,0}/P_r)$. Since $\rho_{r,P}$ is residually either full or of symmetric cube type, by the classification preceding Lemma \ref{sym3zar} one of the following must hold:
\begin{enumerate}[label=(\roman*)]
\item the algebraic group $\ovl{G_{r,P}^\circ}^\Zar$ is isomorphic to $\Sym^3\SL_2$ over $\I_{r,0}/P_r$;
\item the algebraic group $\ovl{G_{r,P}^\circ}^\Zar$ is isomorphic to $\Sp_4$ over $\I_{r,0}/P_r$. 
\end{enumerate}

In the two cases let $H^0$ denote the normal open subgroup of $H_0$ satisfying $\im\rho_{r,P}\vert_{H^0}=G_{r,P}^\circ$. Since $H_0$ is open and normal in $G_\Q$, $H^0$ is also open and normal in $G_\Q$. In case (i) there exists a representation $\rho_{r,P}^0\colon H^0\to\GL_2(\I_{r,0}/P_r)$ such that $\rho_{r,P}\vert_{H^0}\cong\Sym^3\rho_{r,P}^0$. Since the image of $\rho_{r,P}\vert_{H^0}$ is Zariski-dense in the copy of $\SL_2(\I_{r,0}/P_r)$ embedded via the symmetric cube map, the image of $\rho_{r,P}^0$ is Zariski-dense in $\SL_2(\I_{r,0}/P_r)$. From Lemma \ref{subsym3} we deduce that $\im\rho_{r,P}^0$ contains a congruence subgroup of $\SL_2(\I_{r,0}/P_r)$. Now the hypotheses of Lemma \ref{subsym3} are satisfied by the representation $\rho_{r,P}^0$ and the group $H^0$, so we conclude that there exists a representation $\rho_{H_0,r,P}^\prime\colon H_0\to\GL_2(\I_{r,0}/P_r)$ such that $\rho_{H_0,r,P}\cong\Sym^3\rho_{H_0,r,P}^\prime$. 
By Proposition \ref{cong0equiv} the prime $P$ must contain $\fc_{\theta,0}$, as desired. 

We show that case (ii) never occurs. Suppose by contradiction that $\ovl{G_{H_0,r,P}^\circ}^\Zar\cong\Sp_4$ over $\I_{r,0}/P_r$. 
By Propositions \ref{sttraces} and \ref{stringseq} we know that the field $\I_{0}/P$ is generated over $\Q_p$ by the traces of $\Ad(\rho_{P}\vert_{H_0})$. Hence the field $\I_{r,0}/P_r$ is generated over $\Q_p$ by the traces of $\Ad\rho_{r,P}$. 
By Theorem \ref{pink} applied to $\im\rho_{r,P}$ there exists a non-zero ideal $\fl_{r,P}$ of $\I_{r,0}/P_r$ such that $G_{r,P}$ contains the principal congruence subgroup $\Gamma_{\I_{r,0}/P_r}(\fl_{r,P})$ of $\Sp_4(\I_{r,0}/P_r)$. 
By definition $\fG_{r,P}=\Q_p\cdot\log(\im\rho_{r,P}\vert_{H_{r}})$ where $H_r$ is an open $G_\Q$, so up to replacing $\fl_{r,P}$ by a smaller non-zero ideal we have 
\begin{equation}\label{logincl} \fl_{r,P}\cdot\fsp_4(\I_{r,0}/P_r)\subset\log(\Gamma_{\I_{r,0}/P_r}(\fl_{r,P}))\subset\log(\iota_{r,0}(G_{P}))\subset\fG_{r,P}. \end{equation}
The algebras $\fG_{r,P}$ are independent of $r$ in the sense of Remark \ref{lieind}, so there exists an ideal $\fl_{P}$ of $\I_0/P$ such that, for every $r$ in the set $\{r_i\}_{i\geq 1}$, the ideal $\fl_{r,P}=\iota_r(\fl_{P})$ satisfies Equation \eqref{logincl}. We choose the ideals $\fl_{r,P}$ of this form.

As before $\Delta$ is the set of roots of $\GSp_4$ with respect to the chosen maximal torus. 
Let $\alpha\in\Delta$. Let $\fU_r^\alpha$ and $\fU^\alpha_{r,P_r}$ be the nilpotent Lie subalgebras respectively of $\fG_r$ and $\fG_{r,P_r}$ corresponding to $\alpha$. We denote by $\pi_{P_r}$ the projection $\fgsp_4(\B_r)\to\fgsp_4(\B_r/P_r\B_r)$. Clearly $\fG_{r,P_r}=\pi_{P_r}(\fG_r)$, so $\fU^\alpha_{r,P_r}=\pi_{P_r}(\fU_r^\alpha)$. Equation \eqref{logincl} gives $\fl_{r,P}\fu^\alpha(\I_{r,0}/P_r)\subset\fU^\alpha_{r,P_r}$. Choose a subset $A^\alpha_{P}$ of $\fu^\alpha(\I_0)$ such that, for every $r$, $\iota_r(A^\alpha_{P})\subset\fU_r^\alpha$ and $\pi_{P_r}(\iota_r(A^\alpha_{P}))=\fl_{r,P}\fu^\alpha(\I_{r,0}/P_r)$. Such a set exists because the algebras $\fU_r^\alpha$ are independent of $r$ by Remark \ref{lieind} and the ideals $\fl_{r,P}$ have been chosen of the form $\iota_r(\fl_{P})$. 
Set $\fA_{P}=\left(\prod_{\alpha\in\Delta}\fA^\alpha_{P}\right)^4$. 
By the same argument as in the proof of Theorem \ref{thexlevel}, the ideal $\fA^\alpha_{P}$ satisfies 
\[ \iota_r(\fA^\alpha_{P})\cdot\fsp_4(\B_r)\subset\fG_r. \]
Since $\fl_\theta\cdot\fsp_4(\B_r)\subset\fG_r$ for every $r$, we also have $(\fl_\theta+\fA^\alpha_{H_0,P})\fsp_4(\B_r)\subset\fG_r$ for every $r$. 

By assumption $\fl_\theta\subset P$, so $\pi_P(\fl_\theta)=0$. By definition of $\fA^\alpha_{H_0,P}$ we have $\pi_P(\fA^\alpha_{P})\supset\pi_P(A_{P})=\fl_{P}$, so $\pi_P(\fl_\theta+\fA_{P})=\fl_{P}$. 
We deduce that $\fl_\theta+\fA^\alpha_{P}$ is strictly larger than $\fl_\theta$. This contradicts the fact that $\fl_\theta$ is the largest among the ideals $\fl$ of $\I_0$ satisfying $\fl\cdot\fsp_4(\B_r)\subset\fG_r$. 
\end{proof}

\begin{cor}\label{fullcong}
When the residual representation $\ovl\rho$ is full, the Galois level $\fl_\theta$ is trivial.
\end{cor}

\begin{proof}
This follows immediately from Theorem \ref{comparison} and Corollary \ref{trivcong}.
\end{proof}

\bigskip

\bigskip

\bigskip

\vspace{1cm}

{\scshape Andrea Conti
%
%
%
%
%

\medskip

Department of Mathematics and Statistics, Concordia University

McConnell Library Building, 1400 De Maisonneuve Blvd. W.

Montreal, Quebec, Canada, H3G 1M8}

\medskip

\textit{E-mail address}: \url{contiand@gmail.com}


\begin{thebibliography}{9999991}

\bigskip

\bibitem[AIP15]{aip} F. Andreatta, A. Iovita, V. Pilloni, \emph{$p$-Adic families of Siegel modular cuspforms}, Ann. of Math. 181 (2015), pp. 623-697.

\bibitem[An87]{andrquad} A. N. Andrianov, \emph{Quadratic Forms and Hecke Operators}, Grundlehren Math. Wiss. 286, Springer (1987).

\bibitem[An09]{andrtwist} A. N. Andrianov, \emph{Twisting of Siegel modular forms with characters, and $L$-functions}, St. Petersburg Math. J. 20(6) (2009), pp. 851-871.

\bibitem[Be12]{bellaiche} J. Bella\"\i che, \emph{Eigenvarieties and adjoint $p$-adic $L$-functions}, preprint, 2012, available at \url{http://people.brandeis.edu/~jbellaic/preprint/coursebook.pdf}.

\bibitem[BC09]{bellchen} J. Bella\"\i che, G. Chenevier, \emph{Families of Galois representations and Selmer groups}, Astérisque 324, Soc. Math. France, Paris, 2009.

\bibitem[Be02]{berger} L. Berger, \emph{Représentations $p$-adiques et équations différentielles}, Invent. Math. 148(2) (2002), pp. 219-284.

\bibitem[Be11]{bergertri} L. Berger, \emph{Trianguline representations}, Bull. Lond. Math. Soc. 43(4) (2011), pp. 619-635.

\bibitem[BC10]{bergchen} L. Berger, G. Chenevier, \emph{Représentations potentiellement triangulines de dimension $2$}, J. Théor. Nombres Bordeaux 22(3) (2010), pp. 557-574.

\bibitem[Bo]{bourbaki} N. Bourbaki, \emph{Elements of Mathematics. Algebra I. Chapters 1-3}, 1st ed., Springer (1989).

\bibitem[BGR84]{bgr} S. Bosch, U. G\"untzer, R. Remmert, \emph{Non-Archimedean Analysis}, Grundleheren Math. Wiss. 261, Springer (1984).

\bibitem[BPS16]{bijpilstr} S. Bijakowski, V. Pilloni, B. Stroh, \emph{Classicité de formes modulaires surconvergentes}, to appear in Annals of Math.

\bibitem[BR16]{braros} R. Brasca, G. Rosso, \emph{Eigenvarieties for non-cuspidal modular forms over certain PEL type Shimura varieties}, preprint, 2016.

\bibitem[Bu07]{buzzard} K. Buzzard, \emph{Eigenvarieties}, in \emph {$L$-functions and Galois representations}, Proc. Conf. Durham 2004, LMS Lect. Notes Series 320, Cambridge University Press 2007, pp. 59-120.
%
%

\bibitem[Ca94]{cartrace} H. Carayol, \emph{Formes modulaires et représentations galoisiennes à valeurs dans un anneau local complet}, Contemp. Math. 165 (1994) pp. 213-237.
%
%

\bibitem[Ch04]{chenfam} G. Chenevier, \emph{Familles $p$-adiques de formes automorphes pour $\GL(n)$}, J. Reine Angew. Math. 570 (2004), pp. 143-217.

\bibitem[Ch05]{chenevier} G. Chenevier, \emph{Une correspondance de Jacquet-Langlands $p$-adique}, Duke Math. J. 126(1) (2005), pp. 161-194.

\bibitem[Ch14]{chendet} G. Chenevier, \emph{The $p$-adic analytic space of pseudocharacters of a profinite group, and pseudorepresentations over arbitrary rings}, in \emph{Automorphic forms and Galois representations}, Vol. 1, Proceedings of the LMS Durham Symposium 2011, London Math. Soc. Lecture Note Ser. 414, Cambridge Univ. Press, Cambridge, 2014, pp. 221-285. 
%

\bibitem[CM98]{colmaz} R. Coleman and B. Mazur, \emph{The eigencurve}, in \emph{Galois Representations in Arithmetic Algebraic Geometry}, London Math. Soc. Lecture Note Ser. 254, Cambridge Univ. Press, Cambridge, 1998, pp. 1-113.

\bibitem[Col08]{colmez} P. Colmez, \emph{Représentations triangulines de dimension $2$}, Astérisque 319, Soc. Math. France, Paris, 2008, pp. 213-258.

\bibitem[Con99]{conradirr} B. Conrad, \emph{Irreducible components of rigid analytic spaces}, Ann. Inst. Fourier 49 (1999), pp. 905-919.

\bibitem[Co16]{contith} A. Conti, \emph{Grande image de Galois pour familles $p$-adiques de formes automorphes de pente positive}, Ph.D. thesis, Université Paris 13, 2016.

\bibitem[CIT15]{cit} A. Conti, A. Iovita, J. Tilouine, \emph{Big image of Galois representations associated with finite slope $p$-adic families of modular forms}, preprint.
%

\bibitem[dJ95]{dejong} A. J. de Jong, \emph{Crystalline Dieudonn\'e theory via formal and rigid geometry}, Publ. Math. Inst. Hautes \'Etudes Sci. 82(1) (1995), pp. 5-96.

%
\bibitem[DiM13]{dimat} G. Di Matteo, \emph{On triangulable tensor products of $B$-pairs and trianguline representations}, preprint, 2013

\bibitem[DG12]{ghadim} E. Ghate, M. Dimitrov, \emph{On classical weight one forms in Hida families}, J. Théorie Nombres Bordeaux 24(3) (2012), pp. 669-690.

\bibitem[Em14]{emerton} M. Emerton, \emph{Local-global compatibility in the $p$-adic Langlands programme for $\GL_2/\Q$}, preprint (2014).

%


\bibitem[Fi02]{fischman} A. Fischman, \emph{On the image of $\Lambda$-adic Galois representations}, Annales de l'Institut Fourier 52(2) (2002), pp. 351-378.
%

%
%
%
%
\bibitem[GT05]{gentil} A. Genestier, J. Tilouine, \emph{Systèmes de Taylor-Wiles pour $\GSp_4$}, in \emph{Formes automorphes II. Le cas du groupe $\GSp(4)$}, Astérisque 302, Soc. Math. France, Paris, 2005, pp. 177-290. 

\bibitem[Go1889]{goursat} \'E. Goursat, \emph{Sur les substitutions orthogonales et les divisions régulières de l'espace}, Ann. Sci. \'Ec. Norm. Supér. 6 (1889), pp. 9-102.
%
%

\bibitem[Hi15]{hida} H. Hida, \emph{Big Galois representations and $p$-adic $L$-functions}, Compos. Math. 151 (2015), pp. 603-654.

\bibitem[HT15]{hidatil} H. Hida, J. Tilouine, \emph{Big image of Galois representations and congruence ideals}, in \emph{Arithmetic Geometry}, pp. 217-254, Proc. Workshop on Serre's Conjecture, Hausdorff Inst. Math., Bonn, eds. L. Dieulefait, D.R. Heath-Brown, G. Faltings, Y.I. Manin, B. Z. Moroz, J.-P. Wintenberger, Cambridge University Press (2015).
%

\bibitem[KPX]{kedpotxia} K. S. Kedlaya, J. Pottharst, L. Xiao, \emph{Cohomology of arithmetic families of $(\varphi,\Gamma)$-modules}, J. Amer. Math. Soc. 27 (2014), pp. 1043-1115. 

\bibitem[KS02]{kimsha} H. H. Kim and F. Shahidi, \emph{Functorial products for $\GL_2\times\GL_3$ and the symmetric cube for $\GL_2$}, Annals of Math.
Second Series, Vol. 155, No. 3, pp. 837-893 (2002).

\bibitem[Ki03]{kisin} M. Kisin, \emph{Overconvergent modular forms and the Fontaine-Mazur conjecture}, Invent. Math. 153 (2003), pp. 363-454.

%

\bibitem[La16]{lang} J. Lang, \emph{On the image of the Galois representation associated to a non-CM Hida family}, Algebra Number Theory 10(1) (2016), pp. 155-194.
%
%

\bibitem[Liv89]{livne} R. Livn\'e, \emph{On the conductors of modulo $\ell$ representations coming from modular forms}, J. Number Theory 31 (1989), pp. 133-141.

\bibitem[Lu14]{ludwigunit} J. Ludwig, \emph{$p$-adic functoriality for inner forms of unitary groups in three variables}, Math. Res. Letters 21(1) (2014), pp. 141-148.

\bibitem[Lu14]{ludwigll} J. Ludwig, \emph{A $p$-adic Labesse-Langlands transfer}, preprint.

\bibitem[Ma89]{mazdef} B. Mazur, \emph{Deforming Galois representations}, in \emph{Galois groups over $\Q$}, M.S.R.I. Publications, Springer-Verlag, Berlin 1989, pp. 385-437.
%

\bibitem[Mo81]{momose} F. Momose, \emph{On the $\ell$-adic representations attached to modular forms}, J. Fac. Sci. Univ. Tokyo Sect. IA Math. 28(1) (1981), pp. 89-109.

\bibitem[Ny96]{nyssen} L. Nyssen, \emph{Pseudo-représentations}, Math. Ann. 306 (1996), pp. 257-283.

\bibitem[OM78]{omeara} O. T. O'Meara, \emph{Symplectic groups}, Mathematical Surveys, No. 16, Amer. Math. Soc., Providence, R. I., 1978.

\bibitem[Pil12]{pillab} V. Pilloni, \emph{Modularité, formes de Siegel et surfaces abéliennes}, J. Reine Angew. Math. 666 (2012), pp. 35-82.
%
%

\bibitem[Pink98]{pink} R. Pink, \emph{Compact subgroups of linear algebraic groups}, J. Algebra 206 (1998), pp. 438-504.

\bibitem[RS07]{ramsha} D. Ramakrishnan, F. Shahidi, \emph{Siegel modular forms of genus $2$ attached to elliptic curves}, Math. Res. Lett. 14, no. 2, pp. 315-332 (2007).

\bibitem[Ri75]{ribetI} K. Ribet, \emph{On $\ell$-adic representations attached to modular forms}, Invent. Math. 28 (1975), pp. 245-276.
%

\bibitem[Ri85]{ribetII} K. Ribet, \emph{On $\ell$-adic representations attached to modular forms. II}, Glasgow Math. J. 27 (1985), pp. 185-194.

\bibitem[Ro96]{rouquier} R. Rouquier, \emph{Caractérisation des caractères et pseudo-caractères}, J. Algebra 180 (1996), pp. 571-586.

\bibitem[Sen73]{senlie} S. Sen, \emph{Lie algebras of Galois groups arising from Hodge-Tate modules}, Ann. of Math., Vol. 97, No. 1 (1973), pp. 160-170.

\bibitem[Sen80]{sencont} S. Sen, \emph{Continuous cohomology and $p$-adic Hodge theory}, Invent. Math. 62(1) (1980), pp. 89-116.

\bibitem[Sen93]{seninf} S. Sen, \emph{An infinite dimensional Hodge-Tate theory}, Bull. Soc. Math. France 121 (1993) pp. 13-34.

\bibitem[Ser70]{serre} J.-P. Serre, \emph{Facteurs locaux des fonctions zêta des variétés algébriques (définitions et conjectures)}, in Sém. DPP 1969/1970, exp. 19 (Oe. 87).
%
%

\bibitem[Ta91]{taylorgal} R. Taylor, \emph{Galois representations associated to Siegel modular forms of low weight}, Duke Math. J. 63 (1991), pp. 281-332.
%
%

\bibitem[Taz85]{tazhet} S. Tazhetdinov, \emph{Subnormal structure of symplectic groups over local rings}, Mathematical notes of the Academy of Sciences of the USSR 37(2) (1985), pp. 164-169. 

\bibitem[Ti06]{tilsiegel} J. Tilouine, \emph{Nearly ordinary rank four Galois representations and $p$-adic Siegel modular forms}, with an appendice by D. Blasius, Comp. Math. 142 (2006), pp. 122-156. 
%
%
%
\bibitem[Ur05]{urban} E. Urban, \emph{Sur les représentations $p$-adiques associées aux représentations cuspidales de $\GSp_4(\Q)$}, in \emph{Formes automorphes II. Le cas du groupe $\GSp(4)$}, Astérisque 302, Soc. Math. France, Paris, 2005, pp. 151-176.

\bibitem[Wa98]{wan} D. Wan, \emph{Dimension variation of classical and $p$-adic modular forms}, Invent. Math. 133 (1998), pp. 449-463.
%
%

\bibitem[Wh12]{white} P.-J. White, \emph{$p$-adic Langlands functoriality for the definite unitary groups}, to appear in J. Reine Angew. Math. 
\end{thebibliography}
\end{document}